\documentclass[reqno,11pt]{book}
\usepackage{graphicx, amsmath, amsthm}
\usepackage{amssymb}  
\usepackage{tikz}
\usetikzlibrary{arrows,shapes}
\usetikzlibrary{calc}
\usepackage{chessboard}
\input{markov.sty}

\begin{document}

%%%%%%%%%%%%%%%%%%%%%%%%%%%%%%%%%%%%%%%%%%%%%%%%%%%%%%%%%%%%%%%%%%%%%%%%%%%

\title{{\Huge Processus al\'eatoires \\ 
et applications} \\
\bigskip
\bigskip
\bigskip
\bigskip
{\Large Master 2 Pro de Math\'ematiques}\\
\bigskip
\bigskip
{\Large Universit\'e d'Orl\'eans}}
\bigskip
\author{Nils Berglund}
\bigskip
\date{Version de Janvier 2014}   

\maketitle
\vfill
\newpage
\thispagestyle{empty}
\cleardoublepage
%\phantom{.}
\setcounter{page}{-1}
\thispagestyle{empty}
\tableofcontents
\newpage
\thispagestyle{empty}

%%%%%%%%%%%%%%%%%%%%%%%%%%%%%%%%%%%%%%%%%%%%%%%%%%%%%%%%%%%%%%%%%%%%%%%%%%%

\part{Cha\^\i nes de Markov}
\label{part_markov}

%%%%%%%%%%%%%%%%%%%%%%%%%%%%%%%%%%%%%%%%%%%%%%%%%%%%%%%%%%%%%%%%%%%%%%%%%%%

\chapter{Cha\^\i nes de Markov sur un ensemble fini}
\label{chap_fini}

%%%%%%%%%%%%%%%%%%%%%%%%%%%%%%%%%%%%%%%%%%%%%%%%%%%%%%%%%%%%%%%%%%%%%%%%%%%

\section{Exemples de \chaine s de Markov}
\label{sec_f_ex}

Les \chaine s de Markov sont intuitivement tr\`es simples \`a d\'efinir. 
Un syst\`eme peut admettre un certain nombre d'\'etats diff\'erents.
L'\'etat change au cours du temps discret. A chaque changement, le nouvel
\'etat est choisi avec une distribution de probabilit\'e fix\'ee au
pr\'ealable, et ne d\'ependant que de l'\'etat pr\'esent. 

\begin{example}[La souris dans le labyrinthe]
\label{ex_souris}
Une souris se d\'eplace dans le labyrinthe de la figure~\ref{fig_souris}. 
Initialement, elle se trouve dans la case 1. A chaque minute, elle change 
de case en choisissant, de mani\`ere \'equiprobable, l'une des cases
adjacentes. D\`es qu'elle atteint soit la nourriture (case 4),  soit sa
tani\`ere (case 5), elle y reste. 

\begin{figure}
 \centerline{
 \includegraphics*[clip=true,height=40mm]{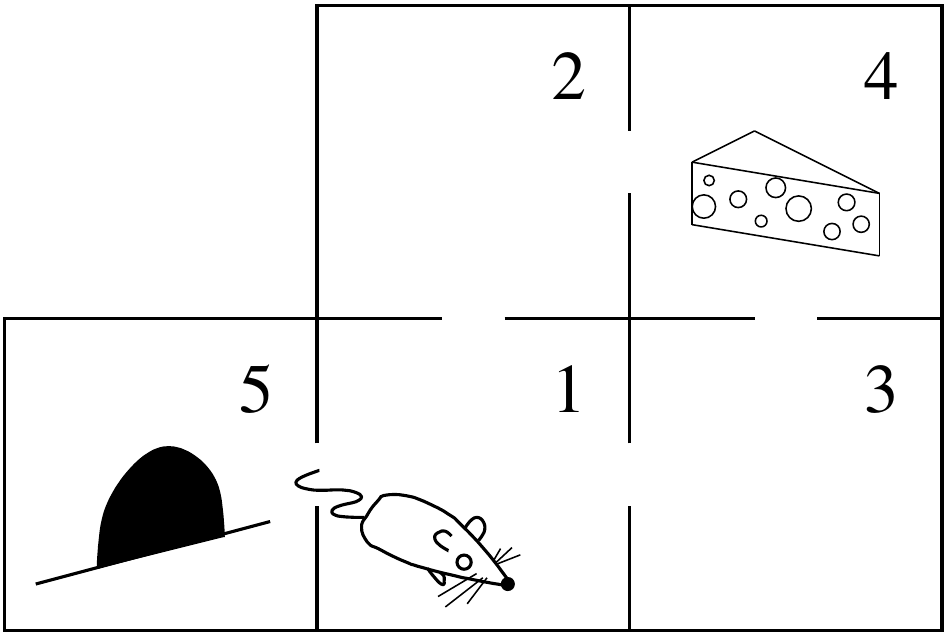}
 }
 \figtext{
 }
 \caption[]{Le labyrinthe dans lequel vit la souris.}
 \label{fig_souris}
\end{figure}

On se pose alors les questions suivantes :
\begin{enum}
\item	Avec quelle probabilit\'e la souris atteint-elle la nourriture
plut\^ot que sa tani\`ere? 
\item	Au bout de combien de temps atteint-elle sa tani\`ere ou la
nourriture? 
\end{enum}
On peut essayer de r\'epondre \`a ces questions en construisant un arbre 
d\'ecrivant les chemins possibles. Par exemple, il est clair que la souris
se retrouve dans sa tani\`ere au bout d'une minute avec probabilit\'e
$1/3$. Sinon, elle passe soit dans la case 2, soit dans la case 3, et
depuis chacune de ces cases elle a une chance sur deux de trouver la
nourriture. Il y a donc une probabilit\'e de $1/6$ que la souris trouve la
nourriture au bout de deux minutes. Dans les autres cas, elle se retrouve
dans la case de d\'epart, ce qui permet d'\'etablir une formule de
r\'ecurrence pour les probabilit\'es cherch\'ees. 

Cette mani\`ere de faire est toutefois assez compliqu\'ee, et devient
rapidement impossible \`a mettre en oeuvre quand la taille du labyrinthe
augmente. Dans la suite, nous allons d\'evelopper une m\'ethode plus
efficace pour r\'esoudre le probl\`eme, bas\'ee sur une
repr\'esentation matricielle.  
\end{example}

\begin{example}[Jeu de Pile ou Face]
\label{ex_PF}
Anatole et Barnab\'e jouent \`a la variante suivante de Pile ou Face. 
Ils jettent une pi\`ece de monnaie (parfaitement \'equilibr\'ee) de
mani\`ere r\'ep\'et\'ee. Anatole gagne d\`es que la pi\`ece tombe trois
fois de suite sur Face, alors que Barnab\'e gagne d\`es que la suite
Pile-Face-Pile appara\^\i t. 

\begin{figure}
%  \centerline{
%  \includegraphics*[clip=true,height=50mm]{figs/PFP}
%  }
%  \figtext{
%  \writefig       3.9     2.2     $1/2$
%  \writefig       5.5     2.5     $1/2$
%  \writefig       8.0     2.4     $1/2$
%  \writefig       3.2     4.9     $1/2$
%  \writefig       5.8     4.0     $1/2$
%  \writefig       9.2     4.0     $1/2$
%  \writefig       6.0     0.6     $1/2$
%  \writefig       9.2     0.6     $1/2$
%  }
\vspace{-10mm}
 \begin{center}
\begin{tikzpicture}[->,>=stealth',shorten >=2pt,shorten <=2pt,auto,node
distance=4.0cm, thick,main node/.style={circle,scale=0.7,minimum size=1.2cm,
fill=blue!20,draw,font=\sffamily\Large}
]

  \node[main node] (PP) {PP};
  \node[main node] (PF) [right of=PP] {PF};
  \node[main node]  (B) [right of=PF] {B};
  \node[main node] (FP) [below of=PP] {FP};
  \node[main node] (FF) [right of=FP] {FF};
  \node[main node]  (A) [right of=FF] {A};

  \path[every node/.style={font=\sffamily\small}]
    (PP) edge [loop above,left,distance=2cm,out=170,in=100] node
{$1/2$} (PP)
    (PP) edge [above] node {$1/2$} (PF)
    (PF) edge [above] node {$1/2$} (B)
    (FF) edge [above] node {$1/2$} (FP)
    (FF) edge [above] node {$1/2$} (A)
    (FP) edge [left] node {$1/2$} (PP)
    (FP) edge [above left] node {$1/2$} (PF)
    (PF) edge [left] node {$1/2$} (FF)
    ;
\end{tikzpicture}
\end{center}
 \caption[]{Graphe associ\'e au jeu de Pile ou Face. Chaque symbole de
deux lettres repr\'esente le r\'esultat des deux derniers jets de pi\`ece.
Anatole gagne si la pi\`ece tombe trois fois de suite sur Face, Barnab\'e
gagne si la pi\`ece tombe sur Pile-Face-Pile.}
 \label{fig_PFP}
\end{figure}
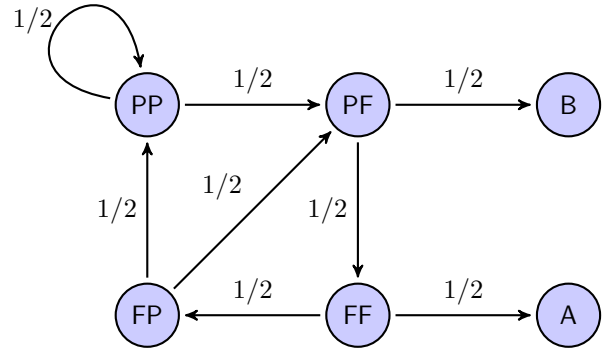

On se pose les questions suivantes :
\begin{enum}
\item 	Avec quelle probabilit\'e est-ce Anatole qui gagne le jeu? 
\item	Au bout de combien de jets de la pi\`ece l'un des deux joueurs
gagne-t-il? 
\end{enum}
La situation est en fait assez semblable \`a celle de l'exemple
pr\'ec\'edent. Un peu de r\'eflexion montre que si personne n'a gagn\'e au
bout de $n$ jets de la pi\`ece, la probabilit\'e que l'un des deux joueurs
gagne au coup suivant ne d\'epend que des deux derniers r\'esultats. On
peut alors d\'ecrire le jeu par une \chaine\ de Markov sur l'ensemble 
\begin{equation}
\cX = \set{\text{PP}, \text{PF} ,\text{FP} ,\text{FF} ,\text{A
gagne}, \text{B gagne}}\;, 
\end{equation}
o\`u par exemple PP signifie que la pi\`ece est
tomb\'ee sur Pile lors des deux derniers jets. On d\'etermine alors les
probabilit\'es de transition entre les cinq \'etats, et on retrouve un
probl\`eme semblable \`a celui de la souris. 
\end{example}

\begin{example}[Mod\`ele d'Ehrenfest]
\label{ex_Ehrenfest}
C'est un syst\`eme motiv\'e par la physique, qui a \'et\'e introduit pour
mod\'eliser de mani\`ere simple la r\'epartition d'un gaz entre deux
r\'ecipients.  
$N$ boules, num\'erot\'ees de $1$ \`a $N$, sont r\'eparties sur deux
urnes. De mani\`ere r\'ep\'et\'ee, on tire au hasard, de fa\c con
\'equiprobable, un num\'ero entre $1$ et $N$, et on change d'urne la boule
correspondante. 

On voudrait savoir comment ce syst\`eme se comporte asymptotiquement en
temps : 
\begin{enum}
\item	Est-ce que la loi du nombre de boules dans chaque urne approche
une loi limite?
\item	Quelle est cette loi? 
\item	Avec quelle fr\'equence toutes les
boules se trouvent-elles toutes dans la m\^eme urne? 
\end{enum}
On peut \`a nouveau d\'ecrire le syst\`eme par une \chaine\ de Markov,
cette fois sur l'espace des \'etats $\cX=\set{0,1,\dots,N}$, o\`u le
num\'ero de l'\'etat correspond au nombre de boules dans l'urne de gauche,
par exemple. 
\end{example}

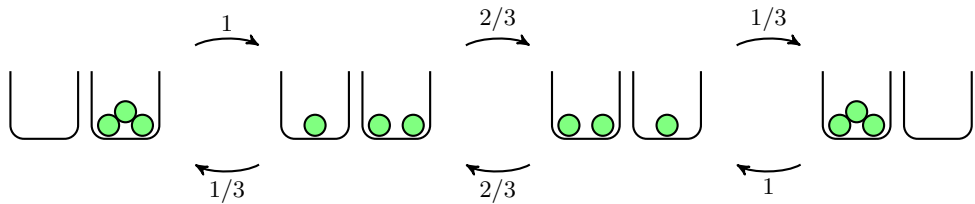
\begin{figure}
%  \centerline{
%  \includegraphics*[clip=true,width=140mm]{figs/ehrenfest}
%  }
%  \figtext{
%  \writefig       3.3     1.4     $1$
%  \writefig       7.0     1.4     $2/3$
%  \writefig      11.0     1.4     $1/3$
%  \writefig       3.2     0.3     $1/3$
%  \writefig       7.2     0.3     $2/3$
%  \writefig      11.3     0.3     $1$
%  }
\vspace{-3mm}
\begin{center}
\begin{tikzpicture}[->,>=stealth',auto,scale=0.9,node
distance=3.0cm, thick,main node/.style={circle,scale=0.7,minimum size=0.4cm,
fill=green!50,draw,font=\sffamily}]
  
  \pos{0}{0} \urntikz
  \pos{1.2}{0} \urntikz
  
  \node[main node] at(0.35,0.2) {};
  \node[main node] at(0.85,0.2) {};
  \node[main node] at(0.6,0.4) {};
  
  \pos{4}{0} \urntikz
  \pos{5.2}{0} \urntikz

  \node[main node] at(4.35,0.2) {};
  \node[main node] at(4.85,0.2) {};
  \node[main node] at(3.4,0.2) {};

  \pos{8}{0} \urntikz
  \pos{9.2}{0} \urntikz

  \node[main node] at(7.15,0.2) {};
  \node[main node] at(7.65,0.2) {};
  \node[main node] at(8.6,0.2) {};

  \pos{12}{0} \urntikz
  \pos{13.2}{0} \urntikz

  \node[main node] at(11.15,0.2) {};
  \node[main node] at(11.65,0.2) {};
  \node[main node] at(11.4,0.4) {};
  
  \node[minimum size=2.2cm] (0) at (0.1,0.5) {};
  \node[minimum size=2.2cm] (1) at (4.1,0.5) {};
  \node[minimum size=2.2cm] (2) at (8.1,0.5) {};
  \node[minimum size=2.2cm] (3) at (12.1,0.5) {};
  
  \path[shorten >=.3cm,shorten <=.3cm,every
        node/.style={font=\sffamily\footnotesize}]
    (0) edge [bend left,above] node {$1$} (1)
    (1) edge [bend left,above] node {$2/3$} (2)
    (2) edge [bend left,above] node {$1/3$} (3)
    (3) edge [bend left,below] node {$1$} (2)
    (2) edge [bend left,below] node {$2/3$} (1)
    (1) edge [bend left,below] node {$1/3$} (0)
    ;
\end{tikzpicture}
\end{center}
\vspace{-7mm}
 \caption[]{Le mod\`ele d'urnes d'Ehrenfest, dans le cas de $3$ boules.}
 \label{fig_ehrenfest}
\end{figure}

\begin{example}[Texte al\'eatoires]
\label{ex_texte}
Voici trois \lq\lq textes\rq\rq\ g\'en\'er\'es de mani\`ere al\'eatoire :

\begin{enumerate}
\item[A.] 
{\sf
YxUV,luUqHCLvE?,MRiKaoiWjyhg nEYKrMFD!rUFUy.qvW;e:FflN.udbBdo!, \\
ZpGwTEOFcA;;RrSMvPjA'Xtn.vP?JNZA;xWP, Cm?;i'MzLqVsAnlqHyk,ghDT  \\
:PwSwrnJojRhVjSe?dFkoVRN!MTfiFeemBXITdj m.h d'ea;Jkjx,XvHIBPfFT \\
s I'SLcSX;'X!S, ODjX.eMoLnQttneLnNE!qGRgCJ:BuYAauJXoOCCsQkLcyPO \\
MulKLRtSm;PNpFfp'PfgvIJNrUr t l aXtlA?;TPhPxU:,ZmVGr,,'DIjqZDBY \\
DrkPRiKDYRknDhivt;, LYXDuxNKpjegMvrtfz:JpNTDj'LFmHzXxotRM u.iya \\
UUrgZRcA QmCZffwsNWhddBUPAhJIFJvs.CkKFLJoXef;kCnXrv'uWNcpULYsnl \\
Kg OURmysAnxFjHawwsSpM H;PWPsMaFYLMFyvRWOjbdPlLQIaaspNZkuO'Ns.l \\
jEXO,lxQ'GS;n;H:DH:VWJN :t'JMTUVpKCkVZ'NyKJMGiIbQFXEgDEcWxMBiyo \\
ybRIWIAC deMJnnL;SBAZ?:.UuGnC:B.!lBUT,pT?tyHHLlCvN, mKZgwlMJOJd \\
HHobua;KU.;kADVM?jr'v.SCq:hZLR;lqkmLkhn:ajhBM,gKexDAro,HlczWTv \\
cFmNPt.MudUWPO, sTrWlJdgjoiJd.:d;CpJkJCW;FIRnpMGa;umFysOMAqQtmT \\
pPaYZKtOFYppeE.KFX?SuvcbaDrQ XECelD;cfoQKf?'jCTUaISS;fV:gqoWfSq \\
k:Tf!YuPBANtKhewiNg'ImOFs:UhcExmBjsAaMhBf UVP, 'dcFk;gxJMQGyXI; \\
nVwwfWxS:YXQMELEIObTJiilUYSlOsg.gCqlrN:nEU:irHM'nOLXWUbJLTU re' \\
kk vAwMgt'KgWSxwxqJe,z'OBCrnoIshSCDlZirla,rWNPkc?UgZm GOBX.QylY \\
jOtuF
}

\item[B.]	
{\sf
nsunragetnetelpnlac.  pieln tJmends d e.imnqu  caa  aneezsconns re.tc oml d
 e  c, paeisfuaul irt ssna l df.ieulat a ese t hre edn ro  m eeel
slsplotasstp etuoMeiiseeaenemzeaeuqpeer  enuoco  sfehnnir p ts 'mpisu qrd
iraLp nFetesa,opQeey rieeaduset MuuisecG il e m  ru daeiafasousfnircot i
eeedracev ever.nsn iaeulu!,mtel lpa rdbjdide  tolr'murunlr bteaaua
ieasilureseuavrmoce ntvqm qnurnaunsa.mraayVarinanr  eumsu cnponf ciuo
.pssre  elreeY snrrq aani psu oqoddaiaaomrssloe'avia,loei va
eroltrsurdeduuoe ffusir 'th'niIt has,slluoooe tee  ?eoxaea slsii i u
edtvsear e,Mesatnd o o rvdocaeagiua  apugiqn rclt  smtee.te, gceade etsn e
v in eag ent so  ra te,  oi seGndd  i eeet!dii e  ese nanu d sp ul afeen
aqelonens ssisaaoe cs     eectadegotuudlru  i  'c, uuuuts 'tt , dir
atermdmuciqedn  esovsioieieerxdroie mqso,es rrvteen,r dtei xcalrionuaae e
vtmplsz miuqa   u aboir br gmcdexptedn pEua't vm vnic eeren ereaa,eegeta u
rss nlmxomas ea nsbnt s,eEpeteae teiasbo cd ee tu em ue quee en, sd
eeneepeot 
}

\item[C.]
{\sf
cesalu'act, bouleuivoie melarous die ndant leuvoiblue poit pesois
deuntaciroverchu llie e lle s r lerchar, laisueuayaissabes vet s cuetr i
as, rdetite se d'iretie, de.. nendoules, le pablur e d ! copomouns ppait
limmix a r aux urars laie Le r lercret ce c. n'are four nsirepapole pa vr
s, nte le efit. itesit, le faun e ju estatusuet usoin prcilaisanonnout ssss
l tosesace cole sientt, dent pontrtires. e, l mentoufssss chat Laneus c
Chontrouc Ce e. Et deses j'ecci uleus mmon s mauit paga lanse l cont
ciquner e c Cha s l'a Jes des s'erattrlunt es de sacouen erends. ve e quns
som'a aisajouraite eux lala pour ! a levionible plaint n ss, danetrc ponce
con du lez, l danoit, dirvecs'u ce ga vesai : chleme eesanl Pa chiontotes
anent fomberie vaud'untitez e esonsan t a ! bondesal'is Ilaies, vapa e !
Lers jestsiee celesu unallas, t. ces. ta ce aielironi mmmileue cecoupe et
dennt vanen A la ajole quieet, scemmu tomtemotit me aisontouimmet Le s
Prage ges peavoneuse ! blec douffomurrd ntis.. rur, ns ablain i pouilait
lertoipr ape. leus icoitth me e e, poiroia s. ! atuepout somise e la as
}

\end{enumerate}

Il est clair qu'aucun de ces textes n'a de signification. Toutefois,
le texte B.\ semble moins arbitraire que le texte A., et C.\ para\^\i
t moins \'eloign\'e d'un texte fran\c cais que B. Il suffit pour cela
d'essayer de lire les textes \`a haute voix. 

Voici comment ces textes ont \'et\'e g\'en\'er\'es. Dans les trois cas, on
utilise le m\^eme alphabet de 60 lettres (les 26 minuscules et majuscules, 
quelques signes de ponctuation et l'espace). 
\begin{enum}
\item	Pour le premier texte, on a simplement tir\'e au hasard, de
mani\`ere ind\'ependante et avec la loi uniforme, des lettres de
l'alphabet. 

\item	Pour le second texte, on a tir\'e les lettres de mani\`ere
ind\'ependante, mais pas avec la loi uniforme. Les probabilit\'es des
diff\'erentes lettres correspondent aux fr\'equences de ces lettres dans
un texte de r\'ef\'erence fran\c cais (en l’occurrence, un extrait du {\sl
Colonel Chabert}\/ de Balzac). Les fr\'equences des diff\'erentes
lettres du texte al\'eatoire sont donc plus naturelles, par exemple la
lettre {\sf e} appara\^\i t plus fr\'equemment (dans $13\%$ des cas) que la
lettre {\sf z} ($0.2\%$). 

\item	Pour le dernier texte, enfin, les lettres n'ont pas \'et\'e
tir\'ees de mani\`ere ind\'ependante, mais d\'ependant de la lettre
pr\'ec\'edente. Dans le m\^eme texte de r\'ef\'erence que
pr\'e\-c\'edemment, on a d\'etermin\'e avec quelle fr\'equence la lettre
{\sf a} est suivie de {\sf a} (jamais), {\sf b} (dans $3\%$ des cas), et
ainsi de suite, et de m\^eme pour toutes les autres lettres. Ces
fr\'equences ont ensuite \'et\'e
choisies comme probabilit\'es de transition lors de la g\'en\'eration du
texte. 
\end{enum}

Ce proc\'ed\'e peut facilement \^etre am\'elior\'e, par exemple en faisant
d\'ependre chaque nouvelle lettre de plusieurs lettres pr\'ec\'edentes. 
Mais m\^eme avec une seule lettre pr\'ec\'edente, il est remarquable que 
les textes engendr\'es permettent assez facilement de reconna\^\i tre la
langue du texte de r\'ef\'erence, comme en t\'emoignent ces deux exemples:

\begin{enumerate}
\item[D.]
{\sf 
deser Eld s at heve tee opears s cof shan; os wikey coure tstheevons irads;
Uneer I tomul moove t nendoot Heilotetateloreagis his ud ang l ars thine
br, we tinond end cksile: hersest tear, Sove Whey tht in t ce tloour ld t
as my aruswend Ne t nere es alte s ubrk, t r s; penchike sowo
Spotoucthistey psushen, ron icoowe l Whese's oft Aneds t aneiksanging t
ungl o whommade bome, ghe; s, ne. torththilinen's, peny. d llloine's anets
but whsto a It hoo tspinds l nafr Aneve powit tof f I afatichif m as tres,
ime h but a wrove Les des wined orr; t he ff teas be hende pith hty ll ven
bube. g Bube d hitorend tr, Mand nd nklichis okers r whindandy, Sovede brk
f Wheye o edsucoure, thatovigh ld Annaix; an eer, andst Sowery looublyereis
isthalle Base whon ey h herotan wict of les, h tou dends m'dys h Wh
on'swerossictendoro whaloclocotolfrrovatel aled ouph rtrsspok,
ear'sustithimiovelime From alshis ffad, Spake's wen ee: hoves aloorth
erthis n t Spagovekl stat hetubr tes, Thuthiss oud s hind t s potrearall's
ts dofe 
}\footnote{Texte de r\'ef\'erence: Quelques sonnets de Shakespeare.}

\item[E.]
{\sf
dendewoch wich iere Daf' lacht zuerckrech, st, Gebr d, Bes.
jenditerullacht, keie Un! etot' in To sendenus scht, ubteinraben Qun Jue
die m arun dilesch d e Denuherelererufein ien. seurdan s ire Zein. es min?
dest, in. maur as s san Gedein it Ziend en desckruschn kt vontimelan. in,
No Wimmmschrstich vom delst, esichm ispr jencht sch Nende Buchichtannnlin
Sphrr s Klldiche dichwieichst. ser Bollesilenztoprs uferm e mierchlls aner,
d Spph! wuck e ing Erenich n sach Men. Sin s Gllaser zege schteun d,
Gehrstren ite Spe Kun h Umischr Ihngertt, ms ie. es, bs de! ieichtt f;
Ginns Ihe d aftalt veine im t'seir; He Zicknerssolanust, fllll. mmichnennd
wigeirdie h Zierewithennd, wast naun Wag, autonbe Wehn eietichank We
dessonindeuchein ltichlich bsch n, Ichritienstam Lich uchodigem Din eieiers
die it f tlo nensseicichenko Mechtarzaunuchrtzubuch aldert; l von. fteschan
nn ih geier Schich Geitelten Deichst Fager Zule fer in vischtrn; Schtih Un
Hit ach, dit? at ichuch Eihra! Hich g ure vollle Est unvochtelirn An 
}\footnote{Texte de r\'ef\'erence: Un extrait du {\sl Faust}\/ de Goethe.}
\end{enumerate}

Cela donne, inversement, une m\'ethode assez \'economique permettant \`a
une machine de d\'eterminer automatiquement dans quelle langue un texte est
\'ecrit.  

\end{example}

\begin{example}[Le mod\`ele d'Ising]
\label{ex_Ising}
Comme le mod\`ele d'Ehrenfest, ce mod\`ele vient de la physique, plus
particuli\`erement de la physique statistique. Il est sens\'e d\'ecrire un
ferroaimant, qui a la propri\'et\'e de s'aimanter spontan\'ement \`a
temp\'erature suffisamment basse. On consid\`ere une partie (connexe)
$\Lambda$ du r\'eseau $\Z^d$ ($d$ \'etant la dimension du syst\`eme, par
exemple $3$), contenant $N$ sites. A chaque site, on attache un \lq\lq
spin\rq\rq\ (une sorte d'aimant \'el\'ementaire), prenant valeurs $+1$ ou
$-1$. Un choix d'orientations de tous les spins s'appelle une
configuration, c'est donc un \'el\'ement de l'espace de configuration
$\cX=\set{-1,1}^\Lambda$. A une configuration $\sigma$, on associe
l'\'energie 
\begin{equation}
\label{intro1}
H(\sigma) = -\sum_{<i,j>\in\Lambda} \sigma_i\sigma_j 
- h \sum_{i\in\Lambda}\sigma_i\;.
\end{equation}
Ici, la notation $<i,j>$ indique que l'on ne somme que sur les paires de
spins plus proches voisins du r\'eseau, c'est--\`a--dire \`a une distance
$1$. Le premier terme est donc d'autant plus grand qu'il y a de spins
voisins diff\'erents. Le second terme d\'ecrit l'interaction avec un champ
magn\'etique ext\'erieur $h$. Il est d'autant plus grand qu'il y a de
spins oppos\'es au champ magn\'etique. 

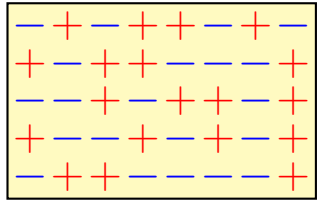
\begin{figure}
%  \centerline{
%  \includegraphics*[clip=true,height=30mm]{figs/Ising}
%  }
%  \figtext{
%  }
 \begin{center}
\begin{tikzpicture}[thick,auto,node distance=0.5cm,every
node/.style={font=\sffamily\LARGE}]

  \draw [fill=yellow!30] (-0.3,-0.3) rectangle (3.8,2.3);

  \node[blue] (00) {$-$};
  \node[red]  (10) [right of=00] {$+$};
  \node[red]  (20) [right of=10] {$+$};
  \node[blue] (30) [right of=20] {$-$};
  \node[blue] (40) [right of=30] {$-$};
  \node[blue] (50) [right of=40] {$-$};
  \node[blue] (60) [right of=50] {$-$};
  \node[red]  (70) [right of=60] {$+$};

  \node[red]  (01) [above of=00] {$+$};
  \node[blue] (11) [right of=01] {$-$};
  \node[blue] (21) [right of=11] {$-$};
  \node[red]  (31) [right of=21] {$+$};
  \node[blue] (41) [right of=31] {$-$};
  \node[red]  (51) [right of=41] {$+$};
  \node[blue] (61) [right of=51] {$-$};
  \node[red]  (71) [right of=61] {$+$};

  \node[blue] (02) [above of=01] {$-$};
  \node[blue] (12) [right of=02] {$-$};
  \node[red]  (22) [right of=12] {$+$};
  \node[blue] (32) [right of=22] {$-$};
  \node[red]  (42) [right of=32] {$+$};
  \node[red]  (52) [right of=42] {$+$};
  \node[blue] (62) [right of=52] {$-$};
  \node[red]  (72) [right of=62] {$+$};

  \node[red]  (03) [above of=02] {$+$};
  \node[blue] (13) [right of=03] {$-$};
  \node[red]  (23) [right of=13] {$+$};
  \node[red]  (33) [right of=23] {$+$};
  \node[blue] (43) [right of=33] {$-$};
  \node[blue] (53) [right of=43] {$-$};
  \node[blue] (63) [right of=53] {$-$};
  \node[red]  (73) [right of=63] {$+$};

  \node[blue] (04) [above of=03] {$-$};
  \node[red]  (14) [right of=04] {$+$};
  \node[blue] (24) [right of=14] {$-$};
  \node[red]  (34) [right of=24] {$+$};
  \node[red]  (44) [right of=34] {$+$};
  \node[blue] (54) [right of=44] {$-$};
  \node[red]  (64) [right of=54] {$+$};
  \node[blue] (74) [right of=64] {$-$};

  \end{tikzpicture}
\end{center}
\vspace{-5mm}
 \caption[]{Une configuration du mod\`ele d'Ising en dimension $d=2$.}
 \label{fig_ising}
\end{figure}

Un principe de base de la physique statistique est que si un syst\`eme est
en \'equilibre thermique \`a temp\'erature $T$, alors il se trouve dans la
configuration $\sigma$ avec probabilit\'e proportionnelle \`a $\e^{-\beta
H(\sigma)}$ (mesure de Gibbs), o\`u $\beta=1/T$. A temp\'erature faible, le
syst\`eme privil\'egie les configurations de basse \'energie, alors que
lorsque la temp\'erature tend vers l'infini, toutes les configurations
deviennent \'equiprobables. 

L'aimantation totale de l'\'echantillon est donn\'ee par la variable
al\'eatoire 
\begin{equation}
\label{intro2}
m(\sigma) = \sum_{i\in\Lambda} \sigma_i\;,
\end{equation}
et son esp\'erance vaut 
\begin{equation}
\label{intro3}
\expec m = 
\frac{\sum_{\sigma\in\cX} m(\sigma)
\e^{-\beta H(\sigma)}}
{\sum_{\sigma\in\cX}\e^{-\beta
H(\sigma)}}\;.
\end{equation}
L'int\'er\^et du mod\`ele d'Ising est qu'on peut montrer l'existence d'une
transition de phase, en dimension $d$ sup\'erieure ou \'egale \`a $2$.
Dans ce cas il existe une temp\'erature critique en-dessous de laquelle
l'aimantation varie de mani\`ere discontinue en fonction de $h$ dans la
limite $N\to\infty$. Pour des temp\'eratures sup\'erieures \`a la valeur
critique, l'aimantation est continue en $h$. 

Si l'on veut d\'eterminer num\'eriquement l'aimantation, il suffit en
principe de calculer la somme~\eqref{intro3}. Toutefois, cette somme
comprend $2^N$ termes, ce qui cro\^it tr\`es rapidement avec la taille du
syst\`eme. Par exemple pour un cube de $10\times10\times10$ spins, le
nombre de termes vaut $2^{1000}$, ce qui est de l'ordre de
$10^{300}$. Un ordinateur calculant $10^{10}$ termes par seconde
mettrait beaucoup plus que l'\^age de l'univers \`a calculer la somme. 

Une alternative est d'utiliser un algorithme dit de Metropolis. Au lieu
de parcourir toutes les configurations possibles de $\cX$, on n'en
parcourt qu'un nombre limit\'e, de mani\`ere bien choisie, \`a l'aide
d'une \chaine\ de Markov. Pour cela, on part dans une configuration
initiale $\sigma$, puis on transforme cette configuration en retournant un
spin choisi au hasard. Plus pr\'ecis\'ement, on n'op\`ere cette
transition qu'avec une certaine probabilit\'e, qui d\'epend de la
diff\'erence d'\'energie entre les configurations de d\'epart et
d'arriv\'ee. L'id\'ee est que si les probabilit\'es de transition sont
bien choisies, alors la \chaine\ de Markov va \'echantillonner l'espace de
configuration de telle mani\`ere qu'il suffira de lui faire parcourir une
petite fraction de toutes les configurations possibles pour obtenir une
bonne approximation de l'aimantation $\expec{m}$. Les questions sont alors 
\begin{enum}
\item	De quelle mani\`ere choisir ces probabilit\'es de transition?
\item	Combien de pas faut-il effectuer pour approcher $\expec{m}$ avec
une pr\'ecision donn\'ee?
\end{enum}
\end{example}

\begin{example}[Le probl\`eme du voyageur de commerce]
\label{ex_voyageur}
C'est un exemple classique de probl\`eme d'optimisation. Un voyageur de
commerce doit visiter $N$ villes, en revenant \`a son point de d\'epart
apr\`es \^etre pass\'e exactement une fois par chaque ville. Comment
choisir l'ordre des villes de mani\`ere \`a minimiser la longueur du
circuit? 

La difficult\'e est que le nombre de circuits possibles cro\^\i t
extr\^emement vite avec le nombre $N$ de villes, beaucoup plus vite
qu'exponentiellement. En effet, il y a $N!$ permutations possibles de
l'ordre des villes. Si l'on ne tient compte ni de la ville de d\'epart, ni
du sens de parcours, il reste $(N-1)!/2$ circuits possibles. Calculer les
longueurs de tous ces circuits devient irr\'ealisable d\`es que $N$
d\'epasse $20$ environ. 

\begin{figure}
 \centerline{
 \includegraphics*[clip=true,height=80mm]{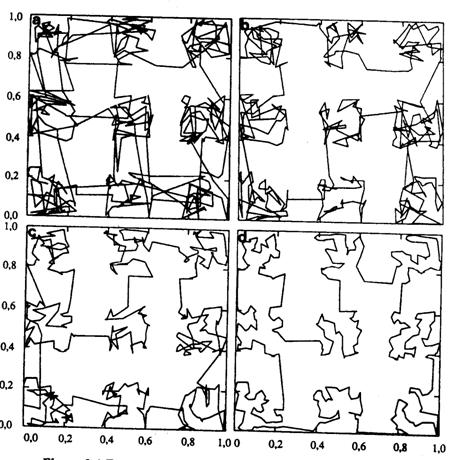}
 }
 \figtext{
 }
 \caption[]{Approximations successives de la solution du probl\`eme du
 voyageur de commerce par la m\'ethode du recuit simul\'e (tir\'e de
 l'article original : S. Kirkpatrick, C. Gelatt et M. Vecchi,
 {\it Optimization by Simulated Annealing}, Science, 220 (1983), pp.
 671--680, copyright 1983 by A.A.A.S.)).}
 \label{fig_voyageur}
\end{figure}

On peut tenter de trouver une solution approch\'ee par approximations
successives. Partant d'un circuit initial, on le modifie l\'eg\`erement,
par exemple en \'echangeant deux villes. Si cette modification raccourcit
la longueur du circuit, on continue avec le circuit modifi\'e. Si elle le
rallonge, par contre, on rejette la modification et on en essaie une
autre. 

Le probl\`eme avec cette m\'ethode est que le syst\`eme peut se retrouver
pi\'eg\'e dans un minimum local, qui est tr\`es diff\'erent du minimum
global recherch\'e de la longueur. On peut en effet se retrouver \lq\lq
bloqu\'e\rq\rq\ dans un circuit plus court que tous ses voisins (obtenus
en permutant deux villes), mais une permutation de plus de deux villes
pourrait raccourcir le circuit. 

Une variante plus efficace de cette m\'ethode est celle du \defwd{recuit
simul\'e}\/. Dans ce cas, on ne rejette pas toutes les modifications qui
allongent le circuit, mais on les accepte avec une certaine probabilit\'e,
qui d\'ecro\^\i t avec l'allongement. De cette mani\`ere, le processus peut
s'\'echapper du minimum local et a une chance de trouver un minimum plus
profond. La terminologie vient de la m\'etallurgie : Dans un alliage, les
atomes des diff\'erents m\'etaux sont dispos\'es de mani\`ere plus ou
moins r\'eguli\`ere, mais avec des imperfections. Moins il y a
d'imperfections, plus l'alliage est solide. En r\'echauffant et
refroidissant plusieurs fois l'alliage, on donne aux atomes la
possibilit\'e de se r\'earranger de mani\`ere plus r\'eguli\`ere,
c'est-\`a-dire en diminuant l'\'energie potentielle. 

A nouveau, on se pose les questions suivantes :
\begin{enum}
\item 	Comment choisir les probabilit\'es d'acceptation des
modifications? 
\item	Comment la probabilit\'e de s'approcher \`a une certaine distance
du minimum cherch\'e d\'epend-elle de la longueur de la simulation?
\end{enum}
\end{example}

%%%%%%%%%%%%%%%%%%%%%%%%%%%%%%%%%%%%%%%%%%%%%%%%%%%%%%%%%%%%%%%%%%%%%%%%%%%

\section{D\'efinitions}
\label{sec_f_def}

\begin{definition}
\label{def_markov1}
Soit $N$ un entier strictement positif. Une matrice $P$ de taille $N\times
N$ est une \defwd{matrice stochastique} si ses \'el\'ements de matrice
$p_{ij}=(P)_{ij}$ satisfont 
\begin{equation}
\label{fdef1}
0 \leqs p_{ij} \leqs 1 \qquad \forall i,j\; 
\end{equation}
et 
\begin{equation}
\label{fdef2}
\sum_{j=1}^N p_{ij} = 1 \qquad \forall i\;.
\end{equation}
\end{definition}

On v\'erifie facilement que si $P$ et $Q$ sont deux matrices stochastiques,
alors le produit $PQ$ est encore une matrice stochastique. En particulier, 
toutes les puissances $P^n$ de $P$ sont encore des matrices
stochastiques. 
Les \'el\'ements $p_{ij}$ vont d\'efinir les probabilit\'es de transition
de la \chaine\ de Markov de l'\'etat $i$ vers l'\'etat $j$. 

\begin{definition}
\label{def_markov2}
Soit $\cX=\set{1,\dots,N}$ un ensemble fini et $P$ une matrice
stochastique de taille $N$. Une \defwd{\chaine\ de Markov} sur $\cX$ de
matrice de transition $P$ est une suite $(X_0,X_1,X_2,\dots)$ de variables
al\'eatoires \`a valeurs dans $\cX$, satisfaisant la \defwd{propri\'et\'e
de Markov}
\begin{align}
\nonumber
\bigpcond{X_n=j}{X_{n-1}=i_{n-1},X_{n-2}=i_{n-2},\dots,X_0=i_0} 
&= \bigpcond{X_n=j}{X_{n-1}=i_{n-1}} \\
&= p_{i_{n-1}j}
\label{fdef3}
\end{align}
pour tout temps $n\geqs 1$ et tout choix $(i_0,i_1,i_{n-1},j)$
d'\'el\'ements de $\cX$. La loi de $X_0$, que nous noterons $\nu$, est
appel\'ee la \defwd{distribution initiale}\/ de la \chaine.
\end{definition}

Pour s'assurer que cette d\'efinition fait bien sens, il faut
v\'erifier que les $X_n$ construits comme ci-dessus sont bien des
variables al\'eatoires, c'est-\`a-dire que la somme sur tous les $j\in\cX$
des probabilit\'es $\prob{X_n=j}$ vaut $1$. Ceci est imm\'ediat par
r\'ecurrence sur $n$. Si les $n$ variables $X_0$, \dots, $X_{n-1}$ sont
des variables al\'eatoires, alors on a :
\begin{align}
\nonumber
\sum_{j\in\cX} \prob{X_n=j} 
&= \sum_{j\in\cX} \prob{X_n=j,X_{n-1}\in\cX,\dots,X_0\in\cX} \\
\nonumber
&= \sum_{j\in\cX} \sum_{i_{n-1}\in\cX} \dots \sum_{i_0\in\cX} 
\prob{X_n=j,X_{n-1}=i_{n-1},\dots,X_0=i_0} \\ 
\nonumber
&= \sum_{i_{n-1}\in\cX} \underbrace{\sum_{j\in\cX} p_{i_{n-1}j}}_{=1}
\underbrace{\sum_{i_{n-2}\in\cX} \dots \sum_{i_0\in\cX} 
\prob{X_{n-1}=i_{n-1},\dots,X_0=i_0}}_{=\prob{X_{n-1}=i_{n-1},
X_{n-2}\in\cX, \dots , X_0\in\cX
}=\prob{X_{n-1}=i_{n-1}}} \\
&= \prob{X_{n-1}\in\cX} = 1\;.
\label{fdef4} 
\end{align}

\begin{notation}
\label{not_markov1}
Si la loi initiale $\nu$ est fix\'ee, nous noterons souvent $\fP_\nu$
la loi de la \chaine\ de Markov associ\'ee. Si $\nu$ est concentr\'ee en
un seul site $i$ ($\nu=\delta_i$), on notera la loi de la \chaine\ $\fP_i$
au lieu de $\fP_{\delta_i}$. Enfin nous \'ecrirons parfois $X_{[n,m]}$ au
lieu de $(X_n,X_{n+1},\dots,X_m)$, et $\prob{X_{[n,m]}=i_{[n,m]}}$ au lieu
de $\prob{X_n=i_n,X_{n+1}=i_{n+1},\dots,X_m=i_m}$. $X_{[n,m]}$ est appel\'e
la \emph{trajectoire}\/ de la \chaine\ entre les temps $n$ et $m$. 
\end{notation}

\begin{example}
Dans le cas de l'exemple~\ref{ex_souris} de la souris dans le labyrinthe, 
la matrice de transition est donn\'ee par 
\begin{equation}
\label{fdef5}
P = 
\begin{pmatrix}
0 & 1/3 & 1/3 & 0 & 1/3 \\
1/2 & 0 & 0 & 1/2 & 0 \\
1/2 & 0 & 0 & 1/2 & 0 \\
0 & 0 & 0 & 1 & 0 \\
0 & 0 & 0 & 0 & 1
\end{pmatrix}\;.
\end{equation} 
Dans le cas de l'exemple~\ref{ex_PF} du jeu de Pile ou Face, la matrice de
transition vaut
\begin{equation}
\label{fdef6}
P = 
\begin{pmatrix}
1/2 & 1/2 & 0 & 0 & 0 & 0 \\
0 & 0 & 0 & 1/2 & 0 & 1/2 \\
1/2 & 1/2 & 0 & 0 & 0 & 0 \\
0 & 0 & 1/2 & 0 & 1/2 & 0 \\
0 & 0 & 0 & 0 & 1 & 0 \\
0 & 0 & 0 & 0 & 0 & 1
\end{pmatrix}\;.
\end{equation}
\end{example}

Voici tout d'abord une caract\'erisation d'une \chaine\ de Markov en
termes de ses trajectoires. 

\begin{theorem}
\label{thm_fdef1}
Soit $\set{X_n}_{n\in\N}$ une suite de variables al\'eatoires \`a valeurs
dans $\cX$, $\nu$ une distribution de probabilit\'e sur $\cX$, et $P$ une
matrice stochastique. 
Alors $\set{X_n}_{n\in\N}$ est une \chaine\ de Markov de matrice de
transition $P$ et de distribution initiale $\nu$ si et seulement si pour
tout $n\geqs0$, et pour tout choix de $i_0,i_1,\dots,i_n$ d'\'el\'ements
de $\cX$, on a 
\begin{equation}
\label{fdef7}
\bigprob{X_{[0,n]}=i_{[0,n]}} = \nu_{i_0}p_{i_0i_1}p_{i_1i_2}\dots
p_{i_{n-1}i_n}\;.
\end{equation}
\end{theorem}
\begin{proof}\hfill
\begin{itemize}
\item[$\Rightarrow:$]
Par r\'ecurrence sur $n$. C'est clair pour $n=0$. Si c'est vrai pour $n$,
alors 
 \begin{align}
\nonumber
\bigprob{X_{[0,n+1]}=i_{[0,n+1]}} &= 
\bigpcond{X_{n+1}=i_{n+1}}{X_{[0,n]}=i_{[0,n]}}
\bigprob{X_{[0,n]}=i_{[0,n]}} \\ 
&= p_{i_ni_{n+1}}\nu_{i_0}p_{i_0i_1}p_{i_1i_2}\dots
p_{i_{n-1}i_n}\;.
\label{fdef8:1}
\end{align}

\item[$\Leftarrow:$] 
Par d\'efinition de la probabilit\'e conditionnelle, on a 
\begin{equation}
\label{fdef8:2}
\bigpcond{X_n=i_n}{X_{[0,n-1]}=i_{[0,n-1]}} =
\frac{\prob{X_{[0,n]}=i_{[0,n]}}}{\prob{X_{[0,n-1]}=i_{[0,n-1]}}}
= p_{i_{n-1}i_n}\;,
\end{equation}
la derni\`ere \'egalit\'e suivant de~\eqref{fdef7}.
\qed
\end{itemize}
\renewcommand{\qed}{}
\end{proof}

L'\'equation~\eqref{fdef7} donne la probabilit\'e de la trajectoire
$X_{[0,n]}$. Le r\'esultat suivant montre que la propri\'et\'e de Markov
reste vraie pour des trajectoires : l'\'evolution sur un intervalle de
temps $[n+1,m]$ ne d\'epend que de l'\'etat au temps $n$, et pas de la
trajectoire pass\'ee de la \chaine.

\begin{prop}
\label{prop_fdef1}
Si $\set{X_n}_{n\in\N}$ est une \chaine\ de Markov sur $\cX$, alors pour
tous temps $n < m\in\N$, tous $i_n\in\cX$, $A\subset\cX^n$ et
$B\subset\cX^{m-n}$ tels que $\prob{X_n=i_n,X_{[0,n-1]}\in A}>0$ on~a 
\begin{equation}
\label{fdef9}
\bigpcond{X_{[n+1,m]}\in B}{X_n=i_n, X_{[0,n-1]}\in A} 
= \bigpcond{X_{[n+1,m]}\in B}{X_n=i_n}\;.
\end{equation}
\end{prop}
\begin{proof}
On a 
\begin{align}
\nonumber
\bigpcond{X_{[n+1,m]}\in B}{X_n=i_n, X_{[0,n-1]}\in A}
&= \frac{\bigprob{X_{[n+1,m]}\in B,X_n=i_n,X_{[0,n-1]}\in A}}
{\bigprob{X_n=i_n,X_{[0,n-1]}\in A}} \\
\nonumber
&= \frac{\displaystyle\sum_{i_{[n+1,m]}\in B} \sum_{i_{[0,n-1]}\in A}
\nu_{i_0}p_{i_0i_1}\dots p_{i_{m-1}i_m}}
{\displaystyle\sum_{i_{[0,n-1]}\in A}
\nu_{i_0}p_{i_0i_1}\dots p_{i_{n-1}i_n}} \\
&= \sum_{i_{[n+1,m]}\in B} p_{i_ni_{n+1}}\dots p_{i_{m-1}i_m}\;,
\label{fdef10:1}
\end{align}
qui ne d\'epend pas de l'ensemble $A$. En particulier, choisissant
$A=\cX^n$ dans l'\'egalit\'e ci-dessus, on trouve 
\begin{equation}
\label{fdef10:2}
\bigpcond{X_{[n+1,m]}\in B}{X_n=i_n}
= \sum_{i_{[n+1,m]}\in B} p_{i_ni_{n+1}}\dots p_{i_{m-1}i_m}\;,
\end{equation}
d'o\`u le r\'esultat, en comparant~\eqref{fdef10:1} et~\eqref{fdef10:2}. 
\end{proof}

Un cas particulier important est celui o\`u
$B=\cX^{m-n-1}\times\set{j}$, c'est-\`a-dire qu'on s'int\'eresse
\`a toutes les trajectoires se terminant en $i_m=j$ au temps $m$. 
Dans ce cas, la relation~\eqref{fdef10:2} donne 
\begin{equation}
\label{fdef11}
\bigpcond{X_m=j}{X_n=i}
= \sum_{i_{n+1}\in \cX} \dots \sum_{i_{m-1}\in
\cX}p_{ii_{n+1}}p_{i_{n+1}i_{n+2}}\dots p_{i_{m-1}j}\;. 
\end{equation}
Par d\'efinition du produit matriciel, la somme ci-dessus n'est autre que
l'\'el\'ement de matrice $(i,j)$ de la matrice $P^{m-n}$, que nous
noterons $p^{(m-n)}_{ij}$. On remarquera que le membre de droite
de~\eqref{fdef11} ne d\'epend que de la diff\'erence $m-n$. On a donc pour
tout $0<n<m$
\begin{equation}
\label{fdef12}
\bigpcond{X_m=j}{X_n=i}
= p^{(m-n)}_{ij}
= \bigpcond{X_{m-n}=j}{X_0=i}
\end{equation}
(\defwd{propri\'et\'e des incr\'ements stationnaires}\/). 
Enfin, pour tout $n>0$, 
\begin{equation}
\label{fdef13}
\bigprobin{\nu}{X_m=j} = \sum_{i\in\cX}\bigprobin{\nu}{X_m=j,X_0=i}
%= \sum_{i\in\cX}\bigprob{X_0=i}\bigpcond{X_n=j}{X_0=i}
= \sum_{i\in\cX}\nu_ip^{(m)}_{ij}\;. 
\end{equation}
La matrice $P^m$ donne donc les probabilit\'es de transition en $m$ pas. 

\begin{example}
\label{ex_fdef2}
Pour la matrice de transition~\eqref{fdef5} de la souris dans le
labyrinthe, on trouve 
\begin{equation}
\label{fdef14}
P^2 = 
\begin{pmatrix}
1/3 & 0 & 0 & 1/3 & 1/3 \\
0 & 1/6 & 1/6 & 1/2 & 1/6 \\
0 & 1/6 & 1/6 & 1/2 & 1/6 \\
0 & 0 & 0 & 1 & 0 \\
0 & 0 & 0 & 0 & 1
\end{pmatrix}\;.
\end{equation} 
La $i^\text{\`eme}$ ligne de la matrice correspond aux probabilit\'es
d'\^etre dans les diff\'erents \'etats si la souris est partie de la case
$i$. Ainsi, si elle est partie de la case $1$, on voit qu'elle se retrouve
au temps $2$ soit en $1$, soit dans sa tani\`ere, soit aupr\`es de la
nourriture, \` a chaque fois avec m\^eme probabilit\'e $1/3$. Si elle est
partie de l'une des cases $2$ ou $3$, elle a une chance sur $2$ d'avoir
trouv\'e la nourriture au temps $2$, et une chance sur $6$ de se retrouver
dans l'une des cases $2$, $3$ ou $5$. 
\end{example}

%%%%%%%%%%%%%%%%%%%%%%%%%%%%%%%%%%%%%%%%%%%%%%%%%%%%%%%%%%%%%%%%%%%%%%%%%%%

\section{\Chaine s de Markov absorbantes}
\label{sec_f_abs}

\begin{definition}
\label{def_fabs1}
On dit qu'un \'etat $j\in\cX$ est \defwd{accessible}\/ depuis un autre
\'etat $i\in\cX$, et on note $i\reaches j$, s'il existe un temps $n\in\N$
tel que $p^{(n)}_{ij}>0$, c'est-\`a-dire que partant de $i$, on atteint
$j$ avec probabilit\'e positive en un nombre fini de pas. On notera $i\sim
j$ si on a \`a la fois $i\reaches j$ et $j\reaches i$. 
\end{definition}

On v\'erifie facilement que la relation $\reaches$ est r\'eflexive et
transitive, et que $\sim$ est une relation d'\'equivalence. 

\begin{definition}
\label{def_fabs2}
Un \'etat $i\in\cX$ est dit \defwd{absorbant}\/ si $p_{ii}=1$ (et donc
n\'ecessairement $p_{ij}=0$ pour tout $j\neq i$). 
Une \chaine\ de Markov est dite \defwd{absorbante}\/ s'il existe, pour
tout \'etat de $\cX$, un \'etat absorbant accessible depuis cet \'etat. 
\end{definition}

Dans le reste de cette section, nous allons consid\'erer des \chaine s
absorbantes avec $r\geqs1$ \'etats absorbants. Les exemples~\ref{ex_souris}
de la souris et~\ref{ex_PF} du jeu de Pile ou Face sont des exemples de
\chaine s absorbantes avec deux \'etats absorbants. 

Nous conviendrons de num\'eroter les \'etats de mani\`ere \`a placer
d'abord les $q=N-r$ \'etats non absorbants, et ensuite les $r$ \'etats
absorbants. La matrice de transition prend alors la \defwd{forme
canonique}\/ 
\begin{equation}
\label{fabs1}
P = 
\begin{pmatrix}
Q & R \\ 0 & I 
\end{pmatrix}\;,
\end{equation}
o\`u $Q$ est une matrice de taille $q\times q$, $R$ est une matrice de
taille $q\times r$, $0$ d\'esigne la matrice nulle de taille $r\times q$,
et $I$ la matrice identit\'e de taille $r$. 
Il est facile de montrer par r\'ecurrence que 
\begin{equation}
\label{fabs2}
P^n = 
\begin{pmatrix}
Q^n & \brak{I+Q+\dots+Q^{n-1}}R \\ 0 & I 
\end{pmatrix}\;. 
\end{equation}

\begin{prop}
\label{prop_fabs1}
Soit $P$ la matrice de transition d'une \chaine\ de Markov absorbante,
\'ecrite sous forme canonique. Alors
\begin{enum}
\item	On a 
\begin{equation}
\label{fabs3}
\lim_{n\to\infty} Q^n=0\;.
\end{equation}
\item	La matrice $I-Q$ est inversible, et son inverse vaut 
\begin{equation}
\label{fabs4}
\bigbrak{I-Q}^{-1} = \sum_{k=0}^\infty Q^k\;.
\end{equation}
\end{enum}
\end{prop}
\begin{proof}\hfill
\begin{enum}
\item
Soit $i\leqs q$ un \'etat non absorbant.  L'\'el\'ement de matrice
$(Q^n)_{ij}$ de $Q^n$ est la probabilit\'e de se trouver dans l'\'etat non
absorbant $j$, apr\`es $n$ pas, partant de $i$. Par cons\'equent,
$(Q^n)_{ij}$ est inf\'erieur ou \'egal \`a la probabilit\'e de ne pas avoir
atteint d'\'etat absorbant en $n$ pas. Soit 
\begin{equation}
\label{fabs4:1}
m_i = \min\setsuch{n\geqs1}{\exists  k>q,\,(P^n)_{ik}>0}
\end{equation}
le nombre minimal de pas n\'ecessaire \`a atteindre un \'etat absorbant
$k$ depuis $i$. Soit 
\begin{equation}
\label{fabs4:2}
p_i = \probin{i}{X_{m_i}\leqs q} < 1
\end{equation}
la probabilit\'e de ne pas atteindre d'\'etat absorbant en $m_i$ pas,
partant de $i$. Soit enfin 
\begin{equation}
\label{fabs4:3}
M = \max_{i=1,\dots,q} m_i\;,
\qquad
p = \max_{i=1,\dots,q} p_i\;. 
\end{equation}
Alors la probabilit\'e de ne pas atteindre d'\'etat absorbant en $M$ pas,
partant de n'impor\-te quel \'etat non absorbant, est born\'ee par $p$. Il
suit que la probabilit\'e de ne pas atteindre d'\'etat absorbant en $Mn$
pas est born\'ee par $p^n$. Cette probabilit\'e tend vers $0$ lorsque $n$
tend vers l'infini. La probabilit\'e de ne pas \^etre absorb\'e apr\`es un
nombre arbitraire $m$ de pas \'etant une fonction d\'ecroissante de
$m$, elle tend n\'ecessairement vers $0$. Par cons\'equent, $(Q^n)_{ij}$
tend vers z\'ero lorsque $n$ tend vers l'infini, pour tout
$j\in\set{1,\dots,q}$.  

\item	Supposons qu'il existe un vecteur $x$ tel que $Qx=x$. Dans ce cas
on a 
\begin{equation}
\label{fabs4:4}
x = Qx = Q^2x = \dots=Q^nx = \dots=\lim_{n\to\infty} Q^nx = 0\;,
\end{equation}
ce qui montre que $Q$ n'admet pas la valeur propre $1$. Par cons\'equent,
$I-Q$ est inversible. Enfin, comme
 \begin{equation}
\label{fabs4:5}
\bigbrak{I-Q}\sum_{k=0}^n Q^k = I - Q^{n+1} \to I 
\qquad\text{lorsque $n\to\infty$}\;,
\end{equation}
on obtient la relation~\eqref{fabs4} en multipliant \`a gauche par
$\brak{I-Q}^{-1}$. 
\qed
\end{enum}
\renewcommand{\qed}{}
\end{proof}

Nous noterons $F$ la matrice $\brak{I-Q}^{-1}$, et nous l'appellerons la
\defwd{matrice fondamentale}\/ de la \chaine. La relation~\eqref{fabs2}
montre que 
\begin{equation}
\label{fabs5}
\lim_{n\to\infty} P^n = 
\begin{pmatrix}
0 & FR \\ 0 & I 
\end{pmatrix}\;. 
\end{equation}
Le fait que $Q^n$ tend vers z\'ero traduit donc le fait que la
probabilit\'e d'absorption tend vers $1$ lorsque le temps tend vers
l'infini. La matrice $B=FR$ devrait repr\'esenter les probabilit\'es de
transition, dans la limite des temps infinis, entre \'etats non absorbants
et absorbants. Ceci est confirm\'e par le r\'esultat suivant.

\goodbreak 

\begin{theorem}
\label{thm_fabs1}
Soit $F$ la matrice fondamentale d'une \chaine\ de Markov absorbante. 
\begin{enum}
\item	L'\'el\'ement de matrice $f_{ij}$ de $F$ est l'esp\'erance du
nombre de passages en $j$ partant de~$i$~:
\begin{equation}
\label{fabs6}
f_{ij} = \Bigexpecin{i}{\sum_{n\geqs0}\indexfct{X_n=j}}\;.
\end{equation}

\item	Soit $\tau=\inf\setsuch{n\geqs1}{X_n>q}$ la variable al\'eatoire
donnant le temps jusqu'\`a absorption. Alors 
\begin{equation}
\label{fabs7}
\expecin{i}{\tau} = \sum_{j=1}^q f_{ij}\;.
\end{equation}

\item	Les \'el\'ements de matrice $b_{ik}$ de $B=FR$ donnent les
probabilit\'es d'\^etre absorb\'es dans les diff\'erents \'etats :
\begin{equation}
\label{fabs8}
b_{ik} = \bigprobin{i}{X_\tau=k}\;.
\end{equation}
\end{enum}
\end{theorem}
\begin{proof}
\hfill
\begin{enum}
\item	Soit la variable de Bernoulli $Y_{n,j}=\indexfct{X_n=j}$. On a
$\expecin{i}{Y_{n,j}}=\probin{i}{Y_{n,j}=1}=(Q^n)_{ij}$, et donc 
\begin{equation}
\label{fabs9:1}
\Bigexpecin{i}{\sum_{n\geqs0}Y_{n,j}} = \sum_{n\geqs0}(Q^n)_{ij}
= (F)_{ij} = f_{ij}\;.
\end{equation}

\item	Sommant la relation ci-dessus sur tous les \'etats non absorbants,
on a 
\begin{align}
\nonumber
\sum_{j=1}^q f_{ij} &= \Bigexpecin{i}{\sum_{n\geqs0}\sum_{j=1}^qY_{n,j}}
= \Bigexpecin{i}{\sum_{n\geqs0}\indexfct{X_n\leqs q}} \\
&= \sum_{n\geqs0} \probin{i}{\tau>n} 
= \sum_{n\geqs0} n\probin{i}{\tau=n} = \expecin{i}{\tau}\;.
\label{fabs9:2}
\end{align}

\item	En d\'ecomposant sur les valeurs possibles $n$ de $\tau-1$, 
puis sur les valeurs possibles $j$ de $X_n$, 
\begin{align}
\nonumber
\probin{i}{X_\tau=k} &= \sum_{n\geqs0}\probin{i}{X_n\leqs q, X_{n+1}=k} \\
\nonumber
&= \sum_{n\geqs0} \sum_{j=1}^q \probin{i}{X_n=j}\pcond{X_{n+1}=k}{X_n=j} \\
&= \sum_{n\geqs0} \sum_{j=1}^q (Q^n)_{ij} (R)_{jk} 
= \sum_{n\geqs0} (Q^nR)_{ik} = (FR)_{ik}\;.
\label{fabs9:3}
\end{align}
\qed
\end{enum}
\renewcommand{\qed}{}
\end{proof}

\begin{example}
Pour l'exemple~\ref{ex_PF} du jeu de Pile ou Face, les matrices $Q$ et $
R$ sont donn\'ees par 
\begin{equation}
\label{fabs10}
Q = 
\begin{pmatrix}
1/2 & 1/2 & 0 & 0 \\
0 & 0 & 0 & 1/2  \\
1/2 & 1/2 & 0 & 0 \\
0 & 0 & 1/2 & 0  
\end{pmatrix}\;,
\qquad
R = 
\begin{pmatrix}
0 & 0 \\
0 & 1/2 \\
0 & 0  \\
1/2 & 0 
\end{pmatrix}\;.
\end{equation}
On calcule alors la matrice fondamentale 
\begin{equation}
\label{fabs11}
F = \brak{I-Q}^{-1} = 
\begin{pmatrix}
7/3 & 4/3 & 1/3 & 2/3 \\
1/3 & 4/3 & 1/3 & 2/3  \\
4/3 & 4/3 & 4/3 & 2/3 \\
2/3 & 2/3 & 2/3 & 4/3  
\end{pmatrix}\;, 
\end{equation}
et la matrice donnant les probabilit\'es d'absorption 
\begin{equation}
\label{fabs12}
B = FR = 
\begin{pmatrix}
1/3 & 2/3 \\
1/3 & 2/3 \\
1/3 & 2/3 \\
2/3 & 1/3  
\end{pmatrix}\;. 
\end{equation}
Ainsi, partant de l'un des \'etats PP, PF ou FP, Anatole gagne avec
probabilit\'e $1/3$, et Barnab\'e gagne avec probabilit\'e $2/3$. Partant
de l'\'etat FF, c'est Barnab\'e qui gagne avec probabilit\'e $1/3$, et 
Anatole qui gagne avec probabilit\'e $2/3$. Comme personne ne gagne lors
des deux premiers jets, et que les quatre \'etats PP, PF, FP et FF
sont atteints avec la m\^eme probabilit\'e, il faut choisir la
distribution initiale $\nu=(1/4,1/4,1/4,1/4,0,0)$. Par cons\'equent,
Anatole gagne le jeu avec probabilit\'e 
\begin{equation}
\label{fabs13}
\probin{\nu}{X_\tau=\text{\lq\lq A gagne\rq\rq}} = 
\sum_{i=1}^4 \nu_i b_{i1} = \frac5{12}\;.
\end{equation}
Quelle est la dur\'ee moyenne du jeu? La relation~\eqref{fabs7} montre que
la somme des \'el\'ements de la ligne $i$ de $F$ donne l'esp\'erance du
temps d'absorption partant de $i$, donc par exemple
$\expecin{1}{\tau}=14/3$. En moyennant sur la distribution initiale, on
trouve $\expecin{\nu}{\tau}=23/6$. La dur\'ee moyenne du jeu est donc de
$2+23/6=35/6$, soit un peu moins de $6$ jets de pi\`ece.
\end{example}

%%%%%%%%%%%%%%%%%%%%%%%%%%%%%%%%%%%%%%%%%%%%%%%%%%%%%%%%%%%%%%%%%%%%%%%%%%%

\section{\Chaine s de Markov irr\'eductibles}
\label{sec_f_irred}

\begin{definition}
\label{def_firred1}
Une \chaine\ de Markov est dite \defwd{irr\'eductible}\/ ou
\defwd{ergodique}\/ si $i\sim j$ $\forall i, j\in\cX$. La \chaine\ est
dite \defwd{r\'eguli\`ere}\/ s'il existe une puissance $P^n$ de $P$ dont
tous les \'el\'ements sont strictement positifs. 
\end{definition}

Une \chaine\ de Markov r\'eguli\`ere est n\'ecessairement irr\'eductible,
car tout \'etat est accessible depuis tout autre en $n$ pas au plus. La
r\'eciproque n'est pas vraie, car dans la d\'efinition de
l'irr\'eductibilit\'e on n'a pas sp\'ecifi\'e le nombre de pas. 

\begin{example}
\label{ex_irred1}
La \chaine\ d\'ecrivant le mod\`ele d'Ehrenfest est irr\'eductible. En
effet, quel que soit le nombre de boules dans l'urne de gauche, on peut
atteindre tout autre \'etat en d\'epla\c cant au plus $N$ boules d'une
urne \`a l'autre. Cependant, la \chaine\ n'est pas r\'eguli\`ere. En
effet, comme \`a chaque pas de temps on d\'eplace exactement une boule, 
le nombre de boules dans l'urne de gauche sera alternativement pair et
impair. Par cons\'equent, chaque \'el\'ement de matrice des puissance
$P^n$ sera nul pour un $n$ sur deux. 
\end{example}

\begin{definition}
\label{def_firred2}
Pour un sous-ensemble $A\subset\cX$, on appelle \defwd{temps de premier
passage}\/ de la \chaine\ dans $A$ la variable al\'eatoire 
\begin{equation}
\label{firred0}
\tau_A = \inf\setsuch{n>0}{X_n\in A}\;.
\end{equation}
Dans le cas o\`u $A=\set{i}$ consiste en un seul point, nous \'ecrirons
aussi $\tau_i$ au lieu de $\tau_{\set{i}}$. 
\end{definition}

Une diff\'erence importante entre \chaine s absorbantes et irr\'eductibles
est que ces der\-ni\`eres finissent toujours par revenir dans chacun de
leurs \'etats. 

\begin{prop}
\label{prop_firred1}
Pour une \chaine\ de Markov irr\'eductible sur un ensemble fini $\cX$, le
temps de premier passage en tout sous-ensemble $A\subset\cX$ est fini
presque s\^urement :
\begin{equation}
\label{firred0B}
\bigprob{\tau_A<\infty} \defby 
\lim_{n\to\infty} \bigprob{\tau_A\leqs n} = 1\;.
\end{equation}
\end{prop}
\begin{proof}
Consid\'erons une autre \chaine\ de Markov de matrice de
transition $\widehat P$, obtenue \`a partir de la \chaine\ de d\'epart en
rendant absorbants les \'etats de $A$ : 
\begin{equation}
\label{firred0B:1}
\hat p_{ij} = 
\begin{cases}
\delta_{ij} & \text{si $i\in A$\;,} \\
p_{ij} & \text{sinon\;.}
\end{cases}
\end{equation}
Les trajectoires de la \chaine\ initiale et de la \chaine\ modifi\'ee
co\"incident jusqu'au temps $\tau_A$. Il suffit donc de
montrer~\eqref{firred0B} pour la \chaine\ absorbante. Or dans ce cas, le
r\'esultat est une cons\'equence directe de la
Proposition~\ref{prop_fabs1}. En effet, la probabilit\'e
$\probin{i}{\tau_A>n}$ de ne pas avoir \'et\'e absorb\'e jusqu'au temps
$n$, partant de $i$, est donn\'ee par la somme des $(Q^n)_{ij}$ sur les
$j\in A$, qui tend vers $0$. 
\end{proof}

Il est important de remarquer que le r\'esultat ci-dessus n'est plus
forc\'ement vrai lorsque $\cX$ n'est pas fini! Nous reviendrons sur ce
point dans le chapitre suivant. 

Nous \'etudions maintenant de plus pr\`es les \chaine s de Markov
r\'eguli\`eres. Leur propri\'et\'e principale est le fait remarquable
suivant. 

\begin{theorem}
\label{thm_firred1}
Soit $P$ la matrice de transition d'une \chaine\ de Markov r\'eguli\`ere.
Alors il existe des nombres $\pi_1,\dots,\pi_N>0$, dont la somme vaut $1$,
tels que 
\begin{equation}
\label{firred1}
\lim_{n\to\infty} P^n = \Pi \defby
\begin{pmatrix}
\pi_1 & \pi_2 & \dots & \pi_N \\
\pi_1 & \pi_2 & \dots & \pi_N \\
\vdots & \vdots & & \vdots \\
\pi_1 & \pi_2 & \dots & \pi_N 
\end{pmatrix}\;.
\end{equation}
De plus, le vecteur ligne $\pi=(\pi_1,\pi_2,\dots,\pi_N)$ satisfait 
\begin{equation}
\label{firred1A}
\pi P = \pi\;.
\end{equation}
\end{theorem}
\begin{proof}
Si la cha\^\i ne n'a qu'un \'etat, le r\'esultat est imm\'ediat, donc nous
pouvons admettre que $N\geqs 2$. Supposons pour commencer que tous les
\'el\'ements de $P^k$ sont strictement positifs pour $k=1$, et soit $d>0$ le
plus petit \'el\'ement de $P$. Alors $d\leqs 1/2$,
puisque $Nd\leqs 1$. Soit $y$ un vecteur colonne tel que 
\begin{equation}
\label{firred1:1}
0 \leqs m_0 \leqs y_i \leqs M_0 
\qquad \forall i\in\set{1,\dots,N}\;.
\end{equation}
Soit $z=Py$. La plus grande valeur possible d'une composante $z_j$ de $z$
est obtenue si $y^T=(m_0,M_0,\dots,M_0)$ et $p_{k1}=d$. Dans ce cas,  la
somme des $N-1$ derniers \'el\'ements de la ligne $j$ de $P$ vaut $1-d$, et
par cons\'equent $z_j = dm_0 + (1-d) M_0$.   On a donc n\'ecessairement  
\begin{equation}
\label{firred1:2}
z_j \leqs dm_0 + (1-d) M_0 \bydef M_1
\qquad \forall j\in\set{1,\dots,N}\;.
\end{equation}
Un raisonnement similaire montre que  
\begin{equation}
\label{firred1:3}
z_j \geqs dM_0 + (1-d) m_0 \bydef m_1
\qquad \forall j\in\set{1,\dots,N}\;.
\end{equation}
Par cons\'equent, nous avons $m_1\leqs z_j\leqs M_1$, avec 
\begin{equation}
\label{firred1:4}
M_1 - m_1 = (1-2d) (M_0-m_0)\;.
\end{equation}
De plus, on voit facilement que $m_1\geqs m_0$ et $M_1\leqs M_0$. 
Apr\`es $n$ it\'erations, les composantes de $P^n y$ seront comprises entre
des nombres $m_n$ et $M_n$, satisfaisant 
\begin{equation}
\label{firred1:5}
M_n - m_n = (1-2d)^n (M_0-m_0)
\end{equation}
et 
\begin{equation}
\label{firred1:6}
m_0 \leqs m_1 \leqs \dots \leqs m_n \leqs M_n \leqs \dots \leqs M_1 \leqs
M_0\;.
\end{equation}
Les suites $\set{m_n}_{n\geqs1}$ et $\set{M_n}_{n\geqs1}$ sont donc
adjacentes, et convergent vers une m\^eme limite $u$. On a donc 
\begin{equation}
\label{firred1:7}
\lim_{n\to\infty} P^n y = 
\begin{pmatrix}
u \\ \vdots \\ u
\end{pmatrix}\;,
\end{equation}
o\`u $u$ d\'epend de $y$. Appliquons cette relation sur les vecteurs de
base $e_1,\dots,e_N$. Il existe des nombres $\pi_i$ tels que  
\begin{equation}
\label{firred1:8}
\lim_{n\to\infty} P^n e_i = 
\begin{pmatrix}
\pi_i \\ \vdots \\ \pi_i
\end{pmatrix}
\end{equation}
pour chaque $i$. Or $P^n e_i$ est la $i$-\`eme colonne de $P^n$, nous avons
donc prouv\'e la relation~\eqref{firred1}. Par ailleurs, comme dans le cas
$y=e_i$ on a $m_0=0$ et $M_0=1$, la relation~\eqref{firred1:3} donne
$m_1\geqs d$, donc $\pi_i\geqs d$. Par cons\'equent, tous les $\pi_i$ sont
strictement positifs. La somme des $\pi$ vaut $1$ car toute puissance de
$P$ est une matrice stochastique.

Enfin, pour montrer~\eqref{firred1A}, il suffit d'observer que 
$\Pi P=\lim_{n\to\infty}P^{n+1}=\Pi$. Chaque ligne de cette \'equation
matricielle est \'equivalente \`a~\eqref{firred1A}. 

Consid\'erons finalement le cas o\`u tous les \'el\'ements de $P^k$ sont
positifs pour un $k>1$. Le raisonnement ci-dessus peut \^etre r\'ep\'et\'e pour
montrer que les composantes de $P^{kn}z$ sont comprises entre deux bornes
$m_n$ et $M_n$ satisfaisant $M_{kn} - m_{kn} = (1-2d)^n (M_0-m_0)$. Pour les
\'etapes interm\'ediaires, on peut appliquer~\eqref{firred1:2}
et~\eqref{firred1:3} avec $d=0$ pour conclure que $M_{n+1}-m_{n+1}\leqs M_n-m_n$
pour tout $n$. Cela montre \`a nouveau que $M_n-m_n$ tend vers z\'ero.
\end{proof}

\begin{remark}
\label{rem_firred1}
Le r\'esultat pr\'ec\'edent montre que toute matrice stochastique
r\'eguli\`ere $P$ admet $1$ comme valeur propre. En fait, nous aurions
d\'ej\`a pu le remarquer avant, car la d\'efinition~\eqref{fdef2} d'une
matrice stochastique (quelconque) implique que $P\vone=\vone$, o\`u
$\vone=\transpose{(1,1\dots,1)}$. La relation~\eqref{firred1} montre en
plus que pour une matrice stochastique r\'eguli\`ere, la valeur propre $1$
est simple, et toutes les autres valeurs propres sont strictement
inf\'erieures \`a $1$ en module. 
\end{remark}

\begin{remark}
\label{rem_firred1b}
On d\'eduit facilement de l'expression~\eqref{firred1} que 
\begin{equation}
 \label{firred1D}
\Pi^2 = \Pi\;. 
\end{equation}
Une matrice satisfaisant cette relation est appel\'ee un \defwd{projecteur}\/. 
En l’occurrence, $\Pi$ projette tout vecteur ligne $\nu$ sur un multiple de
$\pi$
(c'est-\`a-dire $\nu\Pi\parallel\pi$), et tout vecteur colonne $x$ sur un
multiple du vecteur $\vone$ (c'est-\`a-dire $\Pi x\parallel\vone$). En
particulier, si $\nu$ est une distribution de probabilit\'e (donc
$\sum_i\nu_i=1$), alors on v\'erifie que $\nu\Pi = \pi$.
\end{remark}

Le vecteur ligne $\pi$ a plusieurs propri\'et\'es importantes : 
\begin{enum}
\item	Par \eqref{firred1} on a, $\forall i, j\in\cX$,  
\begin{equation}
\label{firred1B}
\lim_{n\to\infty} \bigprobin{i}{X_n=j} = \lim_{n\to\infty} (P^n)_{ij} =
\pi_j\;.
\end{equation}
$\pi$ d\'ecrit donc la distribution de probabilit\'e asymptotique de la
\chaine, qui est ind\'epen\-dante de l'\'etat initial.

\item	L'\'equation \eqref{firred1A} implique que pour tout temps $n$, 
\begin{equation}
\label{firred1C}
\bigprobin{\pi}{X_n=j} = \sum_{i\in\cX} \pi_i (P^n)_{ij} =
(\pi P^n)_j = \pi_j\;,
\end{equation}
ce qui motive la d\'efinition suivante. 
\end{enum}

\begin{definition}
\label{def_firred3}
La distribution de probabilit\'e $\pi=(\pi_1,\pi_2,\dots,\pi_N)$
satisfaisant la relation~\eqref{firred1} est appel\'ee \defwd{distribution
stationnaire}\/ (ou \defwd{invariante}) de la \chaine\ de Markov. 
\end{definition}

Enfin, on a le r\'esultat g\'en\'eral suivant : 

\begin{theorem}
\label{thm_firred2}
Pour une \chaine\ r\'eguli\`ere et toute distribution initiale $\nu$, on a
\begin{equation}
\label{firred2}
\lim_{n\to\infty} \probin{\nu}{X_n=j} = \pi_j\quad\forall j\in\cX\;.
\end{equation}
\end{theorem}
\begin{proof}
Une premi\`ere preuve tr\`es simple consiste \`a observer que la loi
asymptotique de $X_n$ est donn\'ee par le vecteur ligne 
\begin{equation}
 \label{firred3:0}
\lim_{n\to\infty} \nu P^n = \nu \Pi = \pi\;, 
\end{equation} 
en vertu de la Remarque~\ref{rem_firred1b}. 
Il est toutefois plus int\'eressant de pr\'esenter une autre preuve, tr\`es
\'el\'egante, due \`a Doeblin. Consid\'erons une autre \chaine\ de Markov,
d\'efinie sur l'espace $\cX\times\cX$. Ses probabilit\'es de transition
$P^\star$ sont donn\'ees par 
\begin{equation}
\label{firred3:1}
p^\star_{(i,j),(k,l)} = p_{ik}p_{jl}\;.
\end{equation}
Nous supposons que la distribution initiale de cette \chaine\ est une
mesure produit $\rho=\nu\otimes\pi$, c'est-\`a-dire que 
\begin{equation}
\label{firred3:2}
\rho((i,j))=\nu_i\pi_j \quad \forall (i,j)\in\cX\times\cX\;.
\end{equation}
Nous d\'enotons cette \chaine\ par $\set{(X_n,Y_n)}_{n\geqs0}$. Par
construction, les variables al\'eatoires $X_0$ et $Y_0$ sont
ind\'ependantes. Il suit alors de la d\'efinition~\eqref{firred3:1} des
probabilit\'es de transition que $X_n$ et $Y_n$ sont ind\'ependantes pour
tout $n$, et que les suites $\set{X_n}_{n\geqs0}$ et $\set{Y_n}_{n\geqs0}$
sont en fait deux \chaine s de Markov sur $\cX$ de matrice de transition
$P$, et de distributions initiales respectivement donn\'ees par $\nu$ et
$\pi$.

La matrice de transition $P^\star$ est \'egalement r\'eguli\`ere : il
suffit de se convaincre que les \'el\'ements de matrice des puissances
$(P^\star)^n$ sont donn\'es par des produits
$\smash{p^{(n)}_{ik}p^{(n)}_{jl}}$. 
Consid\'erons alors l'ensemble 
\begin{equation}
\label{firred3:3}
A = \setsuch{(i,i)}{i\in\cX} \subset \cX\times\cX\;.
\end{equation}
Le temps de premier passage $\tau_A$ peut aussi s'\'ecrire 
\begin{equation}
\label{firred3:4}
\tau_A = \inf\setsuch{n>0}{X_n=Y_n}\;.
\end{equation}
Nous pr\'etendons que les deux \chaine s ont la m\^eme loi pour
$n\geqs\tau_A$. Plus pr\'ecis\'ement, 
\begin{equation}
\label{firred3:5}
\probin{\rho}{X_n=j, \tau_A\leqs n} = 
\probin{\rho}{Y_n=j, \tau_A\leqs n} 
\qquad \forall j\in\cX\;, \forall n\geqs0\;.
\end{equation}
Pour montrer cela, nous introduisons un nouveau processus
$\set{Z_n}_{n\in\N}$ d\'efini par 
\begin{equation}
\label{firred3:6}
Z_n = 
\begin{cases}
X_n & \text{pour $n\leqs\tau_A$\;,}\\
Y_n & \text{pour $n>\tau_A$\;.}
\end{cases}
\end{equation}
On v\'erifie par un calcul direct, en d\'ecomposant sur les valeurs
possibles de $\tau_A$, que 
\begin{equation}
\label{firred3:7}
\probin{\rho}{[Z_{[0,n]}=i_{[0,n]}}
= \nu_{i_0}\prod_{m=1}^n p_{i_{m-1}i_m}
\end{equation}
pour tout $n\geqs 0$ et tout choix de $i_{[0,n]}\in\cX^{n+1}$. Par le
Th\'eor\`eme~\ref{thm_fdef1}, il suit que $\set{Z_n}_{n\in\N}$ est une
\chaine\ de Markov de distribution initiale $\nu$ et matrice de transition
$P$, et est donc \'egale en loi \`a $\set{X_n}_{n\in\N}$. Ceci 
prouve~\eqref{firred3:5}. Finalement, on a 
\begin{align}
\nonumber
\probin{\nu}{X_n=j} &= \probin{\rho}{X_n=j, \tau_A\leqs n}
+ \probin{\rho}{X_n=j, \tau_A > n}\;,\\
\pi_j = \probin{\pi}{Y_n=j} &= \probin{\rho}{Y_n=j, \tau_A\leqs n}
+ \probin{\rho}{Y_n=j, \tau_A > n}\;.
\label{firred3:8}
\end{align}
En prenant la diff\'erence et en utilisant~\eqref{firred3:5}, il vient 
\begin{align}
\nonumber
\bigabs{\probin{\nu}{X_n=j}-\pi_j} 
&\leqs \bigabs{\probin{\rho}{X_n=j, \tau_A > n} 
- \probin{\rho}{Y_n=j, \tau_A > n}} \\
&\leqs 2\probin{\rho}{\tau_A > n}\;.
\label{firred3:9}
\end{align}
Or cette derni\`ere quantit\'e tend vers z\'ero lorsque $n\to\infty$,
puisque $\tau_A$ est fini presque s\^urement en vertu de la
Proposition~\ref{prop_firred1}. 
\end{proof}

\begin{remark}
\label{rem_firred2}
La preuve de Doeblin permet aussi d'exprimer la vitesse de convergence
vers la distribution stationnaire \`a l'aide du temps $\tau_A$ introduit
dans la preuve. En effet, en sommant la premi\`ere ligne
de~\eqref{firred3:9} sur tous les $j\in\cX$, on obtient  
\begin{equation}
\label{firred4}
\sum_{j\in\cX} \bigabs{\probin{\nu}{X_n=j}-\pi_j} \leqs 2
\probin{\nu\otimes\pi}{\tau_A>n}\;.
\end{equation}
Le membre de gauche peut \^etre consid\'er\'e comme la distance $\ell_1$
entre la distribution de $X_n$ et la distribution stationnaire. 
Ce genre d'argument est appel\'e un argument de couplage. 
%
% \begin{proof}
% Posons $f_i=\probin{\nu}{X_n=i}$ et $B=\setsuch{i\in\cX}{f_i>\nu_i}$.
% Alors on a 
% \begin{align}
% \nonumber
% \sum_{i\in\cX} \bigabs{f_i-\pi_i}
% &= \sum_{i\in B} (f_i-\pi_i) - \sum_{i\in\cX\setminus B} (f_i-\pi_i) \\
% \nonumber
% &= 2\sum_{i\in B} (f_i-\pi_i) - \underbrace{\sum_{i\in \cX}
% (f_i-\pi_i)}_{=1-1=0} \\
% \nonumber
% &= 2\bigbrak{\prob{X_n\in B} - \prob{Y_n\in B}} \\
% \nonumber
% &\leqs 2\bigbrak{\prob{X_n\in B,Y_n\in B} + \prob{X_n\in B,Y_n\neq X_n} -
% \prob{Y_n\in B}} \\
% \nonumber
% &\leqs 2\bigbrak{\prob{X_n\in B,Y_n\in B} + \prob{Y_n\neq X_n} -
% \prob{Y_n\in B}} \\
% &\leqs 2\prob{Y_n\neq X_n}\;,
% \label{firred5}
% \end{align}
% d'o\`u le r\'esultat, vu que $\set{Y_n\neq
% X_n}\Rightarrow\set{\tau_A>n}$.
% \end{proof}
\end{remark}

Nous revenons maintenant au cas g\'en\'eral de \chaine s de Markov
irr\'eductibles. Dans ce cas, la loi de $X_n$ ne converge pas
n\'ecessairement vers une loi $\pi$ donn\'ee. Toutefois, une partie des
r\'esultats pr\'ec\'edents reste vraie : 

\begin{prop}
\label{prop_firred2}
Soit $P$ la matrice de transition d'une \chaine\ de Markov irr\'eductible.
Alors $P$ admet $1$ comme valeur propre simple. L'unique vecteur propre
\`a gauche $\pi$ de $P$ pour la valeur propre $1$ tel que $\sum_i\pi_i=1$
sera \`a nouveau appel\'e la\/ \defwd{distribution stationnaire}\/ de la
\chaine.
\end{prop}
\begin{proof}
Consid\'erons la matrice stochastique $Q=\frac12[P+I]$. Soit 
\begin{equation}
\label{firred6:1}
m = \max_{i,j\in\cX} \bigset{\min\setsuch{n\geqs1}{p^{(n)}_{ij}>0}}\;.
\end{equation}
Consid\'erons la matrice 
\begin{equation}
\label{firred6:2}
Q^m = \frac{1}{2^m} \biggbrak{I + \binom m1 P + \binom m2 P^2 + \dots + 
\binom m{m-1} P^{m-1} + P^m}\;.
\end{equation}
Pour tout couple $(i,j)$, il existe un terme de cette somme dont
l'\'el\'ement de matrice $(i,j)$ soit strictement positif. Comme tous les
autres \'el\'ements de matrice sont non-n\'egatifs, on conclut que
$(Q^m)_{i,j}>0$. Par cons\'equent, $Q$ est la matrice de transition d'une
\chaine\ r\'eguli\`ere. Par le th\'eor\`eme~\ref{thm_firred1}, il existe
une unique distribution de probabilit\'e $\pi$ telle que $\pi Q=\pi$, ce
qui implique $\frac12\brak{\pi+\pi P}=\pi$, donc $\pi P=\pi$. 
\end{proof}

\begin{example}
\label{ex_irred2}
On v\'erifie facilement par calcul direct que la distribution stationnaire 
du mod\`ele d'Ehrenfest est binomiale de param\`etre $1/2$ : 
$\nu_i = 2^{-N}\binom Ni$. 
Nous verrons plus loin une interpr\'etation plus intuitive de ce
r\'esultat.
\end{example}

Quelle est l'interpr\'etation de la distribution stationnaire? D'une part,
nous savons d\'ej\`a que si $X_n$ suit la loi $\pi$ \`a un temps $n$,
alors $X_m$ suivra la m\^eme loi $\pi$ \`a tous les temps ult\'erieurs
$m>n$. En revanche, les Th\'eor\`emes~\ref{thm_mas1} et~\ref{thm_mas2} ne
sont plus n\'ecessairement vrais dans ce cas : Il suffit de consid\'erer
l'exemple du mod\`ele d'Ehrenfest. Toutefois, on a encore convergence
vers la distribution stationnaire dans le sens de la moyenne ergodique (ou
moyenne de Cesaro) : 

\begin{theorem}
\label{thm_firred3}
Pour une \chaine\ de Markov irr\'eductible, et pour toute distribution
initiale $\nu$, la fr\'equence moyenne de passage en tout \'etat $j$
converge vers $\pi_j$: 
\begin{equation}
\label{firred7}
\lim_{n\to\infty} \frac1n
\biggexpecin{\nu}{\sum_{m=0}^{n-1}\indexfct{X_m=j}} = \pi_j
\qquad
\forall j\in\cX\;.
\end{equation}
\end{theorem}
\begin{proof}
Soit $\Pi$ la matrice dont toutes les lignes sont \'egales
\`a $\pi$, cf.~\eqref{firred1}. Alors on a $\Pi P=\Pi$, et le fait que
$P\vone=\vone$ implique qu'on a \'egalement $P\Pi=\Pi$. Il suit que 
\begin{equation}
\label{firred9:1}
(I+P+\dots+P^{n-1}) (I-P+\Pi) = I -P^n + n\Pi\;.
\end{equation}
Montrons que la matrice $I-P+\Pi$ est inversible. Soit $x$ un vecteur
colonne tel que $(I-P+\Pi)x=0$. Alors on a 
\begin{equation}
\label{firred9:2}
0 = \pi (I-P+\Pi) x = \underbrace{\pi(I-P)}_{=0} x + \pi \Pi x = \pi x\;,
\end{equation}
puisque $\pi\Pi=\pi$ en raison du fait que $\sum_i\pi_i=1$. Il suit que
$\Pi x=0$, et donc $(I-P)x=0$. Comme $P$ admet $1$ comme valeur propre
simple, avec vecteur propre \`a droite $\vone$, ceci implique que
$x\parallel\vone$, ce qui n'est possible que si $x=0$ puisque $\pi x=0$ et
tous les $\pi_i$ sont positifs. La matrice $I-P+\Pi$ est donc bien
inversible. 

Soit $Z=(I-P+\Pi)^{-1}$. Comme $\pi(I-P+\Pi)=\pi$, on a aussi $\pi=\pi Z$
et $\Pi = \Pi Z$. En multipliant~\eqref{firred9:1} \`a droite par $Z$, il
vient 
\begin{equation}
\label{firred9:3}
I + P + \dots + P^{n-1} = (I-P^n) Z + n\Pi Z = (I-P^n) Z + n\Pi\;.
\end{equation}
Or nous avons, pour tout \'etat initial $i$, 
\begin{equation}
\label{firred9:4}
\frac1n\biggexpecin{i}{\sum_{m=0}^{n-1}\indexfct{X_m=j}}
= \frac1n\sum_{m=0}^{n-1} \bigpar{P^m}_{ij} 
= \biggbrak{\frac1n(I-P^n) Z + \Pi}_{ij}\;.
\end{equation}
Comme les \'el\'ements de matrice de $P^n$ sont uniform\'ement born\'es
par $1$, cette quantit\'e converge vers $(\Pi)_{ij}=\pi_j$ lorsque
$n\to\infty$. Pour une distribution initiale quelconque $\nu$, on obtient
de la m\^eme mani\`ere la convergence vers $(\nu\Pi)_j=\pi_j$. 
\end{proof}

\begin{remark}
Cette preuve peut ne pas sembler tr\`es transparente. On peut en fait
l'\'eclairer avec quelques notions de calcul matriciel. Soit la matrice
$Q=P-\Pi$, d\'ecrivant l'\'ecart entre la matrice stochastique et le projecteur
sur la distribution stationnaire. Il suit des \'egalit\'es $\Pi
P=P\Pi=\Pi^2=\Pi$ que $\Pi Q=Q\Pi=0$, et on en d\'eduit 
\begin{equation}
 \label{firred9:A}
P^n = \Pi + Q^n
\end{equation} 
(voir aussi l'exercice~\ref{exo_mf09}). 
Dans le cas o\`u $P$ est r\'eguli\`ere, $Q^n$ tend vers $0$ lorsque
$n\to\infty$. Si $P$ est irr\'eductible mais pas r\'eguli\`ere, la
fonction $n\mapsto Q^n$ est oscillante. La preuve montre cependant que $Q$
n'admet pas la valeur propre $1$, et que la moyenne des $Q^n$ tend vers z\'ero,
donc la moyenne des $P^n$ tend vers $\Pi$.  
\end{remark}

La distribution stationnaire $\pi$ a \'egalement un lien int\'eressant
avec l'esp\'erance du temps de premier retour en un site $i$, appel\'e
\defwd{temps de r\'ecurrence moyen}\/ en $i$ :

\begin{theorem}
\label{thm_firred4}
Pour une \chaine\ de Markov irr\'eductible de distribution stationnaire
$\pi$, les temps de r\'ecurrence moyens sont donn\'es par 
\begin{equation}
\label{firred17}
\expecin{i}{\tau_i} = \frac1{\pi_i}\;.
\end{equation}
\end{theorem}
\begin{proof}
Nous commen\c cons par \'etablir une \'equation reliant divers temps de
premier passage moyens. Pour $i,j\in\cX$, on a 
\begin{align}
\nonumber
\expecin{i}{\tau_j}
&= \probin{i}{\tau_j=1} + \sum_{n\geqs2} n\probin{i}{\tau_j=n}\\
\nonumber
&= p_{ij} + \sum_{n\geqs2} n \sum_{k\neq j} \probin{i}{\tau_j=n,X_1=k}\\
\nonumber
&= p_{ij} + \sum_{n\geqs2} n \sum_{k\neq j} p_{ik} \probin{k}{\tau_j=n-1}\\
\nonumber
&= p_{ij} + \sum_{k\neq j} p_{ik} \sum_{m\geqs1} (m+1)
\probin{k}{\tau_j=m}\\
\nonumber
&= p_{ij} + \sum_{k\neq j} p_{ik} 
\biggbrak{\expecin{k}{\tau_j} + \underbrace{\sum_{m\geqs1}
\probin{k}{\tau_j=m}}_{=1}} \\
&= 1 + \sum_{k\neq j} p_{ik}\expecin{k}{\tau_j}\;.
\label{firred18:1}
\end{align}
Cette relation peut \^etre r\'ecrite sous la forme 
\begin{equation}
\label{firred18:2}
1-\expecin{i}{\tau_j} = -\sum_{k\in\cX} (1-\delta_{kj})
p_{ik}\expecin{k}{\tau_j}\;.
\end{equation}
Il suit que 
\begin{align}
\nonumber
1-\pi_j\expecin{j}{\tau_j}
&= \sum_{i\in\cX} \pi_i \bigbrak{1-\delta_{ij}\expecin{i}{\tau_j}} \\
\nonumber
&= \sum_{i\in\cX} \pi_i
\bigbrak{1-\expecin{i}{\tau_j} +(1-\delta_{ij})\expecin{i}{\tau_j}} \\
\nonumber
&= \sum_{k\in\cX} \sum_{i\in\cX} \pi_i
(\delta_{ik}-p_{ik})(1-\delta_{kj})\expecin{k}{\tau_j} \;.
\label{firred18:3}
\end{align}
La somme sur $i$ s'annule, puisque $\pi_k=\sum_{i}\pi_ip_{ik}$.
\end{proof}

\begin{example}
\label{ex_irred3}
Dans le cas du mod\`ele d'Ehrenfest avec $N$ boules, le temps de
r\'e\-cur\-rence moyen vers l'\'etat \`a $i$ boules est donn\'e par  
\begin{equation}
\label{firred19}
\expecin{i}{\tau_i} = \frac1{\nu_i} 
= 2^N\frac{i!(N-i)!}{N!}\;.
\end{equation}
En particulier, le temps moyen entre configurations o\`u toutes les boules
sont dans l'urne de gauche est de $2^N$. Ce temps devient gigantesque
pour des nombres de boules de l'ordre du nombre d'Avogadro, c'est-\`a-dire
du nombre de mol\'ecules dans un \'echantillon d'une mole de gaz. Ce
mod\`ele simple peut donc justifier pourquoi, lorsque deux r\'ecipients
contenant des gaz sont mis en contact, on n'observe jamais toutes les
mol\'ecules dans le m\^eme r\'ecipient.
\end{example}

%%%%%%%%%%%%%%%%%%%%%%%%%%%%%%%%%%%%%%%%%%%%%%%%%%%%%%%%%%%%%%%%%%%%%%%%%%%

\section{\Chaine s de Markov r\'eversibles}
\label{sec_rev}

Dans cette section, l'ensemble $\cX$ peut \^etre infini (mais
doit \^etre d\'enombrable). 
Rappelons que $E^\cX$ d\'enote l'ensemble des applications $f:\cX\to E$. 

\begin{definition}
\label{def_rev1}
Soit $P$ une matrice stochastique. Un vecteur
$\alpha=\set{\alpha_i}_{i\in\cX}\in[0,\infty)^\cX$, $\alpha\neq0$, est dit
\defwd{r\'eversible}\/ par rapport \`a $P$ si 
\begin{equation}
\label{rev1}
\alpha_i p_{ij} = \alpha_j p_{ji} 
\qquad \forall i, j\in\cX\;.
\end{equation}
Une \chaine\ de Markov est dite \defwd{r\'eversible}\/ si sa matrice admet
un vecteur r\'eversible. 
\end{definition}

La condition~\eqref{rev1} est appel\'ee \defwd{condition d'\'equilibre
d\'etaill\'e}\/ en physique. Elle signifie que si les \'etats $i$ et $j$
sont occup\'es avec probabilit\'es proportionnelles \`a $\alpha_i$ et
$\alpha_j$ respectivement, alors les taux de transition de $i$ \`a $j$ et
de $j$ \`a $i$ sont \'egaux.

\begin{theorem}
\label{thm_rev1}
Soit $P$ %$=\set{p_{ij}}_{i,j\in\cX}$ 
une matrice stochastique et
$\alpha\in[0,\infty)^\cX$ un vecteur non nul.
\begin{enum}
\item	Si $\alpha$ est r\'eversible par rapport \`a $P$, alors $\alpha$
est une mesure invariante. 
\item	Si $\alpha$ est r\'eversible par rapport \`a $P$, et
$\sum_{j\in\cX}\alpha_j<\infty$, alors la mesure $\pi$ d\'efinie par
$\pi_i=\alpha_i/\sum_{j\in\cX}\alpha_j$ est une distribution stationnaire.
\item	Si $\pi$ est une distribution stationnaire, alors 
\begin{equation}
\label{rev2}
\probin{\pi}{X_0=i_0,X_1=i_1,\dots,X_n=i_n} = 
\probin{\pi}{X_0=i_n,X_1=i_{n-1},\dots,X_n=i_0}
\end{equation}
pour tout $n\in\N$ et tout choix de $i_0,\dots,i_n\in\cX$.
\end{enum}
\end{theorem}
\begin{proof}
\hfill
\begin{enum}
\item	On a
\begin{equation}
\label{rev2:1}
\sum_{i\in\cX} \alpha_i p_{ij} = \alpha_j \sum_{i\in\cX} p_{ji} =
\alpha_j\;.
\end{equation}
\item	Suit imm\'ediatement de 1.
\item	Par le Th\'eor\`eme~\ref{thm_fdef1},
\begin{align}
\nonumber
\probin{\pi}{X_0=i_0,X_1=i_1,\dots,X_n=i_n}
&= \pi_{i_0} p_{i_0i_1} p_{i_1i_2} \dots p_{i_{n-1}i_n} \\
\nonumber
&= p_{i_1i_0} \pi_{i_1} p_{i_1i_2} \dots p_{i_{n-1}i_n} \\
&= \dots = p_{i_1i_0} p_{i_2i_1}\dots p_{i_ni_{n-1}} \pi_{i_n}\;,
\label{rev2:2}
\end{align}
qui est \'egal \`a $\probin{\pi}{X_0=i_n,X_1=i_{n-1},\dots,X_n=i_0}$.
\qed
\end{enum}
\renewcommand{\qed}{}
\end{proof}

La relation~\eqref{rev2} signifie qu'une trajectoire a la m\^eme
probabilit\'e que la trajectoire renvers\'ee dans le temps. C'est ce qui
justifie le terme de r\'eversibilit\'e.

\begin{example}[Marche al\'eatoire du cavalier]
On suppose qu'un cavalier se d\'eplace sur un \'echiquier, en choisissant
\`a chaque unit\'e de temps, de mani\`ere \'equiprobable, l'un des
mouvements autoris\'es par les r\`egles des Echecs. Combien de temps se
passe-t-il en moyenne entre deux passages du cavalier au coin inf\'erieur
gauche de l'\'echiquier?

\begin{figure}
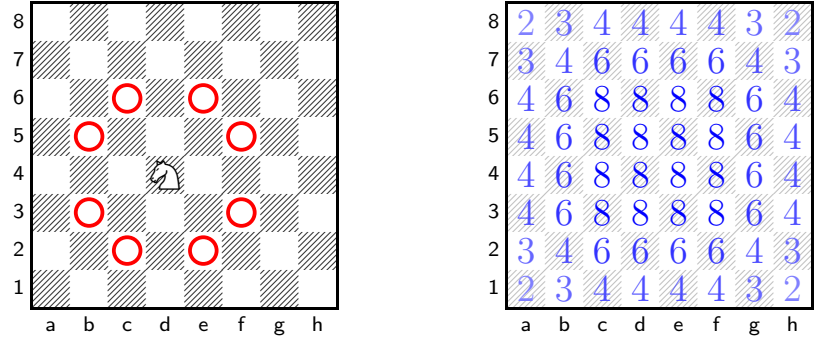

%  \centerline{
%  \includegraphics*[clip=true,height=42mm]{figs/echec1}
%  }
% \figtext{
% }
% \begin{tikzpicture}
% \pgfmathsetmacro{\boardsize}{8}
%  
%     \foreach \i in {1,...,\boardsize}{
%         \foreach \j in {1,...,\boardsize}{
%             \pgfmathsetmacro{\weight}{(1 + (-1)^(\i+\j))*50};
%             \node[rectangle,fill=gray!\weight,minimum size=1cm] (node\i-\j) at
%              (\i,\j) {};
%              \draw (\i-0.5,\j-0.5) rectangle (\i +0.5,\j +0.5);
%         }
%     }
%     
% \end{tikzpicture}
\begin{center}
\vspace{-5mm}
\chessboard[smallboard,
  boardfontsize=14.4pt,
  setwhite={nd4},showmover=false,
  color=red,
  padding=-0.2em,
  pgfstyle=circle,
  markfields={b3,b5,c2,c6,e2,e6,f3,f5}
  ]
\hspace{10mm}
\setchessboard{
blackfieldcolor=black!30,
setfontcolors}
\chessboard[smallboard,
  showmover=false,
  boardfontsize=14.4pt,
  pgfstyle=text,
  color=blue,
  text=$8$\bfseries\sffamily,
  markregion=c3-c3,
  markregion=d3-d3,
  markregion=e3-e3,
  markregion=f3-f3,
  markregion=c4-c4,
  markregion=d4-d4,
  markregion=e4-e4,
  markregion=f4-f4,
  markregion=c5-c5,
  markregion=d5-d5,
  markregion=e5-e5,
  markregion=f5-f5,
  markregion=c6-c6,
  markregion=d6-d6,
  markregion=e6-e6,
  markregion=f6-f6,
  color=blue!80,
  text=$6$\bfseries\sffamily,
  markregion=c2-c2,
  markregion=d2-d2,
  markregion=e2-e2,
  markregion=f2-f2,
  markregion=c7-c7,
  markregion=d7-d7,
  markregion=e7-e7,
  markregion=f7-f7,
  markregion=b3-b3,
  markregion=b4-b4,
  markregion=b5-b5,
  markregion=b6-b6,
  markregion=g3-g3,
  markregion=g4-g4,
  markregion=g5-g5,
  markregion=g6-g6,
  color=blue!70,
  text=$4$\bfseries\sffamily,
  markregion=c1-c1,
  markregion=d1-d1,
  markregion=e1-e1,
  markregion=f1-f1,
  markregion=c8-c8,
  markregion=d8-d8,
  markregion=e8-e8,
  markregion=f8-f8,
  markregion=a3-a3,
  markregion=a4-a4,
  markregion=a5-a5,
  markregion=a6-a6,
  markregion=h3-h3,
  markregion=h4-h4,
  markregion=h5-h5,
  markregion=h6-h6,
  markregion=b2-b2,
  markregion=g2-g2,
  markregion=b7-b7,
  markregion=g7-g7,
  color=blue!60,
  text=$3$\bfseries\sffamily,
  markregion=b1-b1,
  markregion=a2-a2,
  markregion=g1-g1,
  markregion=h2-h2,
  markregion=b8-b8,
  markregion=a7-a7,
  markregion=g8-g8,
  markregion=h7-h7,
  color=blue!50,
  text=$2$\bfseries\sffamily,
  markregion=a1-a1,
  markregion=h1-h1,
  markregion=a8-a8,
  markregion=h8-h8
  ]  
\end{center}
\vspace{-5mm}
 \caption[]{Mouvements permis du cavalier sur l'\'echiquier. Nombre de
 mouvements possibles \`a partir de chaque case.}
 \label{fig_echecs}
\end{figure}

Soit $n_i$ le nombre de mouvements possibles \`a partir de la case $i$
(\figref{fig_echecs}). La trajectoire du cavalier est donc d\'ecrite par
une \chaine\ de Markov de probabilit\'es de transition
\begin{equation}
\label{rev3}
p_{ij} = 
\begin{cases}
\myvrule{16pt}{14pt}{0pt} 
\dfrac1{n_i} & \text{si le mouvement $i\mapsto j$ est permis\;,} \\
0 & \text{sinon\;.}
\end{cases}
\end{equation}
On v\'erifie alors facilement que $n=\set{n_i}_{i\in\cX}$ est un vecteur
r\'eversible de la \chaine. Par cons\'equent, 
\begin{equation}
\label{rev4}
\pi_i = \frac{n_i}{\sum_{j\in\cX}n_j} = \frac{n_i}{336}
\end{equation}
est la distribution stationnaire de la \chaine. Il suit alors du
Th\'eor\`eme~\ref{thm_stat1} que le temps de r\'ecurrence moyen vers le
coin inf\'erieur gauche est donn\'e par $1/\pi_{(1,1)} = 336/2 = 168$. 
\end{example}

\begin{example}[Mod\`ele d'Ehrenfest]
\label{ex_rev2}
Nous sommes maintenant en mesure d'expliquer pourquoi la distribution
invariante du mod\`ele d'Ehrenfest est binomiale (cf.
Exemples~\ref{ex_Ehrenfest} et~\ref{ex_irred2}). Pour cela, au lieu de
consid\'erer le mod\`ele comme une \chaine\ de Markov sur l'ensemble
$\set{0,\dots,N}$, nous le consid\'erons comme une \chaine\ sur
$\cX=\set{0,1}^N$ (qu'on peut consid\'erer comme un hypercube de
dimension $N$). La composante $x_i$ de l'\'etat $x=(x_1,\dots,x_N)\in\cX$
vaut $0$ si la $i^\text{\`eme}$ boule est dans l'urne de gauche, et $1$ si
elle est dans l'urne de droite. 

Depuis chaque \'etat $x\in\cX$, on peut atteindre exactement $N$ autres
\'etats, obtenus en changeant exactement une composante de $x$, chaque fois
avec probabilit\'e $1/N$. Par cons\'equent, tout vecteur constant est
r\'eversible, et la distribution stationnaire est uniforme: $\pi_x=2^{-N}
\forall x\in\cX$. Toutefois, il peut y avoir beaucoup d'\'etats de $\cX$
correspondant \`a un nombre donn\'e de boules dans une urne. En fait, 
\begin{equation}
\label{rev5}
\probin{\pi}{\text{$m$ boules dans l'urne de droite}}
= \sum_{x\colon \sum x_i=m}\pi_x
= \binom{N}{m} \frac{1}{2^N}\;.
\end{equation}
On retrouve le fait que la distribution stationnaire du mod\`ele
d'Ehrenfest est binomiale.
\end{example}

%%%%%%%%%%%%%%%%%%%%%%%%%%%%%%%%%%%%%%%%%%%%%%%%%%%%%%%%%%%%%%%%%%%%%%%%%%%

\goodbreak
\section{Exercices}
\label{sec_exo_markov_fini}

\begin{exercice}
\label{exo_mf01} 

Un alpiniste veut faire l'ascension du Mont Blanc. Lors de son ascension,
il d\'ecide de passer la nuit au refuge de T\^ete Rousse, et \'egalement
au refuge de l'Aiguille du Go\^uter. Chaque matin, il observe la
m\'et\'eo. Si celle-ci lui para\^\i t bonne, alors il poursuit son
ascension jusqu'\`a l'\'etape suivante. En revanche, si la m\'et\'eo semble
mauvaise, il redescend d'une \'etape. 
On suppose que l'alpiniste est initialement au refuge de T\^ete Rousse, et
que s'il est oblig\'e de redescendre au Nid d'Aigle, alors il abandonne
son projet d'ascension. Enfin, on suppose que la m\'et\'eo est bonne avec
probabilit\'e $p$, et mauvaise avec probabilit\'e $q=1-p$, et
ind\'ependante de la m\'et\'eo des jours pr\'ec\'edents. 

% \bigskip
% \centerline{\includegraphics*[clip=true,height=50mm]{figs/MontBlanc}}
% \medskip

\begin{center}
\begin{tikzpicture}[->,>=stealth',shorten >=5pt,shorten <=5pt,auto,node
distance=3.0cm, thick,every node/.style={font=\sffamily\footnotesize}]

\shade (0.0,0.0) -- (1.5,1.8) -- (2.5,1.8) -- (4.0,3.8) -- (5.0,3.8)  
-- (6.5,5.8) -- (8.5,5.8) -- (10.0,0) -- (0.0,0.0);

  \node (NA) at(0,0) {Nid d'Aigle (2500 m)};
  \node (TR) at(2.5,2) {T\^ete Rousse (3100 m)};
  \node (AG) at(5.0,4) {Aiguille du Go\^uter (3800 m)};
  \node (MB) at(7.5,6) {Mont Blanc (4800 m)};

  \path[every node/.style={font=\sffamily\small}]
    (TR) edge [bend left,below right] node {$q$} (NA)
    (AG) edge [bend left,below right] node {$q$} (TR)
    (TR) edge [bend left,above left] node {$p$} (AG)
    (AG) edge [bend left,above left] node {$p$} (MB)
    ;
\end{tikzpicture}
\end{center}

\begin{enum}
\item	Montrer que le probl\`eme peut \^etre d\'ecrit par une cha\^\i ne de
Markov absorbante, et calculer sa matrice fondamentale.
\item	Calculer, en fonction de $p$, la probabilit\'e que l'alpiniste
atteigne le sommet.
\item	D\'eterminer la valeur $p^\star$ de $p$ pour qu'il 
ait une chance sur deux d'atteindre le sommet.
\item	Calculer le nombre moyen de jours de l'ascension pour $p=p^\star$.
\end{enum}
\end{exercice}

\goodbreak

\begin{exercice}
\label{exo_mf02} 

R\'esoudre le probl\`eme de la souris dans le labyrinthe (Exemple~3.1.1.)~:

\begin{enum}
\item	D\'eterminer la matrice fondamentale de la cha\^\i ne.
\item	Calculer la probabilit\'e que la souris atteigne la nourriture.
\item	Calculer le temps moyen du parcours de la souris.
\end{enum}
\end{exercice}

\goodbreak

\begin{exercice} 
\label{exo_mf03} 

Deux joueurs A et B s'affrontent dans une partie de tennis. Chaque point jou\'e 
est gagn\'e par le joueur A avec une probabilit\'e de $3/5$, sinon il est
gagn\'e par B. On suppose les points ind\'ependants. 

Initialement, les deux joueurs sont \`a \'egalit\'e. Pour gagner la partie, un
joueur doit obtenir une avance de deux points sur son opposant. 

\begin{enum}
\item	Mod\'eliser le jeu par une \chaine\ de Markov absorbante \`a $5$
\'etats: Egalit\'e (deuce), Avantage A, Avantage B, A gagne, et B gagne.
Donner la matrice de transition de cette \chaine. 

\item	Montrer que la matrice fondamentale de la \chaine\ est donn\'ee par 
\[
N = \frac{1}{13}
\begin{pmatrix}
25 & 15 & 10 \\
10 & 19 &  4 \\
15 &  9 & 19
\end{pmatrix}\;.
\]

\item	Calculer la probabilit\'e que A gagne, si les joueurs sont initialement
\`a \'egalit\'e. 

\item	Calculer la dur\'ee moyenne du jeu si les joueurs sont initialement
\`a \'egalit\'e. 
\end{enum}

\end{exercice}

\goodbreak

\begin{exercice} 
\label{exo_mf031} 

Bilbo le hobbit est perdu dans les cavernes des orques, o\`u r\`egne une
obscurit\'e totale. Partant de la caverne 1, il choisit de mani\`ere
\'equiprobable l'une des galeries partant de cette caverne, et continue de cette
mani\`ere jusqu'\`a ce qu'il aboutisse soit \`a la caverne de Gollum (caverne
4), soit \`a l'air libre (num\'ero 5). 

\bigskip
\centerline{
 \includegraphics*[clip=true,height=70mm]{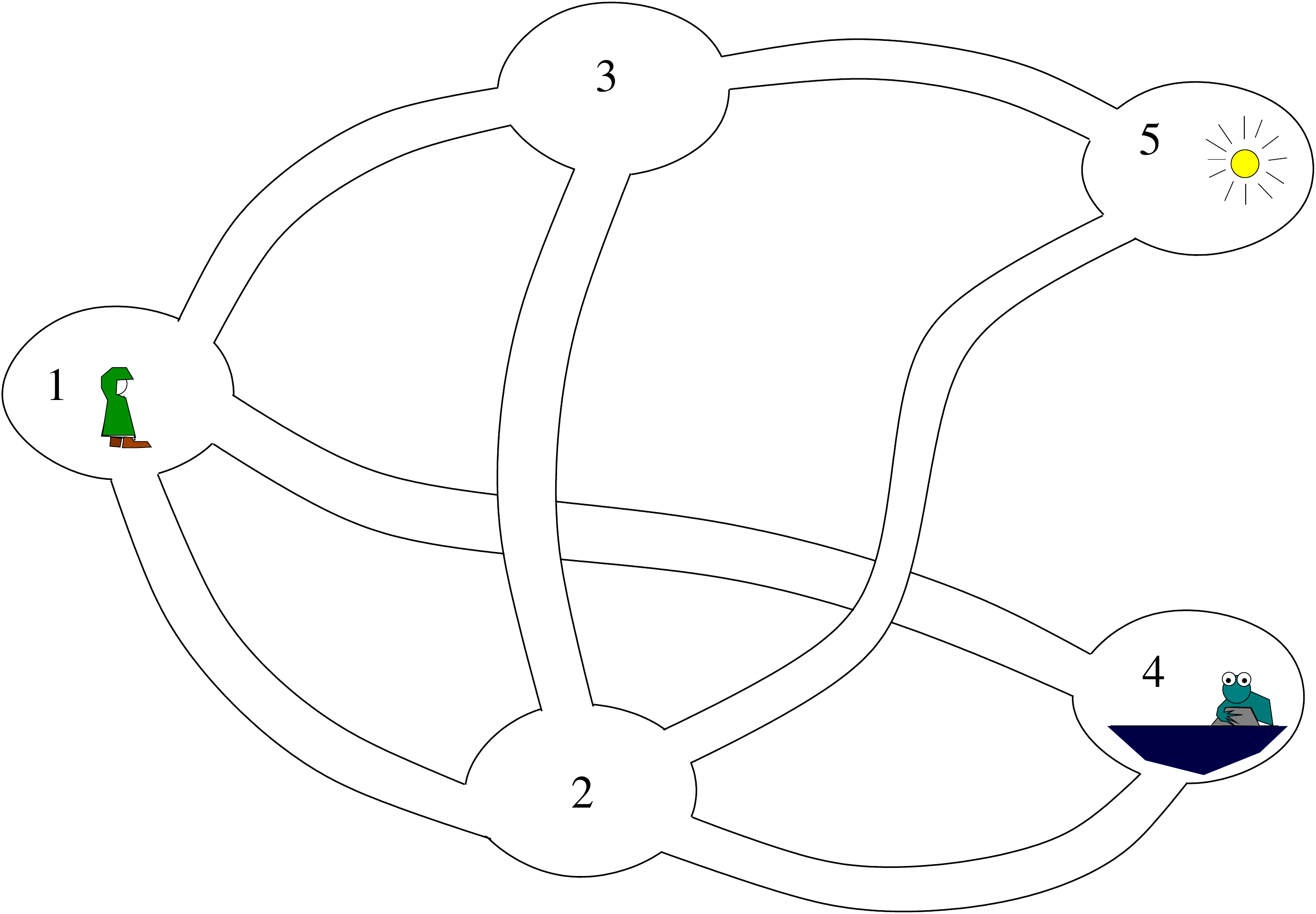}
 }
\medskip

\begin{enum}
\item	D\'ecrire l'itin\'eraire de Bilbo par une cha\^ine de Markov, dont on
donnera la matrice de transition et la matrice fondamentale. 

\item	Partant de 1, quelle est la probabilit\'e que Bilbo trouve la sortie
plut\^ot que de d\'eboucher dans la caverne de Gollum? 

\item 	Combien de galeries Bilbo aura-t-il travers\'e en moyenne, avant de
d\'eboucher dans la caverne de Gollum ou a l'air libre? 
\end{enum}

\end{exercice}

\goodbreak

\begin{exercice} 
\label{exo_mf04} 

On consid\`ere une \chaine\ de Markov sur $\cX=\set{1,2,3,4}$, de matrice
de transition
\[
P = 
\begin{pmatrix}
0 & 1/2 & 1/2 & 0 \\
1/16 & 7/16 & 0 & 1/2 \\
1/16 & 0 & 7/16 & 1/2 \\
0 & 1/4 & 1/4 & 1/2 
\end{pmatrix}
\]
\begin{enum}
\item	Montrer que cette \chaine\ est irr\'eductible.
\item	La \chaine\ est-elle r\'eguli\`ere?
\item	Calculer la distribution stationnaire de la \chaine. 
\end{enum}

\end{exercice}

\goodbreak

\begin{exercice}
\label{exo_mf05} 

On consid\`ere la \chaine\ de Markov suivante~:
\medskip

% \centerline{
%  \includegraphics*[clip=true,height=30mm]{figs/chain1_exam_09}
%  }
 \begin{center}
\begin{tikzpicture}[->,>=stealth',shorten >=2pt,shorten <=2pt,auto,node
distance=4.0cm, thick,main node/.style={circle,scale=0.7,minimum size=1.2cm,
fill=blue!20,draw,font=\sffamily\Large}]

  \node[main node] (1) {1};
  \node[main node] (2) [right of=1] {2};
  \node[main node] (3) [right of=2] {3};
  \node[main node] (4) [right of=3] {4};

  \path[every node/.style={font=\sffamily\small}]
    (1) edge [bend left, above] node {$1$} (2)
    (2) edge [bend left, above] node {$1/4$} (3)
    (3) edge [bend left, above] node {$1/2$} (4)
    (4) edge [bend left, below] node {$1$} (3)
    (3) edge [bend left, below] node {$1/2$} (2)
    (2) edge [bend left, below] node {$1/4$} (1)
    (2) edge [loop right, above,distance=1.5cm,out=120,in=60] node {$1/2$} (2)
  ;
\end{tikzpicture}
\end{center}

\begin{enum}
\item 	Donner la matrice de transition $P$ de la \chaine.
\item	La \chaine\ est-elle irr\'eductible?
\item	La \chaine\ est-elle r\'eguli\`ere? 
\item	D\'eterminer la distribution stationnaire de la \chaine.
\item	La \chaine\ est-elle r\'eversible?
\end{enum}

\end{exercice}

\goodbreak

\begin{exercice} 
\label{exo_mf051} 

Soit la cha\^ine de Markov sur $\cX=\set{1,2,3,4}$ dont le graphe est le
suivant:

\begin{center}
\begin{tikzpicture}[->,>=stealth',shorten >=2pt,shorten <=2pt,auto,node
distance=4.0cm, thick,main node/.style={circle,scale=0.7,minimum size=1.2cm,
fill=blue!20,draw,font=\sffamily\Large}]

  \node[main node] (1) {1};
  \node[main node] (3) [above right of=1] {3};
  \node[main node] (2) [below right of=3] {2};
  \node[main node] (4) [below right of=1] {4};

  \path[every node/.style={font=\sffamily\small}]
    (1) edge [bend left,above] node {$1/2$} (2)
    (2) edge [bend left,below] node {$1/2$} (1)
    (1) edge [bend right, below left] node {$1/2$} (4)
    (4) edge [bend right, below right] node {$1/2$} (2)
    (2) edge [bend right, above right] node {$1/2$} (3)
    (3) edge [bend right, above left] node {$1/2$} (1)
    (3) edge [loop left, above,distance=1cm,out=60,in=120] node {$1/2$} (3)
    (4) edge [loop left, below,distance=1cm,out=-120,in=-60] node {$1/2$} (4)
   ;
\end{tikzpicture}
\end{center}

\begin{enum}
\item	La cha\^ine est-elle irr\'eductible?
\item	La cha\^ine est-elle r\'eguli\`ere?
\item 	D\'eterminer la distribution stationnaire $\pi$ de la cha\^ine.
\item 	La cha\^ine est-elle r\'eversible? 
\end{enum}

\end{exercice}

\goodbreak

\begin{exercice}
\label{exo_mf06} 

Le temps qu'il fait dans une certaine ville est class\'e en trois types: Pluie,
neige et beau temps. On suppose que s'il pleut, il y a une chance sur trois
qu'il pleuve le lendemain, une chance sur six qu'il neige, et une chance sur
deux qu'il fasse beau. S'il neige, il y a une chance sur deux qu'il pleuve le
lendemain, et une chance sur deux qu'il neige. S'il fait beau, il y a une
chance sur quatre qu'il fasse beau le lendemain, et trois chance sur quatre
qu'il pleuve. On suppose que le temps du jour $n$ ne d\'epend que du temps
qu'il fait le jour $n-1$. 

\begin{enum}
\item	Formuler le probl\`eme comme une \chaine\ de Markov en temps discret.
De quel type de \chaine\ s'agit-il?
\item	D\'eterminer la probabilit\'e asymptotique qu'il pleuve, qu'il neige ou
qu'il fasse beau.
\item	Quel est l'intervalle moyen entre deux jours de neige?
\end{enum}

\end{exercice}

\goodbreak

\begin{exercice}
\label{exo_mf07} 

Ang\`ele poss\`ede 3 parapluies. Chaque jour, elle va au bureau le matin, et
revient \`a son domicile le soir. Pour chaque trajet, elle emporte avec elle un
parapluie s'il pleut, et s'il y en a au moins un sur place. Elle n'emporte pas
de parapluie s'il ne pleut pas. On suppose que la probabilit\'e qu'il pleuve au
d\'ebut de chaque trajet est de $1/3$, et qu'elle est ind\'ependante de la
m\'et\'eo lors de tous les autres trajets. Soit $X_n$ le nombre de parapluies
qu'Ang\`ele poss\`ede sur place avant de d\'ebuter le $n$i\`eme trajet. 

\begin{enum}
\item	Montrer que $\set{X_n}_n$ est une \chaine\ de Markov, et donner sa
matrice de transition.
\item	De quel type de \chaine\ s'agit-il? 
\item	Quelle est la probabilit\'e, asymptotiquement au bout d'un grand nombre
de voyages, qu'Ang\`ele ne dispose pas de parapluie sur place au moment de
partir?
\item	Quelle est la probabilit\'e asymptotique qu'elle se fasse mouiller
b\^etement, c'est-\`a-dire qu'elle n'ait pas de parapluie \`a sa disposition
alors qu'il pleut d\`es son d\'epart? 
\end{enum}
\end{exercice}

\goodbreak

\begin{exercice}
\label{exo_mf08} 

Montrer que la distribution binomiale est stationnaire
pour le mod\`ele d'Ehrenfest.
\end{exercice}

\goodbreak

\begin{exercice}
\label{exo_mf09} 

Toute matrice stochastique de taille $2\times2$ peut s'\'ecrire
\[
P = 
\begin{pmatrix}
1-p & p \\
q & 1-q
\end{pmatrix}
\qquad
\text{avec $p, q\in[0,1]$. }
\]

\begin{enum}
\item	Discuter, en fonction de $p$ et $q$, quand $P$ est absorbante,
irr\'eductible, r\'eguli\`ere. 
\item	Montrer que la matrice 
\[
\Pi = \frac{1}{p+q}
\begin{pmatrix}
q & p \\
q & p
\end{pmatrix}
\]
est un projecteur ($\Pi^2=\Pi$) qui commute avec $P$. Calculer son noyau et
son image. 

\item	Soit $Q=P-\Pi$. Montrer que $Q\Pi=\Pi Q=0$. Calculer $Q^2$, puis
$Q^n$ pour tout $n$. 

\item	D\'eduire des r\'esultats pr\'ec\'edents $P^n$ pour tout $n$.
Discuter la limite de $P^n$ lorsque $n\to\infty$ en fonction de $p$ et $q$.
\end{enum}
\end{exercice}

\goodbreak

\begin{exercice}
\label{exo_mf10} 

Une puce se d\'eplace sur le r\'eseau ci-dessous, en choisissant \`a chaque
saut l'une des cases adjacentes, au hasard de mani\`ere uniforme. 

% \bigskip
% 
% \centerline{
%  \includegraphics*[clip=true,height=40mm]{figs/lattice}
%  }
 
\begin{center}
\begin{tikzpicture}[-,auto,node
distance=3.0cm, thick,main node/.style={circle,scale=0.7,minimum size=0.5cm,
fill=green!20,draw,font=\sffamily\Large}]

  \node[main node] (00) at(0.0,0.0) {};
  \node[main node] (10) at(-0.5,-0.75) {};
  \node[main node] (11) at(0.5,-0.75) {};
  \node[main node] (20) at(-1.0,-1.5) {};
  \node[main node] (21) at(0.0,-1.5) {};
  \node[main node] (22) at(1.0,-1.5) {};
  \node[main node] (30) at(-1.5,-2.25) {};
  \node[main node] (31) at(-0.5,-2.25) {};
  \node[main node] (32) at(0.5,-2.25) {};
  \node[main node] (33) at(1.5,-2.25) {};
  \node[main node] (40) at(-2.0,-3.0) {};
  \node[main node] (41) at(-1.0,-3.0) {};
  \node[main node] (42) at(0.0,-3.0) {};
  \node[main node] (43) at(1.0,-3.0) {};
  \node[main node] (44) at(2.0,-3.0) {};

  \path[every node/.style={font=\sffamily\small}]
    (00) edge (10) edge (11) 
    (10) edge (20) edge (21) edge (11)
    (11) edge (21) edge (22)
    (20) edge (30) edge (31) edge (21)
    (21) edge (31) edge (32) edge (22)
    (22) edge (32) edge (33)
    (30) edge (40) edge (41) edge (31)
    (31) edge (41) edge (42) edge (32)
    (32) edge (42) edge (43) edge (33)
    (33) edge (43) edge (44)
    (40) edge (41)
    (41) edge (42)
    (42) edge (43)
    (43) edge (44)
   ;
\end{tikzpicture}
\end{center}

\vspace{-3mm}
\noindent
D\'eterminer le temps de r\'ecurrence moyen vers le coin inf\'erieur
gauche.
\end{exercice}

\begin{exercice}
\label{exo_mf11} 
On consid\`ere 
\begin{enum}
\item	un roi;
\item	une dame;
\item	un fou
\end{enum}
se d\'epla\c cant sur un \'echiquier, en
choisissant \`a chaque fois, de mani\`ere \'equiprobable,  l'un des
mouvements permis par les r\`egles des \'echecs.

\begin{center}
\vspace{-5mm}
\chessboard[tinyboard,
  setwhite={Kc4},showmover=false,
  color=red,
  padding=-0.2em,
  pgfstyle=circle,
  markfields={b3,b4,b5,c3,c5,d3,d4,d5}
  ]
\hspace{10mm}
\chessboard[tinyboard,
  setwhite={Qc4},showmover=false,
  color=red,
  padding=-0.2em,
  pgfstyle=circle,
  markfields={b3,a2,b5,a6,d3,e2,f1,d5,e6,f7,g8,c1,c2,c3,c5,c6,c7,c8,a4,b4,d4,
  e4,f4,g4,h4}
  ]
\hspace{10mm}
\chessboard[tinyboard,
  setwhite={bc4},showmover=false,
  color=red,
  padding=-0.2em,
  pgfstyle=circle,
  markfields={b3,a2,b5,a6,d3,e2,f1,d5,e6,f7,g8}
  ]
\end{center}

\vspace{-3mm}
\noindent
D\'eterminer dans chaque cas le temps de r\'ecurrence moyen vers le coin
inf\'erieur gauche. 
\end{exercice}

\begin{exercice}
\label{exo_mf12} 

Soit $X_n$ la \chaine\ de Markov sur $\cX=\set{1,2,\dots,N}$ de probabilit\'es
de transition
\[
p_{ij} = 
\begin{cases}
p & \text{si $j=i+1$ ou $i=N$ et $j=1$\;} \\
1-p & \text{si $j=i-1$ ou $i=1$ et $j=N$\;.} 
\end{cases}
\]
\begin{enum}
 \item	D\'eterminer la distribution stationnaire de la \chaine.
 \item	Pour quelles valeurs de $p$ la \chaine\ est-elle r\'eversible?
\end{enum}
\end{exercice}

\begin{exercice}
\label{exo_mf13} 

Soit $\cG=(V,E)$ un graphe non orient\'e connexe fini. Soit $X_n$ la \chaine\ de
Markov sur $V$ construite en choisissant pour $X_{n+1}$, de mani\`ere
\'equiprobable, l'un des sommets adjacents \`a $X_n$. 
\begin{enum}
\item	Montrer que le nombre de voisins de chaque site forme un vecteur
r\'eversible.
\item	En d\'eduire une expression pour la distribution stationnaire de la
\chaine.
\end{enum}
\end{exercice}

%%%%%%%%%%%%%%%%%%%%%%%%%%%%%%%%%%%%%%%%%%%%%%%%%%%%%%%%%%%%%%%%%%%%%%%%%%%

%\end{document}

\chapter{Cha\^\i nes de Markov sur un ensemble d\'enombrable}
\label{chap_den}

%%%%%%%%%%%%%%%%%%%%%%%%%%%%%%%%%%%%%%%%%%%%%%%%%%%%%%%%%%%%%%%%%%%%%%%%%%%

\section{Marches al\'eatoires}
\label{sec_rw}

Les marches al\'eatoires constituent un exemple relativement simple,
et n\'eanmoins tr\`es important de \chaine s de Markov sur un ensemble
d\'enombrable infini. Dans ce cas, en effet, $\cX=\Z^d$ est un r\'eseau
infini, de dimension $d\in\N^*$. D'habitude, on consid\`ere que la
\chaine\ d\'emarre en $X_0=0$. Ensuite, elle choisit \`a chaque instant
l'un des $2d$ sites voisins, selon une loi fix\'ee d'avance. 

\begin{definition}
\label{def_rw1}
Une \defwd{marche al\'eatoire}\/ sur $\Z^d$ est une \chaine\ de Markov \`a
valeurs dans $\Z^d$, de distribution initiale $\nu=\delta_0$, et de
probabilit\'es de transition satisfaisant 
\begin{equation}
\label{rw1}
p_{ij} = 0
\qquad
\text{si $i=j$ ou $\norm{i-j}>1$\;.}
\end{equation}
La marche est dite \defwd{sym\'etrique}\/ si 
\begin{equation}
\label{rw2}
p_{ij} = \frac1{2d}
\qquad
\text{pour $\norm{i-j}=1$\;.}
\end{equation}
\end{definition}

\begin{figure}
%  \centerline{
%  \includegraphics*[clip=true,height=50mm]{figs/rw2d}
%  }
%  \figtext{
%  }
 \begin{center}
\begin{tikzpicture}[-,scale=0.5,auto,node
distance=1.0cm, thick,main node/.style={draw,circle,fill=white,minimum
size=3pt,inner sep=0pt}]

  \path[->,>=stealth'] 
     (-4,0) edge (8,0)
     (0,-5) edge (0,3)
  ;

  \draw[very thick] (0,0) node[main node,thick] {} 
  -- (0,1) node[main node,thick] {} 
  -- (1,1) node[main node,thick] {} 
  -- (1,0) node[main node,thick] {} 
  -- (2,0) node[main node,thick] {} 
  -- (2,-1) node[main node,thick] {} 
  -- (1,-1) node[main node,thick] {} 
  -- (1,-2) node[main node,thick] {} 
  -- (2,-2) node[main node,thick] {} 
  -- (2,-3) node[main node,thick] {} 
  -- (1,-3) node[main node,thick] {} 
  -- (0,-3) node[main node,thick] {} 
  -- (-1,-3) node[main node,thick] {} 
  -- (-2,-3) node[main node,thick] {} 
  -- (-2,-2) node[main node,thick] {} 
  -- (-1,-2) node[main node,thick] {} 
  -- (-1,-3) node[main node,thick] {} 
  -- (-1,-4) node[main node,thick] {} 
  -- (0,-4) node[main node,thick] {} 
  -- (0,-3) node[main node,thick] {} 
  -- (1,-3) node[main node,thick] {} 
  -- (1,-4) node[main node,thick] {} 
  -- (2,-4) node[main node,thick] {} 
  -- (3,-4) node[main node,thick] {} 
  -- (4,-4) node[main node,thick] {} 
  -- (5,-4) node[main node,thick] {} 
  -- (5,-3) node[main node,thick] {} 
  -- (5,-2) node[main node,thick] {} 
  -- (4,-2) node[main node,thick] {} 
  -- (4,-3) node[main node,thick] {} 
  -- (5,-3) node[main node,thick] {} 
  -- (6,-3) node[main node,thick] {} 
  ;
\end{tikzpicture}
\end{center}
\vspace{-5mm}
 \caption[]{Une trajectoire d'une marche al\'eatoire en dimension $d=2$.}
 \label{fig_rw2d}
\end{figure}
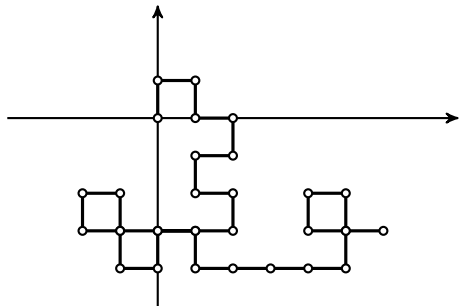

Les trajectoires de la marche al\'eatoire sont des suites de points de
$\Z^d$ \`a distance $1$, qu'on a coutume d'identifier \`a la ligne
bris\'ee reliant ces points (\figref{fig_rw2d}). 

Dans le cas sym\'etrique, il suit directement du
Th\'eor\`eme~\ref{thm_fdef1} que chaque segment de trajectoire $X_{[0,n]}$
a probabilit\'e~$(2d)^{-n}$. On peut facilement d\'eterminer quelques
propri\'et\'es de la loi de $X_n$. 

\begin{prop}
\label{prop_rw1}
Pour la marche al\'eatoire sym\'etrique sur $\Z^d$, les variables
al\'ea\-toires $X_n$ satisfont 
\begin{equation}
\label{rw3}
\expec{X_n} = 0
\qquad
\text{et}
\qquad
\cov(X_n) = \frac nd I
\end{equation}
pour tout temps $n$.
De plus, lorsque $n\to\infty$ on a 
\begin{equation}
\label{rw4}
\frac{X_n}{\sqrt{n}} \stackrel{\cL}{\to} \cN\Bigpar{0,\frac1d I}\;,
\end{equation}
o\`u $\cL$ d\'esigne la convergence en loi, et $\cN(0,\Sigma)$ d\'enote la
loi normale centr\'ee de matrice de covariance $\Sigma$.
\end{prop}

\begin{proof}
On v\'erifie facilement que les variables al\'eatoires $Y_n=X_n-X_{n-1}$
sont i.i.d., d'esp\'erance nulle et de matrice de covariance $\frac1d I$,
ce qui implique~\eqref{rw3}. La relation \eqref{rw4} suit
alors directement du th\'eor\`eme central limite.  
\end{proof}

En cons\'equence, la position de la marche al\'eatoire au temps $n$ se
trouvera avec grande probabilit\'e dans une boule de rayon d'ordre
$\sqrt{n}$ autour de l'origine. On dit que la marche al\'eatoire a un
comportement \defwd{diffusif} (par opposition \`a \defwd{ballistique},
o\`u la distance \`a l'origine cro\^\i trait proportionnellement \`a $n$).

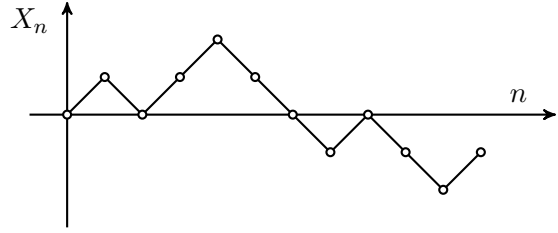
\begin{figure}
%  \centerline{
%  \includegraphics*[clip=true,height=25mm]{figs/marche1}
%  }
%  \figtext{
%  	\writefig	3.5	2.5	{$X_n$}
%  	\writefig	10.4	1.8	{$n$}
%  }
\begin{center}
\begin{tikzpicture}[-,scale=0.5,auto,node
distance=1.0cm, thick,main node/.style={draw,circle,fill=white,minimum
size=3pt,inner sep=0pt}]

  \path[->,>=stealth'] 
     (-1,0) edge (13,0)
     (0,-3) edge (0,3)
  ;

  \node at (12.0,0.5) {$n$};
  \node at (-1.0,2.5) {$X_n$};

  \draw (0,0) node[main node] {} 
  -- (1,1) node[main node] {} 
  -- (2,0) node[main node] {} 
  -- (3,1) node[main node] {} 
  -- (4,2) node[main node] {} 
  -- (5,1) node[main node] {} 
  -- (6,0) node[main node] {} 
  -- (7,-1) node[main node] {} 
  -- (8,0) node[main node] {} 
  -- (9,-1) node[main node] {} 
  -- (10,-2) node[main node] {} 
  -- (11,-1) node[main node] {} 
  ;
\end{tikzpicture}
\end{center}
\vspace{-5mm}
 \caption[]{Une r\'ealisation d'une marche al\'eatoire
unidimensionnelle.}
 \label{fig_marche1}
\end{figure}

Nous consid\'erons maintenant plus particuli\`erement le cas de la marche
al\'eatoire unidimensionnelle ($d=1$) sym\'etrique. Dans ce cas, on voit
facilement que la loi de $X_n$ est binomiale centr\'ee :
\begin{equation}
\prob{X_n=k} 
= \frac1{2^n}\binom{n}{\frac{n+k}2}
\qquad
\forall k\in\set{-n,-n+2,\dots,n-2,n}\;.
\label{mas3}
\end{equation}
En particulier, la probabilit\'e que le processus se trouve en $0$ au
$n^{\text{\`eme}}$ pas est donn\'ee par 
\begin{equation}
\label{mas4}
\prob{X_n=0} = 
\begin{cases}
0 & \text{si $n$ est impair\;,} \\
\dfrac{(2m)!}{2^{2m}(m!)^2}
& \text{si $n=2m$ est pair\;.}
\end{cases}
\end{equation}
Remarquons que la formule de Stirling implique que pour $m$ grand, 
\begin{equation}
\label{mas5}
\prob{X_{2m}=0} \sim \frac1{2^{2m}}
\frac{\sqrt{4\pi m}\e^{-2m}(2m)^{2m}}{2\pi m\e^{-2m}m^{2m}}
= \frac1{\sqrt{\pi m}}\;.
\end{equation}
En tout temps pair, l'origine est l'endroit le plus probable o\`u trouver
la marche, mais cette probabilit\'e d\'ecro\^\i t avec le temps.

\begin{figure}[t]
%  \centerline{
%  \includegraphics*[clip=true,height=30mm]{figs/marche2} 
%  }
%  \figtext{
%  	\writefig	3.5	3.0	{$X_n$}
%  	\writefig	10.0	1.35	{$\tau_0$}
%  	\writefig	11.0	1.35	{$n$}
%  }
\begin{center}
\begin{tikzpicture}[-,scale=0.5,auto,node
distance=1.0cm, thick,main node/.style={draw,circle,fill=white,minimum
size=3pt,inner sep=0pt}]

  \path[->,>=stealth'] 
     (-1,0) edge (14,0)
     (0,-3) edge (0,4)
  ;

  \node at (13.0,0.5) {$n$};
  \node at (-1.0,2.5) {$X_n$};
  \node at (12.0,-0.7) {$\tau_0$};

  \draw (0,0) node[main node] {} 
  -- (1,1) node[main node] {} 
  -- (2,2) node[main node] {} 
  -- (3,1) node[main node] {} 
  -- (4,2) node[main node] {} 
  -- (5,3) node[main node] {} 
  -- (6,2) node[main node] {} 
  -- (7,3) node[main node] {} 
  -- (8,2) node[main node] {} 
  -- (9,1) node[main node] {} 
  -- (10,2) node[main node] {} 
  -- (11,1) node[main node] {} 
  -- (12,0) node[main node] {} 
  ;
\end{tikzpicture}
\end{center}
\vspace{-5mm}
 \caption[]{Une r\'ealisation d'une marche al\'eatoire
 unidimensionnelle pour laquelle $\tau_0=12$.}
 \label{fig_marche2}
\end{figure}
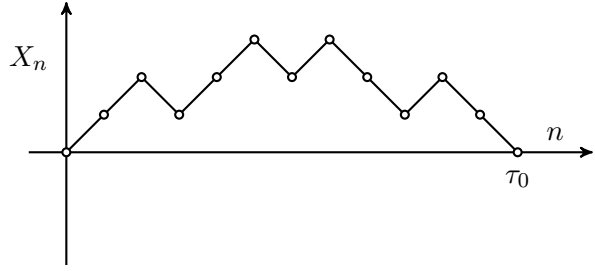

Cependant, la loi de chaque $X_n$ ne d\'etermine pas le processus, les
$X_n$ n'\'etant pas ind\'e\-pendants, et on peut \'etudier beaucoup
d'autres propri\'et\'es de la marche al\'eatoire. Une premi\`ere
quantit\'e int\'eressante est le temps $\tau_0$ du premier retour du
processus en 0 (\figref{fig_marche2}) : 
\begin{equation}
\label{mas9}
\tau_0 = \inf\setsuch{n\geqs 1}{X_n=0}\;.
\end{equation}
Il est clair que $\tau_0$ ne peut
prendre que des valeurs paires. De plus, si $\tau_0=n$ alors $X_n=0$, donc 
$\prob{\tau_0=n}\leqs\prob{X_n=0}$. En fait, il nous faut d\'eterminer 
\begin{equation}
\label{mas10}
\prob{\tau_0=n} 
= \prob{X_1\neq0,X_2\neq0,\dots,X_{n-1}\neq0,X_n=0}\;.
\end{equation}

\begin{theorem}
\label{thm_mas1}
La loi de $\tau_0$ est donn\'ee par  
\begin{equation}
\label{mas11}
\prob{\tau_0=n} = 
\begin{cases}
0 & \text{pour $n$ impair\;,} \\
\myvrule{18pt}{6pt}{0pt} 
\dfrac1n \prob{X_{n-2}=0} & \text{pour $n$ pair\;.}
\end{cases}
\end{equation}
\end{theorem}

\begin{proof}
Supposons que $\tau_0=n$. 
Comme le processus ne peut pas changer de signe sans passer par $0$, on a 
\begin{align}
\nonumber
\prob{\tau_0=n} 
={}& \prob{X_1>0,X_2>0,\dots,X_{n-1}>0,X_n=0} \\
\nonumber
&{}+
\prob{X_1<0,X_2<0,\dots,X_{k-1}<0,X_n=0} \\
\nonumber
={}& 2 \prob{X_1>0,X_2>0,\dots,X_{n-1}>0,X_n=0} \\
\nonumber
={}& 2 \prob{X_1=1,X_2>0,\dots,X_{n-2}>0,X_{n-1}=1,X_n=0} \\
%\nonumber
={}& 2 \pcond{X_n=0}{X_{n-1}=1} 
\prob{X_1=1,X_2>0,\dots,X_{n-2}>0,X_{n-1}=1} \;,
\label{mas11a}
\end{align}
o\`u nous avons utilis\'e la propri\'et\'e de Markov dans la derni\`ere
ligne. La propri\'et\'e des incr\'ements stationnaires (cf.~\eqref{fdef12})
implique 
\begin{equation}
\pcond{X_n=0}{X_{n-1}=1} = \prob{X_1=-1} = \frac12\;.
\label{mas11b}
\end{equation}
Il suit que 
\begin{align}
\label{mas11c}
\prob{\tau_0=n} 
&= \prob{X_1=1,X_2>0,\dots,X_{n-2}>0,X_{n-1}=1}\\
&= \prob{X_1=X_{n-1}=1} - \prob{X_1=X_{n-1}=1,\,
\exists m\in\set{2,\dots,n-2}:\,X_m=0}\;.
\nonumber
\end{align}
Nous utilisons maintenant un argument important, appel\'e le
\defwd{principe de r\'eflexion}\/~: A tout chemin allant de $(1,1)$ \`a
$(n-1,1)$ passant par $0$, on peut faire correspondre un unique chemin de
$(-1,1)$ \`a $(n-1,1)$, obtenu en r\'efl\'echissant par rapport \`a l'axe
des abscisses la partie du chemin ant\'erieure au premier passage en $0$
(\figref{fig_marche3}).
On a donc 
\begin{equation}
\label{mas11d}
\prob{X_1=X_{n-1}=1,\,\exists m\in\set{2,\dots,n-2}:\,X_m=0}
= \prob{X_1=-1,X_{n-1}=1}\;.
\end{equation}
Finalement, en appliquant de nouveau la propri\'et\'e des incr\'ements
stationnaires, on a
\begin{align}
\nonumber
\prob{X_1=1,X_{n-1}=1} &= \pcond{X_{n-1}=1}{X_1=1}\prob{X_1=1}
= \prob{X_{n-2}=0}\cdot\frac12\;, \\
\prob{X_1=-1,X_{n-1}=1} &= \pcond{X_{n-1}=1}{X_1=-1}\prob{X_1=-1}
= \prob{X_{n-2}=2}\cdot\frac12\;.
\label{mas11e}
\end{align}
En rempla\c cant dans~\eqref{mas11c}, il vient 
\begin{equation}
\label{mas11f}
\prob{\tau_0=n} = \frac12 \bigbrak{\prob{X_{n-2}=0} - \prob{X_{n-2}=2}}\;.
\end{equation}
Le reste de la preuve est un calcul direct. Comme
\begin{equation}
\label{mas11g}
\frac{\prob{X_{n-2}=2}}{\prob{X_{n-2}=0}}
= \frac{\binom{n-2}{n/2}}{\binom{n-2}{n/2-1}}
= \frac{\bigpar{\frac n2-1}!\bigpar{\frac n2-1}!}
{\bigpar{\frac n2}!\bigpar{\frac n2-2}!}
= \frac{\frac n2 - 1}{\frac n2} = 1 - \frac2n\;,
\end{equation}
on obtient 
\begin{equation}
\label{mas11h}
\prob{\tau_0=n} = \frac12 \prob{X_{n-2}=0}\biggbrak{1 - 1 + \frac2n} 
= \frac1n \prob{X_{n-2}=0}\;,
\end{equation}
ce qui conclut la d\'emonstration. 
\end{proof}

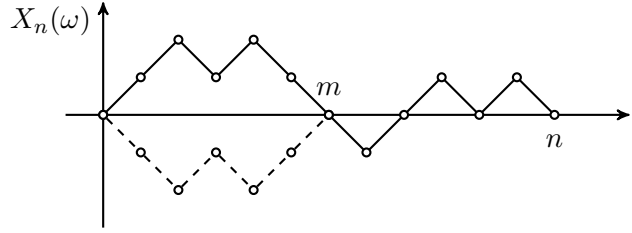
\begin{figure}
%  \centerline{
%  \includegraphics*[clip=true,height=30mm]{figs/marche3} 
%  }
%  \figtext{
%  	\writefig	3.0	3.0	{$X_n(\omega)$}
%  	\writefig	10.0	1.35	{$n$}
%  	\writefig	7.1	1.9	{$m$}
%  }
 \begin{center}
\begin{tikzpicture}[-,scale=0.5,auto,node
distance=1.0cm, thick,main node/.style={draw,circle,fill=white,minimum
size=3pt,inner sep=0pt}]

  \path[->,>=stealth'] 
     (-1,0) edge (14,0)
     (0,-3) edge (0,3)
  ;

  \node at (6.0,0.7) {$m$};
  \node at (-1.4,2.5) {$X_n(\omega)$};
  \node at (12.0,-0.7) {$n$};

  \draw (0,0) node[main node] {} 
  -- (1,1) node[main node] {} 
  -- (2,2) node[main node] {} 
  -- (3,1) node[main node] {} 
  -- (4,2) node[main node] {} 
  -- (5,1) node[main node] {} 
  -- (6,0) node[main node] {} 
  -- (7,-1) node[main node] {} 
  -- (8,0) node[main node] {} 
  -- (9,1) node[main node] {} 
  -- (10,0) node[main node] {} 
  -- (11,1) node[main node] {} 
  -- (12,0) node[main node] {} 
  ;
  
  \node[main node] (0) at (0,0) {};
  \node[main node] (1) at (1,-1) {};
  \node[main node] (2) at (2,-2) {};
  \node[main node] (3) at (3,-1) {};
  \node[main node] (4) at (4,-2) {};
  \node[main node] (5) at (5,-1) {};
  \node[main node] (6) at (6,-0) {};
  
  \path[dashed]
    (0) edge (1)
    (1) edge (2)
    (2) edge (3)
    (3) edge (4)
    (4) edge (5)
    (5) edge (6)
  ;
\end{tikzpicture}
\end{center}
\vspace{-3mm}
 \caption[]{Pour chaque r\'ealisation d'une marche al\'eatoire avec
 $\tau_0=m<n$ telle que $X_1=1$, il existe une autre r\'ealisation telle
 que $\tau_0=m$ et $X_1=-1$, obtenue par r\'eflexion par rapport \`a
 l'axe des abscisses.}
 \label{fig_marche3}
\end{figure}

Le tableau suivant donne les premi\`eres valeurs de la loi et de la fonction
de r\'epartition de $\tau_0$ :
\begin{center}
\begin{tabular}{|c||c|c|c|c|c|c|c|}
\hline
\myvrule{10pt}{0pt}{0pt}
$n$ & $2$ & $4$ & $6$ & $8$ & $10$ & $12$ & $14$ \\
\hline
\myvrule{12pt}{8pt}{0pt}
$\prob{\tau_0=n}$ 
& $\frac12$ & $\frac18$ & $\frac1{16}$ & $\frac5{128}$ & $\frac7{256}$ &
$\frac{21}{1024}$ & $\frac{33}{2048}$ \\
\myvrule{8pt}{6pt}{0pt}
& $=0.5$ & $=0.125$ & $\cong0.063$ & $\cong0.039$ & $\cong0.027$ 
& $\cong0.021$ & $\cong0.016$ \\
\hline
\myvrule{12pt}{8pt}{0pt}
$\prob{\tau_0\leqs n}$ 
& $=0.5$ & $=0.625$ & $\cong0.688$ & $\cong0.727$ & $\cong0.754$ 
& $\cong0.774$ & $\cong0.791$ \\
\hline
\end{tabular}
\end{center}
Il est donc assez probable de revenir rapidement en $0$, puis la loi prend
des valeurs plut\^ot faibles, tout en d\'ecroissant lentement. Il suit
de~\eqref{mas5} que pour des grands $n$, $\prob{\tau_0=n}$ d\'ecro\^\i t
comme
$1/n^{3/2}$. Ce fait a une cons\'equence surprenante :

\goodbreak

\begin{cor}
\label{cor_mas}
$\expec{\tau_0}=+\infty$. 
\end{cor}
\begin{proof}
On a 
\begin{equation}
\label{mas12}
\expec{\tau_0} = \sum_{n\geqs 1}n\prob{\tau_0=n} 
= \sum_{m\geqs1} 2m \frac1{2m} \prob{X_{2m-2}=0} 
\sim \sum_{m\geqs1} \frac1{\sqrt{\pi m}} 
= +\infty\;.
\end{equation}
\end{proof}

En d'autres termes, la marche al\'eatoire finit toujours par revenir en
$0$, mais la loi de $\tau_0$ d\'ecro\^\i t trop lentement pour que son
esp\'erance soit finie. Cela est li\'e au fait que si la marche al\'eatoire
s'\'eloigne beaucoup de $0$, il lui faut longtemps pour y revenir. 

Par un raisonnement analogue, on peut d\'eterminer la loi du temps de
passage  
\begin{equation}
\label{mas13}
\tau_i = \inf\setsuch{n\geqs0}{X_n=i}
\end{equation}
Nous donnons simplement le r\'esultat, la d\'emonstration est laiss\'ee en
exercice. 

\begin{theorem}
\label{thm_mas2}
La loi de $\tau_i$ est donn\'ee par 
\begin{equation}
\label{mas14}
\prob{\tau_i=n} = 
\begin{cases}
\myvrule{18pt}{16pt}{0pt} 
\dfrac{\abs{i}}n \prob{X_n=i} & \text{pour $n\in\set{\abs{i}, \abs{i}+2,
\dots}$\;,} \\
0 & \text{sinon\;.} 
\end{cases}
\end{equation}
\end{theorem}

Pour des raisons similaires \`a celles du cas du retour en $0$, la loi de
$\tau_L$ d\'ecro\^\i t en $1/n^{3/2}$, et son esp\'erance est donc infinie.

%%%%%%%%%%%%%%%%%%%%%%%%%%%%%%%%%%%%%%%%%%%%%%%%%%%%%%%%%%%%%%%%%%%%%%%%%%%

\section{G\'en\'eralit\'es sur les processus stochastiques}
\label{sec_gps}

Jusqu'\`a pr\'esent, nous avons parl\'e de \chaine s de Markov, qui sont
des processus stochastiques particuliers, sans pr\'eciser l'espace
probabilis\'e sous-jacent. Cela ne pose pas de probl\`eme tant qu'on parle
de segments de trajectoire finis : Il suffit de consid\'erer des espaces
produits comprenant un nombre fini de termes. Le cas de trajectoires
infinies n\'ecessite quelques pr\'ecautions suppl\'ementaires. Nous
donnons ici un survol de la cons\-truction g\'en\'erale des processus
stochastiques en temps discret.

Soit $(E,\cE)$ un espace mesurable. L'ensemble $E^n = E\times E\times
\dots \times E$ peut \^etre muni
d'une tribu $\cE^{\otimes n}$, d\'efinie comme la tribu engendr\'ee par
tous les \'ev\'enements du type 
\begin{equation}
\label{gps1}
A^{(i)} = \setsuch{\omega\in E^n}{\omega_i\in A}\;, 
\qquad A\in\cE\;,
\end{equation}
appel\'es \defwd{cylindres}.
On d\'enote par $E^\N$ l'ensemble des applications $x: \N\to E$,
c'est-\`a-dire l'ensemble des suites $(x_0, x_1, x_2, \dots)$ \`a valeurs
dans $E$. Cet ensemble peut \`a nouveau \^etre muni d'une tribu construite
\`a partir de tous les cylindres, not\'ee $\cE^{\otimes \N}$.

\begin{definition}
\label{def_gps1}
\hfill
\begin{itemiz}
\item	Un \defwd{processus stochastique}\/ \`a valeurs dans $(E,\cE)$ est
une suite $\set{X_n}_{n\in\N}$ de variables al\'eatoires \`a valeurs dans
$(E,\cE)$, d\'efinies sur un m\^eme espace probabilis\'e $(\Omega,\cF,\fP)$
(autrement dit, chaque $X_n$ est une application $\cF$-$\cE$-mesurable de
$\Omega$ dans $E$). C'est donc \'egalement une variable al\'eatoire \`a
valeurs dans $(E^\N,\cE^{\otimes\N})$. 

\item	Soit $\Q$ une mesure de probabilit\'e sur
$(E^\N,\cE^{\otimes\N})$. Les \defwd{distributions de dimension finie}\/
de $\Q$ sont les mesures sur $(E^{n+1},\cE^{\otimes n+1})$ d\'efinies par 
\begin{equation}
\label{gps2}
\Q^{(n)} = \Q \circ (\pi^{(n)})^{-1}\;,
\end{equation}
o\`u $\pi^{(n)}$ est la projection $\pi^{(n)}:E^\N\to E^{n+1}$, $(x_0, x_1,
x_2, \dots) \mapsto (x_0, \dots, x_n)$.
\end{itemiz}
\end{definition}

On se convainc facilement que la suite des $\set{\Q^{(n)}}_{n\in\N}$
d\'etermine $\Q$ univoquement. Inversement, pour qu'une suite donn\'ee
$\set{\Q^{(n)}}_{n\in\N}$ corresponde effectivement \`a une mesure $\Q$,
les $\Q^{(n)}$ doivent satisfaire une \defwd{condition de compatibilit\'e}
: 

Soit $\ph_n$ la projection $\ph_n:E^{n+1}\to E^n$, $(x_0, \dots, x_n)
\mapsto (x_0, \dots, x_{n-1})$. Alors on a
$\pi^{(n-1)}=\ph_n\circ\pi^{(n)}$, donc pour tout $A\in\cE^{\otimes n}$, 
$(\pi^{(n-1)})^{-1}(A) = (\pi^{(n)})^{-1}(\ph_n^{-1}(A))$. La condition
de compatibilit\'e s'\'ecrit donc 
\begin{equation}
\label{gps3}
\Q^{(n-1)} = \Q^{(n)} \circ \ph_n^{-1}\;.
\end{equation}
Le diagramme suivant illustre la situation (toutes les projections
\'etant mesurables, on peut les consid\'erer \`a la fois comme
applications entre ensembles et entre tribus) : 

% \begin{minipage}{14.0cm}
% \begin{center}
%  \vspace{7mm}
%  \includegraphics*[clip=true,height=40mm]{figs/cd1} \\
%  \figtext{
%  	\writefig	5.25	4.65	{$(E^\N,\cE^{\otimes\N})$}
%  	\writefig	6.2	3.6	{$\pi^{(n)}$}
%  	\writefig	8.4	3.8	{$\Q$}
%  	\writefig	8.0	2.6	{$\Q^{(n)}$}
%  	\writefig	3.15	2.35	{$\pi^{(n-1)}$}
%  	\writefig	4.75	2.35	{$(E^{n+1},\cE^{\otimes n+1})$}
%  	\writefig	9.7	2.35	{$[0,1]$}
%  	\writefig	6.2	1.25	{$\ph_n$}
%  	\writefig	8.4	0.95	{$\Q^{(n-1)}$}
%  	\writefig	5.25	0.1	{$(E^n,\cE^{\otimes n})$}
%  }
%  \vspace{5mm}
% \end{center}
% \end{minipage}

\begin{center}
\begin{tikzpicture}[->,>=stealth',shorten >=2pt,shorten <=2pt,auto,node
distance=2.5cm, thick
]

  \node (N) {$(E^\N,\cE^{\otimes\N})$};
  \node (n+1) [below of=N] {$(E^{n+1},\cE^{\otimes n+1})$};
  \node (n) [below of=n+1] {$(E^n,\cE^{\otimes n})$};
  \node[node distance=0.9cm] (N-) [left of=N] {};
  \node[node distance=0.9cm] (n-) [left of=n] {};
  \node[node distance=0.8cm] (N+) [right of=N] {};
  \node[node distance=0.8cm] (n+) [right of=n] {};
  \node[node distance=1.1cm] (n+1+) [right of=n+1] {};
  \node[node distance=4.5cm] (01) [right of=n+1] {$[0,1]$};

   \path
     (N) edge [right] node {$\pi^{(n)}$} (n+1)
     (n+1) edge [right] node {$\ph_n$} (n)
     (N-) edge [bend right, distance=2.0cm, left] node {$\pi^{(n-1)}$} (n-)
     (N+) edge [above right] node {$\Q$} (01)
     (n+1+) edge [above] node {$\Q^{(n)}$} (01)
     (n+) edge [below right] node {$\Q^{(n-1)}$} (01)
;
\end{tikzpicture}
\end{center}

\noindent
Nous allons voir comment construire une suite de $\Q^{(n)}$ satisfaisant
la condition~\eqref{gps3}. 

\begin{definition}
\label{def_gps2}
Soient $(E_1,\cE_1)$ et $(E_2,\cE_2)$ deux espaces mesurables. Un
\defwd{noyau marko\-vien de $(E_1,\cE_1)$ vers $(E_2,\cE_2)$}\/  est une
application $K: E_1\times\cE_2 \to [0,1]$ satisfaisant les deux conditions 
\begin{enum}
\item	Pour tout $x\in E_1$, $K(x,\cdot)$ est une mesure de
probabilit\'e sur $(E_2,\cE_2)$. 
\item	Pour tout $A\in\cE_2$, $K(\cdot,A)$ est une application
$\cE_1$-mesurable. 
\end{enum}
\end{definition}

\begin{example}
\label{ex_gps1}
\hfill
\begin{enum}
\item	Soit $\mu$ une mesure de probabilit\'e sur $(E_2,\cE_2)$. Alors
$K$ d\'efini par $K(x,A)=\mu(A)$ pour tout $x\in E_1$ est un noyau
markovien. 

\item	Soit $f:E_1\to E_2$ une application mesurable. Alors $K$ d\'efini
par  $K(x,A)=\indicator{A}(f(x))$ est un noyau markovien.

\item	Soit $\cX=\set{1,\dots,N}$ un ensemble fini, et posons
$E_1=E_2=\cX$ et $\cE_1=\cE_2=\cP(\cX)$. Alors $K$ d\'efini par 
\begin{equation}
\label{gps4}
K(i,A) = \sum_{j\in A} p_{ij}\;,
\end{equation}
o\`u $P=(p_{ij})_{i,j\in\cX}$ est une matrice stochastique, est un noyau
markovien.
\end{enum}
\end{example}

Si $\mu$ est une mesure de probabilit\'e sur $(E_1,\cE_1)$ et $K$ est un
noyau markovien de $(E_1,\cE_1)$ vers $(E_2,\cE_2)$, on d\'efinit une
mesure de probabilit\'e $\mu\otimes K$ sur $\cE_1\otimes\cE_2$ par 
\begin{equation}
\label{gps5}
(\mu\otimes K)(A) \defby \int_{E_1} K(x_1,A_{x_1})
\mu(\6x_1)\;,
\end{equation}
o\`u $A_{x_1} = \setsuch{x_2\in E_2}{(x_1,x_2)\in A} \in \cE_2$ est la
\defwd{section}\/ de $A$ en $x_1$. On v\'erifie que c'est bien une mesure
de probabilit\'e. Afin de comprendre sa signification, calculons ses
marginales. Soient $\pi_1$ et $\pi_2$ les projections d\'efinies par
$\pi_i(x_1,x_2)=x_i$, $i=1,2$. 
\begin{enum}
\item	Pour tout ensemble mesurable $A_1\in\cE_1$, on a 
\begin{align}
\nonumber
\bigpar{(\mu\otimes K)\circ\pi_1^{-1}}(A_1) 
&= (\mu\otimes K)(A_1\times E_2) \\
\nonumber
&= \int_{E_1} \indicator{A_1}(x_1)K(x_1,E_2) \mu(\6x_1) \\
&= \int_{A_1} K(x_1,E_2) \mu(\6x_1)
= \mu(A_1)\;,
\label{gps6}
\end{align}
o\`u on a utilis\'e le fait que la section $(A_1\times\cE_1)_{x_1}$ est
donn\'ee par $E_2$ si $x_1\in A_1$, et $\emptyset$ sinon. Ceci implique 
\begin{equation}
\label{gps7}
(\mu\otimes K)\circ\pi_1^{-1} = \mu\;.
\end{equation}
La premi\`ere marginale de $\mu\otimes K$ est donc simplement $\mu$. 

\item	Pour tout ensemble mesurable $A_2\in\cE_2$, on a 
\begin{align}
\nonumber
\bigpar{(\mu\otimes K)\circ\pi_2^{-1}}(A_2) 
&= (\mu\otimes K)(E_1\times A_2) \\
&= \int_{E_1} K(x_1,A_2) \mu(\6x_1) \;.
\label{gps8}
\end{align}
La seconde marginale de $\mu\otimes K$ s’interpr\`ete comme suit: c'est la
mesure sur $E_2$ obtenue en partant avec la mesure $\mu$ sur $E_1$, et en
\lq\lq allant de tout $x\in E_1$ vers $A_2\in\cE_2$ avec
probabilit\'e $K(x_1,A_2)$\rq\rq. 
\end{enum}

Enfin, par une variante du th\'eor\`eme de Fubini--Tonelli, on v\'erifie
que  pour toute fonction $(\mu\otimes K)$-int\'egrable $f:E_1\times
E_2\to\R$, on a  
\begin{equation}
\label{gps9}
\int_{E_1\times E_2} f \6\,(\mu\otimes K) 
= \int_{E_1} \biggpar{\int_{E_2} f(x_1,x_2) K(x_1,\6x_2)} \mu(\6x_1)\;.
\end{equation}

Nous pouvons maintenant proc\'eder \`a la construction de la suite
$\set{\Q^{(n)}}_{n\in\N}$ de distributions de dimension finie,
satisfaisant la condition de compatibilit\'e~\eqref{gps3}. Sur
l'espace mesurable $(E,\cE)$, on se donne une mesure de probabilit\'e
$\nu$, appel\'ee \defwd{mesure initiale}. On se donne pour tout $n\in\N$
un noyau markovien $K_n$ de $(E^{n+1},\cE^{\otimes {n+1}})$ vers $(E,\cE)$. 
On d\'efinit alors la suite $\set{\Q^{(n)}}_{n\in\N}$ de mesures de
probabilit\'e sur $(E^{n+1},\cE^{n+1})$ r\'ecursivement par 
\begin{align}
\nonumber
\Q^{(0)} &= \nu \;,\\
\Q^{(n)} &= \Q^{(n-1)}\otimes K_{n-1}\;, 
& 
n&\geqs 2\;.
\label{gps10}
\end{align}
Par~\eqref{gps7}, on a $\Q^{(n)}\circ\ph_n^{-1}=(\Q^{(n-1)}\otimes
K_{n-1})\circ\ph_n^{-1}=\Q^{(n-1)}$, donc la condition de compatibilit\'e
est bien satisfaite. 

L'interpr\'etation de~\eqref{gps10} est simplement que chaque noyau
markovien $K_n$ d\'ecrit les probabilit\'es de transition entre les temps
$n$ et $n+1$, et permet ainsi de d\'efinir une mesure sur les
segments de trajectoire plus longs d'une unit\'e. Remarquons enfin qu'on
peut \'egalement construire pour tout $m, n$ un noyau $K_{n,m}$ de
$(E^n,\cE^{\otimes n})$ vers $(E^m,\cE^{\otimes m})$ tel que
$\Q^{(n+m)}=\Q^{(n)}\otimes K_{n,m}$. 

On peut noter par ailleurs que si $\psi_n:E^{n+1}\to E$ d\'esigne la
projection sur la derni\`ere composante $(x_0,\dots,x_n)\mapsto x_n$,
alors la formule~\eqref{gps8} montre que la loi de $X_n$, qui est
donn\'ee par la marginale $\nu_n=\Q^{(n)}\circ\psi_n^{-1}$, s'exprime
comme 
\begin{equation}
\label{gps10B}
\prob{X_n\in A} = 
\nu_n(A) = \int_{E^n} K_{n-1}(x,A) \Q^{(n-1)}(\6x)\;.
\end{equation}
La situation est illustr\'ee par le diagramme suivant :

% \begin{minipage}{14.0cm}
% \begin{center}
%  \vspace{7mm}
%  \includegraphics*[clip=true,height=40mm]{figs/cd2} \\
%  \figtext{
%  	\writefig	4.85	4.65	{$(E,\cE)$}
%  	\writefig	5.3	3.6	{$\psi_n$}
%  	\writefig	7.5	3.8	{$\nu_n$}
%  	\writefig	7.1	2.6	{$\Q^{(n)}$}
%   	\writefig	3.85	2.35	{$(E^{n+1},\cE^{\otimes n+1})$}
%  	\writefig	8.8	2.35	{$[0,1]$}
%  	\writefig	5.3	1.25	{$\ph_n$}
%  	\writefig	7.5	0.95	{$\Q^{(n-1)}$}
%  	\writefig	4.35	0.1	{$(E^n,\cE^{\otimes n})$}
%  }
%  \vspace{7mm}
% \end{center}
% \end{minipage}

\begin{center}
\begin{tikzpicture}[->,>=stealth',shorten >=2pt,shorten <=2pt,auto,node
distance=2.5cm, thick
]

  \node (N) {$(E,\cE)$};
  \node (n+1) [below of=N] {$(E^{n+1},\cE^{\otimes n+1})$};
  \node (n) [below of=n+1] {$(E^n,\cE^{\otimes n})$};
  \node[node distance=0.8cm] (N+) [right of=N] {};
  \node[node distance=0.8cm] (n+) [right of=n] {};
  \node[node distance=1.1cm] (n+1+) [right of=n+1] {};
  \node[node distance=4.5cm] (01) [right of=n+1] {$[0,1]$};

   \path
     (n+1) edge [right] node {$\psi_n$} (N)
     (n+1) edge [right] node {$\ph_n$} (n)
     (N+) edge [above right] node {$\nu_n$} (01)
     (n+1+) edge [above] node {$\Q^{(n)}$} (01)
     (n+) edge [below right] node {$\Q^{(n-1)}$} (01)
;
\end{tikzpicture}
\end{center}

Nous donnons maintenant, sans d\'emonstration, le r\'esultat g\'en\'eral
assurant la l\'egitimit\'e de toute la proc\'edure. 

\begin{theorem}[Ionescu--Tulcea]
Pour la suite de mesures $\set{\Q^{(n)}}_{n\in\N}$ construites
selon~\eqref{gps10}, il existe une unique mesure de probabilit\'e $\Q$ sur
$(E^\N,\cE^{\otimes \N})$ telle que $\Q^{(n)}=\Q\circ(\pi^{(n)})^{-1}$
pour tout $n$, c'est-\`a-dire que les $\Q^{(n)}$ sont les distributions de
dimension finie de $\Q$. 
\end{theorem}

\begin{example}
\label{ex_gps2}
\hfill
\begin{enum}
\item	{\bf Mesures produit:} On se donne une suite
$\set{\mu_n}_{n\in\N}$ de mesures de probabilit\'e sur l'espace mesurable
$(E,\cE)$. Soit, pour tout $n$, 
$\Q^{(n)}=\mu_0\otimes\mu_1\otimes\dots\otimes\mu_n$ la mesure produit. 
C'est une mesure de la forme ci-dessus, avec noyau markovien
\begin{equation}
\label{gps11}
K_n(x,A) = \mu_{n+1}(A) 
\qquad \forall x\in E^n, \forall A\in\cE\;.
\end{equation}
La relation~\eqref{gps10B} montre que la loi $\nu_n$ de $X_n$ est donn\'ee
par $\mu_n$. On dit que les variables al\'eatoires $X_n$ sont
\defwd{ind\'ependantes}. Si tous les $\mu_n$ sont les m\^emes, on dit
qu'elles sont \defwd{ind\'ependantes et identiquement distribu\'ees
(i.i.d.)}\/.

\item	{\bf Syst\`eme dynamique:} On se donne une application mesurable
$f:E\to E$, une mesure de probabilit\'e initiale $\nu$ sur $E$. Soit pour
tout $n$ le noyau markovien 
\begin{equation}
\label{gps12}
K_n(x,A) = \indicator{A}{f(x_n)}\;, 
\end{equation}
et construisons les $\Q^{(n)}$ comme ci-dessus. Alors la
formule~\eqref{gps10B} montre qu'on a pour tout $A\in E$ 
\begin{align}
\nonumber
\nu_{n+1}(A) 
&= \int_{E^{n+1}} \indicator{A}{f(x_n)} \Q^{(n)}(\6x) \\
\nonumber
&= \int_{E} \indicator{A}{f(x_n)} \nu_n(\6x_n) \\
&= \int_{f^{-1}(A)} \nu_n(\6x_n) 
= \nu_n(f^{-1}(A))\;.
\label{gps13}
\end{align}
Il suit que 
\begin{equation}
\label{gps14}
\nu_n = \nu\circ f^{-n}
\qquad 
\forall n\in\N\;.
\end{equation}
Cette situation correspond \`a un syst\`eme dynamique d\'eterministe. Par
exemple, si $f$ est bijective et $\nu=\delta_{x_0}$ est concentr\'ee en un
point, on a $\nu_n = \delta_{f^n(x_0)}$.

\item	{\bf \Chaine s de Markov:}
Soit $\cX$ un ensemble fini ou d\'enombrable, muni de la tribu $\cP(\cX)$,
et $P=(p_{ij})_{i,j\in\cX}$ une matrice stochastique sur $\cX$,
c'est-\`a-dire que $0\leqs p_{ij}\leqs 1$ $\forall i,j\in\cX$ 
et $\sum_{j\in\cX}p_{ij}=1$ $\forall i\in\cX$. On se donne une mesure de
probabilit\'e $\nu$ sur $\cX$, et une suite $\smash{\Q^{(n)}}$ construite
\`a partir de~\eqref{gps10} avec pour noyaux markoviens 
\begin{equation}
\label{gps15}
K_n(i_{[0,n-1]},i_n) = p_{i_{n-1}i_n}\;.
\end{equation}
Le processus stochastique de mesure $\Q$, dont les distributions de
dimension finie sont les $\Q^{(n)}$, est la \chaine\ de Markov sur $\cX$
de distribution initiale $\nu$ et de matrice de transition $P$. Dans ce
cas la relation~\eqref{gps10B} se traduit en
\begin{equation}
\label{gps16}
\prob{X_n\in A} = \nu_n(A) = \sum_{i\in\cX} \sum_{j\in A} p_{ij}
\nu_{n-1}(\set{i}) = \sum_{j\in A} \sum_{i\in\cX} \prob{X_{n-1}=i}p_{ij}\;.
\end{equation}
\end{enum}
\end{example}

%%%%%%%%%%%%%%%%%%%%%%%%%%%%%%%%%%%%%%%%%%%%%%%%%%%%%%%%%%%%%%%%%%%%%%%%%%%

\section{R\'ecurrence, transience et p\'eriode}
\label{sec_rt}

Nous revenons maintenant \`a l'\'etude des propri\'et\'es de \chaine s de
Markov sur un ensemble d\'enombrable $\cX$. Rappelons que le \defwd{temps
de premier passage}\/ de la \chaine\ en un site $i\in\cX$ est la variable
al\'eatoire 
\begin{equation}
\label{rt1}
\tau_i = \inf\setsuch{n\geqs 1}{X_n=i}\;,
\end{equation}
avec la convention que $\tau_i=\infty$ si $X_n\neq i$ $\forall n\geqs1$.
Dans le cas o\`u la \chaine\ d\'emarre dans l'\'etat $i$ au temps $0$,
$\tau_i$ s'appelle \'egalement le \defwd{temps de premier retour en
$i$}\/. 

Dans le cas o\`u $\cX$ est fini, nous avons vu que pour une \chaine\
irr\'eductible, $\tau_i$ \'etait fini presque s\^urement (cf.
Proposition~\ref{prop_firred1}). Dans le cas o\`u $\cX$ est infini, ce
n'est plus forc\'ement le cas. En effet, la preuve que nous avons donn\'ee
de la Proposition~\ref{prop_firred1} utilise la
Proposition~\ref{prop_fabs1} sur les \chaine s absorbantes, dont la preuve
ne marche plus dans le cas infini (la d\'efinition~\eqref{fabs4:3} de $p$
n'interdit pas que $p=1$). On est donc amen\'e \`a introduire la
distinction suivante.

\begin{definition}
\label{def_rt1}
Un \'etat $i\in\cX$  est dit \defwd{r\'ecurrent}\/ si 
\begin{equation}
\label{rt2}
\probin{i}{\tau_i<\infty} \defby
\lim_{N\to\infty} \probin{i}{\tau\leqs N} = 
\sum_{n=1}^\infty
\probin{i}{\tau_i=n} = 1\;.
\end{equation}
Dans le cas contraire, il est dit \defwd{transient}\/. 
La \chaine\ de Markov est appel\'ee \defwd{r\'ecurrente}\/, respectivement
\defwd{transiente}\/, si tous ses \'etats sont r\'ecurrents,
respectivement transients.
\end{definition}

Le fait qu'un \'etat soit r\'ecurrent signifie que la \chaine\ revient
vers cet \'etat presque s\^urement, et donc qu'elle y revient infiniment
souvent. Le fait qu'un \'etat soit transient signifie que la \chaine\ a
une probabilit\'e positive de ne jamais retourner dans cet \'etat.

La condition~\eqref{rt2} n'est en g\'en\'eral pas ais\'ee \`a v\'erifier,
car elle pr\'esuppose la connaissance exacte de la loi de $\tau_i$.
Toutefois, on peut obtenir une condition \'equivalente beaucoup plus
facile \`a v\'erifier. Pour cela, nous commen\c cons par d\'emontrer une
\'equation dite de renouvellement.

\begin{prop}
\label{prop_rt1}
Pour tout $i, j\in\cX$ et tout temps $n\in\N$ on a la relation 
\begin{equation}
\label{rt3}
\probin{i}{X_n=j} = \sum_{m=1}^n \probin{i}{\tau_j=m}
\probin{j}{X_{n-m}=j}\;.
\end{equation}
\end{prop}
\begin{proof}
En d\'ecomposant sur les temps de premier passage en $j$, il vient 
\begin{align}
\nonumber
\probin{i}{X_n=j} 
&= \sum_{m=1}^n \probin{i}{j\notin X_{[1,m-1]},X_m=j,X_n=j} \\
&= \sum_{m=1}^n 
\underbrace{\pcondin{i}{X_n=j}{j\notin
X_{[1,m-1]},X_m=j}}_{=\pcondin{i}{X_n=j}{X_m=j}=\probin{j}{X_{n-m}=j}}
\underbrace{\probin{i}{j\notin
X_{[1,m-1]},X_m=j}}_{=\probin{i}{\tau_j=m}}\;,
\label{rt3:1}
\end{align}
o\`u nous avons utilis\'e la propri\'et\'e des incr\'ements ind\'ependants.
\end{proof}

Nous pouvons maintenant prouver un crit\`ere de r\'ecurrence plus
simple \`a v\'erifier que la d\'efinition~\eqref{rt2}.

\begin{theorem}
\label{thm_rt1}
Les deux conditions suivantes sont \'equivalentes:
\begin{enum}
\item	L'\'etat $i$ est r\'ecurrent.
\item	On a 
\begin{equation}
\label{rt4}
\sum_{n=0}^\infty \probin{i}{X_n=i} = +\infty\;.
\end{equation}
\end{enum}
\end{theorem}
\begin{proof}
\hfill
\begin{itemiz}
\item[$\Rightarrow$:]
L'\'equation de renouvellement~\eqref{rt3} permet d'\'ecrire 
\begin{align}
\nonumber
S\defby \sum_{n=0}^\infty \probin{i}{X_n=i} 
&= 1 + \sum_{n=1}^\infty \probin{i}{X_n=i} \\
\nonumber
&= 1 + \sum_{n=1}^\infty \sum_{m=1}^n \probin{i}{\tau_i=m}
\probin{i}{X_{n-m}=i} \\
\nonumber
&= 1 + \sum_{m=1}^\infty \probin{i}{\tau_i=m} \sum_{n=m}^\infty
\probin{i}{X_{n-m}=i} \\
&= 1 + \underbrace{\sum_{m=1}^\infty \probin{i}{\tau_i=m}}_{=1}
\sum_{n=0}^\infty \probin{i}{X_n=i} = 1+S\;.
\label{rt4:1}
\end{align}
Comme $S\in[0,\infty]$, l'\'egalit\'e $S=1+S$ implique n\'ecessairement
$S=+\infty$. 

\item[$\Leftarrow$:]
On ne peut pas directement inverser les implications ci-dessus. Cependant,
on peut montrer la contrapos\'ee en d\'efinissant pour tout $0<s<1$ les
s\'eries enti\`eres
\begin{align}
\nonumber
\psi(s) &= \sum_{n=0}^\infty \probin{i}{X_n=i} s^n\;, \\
\phi(s) &= \sum_{n=1}^\infty \probin{i}{\tau_i=n} s^n\;.
\label{rt4:2}
\end{align}
Ces s\'eries ont un rayon de convergence sup\'erieur ou \'egal \`a $1$ car
leurs coefficients sont inf\'erieurs ou \'egaux \`a $1$. 
Un calcul analogue au calcul~\eqref{rt4:1} ci-dessus donne alors 
\begin{align}
\nonumber
\psi(s)  
&= 1 + \sum_{m=1}^\infty \probin{i}{\tau_i=m} \sum_{n=m}^\infty
\probin{i}{X_{n-m}=i}s^n \\
\nonumber
&= 1 + \sum_{m=1}^\infty \probin{i}{\tau_i=m}s^m
\sum_{n=0}^\infty \probin{i}{X_n=i}s^{n} \\
&= 1 + \psi(s)\phi(s)\;,
\label{rt4:3}
\end{align}
d'o\`u
\begin{equation}
\label{rt4:4}
\psi(s) = \frac{1}{1-\phi(s)}\;.
\end{equation}
Par cons\'equent, si $\probin{i}{\tau_i<\infty}=\phi(1)<1$, alors on
obtient, en prenant la limite $s\nearrow1$, 
\begin{equation}
\label{rt4:5}
\sum_{n=0}^\infty \probin{i}{X_n=i} 
= \lim_{s\nearrow1}\psi(s) = \frac{1}{1-\phi(1)} < \infty\;.
\end{equation}
\qed
\end{itemiz}
\renewcommand{\qed}{}
\end{proof}

Une application de ce crit\`ere de r\'ecurrence est le r\'esultat
important suivant.

\begin{cor}
\label{cor_rt1}
La marche al\'eatoire sym\'etrique sur $\Z^d$ est r\'ecurrente pour $d=1$
et $d=2$ et transiente pour $d\geqs3$. 
\end{cor}
\begin{proof}
Comme la marche al\'eatoire est invariante par translation, il suffit de
v\'erifier que l'origine est r\'ecurrente, respectivement transiente.
\begin{enum}
\item	En dimension $d=1$, nous avons vu dans la Section~\ref{sec_rw} que
$\probin{0}{X_{2m}=0}$ se comportait en $1/\sqrt{\pi m}$ pour $m$ grand.
Il suit que 
\begin{equation}
\label{rt5:1}
\sum_{n=0}^\infty \probin{0}{X_n=0} = +\infty\;,
\end{equation}
donc par le th\'eor\`eme pr\'ec\'edent, l'origine est r\'ecurrente.

\item	En dimension $d=2$, on peut encore calculer explicitement
$\probin{0}{X_{2m}=0}$. Parmi les $4^{2m}$ trajectoires de longueur $2m$,
il faut d\'eterminer combien sont revenues \`a l'origine. Munissons le
r\'eseau $\Z^2$ des quatre directions Nord, Sud, Est et Ouest. Une
trajectoire de longueur $2m$ partant et revenant de l'origine doit avoir
fait autant de pas (disons $k$) vers le Nord que vers le Sud, et autant de
pas (c'est-\`a-dire $m-k$) vers l'Est que vers l'Ouest. Le nombre de telles
trajectoires est donc
donn\'e par 
\begin{equation}
\label{rt5:2}
\sum_{k=0}^m \binom{2m}{2k} \binom{2k}{k} \binom{2(m-k)}{m-k} 
= \sum_{k=0}^m \frac{(2m)!}{\brak{k!(m-k)!}^2}
= \binom{2m}{m} \sum_{k=0}^m \binom{m}{k}^2\;.
\end{equation}
Cette derni\`ere somme peut se simplifier gr\^ace aux identit\'es suivantes:
\begin{align}
\nonumber
\binom{2m}{m} 
&= \binom{2m-1}{m} + \binom{2m-1}{m-1}
= \binom{2m-2}{m} + 2\binom{2m-2}{m-1} + \binom{2m-2}{m-2} \\
&= \dots = \sum_{k=0}^m \binom mk \binom{m}{m-k}
= \sum_{k=0}^m \binom mk^2\;.
\label{rt5:3}
\end{align}
Il suit que 
\begin{equation}
\label{rt5:4}
\probin{0}{X_{2m}=0} = \frac1{4^{2m}} \binom{2m}{m}^2 \sim \frac1{\pi m}\;,
\end{equation}
o\`u la derni\`ere \'equivalence d\'ecoule de la formule de Stirling. Ceci
implique la divergence de la somme des $\probin{0}{X_n=0}$, donc la
r\'ecurrence de l'origine.

\item	En dimension $d=3$, on obtient de mani\`ere analogue que le nombre
de chemins revenus \`a l'origine apr\`es $6m$ pas est donn\'e par 
\begin{equation}
\label{rt5:5}
\binom{6m}{3m} \sum_{k_1+k_2+k_3=3m}
\biggpar{\frac{(3m)!}{k_1!k_2!k_3!}}^2\;.
\end{equation}
Cette fois-ci, nous ne disposons pas d'une expression exacte pour la
somme. Nous pouvons toutefois l'estimer comme suit: 
\begin{align}
\nonumber
\sum_{k_1+k_2+k_3=3m} \biggpar{\frac{(3m)!}{k_1!k_2!k_3!}}^2
&\leqs \biggpar{\max_{k_1+k_2+k_3=3m}\frac{(3m)!}{k_1!k_2!k_3!}}
\sum_{k_1+k_2+k_3=3m}\frac{(3m)!}{k_1!k_2!k_3!} \\
&= \frac{(3m)!}{(m!)^3} \cdot 3^{3m}\;.
\label{rt5:6}
\end{align}
En effet, le maximum est atteint lorsque tous les $k_i$ sont \'egaux,
alors que la somme est \'egale au nombre de mots de longueur $3m$ qu'on
peut \'ecrire avec trois lettres diff\'erentes, avec r\'ep\'etition (les
$k_i$ correspondant au nombre de lettres de chaque type choisies). En
utilisant la formule de Stirling, on trouve alors 
\begin{equation}
\label{rt5:7}
\probin{0}{X_{6m}=0} \leqs \frac{3^{3m}}{6^{6m}} \binom{6m}{3m}
\frac{(3m)!}{(m!)^3}
\sim \frac1{2(\pi m)^{3/2}}\;.
\end{equation}
Comme par ailleurs 
\begin{align}
\probin{0}{X_{6m}=0} 
&\geqs
\underbrace{\pcondin{0}{X_{6m}=0}{X_{6m-2}=0}}_{=\probin{0}{X_2=0}=1/6}
\probin{0}{X_{6m-2}=0} \\
&\geqs \Bigpar{\frac1{6}}^2 \probin{0}{X_{6m-4}=0}\;,
\label{rt5:8}
\end{align}
les termes de la s\'erie des $\probin{0}{X_{2n}=0}$ d\'ecroissent en
$n^{-3/2}$, ce qui implique que la s\'erie est sommable, et donc que
l'origine est un point transient. 

\item	En dimension $d\geqs4$, on montre de mani\`ere analogue que
$\probin{0}{X_{2n}=0}$ d\'ecro\^\i t comme $n^{-d/2}$, ce qui implique
\`a nouveau la transience de l'origine.
\qed 
\end{enum}
\renewcommand{\qed}{}
\end{proof}

Nous avons utilis\'e \`a plusieurs reprises le fait qu'une marche
al\'eatoire ne peut revenir au m\^eme endroit qu'aux temps pairs. On dit
qu'elle a la p\'eriode $2$. Plus g\'en\'eralement, on introduit la
d\'efinition suivante. 

\begin{definition}
\label{def_rt2}
La \defwd{p\'eriode}\/ d'un \'etat $i\in\cX$ est le nombre 
\begin{equation}
\label{rt6}
d_i = \pgcd \bigsetsuch{n\geqs 1}{\probin{i}{X_n=i} > 0}\;.
\end{equation}
Si $d_i=1$, on dit que l'\'etat $i$ est \defwd{ap\'eriodique}\/. Si tout
$i\in\cX$ est ap\'eriodique, on dit que la \chaine\ est ap\'eriodique. 
\end{definition}

\begin{remark}
\label{rem_rt1}
Une \chaine\ r\'eguli\`ere est ap\'eriodique. En effet, pour tout \'etat
$i$, il existe un \'etat $j$ tel que $p_{ij}>0$. Par d\'efinition, il
existe un temps $n$ tel que $\probin{k}{X_n=\ell}>0$ pour tout
$k,\ell\in\cX$. Par cons\'equent, on a $\probin{i}{X_n=i}>0$ et aussi
\begin{equation}
\label{rt7}
\probin{i}{X_{n+1}=i}\geqs \probin{i}{X_1=j,X_{n+1}=i}
= p_{ij}\probin{j}{X_n=i}>0\;.  
\end{equation}
Ceci implique que $d_i=\pgcd\set{n,n+1}=1$. 
\end{remark}

%\begin{remark}
%\label{rem_rt2}
Nous avons introduit la relation d'\'equivalence $i\sim j$ signifiant que
$i$ est accessible depuis $j$ est inversement. On montre assez
facilement que les propri\'et\'es de r\'ecurrence/tran\-sience et la
p\'eriode sont constantes sur les classes d'\'equivalence. On peut alors
parler de classes r\'ecurrentes ou transientes, et de la p\'eriode d'une
classe. Si la \chaine\ est irr\'eductible, alors elle est respectivement
r\'ecurrente, transiente ou ap\'eriodique si et seulement si elle admet un
\'etat r\'ecurrent, transient ou ap\'eriodique.
%\end{remark}

\begin{prop}
\label{prop_rt2}
Si $i$ et $j$ sont dans la m\^eme classe r\'ecurrente, alors 
\begin{equation}
\label{rt8}
\probin{i}{\tau_j<\infty} = \probin{j}{\tau_i<\infty} = 1\;.
\end{equation}
\end{prop}
\begin{proof}
Soit $A_M = \bigcup_{m=1}^M \set{X_m=j}$ l'\'ev\'enement \lq\lq la \chaine\
visite le site $j$ lors des $M$ premiers pas\rq\rq. Alors 
\begin{equation}
\label{rt8:1}
\lim_{M\to\infty} \fP_i(A_M) = \sum_{m=1}^\infty \probin{j}{\tau_j=m} =
1\;.
\end{equation}
Soit $n_0$ le plus petit entier tel que $\probin{j}{X_{n_0}=i}>0$. Alors
pour tout $M>n_0$, 
\begin{align}
\nonumber
\fP_j\bigpar{A_M\cap\set{X_{n_0}=i}}
&= \sum_{n=1}^{M-n_0} \probin{j}{X_{n_0}=i, \tau_j=n_0+n} \\
\nonumber
&= \sum_{n=1}^{M-n_0} \probin{j}{X_{n_0}=i, j\notin X_{[1,n_0]}}
\probin{i}{\tau_j=n} \\
&\leqs \probin{j}{X_{n_0}=i} \sum_{n=1}^{M-n_0}\probin{i}{\tau_j=n}\;.
\label{rt8:2}
\end{align}
La premi\`ere \'egalit\'e suit du fait que la \chaine\ ne peut pas
retourner en $j$ avant $n_0$ et visiter $i$ au temps $n_0$, par
d\'efinition de $n_0$. Nous faisons maintenant tendre $M$ vers l'infini
des deux c\^ot\'es de l'in\'egalit\'e. Le membre de gauche tend vers
$\probin{j}{X_{n_0}=i}$ en vertu de~\eqref{rt8:1}. Il vient donc 
\begin{equation}
\label{tr8:3}
\probin{j}{X_{n_0}=i} \leqs \probin{j}{X_{n_0}=i}
\probin{i}{\tau_j<\infty}\;.
\end{equation}
Comme $\probin{j}{X_{n_0}=i}\neq 0$ et $\probin{i}{\tau_j<\infty}\leqs 1$,
on a n\'ecessairement $\probin{i}{\tau_j<\infty}=1$.
\end{proof}

%%%%%%%%%%%%%%%%%%%%%%%%%%%%%%%%%%%%%%%%%%%%%%%%%%%%%%%%%%%%%%%%%%%%%%%%%%%

\section{Distributions stationnaires}
\label{sec_stat}

Nous consid\'erons une \chaine\ de Markov irr\'eductible sur un ensemble
d\'enombrable $\cX$, de matrice de transition $P=(p_{ij})_{i,j\in\cX}$.

\begin{definition}
\label{def_stat1}
Une distribution de probabilit\'e $\pi$ sur $\cX$ est dite
\defwd{stationnaire}\/ si elle satisfait 
\begin{equation}
\label{stat1}
\pi_j = \sum_{i\in\cX} \pi_i p_{ij}
\qquad
\forall j\in\cX\;.
\end{equation}
Plus g\'en\'eralement, une mesure $\mu$ sur $\cX$ (qui n'est pas
n\'ecessairement une mesure de probabilit\'e) satisfaisant $\mu_j =
\sum_{i\in\cX} \mu_i p_{ij}$ pour tout $j\in\cX$ est appel\'ee une
\defwd{mesure invariante}\/ de la \chaine. 
\end{definition}

Dans le cas o\`u $\cX$ est fini, nous avons vu qu'une \chaine\
irr\'eductible admettait toujours une distribution stationnaire. Dans le
cas infini, ce n'est plus n\'ecessairement le cas. Nous verrons par
exemple que les marches al\'eatoires sur $\Z^d$ n'admettent pas de
distribution stationnaire (en revanche, elles admettent beaucoup de mesures
invariantes).

Nous allons maintenant d\'eriver une condition n\'ecessaire et suffisante
pour qu'une \chaine\ de Markov irr\'eductible admette une distribution
stationnaire, qui sera toujours unique dans ce cas. Un r\^ole important
est jou\'e par la quantit\'e 
\begin{equation}
\label{stat2}
\gamma^{(k)}_i = \Bigexpecin{k}{\sum_{n=1}^{\tau_k} \indexfct{X_n=i}}\;,
\end{equation}
c'est-\`a-dire le nombre moyen de passages en $i$ entre deux passages en
$k$. Intuitivement, si $k$ est r\'ecurrent alors la \chaine\ revient
infiniment souvent en $k$, et donc $\smash{\gamma^{(k)}_i}$ devrait mesurer
le temps moyen pass\'e en $i$. On peut s'attendre \`a ce que ce temps
corresponde \`a une mesure invariante, et c'est effectivement le cas~:

\begin{prop}
\label{prop_stat1}
Supposons la \chaine\ irr\'eductible et r\'ecurrente. Alors on a 
$\forall k\in\cX$:
\begin{enum}
\item	$\smash{\gamma^{(k)}_k} = 1$;
\item	$\smash{\gamma^{(k)}}$ est une mesure invariante;
\item	Pour tout $i\in\cX$, on a $0<\smash{\gamma^{(k)}_i}<\infty$; 
\item	$\smash{\gamma^{(k)}}$ est l'unique mesure invariante telle que
$\smash{\gamma^{(k)}_k} = 1$. 
\end{enum}
\end{prop}
\begin{proof}
\hfill
\begin{enum}
\item	Evident, puisque $\tau_k$ est fini presque s\^urement, 
$X_{\tau_k}=k$ et $X_n\neq k$ pour $1\leqs
n<\tau_k$.
\item	Nous avons 
\begin{align}
\nonumber
\gamma^{(k)}_i 
&= \Bigexpecin{k}{\sum_{n=1}^\infty \indexfct{X_n=i,n\leqs\tau_k}} 
%\nonumber
= \sum_{n=1}^\infty \probin{k}{X_n=i,n\leqs\tau_k} \\
\nonumber
&= \sum_{j\in\cX} \sum_{n=1}^\infty 
\probin{k}{X_{n-1}=j,n\leqs\tau_k}p_{ji} \\
&= \sum_{j\in\cX} p_{ji} \sum_{m=0}^\infty
\probin{k}{X_m=j,m\leqs\tau_k-1}\;.
\label{stat3:1}
\end{align}
Or la seconde somme dans cette expression peut s'\'ecrire 
\begin{equation}
\label{stat3:2}
\Bigexpecin{k}{\sum_{m=0}^{\tau_k-1} \indexfct{X_m=j}}
= \Bigexpecin{k}{\sum_{m=1}^{\tau_k} \indexfct{X_m=j}}
= \gamma^{(k)}_j\;,
\end{equation}
vu que $\probin{k}{X_0=j}=\delta_{kj}=\probin{k}{X_{\tau_k}=j}$.
Ceci prouve l'invariance de la mesure $\smash{\gamma^{(k)}}$.

\item	L'invariance de la mesure implique que pour tout $n\geqs0$, 
\begin{equation}
\label{stat3:3}
\gamma^{(k)}_i = \sum_{j\in\cX}\gamma^{(k)}_j \probin{j}{X_n=i}\;. 
\end{equation}
En particulier, $1=\gamma^{(k)}_k\geqs \gamma^{(k)}_j \probin{j}{X_n=k}$
pour tout $j$. Comme par irr\'eductibilit\'e, il existe un $n$ tel que
$\probin{j}{X_n=k}>0$, on en d\'eduit que $\smash{\gamma^{(k)}_j}<\infty$
pour tout $j$. D'autre part, on a aussi $\smash{\gamma^{(k)}_i} \geqs
\probin{k}{X_n=i}$, qui est strictement positif pour au moins un $n$.

\item	Soit $\lambda$ une mesure invariante telle que $\lambda_k=1$.
Alors pour tout $j$ on a 
\begin{equation}
\label{stat3:4}
\lambda_j = \sum_{i\neq k} \lambda_i p_{ij} + p_{kj}
\geqs p_{kj}\;.
\end{equation}
Il vient alors, en minorant $\lambda_i$ par $p_{ki}$ dans l'expression
ci-dessus, 
\begin{align}
\nonumber
\lambda_j &\geqs \sum_{i\neq k} p_{ki}p_{ij}
+ p_{kj}\\
&= \probin{k}{X_2=j,\tau_k\geqs 2} + \probin{k}{X_1=j,\tau_k\geqs 1}
\label{stat3:5}
\end{align}
Par r\'ecurrence, on trouve donc pour tout $n\geqs1$ ($a\wedge b$ d\'esigne
le minimum de $a$ et $b$)
\begin{equation}
\lambda_j \geqs \sum_{m=1}^{n+1} \probin{k}{X_m=j,\tau_k\geqs m}
= \biggexpecin{k}{\sum_{m=1}^{(n+1)\wedge\tau_k}\indexfct{X_m=j}}\;.
\label{stat3:6}
\end{equation}
Lorsque $n$ tend vers l'infini, le membre de droite tend vers
$\smash{\gamma^{(k)}_j}$. On a donc $\lambda_j\geqs
\smash{\gamma^{(k)}_j}$ pour tout $j$. Par cons\'equent,
$\mu=\lambda-\smash{\gamma^{(k)}}$ est une mesure invariante, satisfaisant
$\mu_k=0$. Comme $\mu_k=\sum_j\mu_j\probin{j}{X_n=k}$ pour tout $n$,
l'irr\'eductibilit\'e implique $\mu_j=0$ $\forall j$, donc
n\'ecessairement $\lambda=\smash{\gamma^{(k)}}$.
\qed
\end{enum}
\renewcommand{\qed}{}
\end{proof}

Nous pouvons maintenant \'enoncer et d\'emontrer le th\'eor\`eme principal
sur les distributions stationnaires.

\begin{theorem}
\label{thm_stat1}
Pour une \chaine\ de Markov irr\'eductible, les propri\'et\'es suivantes
sont \'equivalentes :
\begin{enum}
\item	Il existe une distribution stationnaire.
\item	Il existe un \'etat $k\in\cX$ tel que 
\begin{equation}
\label{stat3}
\mu_k \defby \expecin{k}{\tau_k} 
%=\sum_{n=0}^\infty n \probin{k}{\tau_k=n} 
< \infty\;.
\end{equation}
\item	La relation~\eqref{stat3} est v\'erifi\'ee pour tout $k\in\cX$. 
\end{enum}
De plus, si ces propri\'et\'es sont v\'erifi\'ees, alors la distribution
stationnaire est unique, et donn\'ee par 
\begin{equation}
\label{stat4}
\pi_i = \frac1{\mu_i}
\qquad\forall i\in\cX\;.
\end{equation}
\end{theorem}
\goodbreak
\begin{proof}
\hfill
\begin{itemizz}
\item[{$2\Rightarrow 1:$}]	
Si $\mu_k<\infty$ alors $k$ est r\'ecurrent, donc la \chaine, \'etant
irr\'eductible, est r\'ecurrente. Par la proposition pr\'ec\'edente,
$\smash{\gamma^{(k)}}$ est l'unique mesure invariante prenant valeur $1$
en $k$. Or nous avons 
\begin{equation}
\label{stat4:1}
\sum_{j\in\cX}\gamma^{(k)}_j 
= \biggexpecin{k}{\sum_{n=1}^{\tau_k}
\underbrace{\sum_{j\in\cX}\indexfct{X_n=j}}_{=1}}
= \expecin{k}{\tau_k} = \mu_k < \infty\;.
\end{equation}
Par cons\'equent, la mesure $\pi$ d\'efinie par
$\pi_j=\gamma^{(k)}_j/\mu_k$ est une mesure de probabilit\'e invariante,
c'est-\`a-dire une distribution stationnaire.

\item[{$1\Rightarrow 3:$}]
Soit $\pi$ une distribution stationnaire, et $k\in\cX$. Alors $\hat\gamma$
d\'efini par $\hat\gamma_j=\pi_j/\pi_k$ est une mesure invariante telle
que $\hat\gamma_k=1$. Par la proposition pr\'ec\'edente, on a
n\'ecessairement $\hat\gamma=\smash{\gamma^{(k)}}$. Il suit par le m\^eme
calcul que ci-dessus
\begin{equation}
\label{stat4:2}
\expecin{k}{\tau_k} = \sum_{j\in\cX} \hat\gamma_j
= \frac{\sum_j\pi_j}{\pi_k} = \frac1{\pi_k} < \infty\;.
\end{equation}

\item[{$3\Rightarrow 2:$}] Evident.
\end{itemizz}
Dans ce cas, l'unicit\'e de la mesure suit de celle de $\gamma^{(k)}$, et
la relation~\eqref{stat4} suit de~\eqref{stat4:2}.
\end{proof}

Ce r\'esultat motive la d\'efinition suivante. 

\begin{definition}
\label{def_stat2}
Un \'etat $i\in\cX$ tel que 
\begin{equation}
\label{stat5}
\expecin{i}{\tau_i} < \infty
\end{equation}
est appel\'e \defwd{r\'ecurrent positif}. Un \'etat r\'ecurrent $i$ qui
n'est pas r\'ecurrent positif est appel\'e \defwd{r\'ecurrent nul}.
La \chaine\ est dite \defwd{r\'ecurrente positive}\/ si tous ses \'etats
le sont. C'est par exemple le cas s'il existe un tel \'etat, et que la
\chaine\ est irr\'eductible. 
\end{definition}

Une \chaine\ irr\'eductible admet donc
une distribution stationnaire si et seulement si elle est r\'ecurrente
positive.

\begin{example}
\label{ex_stat1}
Les marches al\'eatoires sur $\Z^d$ n'admettent pas de distribution
stationnaire. En effet, si $\pi$ \'etait une distribution stationnaire,
l'invariance par translation impliquerait que tout translat\'e de $\pi$
serait encore une distribution stationnaire. Mais nous savons que si une
telle distribution existait, alors elle serait unique. $\pi$ devrait donc
\^etre uniforme, mais il n'existe pas de mesure de probabilit\'e uniforme
sur $\Z^d$. 

En revanche, toutes les mesures uniformes sont invariantes. Les fonctions
pas n\'ecessai\-rement positives satisfaisant~\eqref{stat1} dans le cas
d'une marche al\'eatoire sym\'etrique sont appel\'ees \defwd{harmoniques}.
Toutes les fonctions affines sont harmoniques, mais en dimension
sup\'erieure ou \'egale \`a 2, il y en a beaucoup d'autres.
\end{example}

Nous avons donc obtenu le r\'esultat suivant :

\begin{theorem}
\label{thm_stat2}
La marche al\'eatoire sym\'etrique est r\'ecurrente nulle en
dimensions $d=1$ et $d=2$. 
\end{theorem}

%%%%%%%%%%%%%%%%%%%%%%%%%%%%%%%%%%%%%%%%%%%%%%%%%%%%%%%%%%%%%%%%%%%%%%%%%%%

\section{Convergence vers la distribution stationnaire}
\label{sec_conv}

Dans le cas fini, nous avons montr\'e que si la \chaine\ \'etait
r\'eguli\`ere, alors la loi de $X_n$ convergeait vers la distribution
stationnaire. Dans le cas d'un espace infini, une \chaine\ de Markov ne
peut jamais \^etre r\'eguli\`ere : les probabilit\'es de transition
\'etant sommables, elles ne peuvent \^etre minor\'ees par une quantit\'e
strictement positive. Il s'av\`ere toutefois que la r\'ecurrence positive
et l'ap\'eriodicit\'e suffisent \`a garantir la convergence vers la
distribution stationnaire.

\begin{theorem}
\label{thm_conv1}
Soit $\set{X_n}_{n\geqs0}$ une \chaine\ de Markov irr\'eductible,
ap\'eriodique et r\'ecur\-rente positive, et soit $\pi$ son unique
distribution stationnaire. Alors pour toute distribution initiale $\nu$, 
on a 
\begin{equation}
\label{conv1}
\lim_{n\to\infty} \bigprobin{\nu}{X_n=j} = \pi_j
\qquad \forall j\in\cX\;.
\end{equation}
\end{theorem}
\begin{proof}
Nous allons g\'en\'eraliser la preuve de Doeblin, d\'ej\`a vue dans le cas
fini (voir le Th\'eor\`eme~\ref{thm_firred2}). 
\begin{itemiz}
\item	Nous introduisons une
\chaine\ de Markov $(X_n,Y_n)_{n\geqs0}$ sur $\cX\times\cX$, de
probabilit\'es de transition
\begin{equation}
\label{conv1:1}
p^\star_{(i,j),(k,l)} = p_{ik}p_{jl}\;,
\end{equation}
et de distribution initiale $\rho=\nu\otimes\pi$. Dans ce cas, $X_n$ et
$Y_n$ sont deux \chaine s ind\'ependantes de matrice de transition $P$, et
de distributions initiales $\nu$ et $\pi$. 

\item	Le seul point non trivial de la g\'en\'eralisation est de montrer
que $P^\star$ est irr\'eductible et ap\'eriodique. Pour cela, fixons un
\'etat $k\in\cX$. Consid\'erons d'abord l'ensemble 
\begin{equation}
\label{conv1:2}
\Gamma_k = \bigsetsuch{n\in\N}{\probin{k}{X_n=k} > 0}\;.
\end{equation}
La propri\'et\'e de Markov implique que si $n, m\in\Gamma_k$, alors
$n+m\in\Gamma_k$. D'autre part, par d\'efinition de l'ap\'eriodicit\'e,
$\pgcd\Gamma_k=1$. Nous pr\'etendons qu'il existe un $n_0$ tel que
tout $t\geqs n_0$ appartienne \`a $\Gamma_k$. 

Pour cela, supposons d'abord qu'il existe $n, m\in\Gamma_k$ premiers entre
eux. Par le th\'eor\`eme de B\'ezout, il existe des entiers $p,q\geqs1$ tels
que $pn-qm=\pm1$. Quitte \`a intervertir $n$ et $m$, on peut supposer que
$pn-qm=1$. Soit $n_0=qnm$. Alors pour $1\leqs r\leqs n$, on a
$n_0+r=qnm+r(pn-qm)=qm(n-r)+rpn\in\Gamma_k$. Il suffit alors d'\'ecrire
tout $t>n_0$ comme $t=n_0+r+ns$ avec $1\leqs r\leqs n$ pour conclure que
$t\in\Gamma_k$.

Il se pourrait que $\pgcd\Gamma_k=1$ sans que cet ensemble ne contienne
deux entiers premiers entre eux. Mais par le th\'eor\`eme de B\'ezout, il
existe forc\'ement un ensemble d'\'el\'ements de $\Gamma_k$ dont une
combinaison lin\'eaire vaut $1$, et le raisonnement ci-dessus s'adapte
facilement \`a ce cas.

\item	Fixons des \'etats $i, j, k, \ell\in\cX$. $P$ \'etant suppos\'ee
irr\'eductible, il existe $r\in\N$ tel que $\probin{i}{X_r=k}>0$. Comme
pour tout $n\geqs n_0$, 
\begin{equation}
\label{conv1:4}
\probin{i}{X_{r+n}=k} \geqs \probin{i}{X_r=k}\probin{k}{X_n=k} > 0\;,
\end{equation}
il suit que $\probin{i}{X_n=k}>0$ pour tous les $n\geqs n_0+r$. Pour des
raisons similaires, il existe $m_0, s\in\N$ tels que
$\probin{j}{X_m=\ell}>0$ pour tous les $m\geqs m_0+s$. Par cons\'equent,
il existe un temps $M$ tel que
$\fP^\star_{(i,j)}\set{(X_t,Y_t)=(k,\ell)}>0$ pour tous les $t\geqs M$.
Ceci implique que la \chaine\ compos\'ee est irr\'eductible et
ap\'eriodique. 

\item	Comme la \chaine\ compos\'ee admet manifestement la
distribution invariante $\pi\otimes\pi$, le Th\'eor\`eme~\ref{thm_stat1}
implique qu'elle est r\'ecurrente positive. 

\item	Le reste de la preuve est identique au cas fini. On introduit le
temps $\tau_A$ de premier passage sur la diagonale
$A=\setsuch{(i,i)}{i\in\cX}$, et on montre comme dans le cas fini que 
\begin{equation}
\label{conv1:5}
\abs{\probin{\nu}{X_n=j}-\pi_j}\leqs 2\probin{\rho}{\tau_A>n}\;.
\end{equation}
La Proposition~\ref{prop_rt2} implique que $\tau_A$ est fini presque
s\^urement, et donc que la diff\'erence ci-dessus tend vers z\'ero pour
$n\to\infty$.
\qed
\end{itemiz}
\renewcommand{\qed}{}
\end{proof}

Un probl\`eme important, mais en g\'en\'eral difficile, est d'estimer la
vitesse de convergence vers la distribution stationnaire. La preuve
de Doeblin fournit (voir le cas fini) l'estimation
\begin{equation}
\label{conv2}
\sum_{j\in\cX} \bigabs{\probin{\nu}{X_n=j}-\pi_j} \leqs 2
\probin{\nu\otimes\pi}{\tau_A>n}\;,
\end{equation}
qui est utile dans les cas o\`u on arrive \`a contr\^oler le temps de
couplage $\tau_A$. 

Une autre approche est bas\'ee sur la th\'eorie spectrale. Nous avons vu
que la matrice de transition $P$ admet la valeur propre $1$, avec comme
vecteurs propres \`a gauche et \`a droite, respectivement, la distribution
stationnaire $\pi$ et le vecteur $\vone$ dont toutes les composantes sont
\'egales \`a $1$:
\begin{equation}
\label{conv3}
\pi P = \pi 
\qquad\text{et}\qquad 
P\vone = \vone\;.
\end{equation}
Soit $\mu$ un vecteur ligne tel que 
\begin{equation}
\label{conv4}
\mu \vone = \sum_{i\in\cX} \mu_i = 0\;.
\end{equation}
Alors $\mu P\vone = \mu\vone = 0$, ce qui montre que le sous-espace
$\vone_\perp=\setsuch{\mu\in\R^\cX}{\sum_i\mu_i=0}$ est invariant (\`a
droite) par $P$~: $\vone_\perp P\subset\vone_\perp$. On remarquera que les
\'el\'ements de $\vone_\perp$ ne sont pas des mesures, mais des mesures
sign\'ees. Toutefois, pour certains $\mu\in\vone_\perp$, la somme $\pi+\mu$
est une mesure de probabilit\'e~: il suffit pour cel\`a que $\mu_i \geqs
-\pi_i$ pour tout $i\in\cX$.

Si $\mu\in\vone_\perp$ est un vecteur propre \`a gauche de $P$, de valeur
propre $\lambda$, et si $P$ est irr\'eductible,
ap\'eriodique et r\'ecur\-rente positive, alors on aura n\'ecessairement
$\abs{\lambda}<1$. En effet, si ce n'\'etait pas le cas, la loi de
la \chaine\ de condition initiale $\pi+\eps\mu$ ($\eps$ un r\'eel
suffisamment petit) ne convergerait pas vers $\pi$. 

Cette observation donne une caract\'erisation de la vitesse de convergence
en termes de \defwd{trou spectral}\/~: Soit $\lambda_0$ la plus grande
valeur propre de $P$ de module strictement inf\'erieur \`a $1$. On sait
que le vecteur propre correspondant se trouve dans $\vone_\perp$
(\'eventuellement avec des composantes complexes). Alors la loi de $X_n$
converge exponentiellement vite vers la distribution stationnaire $\pi$,
\`a vitesse $\abs{\lambda_0}^n$. Toutefois, la d\'etermination du trou
spectral est en g\'en\'eral un probl\`eme difficile si l'ensemble $\cX$
est grand. 

Les \chaine s de Markov r\'eversibles se pr\^etent mieux \`a une \'etude
spectrale que les \chaine s non r\'eversibles. Pour le voir, supposons la
\chaine\ irr\'eductible et r\'ecurrente positive, de distribution
stationnaire $\pi$, et introduisons le produit scalaire
\begin{equation}
\label{rev6}
\pscal fg_\pi = \sum_{i\in\cX} \pi_i \cc{f}_i g_i\;.
\end{equation}
Alors on a 
\begin{equation}
\label{rev7}
\pscal f{Pg}_\pi = \sum_{i\in\cX} \pi_i \cc{f}_i \sum_{j\in\cX} p_{ij}g_j
= \sum_{j\in\cX} \pi_j \sum_{i\in\cX} p_{ji} \cc{f}_i g_j
= \pscal {Pf}g_\pi\;.
\end{equation}
Autrement dit, l'op\'erateur lin\'eaire $P$ est autoadjoint dans
$\ell^2(\C,\pi)$. 

Un r\'esultat classique de la th\'eorie des espaces de Hilbert dit que
toutes les valeurs propres de $P$ sont r\'eelles, et que les espaces
propres associ\'es sont orthogonaux. En effet, soient $x_1$ et $x_2$ deux
vecteurs propres \`a droite de $P$, de valeurs propres respectives
$\lambda_1$ et $\lambda_2$. Alors 
\begin{equation}
\label{rev8}
(\cc\lambda_1 - \lambda_2) \pscal{x_1}{x_2}_\pi 
= \pscal{\lambda_1x_1}{x_2}_\pi - \pscal{x_1}{\lambda_2x_2}_\pi 
= \pscal{Px_1}{x_2}_\pi - \pscal{x_1}{Px_2}_\pi = 0\;.
\end{equation}
D'une part, prenant $x_1=x_2$, on obtient que $\lambda_1$ est r\'eelle.
D'autre part, si $\lambda_1\neq\lambda_2$, on obtient l'orthogonalit\'e de
$x_1$ et $x_2$. On sait de plus que $P$ est diagonalisable. 

On a alors une repr\'esentation variationnelle du trou spectral~:
\begin{equation}
\label{rev9}
\abs{\lambda_0} = \sup_{x \colon \pscal{x}{\vone}_\pi=0}
\frac{\pscal{x}{Px}_\pi}{\pscal{x}{x}_\pi}\;.
\end{equation}

%%%%%%%%%%%%%%%%%%%%%%%%%%%%%%%%%%%%%%%%%%%%%%%%%%%%%%%%%%%%%%%%%%%%%%%%%%%

%\goodbreak
\section{Exercices}
\label{sec_exo_markov_denombrable}

\begin{exercice}
\label{exo_md01} 
Montrer que la loi du temps de premier passage en $i$ ($i\neq0$) de la
marche al\'eatoire sym\'etrique unidimensionnelle est donn\'ee par 
\[
\bigprob{\tau_i=n} = 
\begin{cases}
\vrule height 18pt depth 16pt width 0pt
\dfrac{\abs{i}}n \prob{X_n=i} & \text{pour $n\in\set{\abs{i}, \abs{i}+2,
\dots}$\;,} \\
0 & \text{sinon\;.} 
\end{cases}
\]
(Th\'eor\`eme~2.1.5).  
En d\'eduire que $\expec{\tau_i}=+\infty$. 

\noindent
{\it Indications:} 
\begin{enum}
\item	Par sym\'etrie, on peut supposer $i>0$. 
\item	Ecrire l'\'ev\'enement $\set{\tau_i=n}$ \`a l'aide des
\'ev\'enements $\set{X_{n-1}=i-1}$ et $\set{\tau_i\leqs n-2}$.
\item	A l'aide du principe de r\'eflexion, montrer que 
\[
\bigprob{\tau_i=n} = \frac12
\bigbrak{\bigprob{X_{n-1}=i-1} - \bigprob{X_{n-1}=i+1}}
\]
et conclure par un calcul direct.
\end{enum}
\end{exercice}

\goodbreak

\begin{exercice}
\label{exo_md02} 

On consid\`ere une \chaine\ de Markov sur $\N$ de probabilit\'es de
transition 
\[
p_{ij} = 
\begin{cases}
p & \text{si $j=i+1$\;,} \\
1-p & \text{si $j=0$\;,} \\
0 & \text{sinon\;.}
\end{cases}
\]

% \bigskip
% 
% \centerline{
%  \includegraphics*[clip=true,height=30mm]{figs/chain_geom}
%  }
 
\begin{center}
\begin{tikzpicture}[->,>=stealth',shorten >=2pt,shorten <=2pt,auto,node
distance=3.0cm, thick,main node/.style={circle,scale=0.7,minimum size=1.1cm,
fill=blue!20,draw,font=\sffamily\Large}]

  \node[main node] (0) {0};
  \node[main node] (1) [right of=0] {1};
  \node[main node] (2) [right of=1] {2};
  \node[main node] (3) [right of=2] {3};
  \node[main node] (4) [right of=3] {4};
  \node[node distance=2.2cm] (5) [right of=4] {$\dots$};

  \path[every node/.style={font=\sffamily\small}]
    (0) edge [bend left,above] node {$p$} (1)
    (0) edge [loop left,left,distance=1cm,out=-150,in=150] node {$1-p$} (0)
    (1) edge [bend left,above] node {$p$} (2)
    (2) edge [bend left,above] node {$p$} (3)
    (3) edge [bend left,above] node {$p$} (4)
    (4) edge [bend left,above] node {$p$} (5)
    (1) edge [bend left,below] node {$\;1-p$} (0);
    
\draw[every node/.style={font=\sffamily\small}] 
    (2) to[out=-120,in=-60] node[below] {$1-p$} (0);

\draw[every node/.style={font=\sffamily\small}] 
    (3) to[out=-110,in=-70] node[below] {$1-p$} (0);

% \draw[every node/.style={font=\sffamily\small}] 
%     (4) to[out=-100,in=-80] node[below] {$1-p$} (0);
\end{tikzpicture}
\end{center}

\begin{enum}
\item	Pour quelles valeurs de $p$ la \chaine\ est-elle irr\'eductible? 

Pour le reste du probl\`eme, on suppose que $p$ est tel que la \chaine\
soit irr\'eductible.

\item	On suppose $X_0=0$. Soit 
\[
\tau_0 = \inf\setsuch{n>0}{X_n=0}
\]
le temps de premier retour en $0$. Montrer que 
\[
\tau_0 = n \quad\Rightarrow\quad X_{m} = m\;,
\qquad m=0,1,\dots,n-1\;.
\]
En d\'eduire la loi de $\tau_0$. 

\item	Montrer que l'\'etat $0$ est r\'ecurrent.

\item	Montrer que l'\'etat $0$ est r\'ecurrent positif.

\item	Montrer que l'\'etat $0$ est ap\'eriodique.

\item	Soit $\pi$ l'unique distribution stationnaire de la \chaine.
Calculer $\pi_0$ \`a l'aide de $\expecin{0}{\tau_0}$. 

\item	Exprimer $\pi_i$ en fonction de $\pi_{i-1}$ et en d\'eduire
$\pi_i$ pour tout $i$. 
\end{enum}

\end{exercice}

\goodbreak

\begin{exercice}
\label{exo_md03} 

Un client arrive dans une banque. Devant l'unique guichet, il y a une queue de
longueur al\'eatoire $L$. La loi de $L$ est donn\'ee par $\prob{L=k}=p_k$,
$k=1,2,\dots$ (on suppose $\sum_{k\geqs1}p_k=1$).

On admet que chaque client est servi pendant un intervalle de temps de longueur
$1$. Une fois le premier client servi, notre client avance d'une unit\'e dans
la file. Une fois servi, on suppose que le client se place \`a nouveau au bout
de la queue. 

On mod\'elise la situation par la \chaine\ de Markov suivante~:

% \bigskip
% 
% \centerline{
%  \includegraphics*[clip=true,height=30mm]{figs/exam_pb2}
%  }
 
\begin{center}
\begin{tikzpicture}[->,>=stealth',shorten >=2pt,shorten <=2pt,auto,node
distance=3.0cm, thick,main node/.style={circle,scale=0.7,minimum size=1.1cm,
fill=blue!20,draw,font=\sffamily\Large}]

  \node[main node] (0) {0};
  \node[main node] (1) [right of=0] {1};
  \node[main node] (2) [right of=1] {2};
  \node[main node] (3) [right of=2] {3};
  \node[main node] (4) [right of=3] {4};
  \node[node distance=2.2cm] (5) [right of=4] {$\dots$};

  \path[every node/.style={font=\sffamily\small}]
    (0) edge [bend left,above] node {$p_1$} (1)
    (1) edge [bend left,below] node {$1$} (0)
    (2) edge [bend left,below] node {$1$} (1)
    (3) edge [bend left,below] node {$1$} (2)
    (4) edge [bend left,below] node {$1$} (3)
    (5) edge [bend left,below] node {$1$} (4);
   
\draw[every node/.style={font=\sffamily\small}] 
    (0) to[out=60,in=120] node[above] {$p_2$} (2);

\draw[every node/.style={font=\sffamily\small}] 
    (0) to[out=70,in=110] node[above] {$p_3$} (3);
\end{tikzpicture}
\end{center}

%\medskip

\begin{enum}
\item	Sous quelle condition sur les $p_k$ la \chaine\ est-elle
irr\'eductible? 

Pour le reste du probl\`eme, on suppose les $p_k$ tels que la \chaine\
soit irr\'eductible.

\item	On suppose $X_0=0$. Soit 
\[
\tau_0 = \inf\setsuch{n>0}{X_n=0}
\]
le temps de premier retour en $0$. D\'eterminer sa loi.

\item	Montrer que l'\'etat $0$ est r\'ecurrent.

\item	Donner une condition n\'ecessaire et suffisante sur les $p_k$ pour que
l'\'etat $0$ soit r\'ecurrent positif.

\item	Donner une condition suffisante sur les $p_k$ pour que
l'\'etat $0$ soit ap\'e\-riodique.

\item	Soit $\pi$ l'unique distribution stationnaire de la \chaine.
Calculer $\pi_0$ \`a l'aide de $\expecin{0}{\tau_0}$. 

\item	Par r\'ecurrence, d\'eterminer $\pi_i$ pour tout $i$. 
\end{enum}

\end{exercice}

\goodbreak

\begin{exercice}
\label{exo_md04} 

On consid\`ere la \chaine\ de Markov suivante sur $\cX=\Z^*$: 
% \bigskip
% 
% \centerline{
%  \includegraphics*[clip=true,height=26mm]{figs/chain2_exam_09}
%  }
 
\begin{center}
\begin{tikzpicture}[->,>=stealth',shorten >=2pt,shorten <=2pt,auto,node
distance=3.0cm, thick,main node/.style={circle,scale=0.7,minimum size=1.1cm,
fill=blue!20,draw,font=\sffamily\Large}]

  \node[main node] (1) {$1$};
  \node[main node] (2) [right of=1] {$2$};
  \node[main node] (3) [right of=2] {$3$};
  \node[main node] (4) [right of=3] {$4$};
  \node[node distance=2.2cm] (5) [right of=4] {$\dots$};
  \node[main node] (-1) [below of=1] {$-1$};
  \node[main node] (-2) [right of=-1] {$-2$};
  \node[main node] (-3) [right of=-2] {$-3$};
  \node[main node] (-4) [right of=-3] {$-4$};
  \node[node distance=2.2cm] (-5) [right of=-4] {$\dots$};

  \path[every node/.style={font=\sffamily\small}]
    (1) edge [above] node {$1/2$} (2)
    (2) edge [above] node {$1/2$} (3)
    (3) edge [above] node {$1/2$} (4)
    (4) edge [above] node {$1/2$} (5)
    (-5) edge [below] node {$1$} (-4)
    (-4) edge [below] node {$1$} (-3)
    (-3) edge [below] node {$1$} (-2)
    (-2) edge [below] node {$1$} (-1)
    (1) edge [bend left, right] node {$1/2$} (-1)
    (2) edge [right] node {$1/2$} (-2)
    (3) edge [right] node {$1/2$} (-3)
    (4) edge [right] node {$1/2$} (-4)
    (-1) edge [bend left, left] node {$1$} (1)
  ;
\end{tikzpicture}
\end{center}

% \medskip
% \noindent

\begin{enum}
\item	La \chaine\ est-elle irr\'eductible?
\item	L'\'etat $1$ est-il r\'ecurrent?
\item	L'\'etat $1$ est-il r\'ecurrent positif? 
\item	L'\'etat $1$ est-il ap\'eriodique?
\item	Soit $\pi$ la distribution stationnaire de la \chaine. Calculer $\pi_1$.
\item	Calculer $\pi_i$ pour tout $i\in\cX$. 
\end{enum}
\end{exercice}

\goodbreak

\begin{exercice}
\label{exo_md05} 

On consid\`ere une marche al\'eatoire unidimensionnelle sym\'etrique sur 
l'en\-semble $\set{0,1,\dots,N}$ avec conditions aux bords absorbantes,
c'est-\`a-dire que l'on suppose les \'etats $0$ et $N$ absorbants. Soit 
\[
\tau = \tau_0 \wedge \tau_N
= \inf\bigsetsuch{n\geqs0}{X_n\in\set{0,N}}
\]
le temps d'absorption, et soit 
\[
p(i) = \probin{i}{X_\tau=N}\;.
\]

\begin{enum}
\item	D\'eterminer $p(0)$ et $p(N)$. 
\item	Montrer que pour tout $i\in\set{1,\dots,N-1}$, on a 
\[
p(i) = \frac12 \bigbrak{p(i-1)+p(i+1)}\;.
\]

Une fonction $f:\Z\supset A\to\R$ telle que 
$f(i) = \frac12 \brak{f(i-1)+f(i+1)}$ pour tout $i\in A$ est appel\'ee
\emph{harmonique}\/ (discr\`ete). 

\item	Montrer (par l'absurde) le \emph{principe du maximum}: Une fonction
harmonique sur $A$ ne peut atteindre son minimum et son maximum qu'au bord
de $A$ (on pourra supposer $A$ de la forme $A=\set{a,a+1,\dots,b-1,b}$,
dans ce cas son bord est $\partial A=\set{a,b}$). 

\item	Montrer que si $f$ et $g$ sont deux fonctions harmoniques sur $A$,
alors toute combinaison lin\'eaire de $f$ et $g$ est encore harmonique.

\item	Montrer que si $f$ et $g$ sont deux fonctions harmoniques sur $A$,
qui co\"\i ncident sur le bord de $A$, alors elles sont \'egales partout
dans $A$ (consid\'erer $f-g$).

\item	Montrer que toute fonction lin\'eaire $f(i)=ci+h$ est harmonique. 

\item	En utilisant les points 1., 2., 5.~et 6., d\'eterminer la fonction
$p$.
\end{enum}
\end{exercice}

\goodbreak

\begin{exercice}
\label{exo_md06} 

On consid\`ere une marche al\'eatoire sym\'etrique sur
$\cX=\set{0,1,\dots,N}$, avec conditions au bord absorbantes,
c'est-\`a-dire que d\`es que la marche atteint l'un des \'etats $0$ ou
$N$, elle y reste ind\'efiniment. Soit 
\[
\tau = \inf\setsuch{n\geqs 0}{X_n\in\set{0,N}}
\]
le temps d'absorption. Par convention, $\tau=0$ si $X_0\in\set{0,N}$. 

\begin{enum}
\item	Montrer que pour tout $i\in\set{1,\dots,N-1}$, 
\[
\probin{i}{\tau=n} = \frac12 
\bigbrak{\probin{i-1}{\tau=n-1} + \probin{i+1}{\tau=n-1}}\;.
\]

\item	Soit $f(i)=\expecin{i}{\tau}$. Que valent $f(0)$ et $f(N)$?

\item	D\'eduire de 1.\ que 
\begin{equation}
\label{eq1}
f(i) = \frac12 \bigbrak{f(i-1) + f(i+1)} + 1\;.
\end{equation}

\item	Montrer que $f(i)=-i^2$ satisfait l'\'equation~\eqref{eq1}. 

\item	Une fonction $g:\cX\to\R$ est dite harmonique si 
\[
g(i) = \frac12 \bigbrak{g(i-1) + g(i+1)} \qquad\forall i\in\cX\;.
\]
Montrer que si $f$ satisfait~\eqref{eq1} et $g$ est harmonique, alors
$f+g$ satisfait~\eqref{eq1}. 

\item	Montrer que si $f_1$ et $f_2$ satisfont~\eqref{eq1}, alors
$f_1-f_2$ est harmonique. D\'eduire du principe du maximum que si $f_1$ et
$f_2$ co\"\i ncident en $0$ et en $N$, alors elles sont \'egales partout. 

\item	Sachant que les fonctions lin\'eaires sont harmoniques,
d\'eterminer $\expecin{i}{\tau}$. 

\end{enum}

\end{exercice}

\goodbreak

\begin{exercice}
\label{exo_md07} 

On consid\`ere une marche al\'eatoire sur $\N$, absorb\'ee en $0$:
% \medskip
% 
% \centerline{
%  \includegraphics*[clip=true,height=12mm]{figs/chain3_exam_09}
%  }
 
\begin{center}
\begin{tikzpicture}[->,>=stealth',shorten >=2pt,shorten <=2pt,auto,node
distance=3.0cm, thick,main node/.style={circle,scale=0.7,minimum size=1.1cm,
fill=blue!20,draw,font=\sffamily\Large}]

  \node[main node] (0) {$0$};
  \node[main node] (1) [right of=0] {$1$};
  \node[main node] (2) [right of=1] {$2$};
  \node[main node] (3) [right of=2] {$3$};
  \node[node distance=2cm] (4) [right of=3] {$\dots$};

  \path[every node/.style={font=\sffamily\small}]
    (0) edge [loop left,left,distance=1.5cm,out=-150,in=150] node {$1$} (0)
    (1) edge [bend left,above] node {$p$} (2)
    (2) edge [bend left,above] node {$p$} (3)
    (3) edge [bend left,above] node {$p$} (4)
    (1) edge [above] node {$1-p$} (0)
    (2) edge [bend left,below] node {$1-p$} (1)
    (3) edge [bend left,below] node {$1-p$} (2)
    (4) edge [bend left,below] node {$1-p$} (3)
    ;
\end{tikzpicture}
\end{center}

% \medskip
\noindent
On note $h(i) = \probin{i}{\tau_0<\infty}$. 

\begin{enum}
\item	D\'eterminer $h(0)$. 
\item	Montrer que 
\begin{equation}
\label{1}
h(i) = (1-p) h(i-1) + p h(i+1)
\end{equation} 
pour tout $i\geqs1$. 
\item	On consid\`ere le cas $p=1/2$. 
En se basant sur les propri\'et\'es des fonctions harmoniques
d\'eriv\'ees dans les exercices pr\'ec\'edents, d\'eterminer $h(i)$ pour tout
$i$.

{\it Indication:} $h(i)$ est une probabilit\'e.  
\item	Nous consid\'erons \`a partir de maintenant le cas g\'en\'eral $0<p<1$.
Montrer que tout $h$ solution de~\eqref{1} satisfait encore le principe
du maximum: pour tout ensemble $A=\set{a,a+1,\dots,b-1,b}\subset\N$, $h$ atteint
son maximum au bord $\partial A=\set{a,b}$ de $A$. 
\item	Montrer que deux solutions de l'\'equation~\eqref{1}, co\"\i ncidant
sur le bord de $A$, sont \'egales partout dans $A$. 
\item	Montrer que $h(i)=\alpha^i$ satisfait l'\'equation~\eqref{1} pour
certaines valeurs de $\alpha$ qu'on d\'eterminera. 
\item	En d\'eduire l'unique solution pour $h$ lorsque $p<1/2$. Trouver deux
solutions possibles lorsque $p>1/2$. 
\end{enum}

\end{exercice}

\goodbreak

\begin{exercice}
\label{exo_md08} 

Soit $p\in[0,1]$. 
On consid\`ere la \chaine\ de Markov suivante sur $\cX=\N$: 
% \bigskip
% 
% \centerline{
%  \includegraphics*[clip=true,height=12mm]{figs/chainI_exam_10}
%  }
 
\begin{center}
\begin{tikzpicture}[->,>=stealth',shorten >=2pt,shorten <=2pt,auto,node
distance=3.0cm, thick,main node/.style={circle,scale=0.7,minimum size=1.1cm,
fill=blue!20,draw,font=\sffamily\Large}]

  \node[main node] (0) {$0$};
  \node[main node] (1) [right of=0] {$1$};
  \node[main node] (2) [right of=1] {$2$};
  \node[main node] (3) [right of=2] {$3$};
  \node[node distance=2cm] (4) [right of=3] {$\dots$};

  \path[every node/.style={font=\sffamily\small}]
    (0) edge [loop left,left,distance=1.5cm,out=-150,in=150] node {$1-p$} (0)
    (0) edge [bend left,above] node {$p$} (1)
    (1) edge [bend left,above] node {$p$} (2)
    (2) edge [bend left,above] node {$p$} (3)
    (3) edge [bend left,above] node {$p$} (4)
    (1) edge [bend left,below] node {$1-p$} (0)
    (2) edge [bend left,below] node {$1-p$} (1)
    (3) edge [bend left,below] node {$1-p$} (2)
    (4) edge [bend left,below] node {$1-p$} (3)
    ;
\end{tikzpicture}
\end{center}
% \medskip
% \noindent

\begin{enum}
\item	Pour quelles valeurs de $p$ la \chaine\ est-elle irr\'eductible?

On suppose dans la suite que $p$ est tel que la \chaine\ soit irr\'eductible.

\item	La \chaine\ est-elle ap\'eriodique?

\item	On suppose que la \chaine\ est r\'eversible, et soit $\alpha$ un
vecteur r\'eversible. Ecrire une relation de r\'ecurrence pour les composantes
de $\alpha$, et en d\'eduire $\alpha_n$ en fonction de $\alpha_0$.

\item	Pour quelles valeurs de $p$ la \chaine\ admet-elle une distribution
stationnaire $\pi$? D\'eter\-miner $\pi$ pour ces valeurs de $p$.

\item	Pour quelles valeurs de $p$ la \chaine\ est-elle r\'ecurrente?
R\'ecurrente positive?

\item	D\'eterminer le temps de r\'ecurrence moyen $\expecin{0}{\tau_0}$. 

\item	Calculer la position moyenne $\expecin{\pi}{X_n}$ pour les valeurs de
$p$ telles que $\pi$ existe. 
\end{enum}

\end{exercice}

\goodbreak

\begin{exercice}
\label{exo_md09} 

On consid\`ere une marche al\'eatoire sym\'etrique sur
$\cX=\set{0,1,\dots,N}$, avec conditions au bord absorbantes,
c'est-\`a-dire que d\`es que la marche atteint l'un des \'etats $0$ ou
$N$, elle y reste ind\'efiniment. Soit 
\[
\tau = \inf\setsuch{n\geqs 0}{X_n\in\set{0,N}}
\]
le temps d'absorption. Par convention, $\tau=0$ si $X_0\in\set{0,N}$. 
Pour $\lambda\in\R$ et $i\in\cX$ on pose 
\[
f(i,\lambda) = \bigexpecin{i}{\e^{-\lambda\tau}\indexfct{X_\tau=N}}
= 
\begin{cases}
\bigexpecin{i}{\e^{-\lambda\tau}} & \text{si $X_\tau=N$\;,} \\
0 & \text{sinon\;.}
\end{cases}
\]

\begin{enum}
\item	Que valent $f(0,\lambda)$ et $f(N,\lambda)$?

\item	Montrer que pour tout $i\in\set{1,\dots,N-1}$, 
\[
\probin{i}{\tau=n} = \frac12 
\bigbrak{\probin{i-1}{\tau=n-1} + \probin{i+1}{\tau=n-1}}\;.
\]

\item	Montrer que pour tout $i\in\set{1,\dots,N-1}$, 
\[
f(i,\lambda) = \frac12\e^{-\lambda}
\bigbrak{f(i-1,\lambda) + f(i+1,\lambda)}\;.
\]

\item	Trouver une relation entre $c$ et $\lambda$ telle que l'\'equation
ci-dessus pour $f$ admette des solutions de la forme $f(i,\lambda)=\e^{ci}$. 
Montrer \`a l'aide d'un d\'eveloppement limit\'e que 
\[
c^2 = 2\lambda + \Order{\lambda^2}\;.
\]

\item	D\'eterminer des constantes $a$ et $b$ telles que 
\[
\bigexpecin{i}{\e^{-\lambda\tau}\indexfct{X_\tau=N}}
= a \e^{ci} + b \e^{-ci}\;.
\]

\item	Effectuer un d\'eveloppement limit\'e  au
premier ordre en $\lambda$ de l'\'egalit\'e ci-dessus.
En d\'eduire 
\[
\probin{i}{X_\tau=N}
\]
et 
\[
\bigexpecin{i}{\tau \indexfct{X_\tau=N}}\;.
\]

\item	Sans faire les calculs, indiquer comment proc\'eder pour d\'eterminer
la variance de $\tau \indexfct{X_\tau=N}$ et l'esp\'erance et la variance de
$\tau$. 
\end{enum}

\medskip{\noindent}
On rappelle les d\'eveloppements limit\'es suivants:
\begin{align*}
\cosh(x) &= \frac{\e^x+\e^{-x}}{2}
= 1 + \frac{1}{2!}x^2 + \Order{x^4}\;, \\
\sinh(x) &= \frac{\e^x-\e^{-x}}{2}
= x + \frac{1}{3!}x^3 + \Order{x^5}\;.
\end{align*}
\end{exercice}

\goodbreak

\begin{exercice}
\label{exo_md10}

Soit $p\in[0,1]$. 
On consid\`ere la marche al\'eatoire sur $\Z$ dont le graphe est le suivant. 

% \bigskip
% 
% \centerline{
%  \includegraphics*[clip=true,width=140mm]{figs/chain_asym_rw}
%  }
 
 \begin{center}
\begin{tikzpicture}[->,>=stealth',shorten >=2pt,shorten <=2pt,auto,node
distance=3.0cm, thick,main node/.style={circle,scale=0.7,minimum size=1.1cm,
fill=blue!20,draw,font=\sffamily\Large}]

  \node[main node] (0) {$0$};
  \node[main node] (1) [right of=0] {$1$};
  \node[main node] (2) [right of=1] {$2$};
  \node[node distance=2cm] (3) [right of=2] {$\dots$};
  \node[main node] (-1) [left of=0] {$-1$};
  \node[main node] (-2) [left of=-1] {$-2$};
  \node[node distance=2cm] (-3) [left of=-2] {$\dots$};

  \path[every node/.style={font=\sffamily\small}]
    (-3) edge [bend left,above] node {$p$} (-2)
    (-2) edge [bend left,above] node {$p$} (-1)
    (-1) edge [bend left,above] node {$p$} (0)
    (0) edge [bend left,above] node {$p$} (1)
    (1) edge [bend left,above] node {$p$} (2)
    (2) edge [bend left,above] node {$p$} (3)
    (-2) edge [bend left,below] node {$1-p$} (-3)
    (-1) edge [bend left,below] node {$1-p$} (-2)
    (0) edge [bend left,below] node {$1-p$} (-1)
    (1) edge [bend left,below] node {$1-p$} (0)
    (2) edge [bend left,below] node {$1-p$} (1)
    (3) edge [bend left,below] node {$1-p$} (2)
    ;
\end{tikzpicture}
\end{center}

%\medskip

\begin{enum}
\item	Pour quelles valeurs de $p$ la \chaine\ est-elle irr\'eductible?

On suppose dans la suite que $p$ est tel que la \chaine\ soit irr\'eductible.

\item	La \chaine\ est-elle ap\'eriodique?

\item 	Calculer explicitement 
\[
\bigprobin{0}{X_{2n}=0}
\]
pour tout $n\in\N$. 

\item 	A l'aide de la formule de Stirling, trouver un \'equivalent 
de $\probin{0}{X_{2n}}=0$ pour $n\to\infty$. 

Pour quelles valeurs de $p$ la cha\^ine est-elle r\'ecurrente? Transiente? 

\item 	Soit maintenant $Y_n$ la cha\^ine de Markov sur $\N$ dont le graphe est
le suivant:
% \bigskip
% 
% \centerline{
%  \includegraphics*[clip=true,width=100mm]{figs/chain_exam_12}
%  }
 
\begin{center}
\begin{tikzpicture}[->,>=stealth',shorten >=2pt,shorten <=2pt,auto,node
distance=3.0cm, thick,main node/.style={circle,scale=0.7,minimum size=1.1cm,
fill=blue!20,draw,font=\sffamily\Large}]

  \node[main node] (0) {$0$};
  \node[main node] (1) [right of=0] {$1$};
  \node[main node] (2) [right of=1] {$2$};
  \node[main node] (3) [right of=2] {$3$};
  \node[node distance=2cm] (4) [right of=3] {$\dots$};

  \path[every node/.style={font=\sffamily\small}]
    (0) edge [bend left,above] node {$1$} (1)
    (1) edge [bend left,above] node {$p$} (2)
    (2) edge [bend left,above] node {$p$} (3)
    (3) edge [bend left,above] node {$p$} (4)
    (1) edge [bend left,below] node {$1-p$} (0)
    (2) edge [bend left,below] node {$1-p$} (1)
    (3) edge [bend left,below] node {$1-p$} (2)
    (4) edge [bend left,below] node {$1-p$} (3)
    ;
\end{tikzpicture}
\end{center}

%\medskip

Montrer que si $X_0 = Y_0 > 0$, alors 
\[
Y_n = X_n 
\quad 
\forall n \leqs \tau_0\;.
\]

\item 	En d\'eduire une condition suffisante sur $p$ pour que la cha\^ine $Y_n$
soit transiente.

\item 	Pour quelles valeurs de $p$ la cha\^ine $Y_n$ admet-elle une
distribution stationnaire $\pi$? En d\'eduire pour quels $p$ la cha\^ine $Y_n$
est r\'ecurrente positive. 

\item 	Montrer que pour $p=1/2$, on a $Y_n=\abs{X_n}$. Que peut-on en
d\'eduire sur les propri\'et\'es de r\'ecurrence/transience de la cha\^ine
$Y_n$? 
 
\end{enum}
\end{exercice}

\goodbreak

\begin{exercice}
\label{exo_md11}

Soit $p\in]0,1[$. 
On consid\`ere la cha\^ine de Markov sur $\set{0,1,\dots N}$ dont le graphe est
le suivant. 

\begin{center}
\begin{tikzpicture}[->,>=stealth',shorten >=2pt,shorten <=2pt,auto,node
distance=3.0cm, thick,main node/.style={circle,scale=0.7,minimum size=1.1cm,
fill=blue!20,draw,font=\sffamily\Large}]

  \node[main node] (0) {$0$};
  \node[main node] (1) [right of=0] {$1$};
  \node[main node] (2) [right of=1] {$2$};
  \node[node distance=2cm] (3) [right of=2] {$\dots$};
  \node[main node, distance=1.8cm] (N1) [right of=3] {{\small $N-1$}};
  \node[main node] (N) [right of=N1] {$N$};

  \path[every node/.style={font=\sffamily\small}]
    (0) edge [loop left,left,distance=1.5cm,out=-150,in=150] node {$1$} (0)
    (N) edge [loop right,right,distance=1.5cm,out=30,in=-30] node {$1$} (N)
    (1) edge [bend left,above] node {$p$} (2)
    (2) edge [bend left,above] node {$p$} (3)
    (3) edge [bend left,above] node {$p$} (N1)
    (N1) edge [bend left,above] node {$p$} (N)
    (1) edge [bend left,below] node {$1-p$} (0)
    (2) edge [bend left,below] node {$1-p$} (1)
    (3) edge [bend left,below] node {$1-p$} (2)
    (N1) edge [bend left,below] node {$1-p$} (3)
    ;
\end{tikzpicture}
\end{center}
%\medskip

\noindent
Soit 
\[
\tau=\inf\setsuch{n\geqs0}{X_n\in\set{0,N}} 
\]
le temps d'absorption et soit 
\[
 f(i) = \probin{i}{X_\tau=N}
\]
la probabilit\'e d'\^etre absorb\'e dans l'\'etat $N$. 

\begin{enum}
\item	Que valent $f(0)$ et $f(N)$? 

\item 	Montrer que pour tout $i\in\set{1,\dots,N-1}$, 
\[
f(i) = p f(i+1) + (1-p) f(i-1)\;. 
\]

\item 	Dans toute la suite, on suppose $p\neq 1/2$. 

Montrer qu'il existe deux r\'eels $z_+$ et $z_-$ tels que l'\'equation pour
$f(i)$ admette des solutions de la forme $f(i)=z_+^i$ et $f(i)=z_-^i$. Exprimer
ces r\'eels en fonction de $\rho=(1-p)/p$.

\item 	Trouver deux constantes $a$ et $b$ telles que 
\[
\probin{i}{X_\tau=N} = a z_+^i + b z_-^i\;.
\]

\item 	Soit 
\[
 g(i) = \expecin{i}{\tau}\;.
\]
Montrer que pour tout $i\in\set{1,\dots,N-1}$, 
\[
 g(i) =  1 + p g(i+1) + (1-p) g(i-1)\;.
\]

\item 	Trouver une solution particuli\`ere de la forme $g(i)= \gamma i$ pour un
$\gamma\in\R$ d\'ependant de~$p$. 

\item 	V\'erifier que la solution g\'en\'erale est de la forme 
\[
 g(i) = \gamma i + \alpha z_+^i + \beta z_-^i
\]
et d\'eterminer $\alpha$ et $\beta$ pour que 
\[
 \expecin{i}{\tau} = \gamma i + \alpha z_+^i + \beta z_-^i\;.
\]

\end{enum}
\end{exercice}

%%%%%%%%%%%%%%%%%%%%%%%%%%%%%%%%%%%%%%%%%%%%%%%%%%%%%%%%%%%%%%%%%%%%%%%%%%%

\chapter{Application aux algorithmes MCMC}
\label{chap_mcmc}

%%%%%%%%%%%%%%%%%%%%%%%%%%%%%%%%%%%%%%%%%%%%%%%%%%%%%%%%%%%%%%%%%%%%%%%%%%%

%\section{M\'ethodes Monte Carlo et MCMC}
%\label{sec_mcmc}

%%%%%%%%%%%%%%%%%%%%%%%%%%%%%%%%%%%%%%%%%%%%%%%%%%%%%%%%%%%%%%%%%%%%%%%%%%%

\section{M\'ethodes Monte Carlo}
\label{sec_mc}

On appelle \defwd{m\'ethodes Monte Carlo}\/ un ensemble d'algorithmes
stochastiques permettant d'estimer num\'eriquement des grandeurs pouvant
\^etre consid\'er\'ees comme des esp\'erances. En voici pour
commencer un exemple tr\`es simple.

\begin{example}[Calcul d'un volume]
On aimerait calculer num\'eriquement le volume $\abs{V}$ d'un
sous-ensemble compact $V$ de $\R^N$. On suppose que $V$ est donn\'e par un
certain nombre $M$ d'in\'egalit\'es: 
\begin{equation}
\label{mcmc1}
V = \bigsetsuch{x\in\R^N}{f_1(x)\geqs0, \dots, f_M(x)\geqs0}\;. 
\end{equation}
Par exemple, si $M=1$ et $f_1(x)=1-\norm{x}^2$, alors $V$ est une boule.
Dans ce cas, bien s\^ur, le volume de $V$ est connu explicitement. On
s'int\'eresse \`a des cas o\`u $V$ a une forme plus compliqu\'ee, par
exemple une intersection d'un grand nombre de boules et de demi-plans.
Dans la suite nous supposerons, sans limiter la g\'en\'eralit\'e, que $V$
est contenu dans le cube unit\'e $[0,1]^N$.

Une premi\`ere m\'ethode de calcul num\'erique du volume consiste \`a
discr\'etiser l'espace. Divisons le cube $[0,1]^N$ en cubes de cot\'e
$\eps$ (avec $\eps$ de la forme $1/K$, $K\in\N$). Le nombre total de
ces cubes est \'egal \`a $1/\eps^N=K^N$. On compte alors le nombre $n$ de
cubes dont le centre est contenu dans $V$, et le volume de $V$ est 
approximativement \'egal \`a $n\eps^N$. Plus pr\'ecis\'ement, on peut
encadrer $\abs{V}$ par $n_-\eps^N$ et $n_+\eps^N$, o\`u $n_-$ est le
nombre de cubes enti\`erement contenus dans $V$, et $n_+$ est le nombre de
cubes dont l'intersection avec $V$ est vide, mais effectuer ces tests n'est
en g\'en\'eral pas ais\'e num\'eriquement.

Quelle est la pr\'ecision de l'algorithme? Si le bord $\partial V$ est
raisonnablement lisse, l'erreur faite pour $\eps$ petit est de l'ordre de
la surface du bord fois $\eps$. Pour calculer $\abs{V}$ avec une
pr\'ecision donn\'ee $\delta$, il faut donc choisir $\eps$ d'ordre
$\delta/\abs{\partial V}$. Cela revient \`a un nombre de cubes d'ordre 
\begin{equation}
\label{mcmc2}
\biggpar{\frac{\abs{\partial V}}{\delta}}^N\;,
\end{equation}
ou encore, comme on effectue $M$ tests pour chaque cube, \`a un nombre de
calculs d'ordre $(M\abs{\partial V}/\delta)^N$. Ce nombre ne pose pas de
probl\`eme pour les petites dimensions ($N=1,2$ ou $3$ par exemple), mais
cro\^\i t vite avec la dimension $N$.

Une alternative int\'eressante pour les $N$ grands est fournie par
l'algorithme de Monte Carlo. Dans ce cas, on g\'en\`ere une suite $X_1,
X_2, \dots, X_n, \dots$ de variables al\'eatoires ind\'epen\-dantes,
identiquement distribu\'ees (i.i.d.), de loi uniforme sur $[0,1]^N$. Ceci
est facile \`a r\'ealiser num\'eriquement, car on dispose de
g\'en\'erateurs de nombres (pseudo-)al\'eatoires distribu\'es
uniform\'ement sur $[0,1]$ (ou plut\^ot sur
$\set{0,1,\dots,n_{\text{max}}}$ o\`u $n_{\text{max}}$ est un entier du
genre $2^{31}-1$, mais en divisant ces nombres par $n_{\text{max}}$, on
obtient de bonnes approximations de variables uniformes sur $[0,1]$). Il
suffit alors de consid\'erer des $N$-uplets de tels nombres. 

Consid\'erons alors les variables al\'eatoires i.i.d. 
\begin{equation}
\label{mcmc3}
Y_i = \indexfct{X_i\in V}\;, 
\qquad i = 1,2,\dots\;.
\end{equation}
On aura 
\begin{equation}
\label{mcmc4}
\expec{Y_i} = \prob{X_i\in V} = \abs{V}\;.
\end{equation}
Les moyennes empiriques
\begin{equation}
\label{mcmc5}
S_n = \frac1n \sum_{i=1}^n Y_i
\end{equation}
ont esp\'erance $\expec{S_n}=\abs{V}$ et variance
$\variance{S_n}=\variance{Y_1}/n$. La loi faible des grands nombres
implique que 
\begin{equation}
\label{mcmc6}
\lim_{n\to\infty} \bigprob{\bigabs{S_n-\expec{S_n}} > \delta} = 0
\end{equation}
pour tout $\delta>0$. Par cons\'equent, $S_n$ devrait donner une bonne
approximation du volume $\abs{V}$ lorsque $n$ est suffisamment grand. (La
loi forte des grands nombres affirme par ailleurs que $S_n$ tend vers
$\abs{V}$ presque s\^urement, c'est-\`a-dire que $S_n$ n'est plus vraiment
al\'eatoire \`a la limite.)

Pour savoir comment choisir $n$ en fonction de la pr\'ecision d\'esir\'ee,
il nous faut estimer la probabilit\'e que $\abs{S_n-\abs{V}} > \delta$,
pour $n$ grand mais fini. Une premi\`ere estimation est fournie par
l'in\'egalit\'e de Bienaym\'e--Chebychev, qui affirme que 
\begin{equation}
\label{mcmc7}
\bigprob{\bigabs{S_n-\expec{S_n}} > \delta} \leqs 
\frac{\variance(S_n)}{\delta^2}
= \frac{\variance(Y_1)}{\delta^2 n}
< \frac{1}{\delta^2 n}\;,
\end{equation}
o\`u nous avons utilis\'e le fait que
$\variance(Y_1)\leqs\expec{Y_1^2}\leqs 1$. On peut donc affirmer que pour
que la probabilit\'e de faire une erreur sup\'erieure \`a $\delta$ soit
inf\'erieure \`a $\eps$, il suffit de choisir 
\begin{equation}
\label{mcmc8}
n > \frac1{\delta^2\eps}\;.
\end{equation}
Comme pour chaque $i$, il faut g\'en\'erer $N$ variables al\'eatoires, et
effectuer $M$ tests, le nombre de calculs n\'ecessaires est d'ordre
$MN/(\delta^2\eps)$. L'avantage de cette m\'ethode est que ce nombre ne
cro\^\i t que lin\'eairement avec $N$, par opposition \`a la croissance
exponentielle dans le cas de la discr\'etisation. On notera toutefois que
contrairement \`a la discr\'etisation, qui donne un r\'esultat certain aux
erreurs pr\`es, la m\'ethode de Monte Carlo ne fournit que des r\'esultats
vrais avec tr\`es grande probabilit\'e (d'o\`u son nom).

Remarquons que l'estimation~\eqref{mcmc7} est assez pessimiste, et peut
\^etre consid\'erablement am\'elior\'ee. Par exemple, le th\'eor\`eme de
la limite centrale montre que 
\begin{equation}
\label{mcmc9}
\lim_{n\to\infty}
\biggprob{\frac{(S_n-\expec{S_n})^2}{\variance(S_n)} > \eta^2}
= \int_{\abs{x}>\eta} \frac{\e^{-x^2/2}}{\sqrt{2\pi}} \6x\;,
\end{equation}
dont le membre de droite d\'ecro\^\i t comme $\e^{-\eta^2/2}$ pour $\eta$
grand. Cela indique que pour $n$ grand, 
\begin{equation}
\label{mcmc10}
\bigprob{\abs{S_n-\abs{V}} > \delta} \simeq
\e^{-n\delta^2/2\variance(Y_1)}\;.
\end{equation}
Ceci permet d'am\'eliorer le crit\`ere~\eqref{mcmc8} en 
\begin{equation}
\label{mcmc11}
n > \const \frac{\log(1/\eps)}{\delta^2}\;,
\end{equation}
d'o\`u un nombre de calculs d'ordre $NM\log(1/\eps)/\delta^2$. Cette
estimation n'est pas une borne rigoureuse, contrairement \`a~\eqref{mcmc8},
parce qu'on n'a pas tenu compte de la vitesse de convergence
dans~\eqref{mcmc9}, qui par ailleurs ne s'applique que pour $\eta$
ind\'ependant de $\eps$. On devrait plut\^ot utiliser des estimations
provenant de la th\'eorie des grandes d\'eviations, que nous n'aborderons
pas ici. Les r\'esultats sont toutefois qualitativement corrects.

A titre d'illustration, supposons qu'on veuille d\'eterminer le volume
d'un domaine de dimension $N=1000$, d\'efini par $M=10$ in\'egalit\'es,
avec une pr\'ecision de $\delta=10^{-4}$. La m\'ethode de discr\'etisation
n\'ecessite un nombre de calculs d'ordre $10^{5000}$, ce qui est
irr\'ealisable avec les ordinateurs actuels. La m\'ethode de Monte Carlo,
en revanche, fournit un r\'esultat de la m\^eme pr\'ecision, s\^ur avec
probabilit\'e $1-10^{-6}$, avec un nombre de calculs d'ordre
$\log(10^6)\cdot 10^{12} \simeq 10^{13}$, ce qui ne prend que quelques
minutes sur un PC. 
\end{example}

La m\'ethode de Monte Carlo se g\'en\'eralise facilement \`a d'autres
probl\`emes que des calculs de volume. Supposons par exemple donn\'e un
espace probabilis\'e $(\Omega,\cF,\mu)$, et une variable al\'eatoire
$Y:\Omega\to\R$. On voudrait estimer l'esp\'erance de $Y$. Pour cela,
l'algorithme de Monte Carlo consiste \`a g\'en\'erer des variables
al\'eatoires ind\'ependantes $Y_1, Y_2, \dots, Y_n, \dots$, toutes de loi
$\mu Y^{-1}$, et de calculer leur moyenne. Cette moyenne doit converger
vers l'esp\'erance cherch\'ee (pourvu que $Y$ soit int\'egrable). 

Cet algorithme n'est toutefois efficace que si l'on arrive \`a g\'en\'erer
les variables al\'eatoires $Y_i$ de mani\`ere efficace. Une fois de plus,
ceci est relativement facile en dimension faible, mais devient rapidement
difficile lorsque la dimension cro\^\i t. 

\begin{example}[Cas unidimensionnel]
\label{ex_mcmc2}
Une variable al\'eatoire uni\-dimensionnelle $Y$ s'ob\-tient facilement \`a
partir d'une variable de loi uniforme. Soit en effet $U$ une variable
uniforme sur $[0,1]$. Sa fonction de r\'epartition est donn\'ee par 
\begin{equation}
\label{mcmc12}
F_U(u) = \prob{U\leqs u} = u 
\qquad
\text{pour $0\leqs u\leqs 1$.} 
\end{equation}
On cherche une fonction $\varphi$ telle que la variable $Y=\varphi(u)$
admette la fonction de r\'eparti\-tion prescrite $F_Y(y)$. Or on a 
\begin{equation}
\label{mcmc13}
F_Y(y) = \prob{Y\leqs y} = \prob{\varphi(U)\leqs y} 
= \prob{U \leqs \varphi^{-1}(y)} = \varphi^{-1}(y)\;.
\end{equation}
Il suffit donc de prendre $Y = F_Y^{-1}(U)$. 

Par exemple, pour g\'en\'erer une variable de loi exponentielle, dont la
fonction de r\'eparti\-tion est donn\'ee par $F_Y(y)=1-\e^{-\lambda y}$, il
suffit de prendre 
\begin{equation}
\label{mcmc14}
Y = -\frac1\lambda \log(1-U)\;.
\end{equation}

Pour la loi normale, cette m\'ethode n\'ecessiterait le calcul approch\'e
de sa fonction de r\'epartition, ce qui est num\'eriquement peu efficace. 
Il existe toutefois une alternative permettant d'\'eviter ce calcul.
Soient en effet $U$ et $V$ deux variables al\'eatoires ind\'ependantes, de
loi uniforme sur $[0,1]$. On introduit successivement les variables 
\begin{align}
\nonumber
R &= \sqrt{-2\log(1-U)}\;,
& 
Y_1 &= R\cos\Phi\;, \\
\Phi &= 2\pi V\;,
&
Y_2 &= R\sin\Phi\;.
\label{mcmc15}
\end{align}
Alors $Y_1$ et $Y_2$ sont des variables al\'eatoires ind\'ependantes, de
loi normale centr\'ee r\'eduite. Pour le voir, on v\'erifie d'abord que
$R$ a la fonction de r\'epartition $1-\e^{-r^2/2}$, donc la densit\'e
$r\e^{-r^2/2}$. Le couple $(R,\Phi)$ a donc la densit\'e jointe
$r\e^{-r^2/2}/(2\pi)$, et la formule de changement de variable montre que
le couple $(Y_1,Y_2)$ a la densit\'e jointe
$\e^{-(y_1^2+y_2^2)/2}/(2\pi)$, qui est bien celle d'un couple de
variables normales centr\'ees r\'eduites ind\'ependantes.
\end{example}

Bien entendu, les esp\'erances de variables al\'eatoires de loi
unidimensionnelle sont soit connues explicitement, soit calculables
num\'eriquement par la simple estimation d'une int\'egrale. Nous nous
int\'eressons ici \`a des situations o\`u la loi de $Y$ n'est pas aussi
simple \`a repr\'esenter. 

%%%%%%%%%%%%%%%%%%%%%%%%%%%%%%%%%%%%%%%%%%%%%%%%%%%%%%%%%%%%%%%%%%%%%%%%%%%

\section{Algorithmes MCMC}
\label{sec_mcmc}

Consid\'erons le cas d'un espace probabilis\'e discret
$(\cX,\cP(\cX),\mu)$, o\`u $\cX$ est un ensemble d\'enombrable, mais
tr\`es grand. Par exemple, dans le cas du mod\`ele d'Ising (cf.
Exemple~\ref{ex_Ising}), $\cX=\set{-1,1}^N$ est de cardinal $2^N$, et
on s'int\'eresse \`a des $N$ d'ordre $1000$ au moins. La mesure de
probabilit\'e $\mu$ est dans ce cas une application de $\cX$ vers $[0,1]$
telle que la somme des $\mu(i)$ vaut $1$. 

On voudrait estimer l'esp\'erance d'une variable al\'eatoire $Y:\cX\to\R$,
comme par exemple l'aimantation dans le cas du mod\`ele d'Ising~: 
\begin{equation}
\label{mcmc16}
\expec{Y} = \sum_{i\in\cX} Y(i)\mu(i)\;.
\end{equation}
La m\'ethode de Monte Carlo usuelle consiste alors \`a g\'en\'erer une
suite de variables al\'eatoires $X_0, X_1, \dots$ sur $\cX$,
ind\'ependantes et de distribution $\mu$, puis de calculer la moyenne des
$Y(X_j)$. 

Pour g\'en\'erer ces $X_j$, on peut envisager de proc\'eder comme suit~: on
d\'efinit un ordre total sur $\cX$, et on d\'etermine la fonction de
r\'epartition 
\begin{equation}
\label{mcmc17}
F_\mu(j) = \sum_{k=1}^j \mu(k)\;.
\end{equation}
Si $U$ est une variable de loi uniforme sur $[0,1]$, alors $F_\mu^{-1}(U)$
suit la loi $\mu$. Toutefois, en proc\'edant de cette mani\`ere, on ne
gagne rien, car le calcul des sommes~\eqref{mcmc17} est aussi long que le
calcul de la somme~\eqref{mcmc16}, que l'on voulait pr\'ecis\'ement
\'eviter!

Les m\'ethodes MCMC (pour \emph{Monte Carlo Markov Chain}\/) \'evitent cet
inconv\'enient. L'id\'ee est de simuler \emph{en m\^eme temps}\/ la
loi $\mu$ et la variable al\'eatoire $Y$, \`a l'aide d'une cha\^\i ne de
Markov sur $\cX$, de distribution invariante $\mu$. 

Soit donc $\set{X_n}_{n\in\N}$ une telle \chaine, suppos\'ee
irr\'eductible, et de distribution initiale $\nu$ arbitraire. On lui
associe une suite $Y_n=Y(X_n)$ de variables al\'eatoires.
Celles-ci peuvent se d\'ecomposer comme suit~: 
\begin{equation}
\label{mcmc18}
Y_n = \sum_{i\in\cX} Y(i) \indexfct{X_n=i}\;.
\end{equation}
Consid\'erons les moyennes 
\begin{equation}
\label{mcmc19}
S_n = \frac1n \sum_{m=0}^{n-1} Y_m\;. 
\end{equation}
Le Th\'eor\`eme~\ref{thm_firred3} permet d'\'ecrire 
\begin{align}
\nonumber
\lim_{n\to\infty} \bigexpecin{\nu}{S_n} 
&= \lim_{n\to\infty} \frac1n \biggexpecin{\nu}{\sum_{m=0}^{n-1}Y_m} \\
\nonumber
&= \sum_{i\in\cX} Y(i) \lim_{n\to\infty} \frac1n 
\biggexpecin{\nu}{\sum_{m=0}^{n-1}\indexfct{X_m=i}} \\
\nonumber
&= \sum_{i\in\cX} Y(i) \mu(i) \\
&= \expec{Y}\;.
\label{mcmc20}
\end{align}
L'esp\'erance de $S_n$ converge bien vers l'esp\'erance cherch\'ee. Pour
pouvoir appliquer l'id\'ee de la m\'ethode Monte Carlo, il nous faut plus,
\`a savoir la convergence (au moins) en probabilit\'e de $S_n$ vers
$\expec{Y}$. On ne peut pas invoquer directement la loi des grands nombres,
ni le th\'eor\`eme central limite, car les $Y_n$ ne sont plus
ind\'ependants. Mais il s'av\`ere que des r\'esultats analogues restent
vrais dans le cas de \chaine s de Markov. 

\begin{theorem}
Supposons la \chaine\ r\'eversible, et de distribution initiale \'egale
\`a sa distribution stationnaire. Soit $\lambda_0$ la plus grande (en
module) valeur propre diff\'erente de $1$ de la matrice de transition de la
\chaine. Alors 
\begin{equation}
\label{mcmc21}
\variance{S_n} \leqs \frac1n
\biggpar{\frac{1+\abs{\lambda_0}}{1-\abs{\lambda_0}}} \variance{Y}\;.
%\biggpar{1+\frac{2\abs{\lambda_0}}{1-\abs{\lambda_0}}} \variance{Y}\;.
\end{equation}
\end{theorem}
\begin{proof}
Comme la \chaine\ d\'emarre dans la distribution stationnaire $\mu$, tous
les $Y_i$ ont m\^eme loi $\mu Y^{-1}$, m\^eme s'ils ne sont pas
ind\'ependants. Il suit que 
\begin{align}
\nonumber
\variance{S_n} &= \frac1{n^2} 
\biggbrak{\sum_{m=0}^{n-1}\variance{Y_m} + 2\sum_{0\leqs p<q<n}
\cov(Y_p,Y_q)} \\
&= \frac1n \variance{Y} + \frac2{n^2} \sum_{m=1}^{n-1} (n-m)
\cov(Y_0,Y_m)\;,
\label{mcmc22}
\end{align}
en vertu de la propri\'et\'e des incr\'ements ind\'ependants. Or si 
$\vone=\smash{\transpose{(1,1,\dots,1)}}$ on a 
\begin{align}
\nonumber
\cov(Y_0,Y_m) 
&= \Bigexpecin{\mu}{(Y_0-\expecin{\mu}{Y_0}) (Y_m-\expecin{\mu}{Y_m})} \\
\nonumber
&= \sum_{i\in\cX} \sum_{j\in\cX} (Y(i)-\expec{Y})
(Y(j)-\expec{Y}) 
\underbrace{\probin{\mu}{X_0=i,X_m=j}}_{=\mu(i)(P^m)_{ij}}\\
\nonumber
&= \sum_{i\in\cX} \mu(i)(Y(i)-\expec{Y})
\brak{P^m(Y-\expec{Y}\vone)}_i\\
\nonumber
&= \pscal{Y-\expec{Y}\vone}{P^m(Y-\expec{Y}\vone)}_\mu\\
&\leqs \abs{\lambda_0}^m
\pscal{Y-\expec{Y}\vone}{Y-\expec{Y}\vone}_\mu
= \abs{\lambda_0}^m \variance{Y}\;.
\label{mcmc23}
\end{align}
Dans l'in\'egalit\'e, nous avons utilis\'e le fait que $Y-\expec{Y}\vone
\in \vone_\perp$ puisque la somme des $\mu_i(Y(i)-\expec{Y})$ est nulle, et
que par cons\'equent ce vecteur se trouve dans le sous-espace
compl\'ementaire au vecteur propre $\vone$. Le r\'esultat suit alors en
rempla\c cant dans~\eqref{mcmc22}, en majorant $n-m$ par $n$ et en sommant
une s\'erie g\'eom\'etrique. 
\end{proof}

Il suit de cette estimation et de l'in\'egalit\'e de Bienaym\'e--Chebychev
que pour calculer $\expec{Y}$ avec une pr\'ecision $\delta$ et avec
probabilit\'e $1-\eps$, il faut choisir 
\begin{equation}
\label{mcmc24}
n \geqs 
%\frac{\variance Y}{\delta^2\eps}
%\biggpar{1+\frac{2\abs{\lambda_0}}{1-\abs{\lambda_0}}} = 
\frac{\variance Y}{\delta^2\eps}
\biggpar{\frac{1+\abs{\lambda_0}}{1-\abs{\lambda_0}}}\;.
\end{equation}
En pratique, on ne peut pas faire d\'emarrer la \chaine\ exactement avec
la distribution invariante. Ceci conduit \`a une convergence un peu plus
lente, mais du m\^eme ordre de grandeur puisque la loi des $Y_n$ converge
exponentiellement vite vers $\mu Y^{-1}$. Les r\'esultats sont bien s\^ur
meilleurs si on choisit bien la condition initiale, c'est-\`a-dire de
mani\`ere \`a ce que la loi des $Y_n$ converge rapidement. 

%%%%%%%%%%%%%%%%%%%%%%%%%%%%%%%%%%%%%%%%%%%%%%%%%%%%%%%%%%%%%%%%%%%%%%%%%%%

\section{L'algorithme de Metropolis}
\label{sec_metro}

Nous avons vu comment estimer l'esp\'erance d'une variable al\'eatoire
$Y$ \`a l'aide d'une \chaine\ de Markov de distribution invariante
donn\'ee par la loi de $Y$. Pour que cet algorithme soit efficace, il faut
encore que l'on puisse trouver facilement, en fonction de cette loi, une 
matrice de transition donnant la distribution invariante souhait\'ee. 

Une mani\`ere simple de r\'esoudre ce probl\`eme est de chercher une
\chaine\ de Markov r\'eversible. 
%Dans ce contexte, on parle souvent
%d'\defwd{algorithme de M\'etropolis}. 
Nous allons illustrer cette m\'ethode dans le cas du mod\`ele d'Ising, mais
on voit facilement comment la g\'en\'eraliser \`a d'autres syst\`emes. 

Dans le cas du mod\`ele d'Ising (Exemple~\ref{ex_Ising}), l'univers est
donn\'e par $\cX=\set{-1,1}^\Lambda$, o\`u $\Lambda$ est un sous-ensemble
(suppos\'e ici fini) de $\Z^d$. La mesure de probabilit\'e sur $\cX$ est
d\'efinie par
\begin{equation}
\label{metro1}
\mu(\sigma) = \frac{\e^{-\beta H(\sigma)}}{Z_\beta}\;,
\qquad
\text{o\`u }
Z_\beta = \sum_{\sigma\in\cX}\e^{-\beta H(\sigma)}\;.
\end{equation}
Le param\`etre $\beta$ d\'esigne la temp\'erature inverse, et la fonction
$H: \cX\to\R$, associant \`a toute configuration son \'energie, est
donn\'ee par 
\begin{equation}
\label{metro2}
H(\sigma) = -\sum_{i,j\in\Lambda\colon\norm{i-j}=1}\sigma_i\sigma_j 
- h \sum_{i\in\Lambda} \sigma_i\;,
\end{equation}
o\`u $h$ est le champ magn\'etique. L'objectif est de calculer
l'esp\'erance de la variable aimantation, donn\'ee par
\begin{equation}
\label{metro3}
m(\sigma) = \frac1{\abs{\Lambda}}\sum_{i\in\Lambda} \sigma_i
\end{equation}
(nous avons introduit un facteur $1/\abs{\Lambda}$ pour assurer que $m$
prenne des valeurs entre $-1$ et~$1$). Afin de satisfaire la condition de
r\'eversibilit\'e, on cherche une matrice de transition $P$ sur $\cX$ dont
les \'el\'ements satisfont 
\begin{equation}
\label{metro4}
\mu(\sigma)p_{\sigma\sigma'} = \mu(\sigma')p_{\sigma'\sigma}
\end{equation}
pour toute paire $(\sigma,\sigma')\in\cX\times\cX$. Cela revient \`a
imposer que 
\begin{equation}
\label{metro5}
\frac{p_{\sigma\sigma'}}{p_{\sigma'\sigma}}
= \e^{-\beta\Delta H(\sigma,\sigma')}\;,
\end{equation}
o\`u $\Delta H(\sigma,\sigma') = H(\sigma') - H(\sigma)$. On notera que
cette condition ne fait pas intervenir la constante de normalisation
$Z_\beta$, ce qui est souhaitable, car le calcul de cette constante est
aussi co\^uteux que celui de $\expec{m}$. 

L'\defwd{algorithme de M\'etropolis}\/ consiste dans un premier temps \`a
d\'efinir un ensemble de transitions permises, c'est-\`a-dire une relation
sym\'etrique $\sim$ sur $\cX$ (on supposera toujours que
$\sigma\not\sim\sigma$). Les plus courantes sont 
\begin{itemiz}
\item	la \defwd{dynamique de Glauber}\/, qui consiste \`a choisir
$\sigma\sim\sigma'$ si et seulement si les deux configurations $\sigma$ et
$\sigma'$ diff\`erent en exactement une composante; on parle de dynamique
de renversement de spin;

\item	la \defwd{dynamique de Kawasaki}\/, qui consiste \`a choisir
$\sigma\sim\sigma'$ si et seulement si $\sigma'$ est obtenue en
intervertissant deux composantes de $\sigma$; on parle de dynamique
d'\'echange de spin. Dans ce cas, la \chaine\ n'est pas irr\'eductible sur
$\cX$, car elle conserve le nombre total de spins $+1$ et $-1$ : elle
est en fait irr\'eductible sur chaque sous-ensemble de configurations \`a
nombre fix\'e de spins de chaque signe.
\end{itemiz}

Une fois la relation $\sim$ fix\'ee, on choisit des probabilit\'es de
transition telles que 
\begin{equation}
\label{metro6}
p_{\sigma\sigma'} = 
\begin{cases}
\myvrule{10pt}{14pt}{0pt} 
p_{\sigma'\sigma} \e^{-\beta\Delta H(\sigma,\sigma')} 
&\text{si $\sigma\sim\sigma'$\;,}\\
\displaystyle
1 - \sum_{\sigma''\sim\sigma} p_{\sigma\sigma''}
&\text{si $\sigma=\sigma'$\;,}\\
0
&\text{autrement\;.}
\end{cases}
\end{equation}
Pour satisfaire la condition lorsque $\sigma\sim\sigma'$, une possibilit\'e
est de prendre
\begin{equation}
\label{metro7}
p_{\sigma\sigma'} = 
\begin{cases}
%\myvrule{10pt}{14pt}{0pt} 
q_{\sigma\sigma'} 
&\text{si $H(\sigma')\leqs H(\sigma)$\;,}\\
q_{\sigma\sigma'} 
\e^{-\beta\Delta H(\sigma,\sigma')}
&\text{si $H(\sigma')> H(\sigma)$\;,}
\end{cases}
\end{equation}
o\`u $q_{\sigma\sigma'}=q_{\sigma'\sigma}$ et
$\sum_{\sigma'\sim\sigma}q_{\sigma\sigma'}=1$. On peut par exemple choisir
les $q_{\sigma\sigma'}$ constants, \'egaux \`a l'inverse du nombre de
transitions permises. Cela revient \`a toujours effectuer la transition si
elle d\'ecro\^\i t l'\'energie, et de ne l'effectuer qu'avec probabilit\'e
$\e^{-\beta\Delta H(\sigma,\sigma')}$ si elle fait cro\^\i tre
l'\'energie. 
Une autre possibilit\'e est de choisir 
\begin{equation}
\label{metro8}
p_{\sigma\sigma'} = \frac{q_{\sigma\sigma'}}{1+\e^{\beta\Delta
H(\sigma,\sigma')}}\;.
\end{equation}
Remarquons que le calcul de la diff\'erence d'\'energie $\Delta H$ est
particuli\`erement simple dans le cas de la dynamique de Glauber, car
seuls le spin que l'on renverse et ses voisins entrent en compte. Ainsi,
si $\sigma^{(k)}$ d\'enote la configuration obtenue en renversant le spin
num\'ero $k$ de $\sigma$, on aura 
\begin{equation}
\label{metro9}
\Delta H(\sigma,\sigma^{(k)}) = 2\sigma_k 
\Biggbrak{\sum_{\;j\colon\norm{j-k}=1}\sigma_j + h}\;,
\end{equation}
qui est une somme de $2d+1$ termes pour un r\'eseau $\Lambda\subset\Z^d$.

Concr\`etement, l'algorithme de Metropolis avec dynamique de Glauber
s'impl\'emente comme suit (avec $N=\abs{\Lambda}$)~: 
\begin{enum}
\item	{\bf Etape d'initialisation~:}
\begin{itemiz}
\item	choisir une configuration initiale $\sigma(0)$ (si possible 
telle que $\delta_{\sigma(0)}$ soit proche de $\mu$);
\item	calculer $m_0=m(\sigma(0))$ (n\'ecessite $N$ calculs);
\item	calculer $H(\sigma(0))$ (n\'ecessite de l'ordre de $dN$ calculs);
\item	poser $S = m_0$.
\end{itemiz}

\item	{\bf Etape d'it\'eration~:} Pour $n=1, 2, \dots, n_{\max}$, 
\begin{itemiz}
\item	choisir un spin $k$ au hasard uniform\'ement dans $\Lambda$;
\item	calculer $\Delta H(\sigma(n-1),\sigma')$, o\`u
$\sigma'=\sigma(n-1)^{(k)}$; 
%est obtenu en renversant le spin choisi;
\item	si $\Delta H(\sigma(n-1),\sigma') \leqs 0$, poser $\sigma(n) =
\sigma'$; 
\item	si $\Delta H(\sigma(n-1),\sigma') > 0$, poser
$\sigma(n)=\sigma'$ avec probabilit\'e
$q_{\sigma\sigma'}\e^{-\beta\Delta H(\sigma(n-1),\sigma')}$,
sinon prendre $\sigma(n+1)=\sigma(n)$;
\item	si on a renvers\'e le spin $k$,
$m_n=m_{n-1}+2\sigma_k(n-1)/N$, sinon $m_n=m_{n-1}$;
\item	ajouter $m_n$ \`a $S$.
\end{itemiz}
\end{enum}

Le quotient $S/(n+1)$ converge alors vers $\expec{m}$, avec une vitesse
d\'etermin\'ee par~\eqref{mcmc21}. La seule quantit\'e difficile \`a
estimer est le trou spectral $1-\abs{\lambda_0}$. Il s'av\`ere que ce trou
est grand sauf pour les temp\'eratures tr\`es basses ou proches de la
temp\'erature critique de la transition de phase. Dans ce dernier cas, il
existe des algorithmes alternatifs, tels que l'algorithme dit de
Swendsen--Wang, qui convergent beaucoup mieux. 

\begin{figure}
\centerline{
\includegraphics*[clip=true,width=70mm]{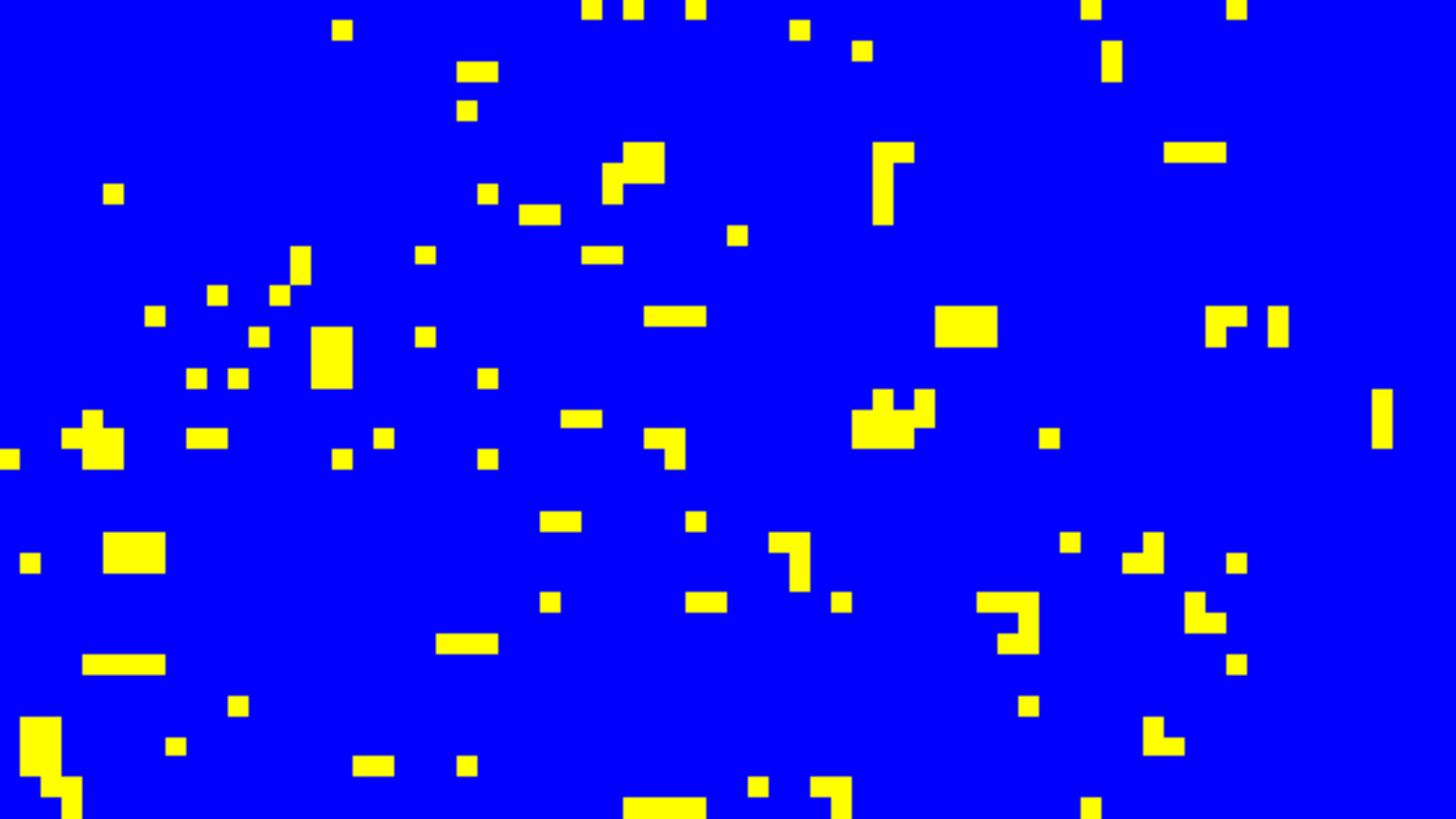}
\hspace{0.1mm}
\includegraphics*[clip=true,width=70mm]{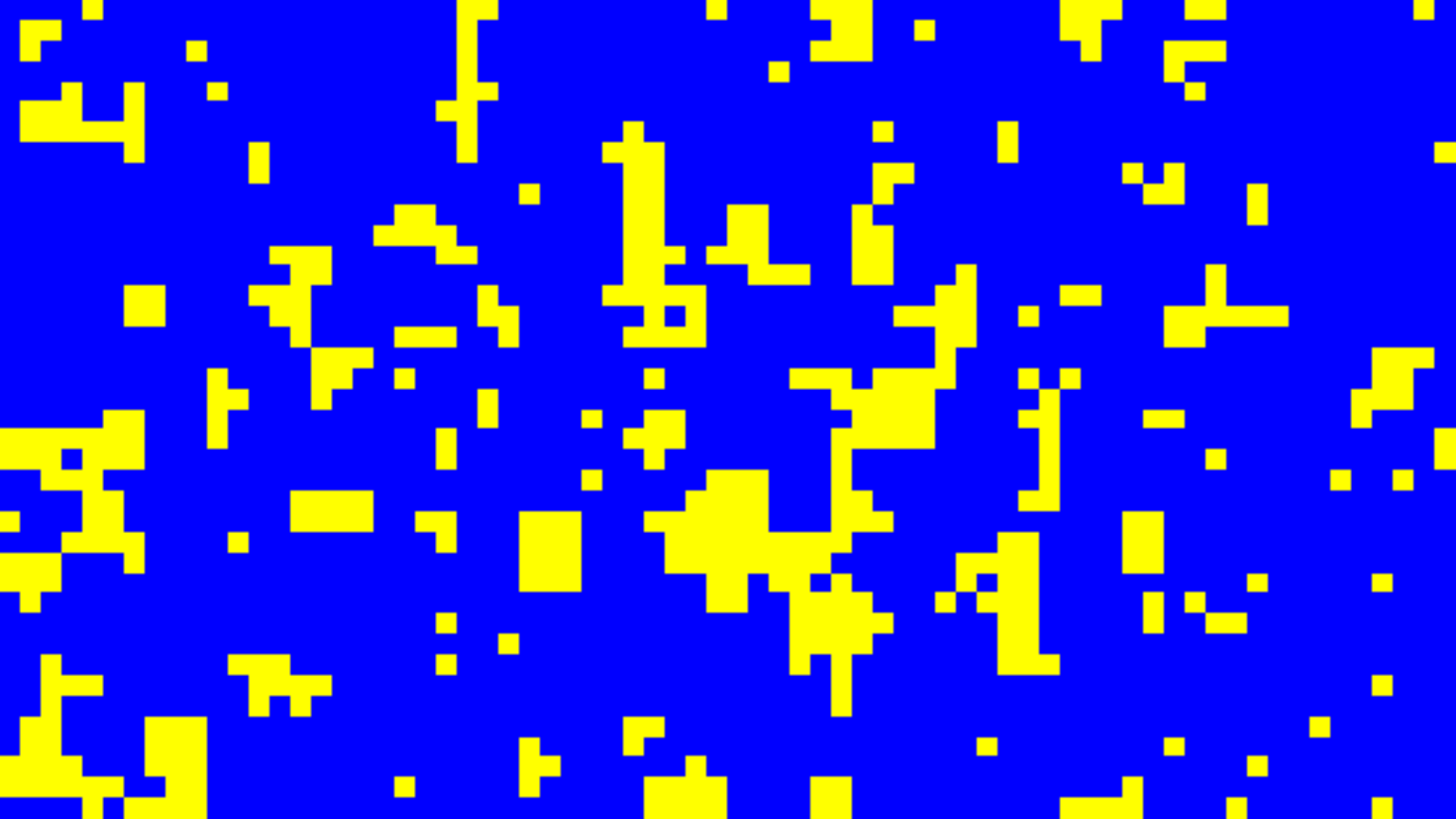}
}
\vspace{2mm}
\centerline{
\includegraphics*[clip=true,width=70mm]{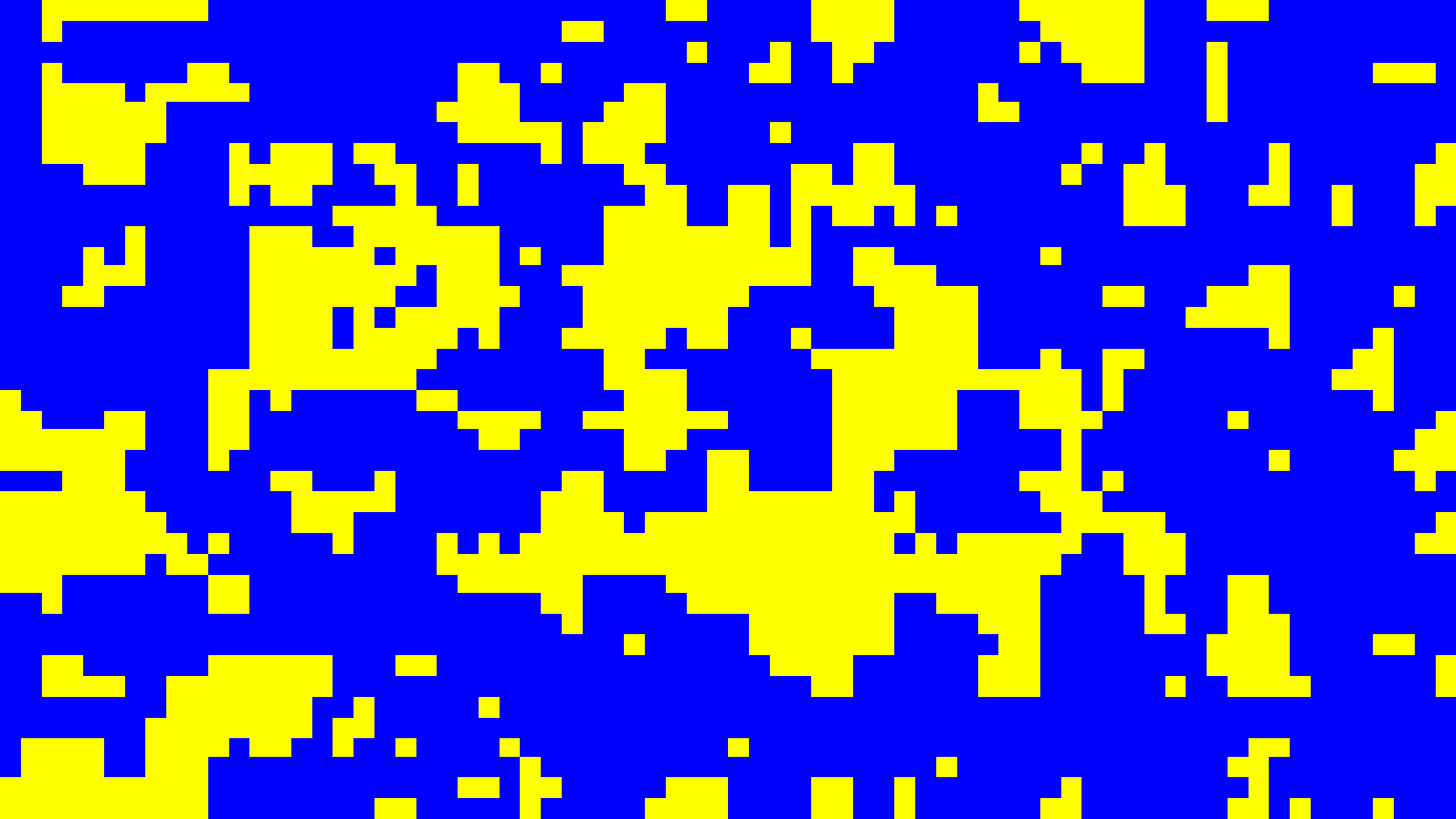}
\hspace{0.1mm}
\includegraphics*[clip=true,width=70mm]{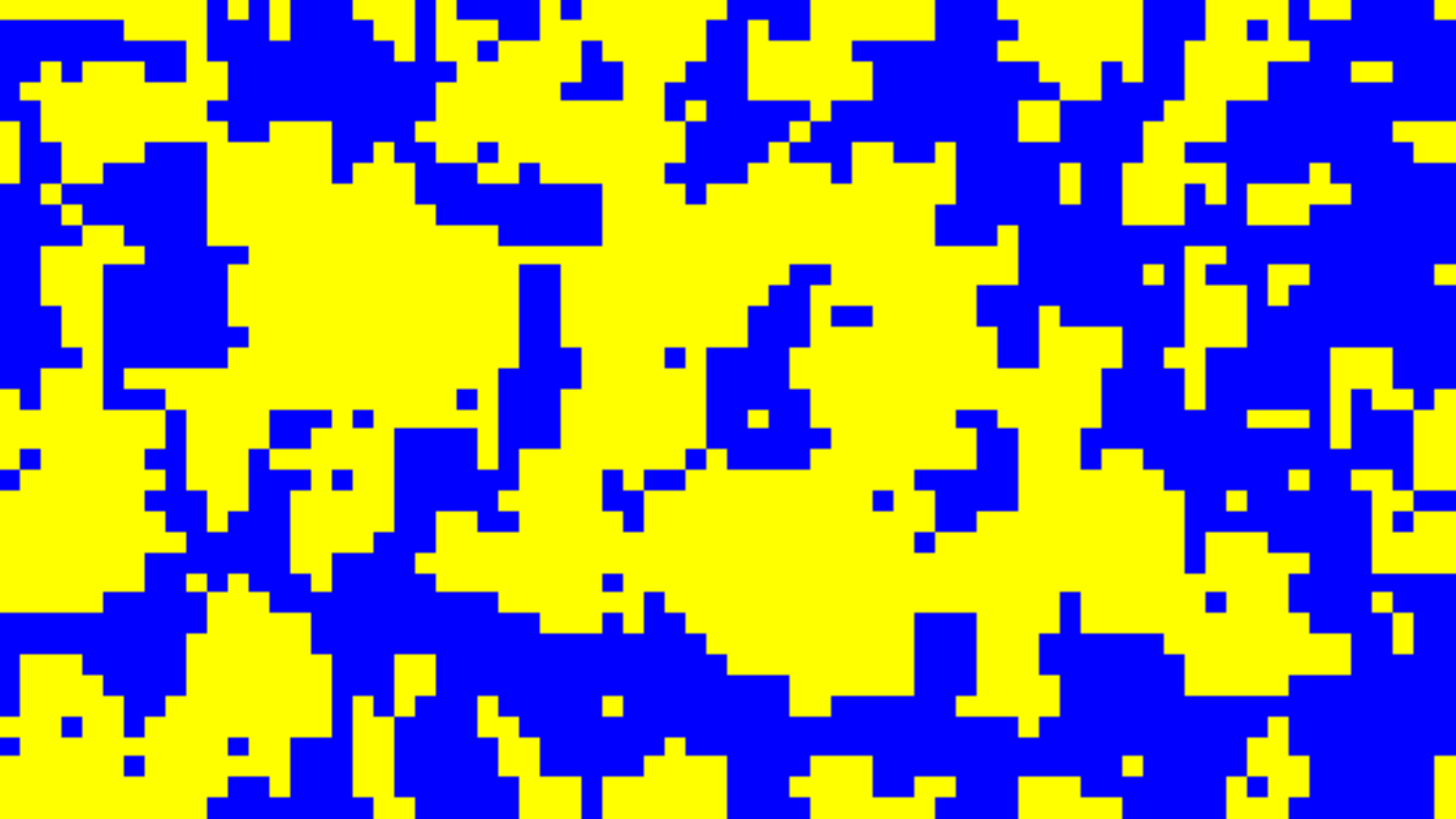}
}
\caption[]{Exemple de simulation d'une dynamique de Glauber. Evolution 
au cours du temps pour $h=1$ et $\beta=0.6$, avec tous les spins initialement
\'egaux \`a $-1$ (bleu). Le champ $h$ positif favorise les spins \'egaux \`a
$+1$ (jaunes).}
 \label{fig_glauber}
\end{figure}

%%%%%%%%%%%%%%%%%%%%%%%%%%%%%%%%%%%%%%%%%%%%%%%%%%%%%%%%%%%%%%%%%%%%%%%%%%%

\section{Le recuit simul\'e}
\label{sec_recuit}

Pour terminer, nous allons illustrer l'id\'ee de l'algorithme du recuit
simul\'e dans le cas du probl\`eme du voyageur de commerce
(Exemple~\ref{ex_voyageur}). Rappelons que pour $N$ villes, on se donne
une fonction $(i,j)\mapsto d(i,j)$ donnant la distance (ou le temps de
voyage) entre les villes $i$ et $j$. Le but est de trouver une permutation
$\sigma$ de $\set{1,2,\dots,N}$ qui minimise la longueur totale d'un
circuit ferm\'e, 
\begin{equation}
\label{sima1}
H(\sigma) = \sum_{i=1}^{N-1} d(\sigma(i),\sigma(i+1)) +
d(\sigma(N),\sigma(1))\;.
\end{equation}
La difficult\'e est que l'ensemble $\cS_N$ des permutations possibles est
de cardinal $N!$ (en fait $(N-1)!/2$ si l'on fixe la ville de d\'epart et
le sens de parcours), et cro\^\i t donc encore plus vite
qu'exponentiellement en $N$. Le tableau suivant donne quelques valeurs de
temps de calcul, en supposant que la longueur d'un chemin se calcule en
une microseconde.

\begin{center}
\begin{tabular}{|r|c|c|}
\hline
$N$ & Nombre de chemins & Temps de calcul \\
\hline
5 & 12 & 12 $\mu$s \\
10 & 181440 & $0.18$ s \\
15 & $43.6 \cdot 10^9$ & 12 heures \\
20 & $60 \cdot 10^{15}$ & 1928 ans \\
25 & $310 \cdot 10^{21}$ & 9.8 milliards d'ann\'ees \\
\hline
\end{tabular}
\end{center}

D'o\`u l'id\'ee d'explorer $\cS_N$ \`a l'aide d'une \chaine\ de Markov. 
On commence par se donner \`a nouveau une relation sym\'etrique $\sim$ sur
$\cS_N$, par exemple $\sigma\sim\sigma'$ si les deux permutations
diff\`erent d'une transposition. Un simple algorithme de descente du
gradient, consistant \`a choisir \`a chaque \'etape l'une des transitions
permises au hasard, et \`a les accepter si elles d\'ecroissent $H$, ne
donne pas de bons r\'esultats, car la \chaine\ peut rester bloqu\'ee dans
des minima locaux. 

Si l'on choisit les probabilit\'es de transition comme
dans~\eqref{metro7}, on \'evite cet inconv\'enient. Toutefois, la \chaine\
ne converge alors pas vers le minimum de $H$, mais vers la mesure de Gibbs
$\e^{-\beta H(\sigma)}/Z_\beta$. Lorsque $\beta$ tend vers $+\infty$,
cette mesure converge effectivement vers une mesure concentr\'ee sur le ou
les minima globaux de $H$. Mais si l'on choisit d'entr\'ee un $\beta$
trop \'elev\'e (c'est-\`a-dire une temp\'erature trop faible), la \chaine\
de Markov risque d'\^etre pi\'eg\'ee tr\`es longtemps au voisinage de
minima locaux. 

L'algorithme du recuit simul\'e consiste \`a choisir une \chaine\ de
Markov inhomog\`ene dans le temps. Cela veut dire qu'au lieu de choisir un
$\beta$ fixe dans~\eqref{metro7}, on prend un $\beta_n$ d\'ependant du
temps, typiquement $\beta_n=\beta_0 K^n$ avec $K>1$. Si $K$ est
suffisamment proche de $1$, c'est-\`a-dire si l'on refroidit le syst\`eme
suffisamment lentement, la \chaine\ converge vers une bonne approximation
du minimum global de $H$ --- c'est du moins ce que l'on observe
num\'eriquement dans beaucoup de cas. 

%%%%%%%%%%%%%%%%%%%%%%%%%%%%%%%%%%%%%%%%%%%%%%%%%%%%%%%%%%%%%%%%%%%%%%%%%%%

% \part{Processus de Poisson}
\part{Processus de sauts et files d'attente}
\label{part_poisson}

%%%%%%%%%%%%%%%%%%%%%%%%%%%%%%%%%%%%%%%%%%%%%%%%%%%%%%%%%%%%%%%%%%%%%%%%%%%

\chapter{Rappels de probabilit\'es}
\label{chap_rappel}

%%%%%%%%%%%%%%%%%%%%%%%%%%%%%%%%%%%%%%%%%%%%%%%%%%%%%%%%%%%%%%%%%%%%%%%%%%%

\section{Loi binomiale et loi de Poisson}
\label{sec_disc}

Nous commen\c cons par rappeler quelques lois de probabilit\'e usuelles qui
joueront un r\^ole important dans la suite. Une \defwd{exp\'erience de
Bernoulli}\/ de longueur $n$ et probabilit\'e de succ\`es $q\in[0,1]$ consiste
\`a r\'ep\'eter $n$ fois, de mani\`ere ind\'ependante, une exp\'erience
\'el\'ementaire qui n'admet que deux issues possibles~: Le succ\`es, qui se
produit avec probabilit\'e $q$, et l'\'echec, qui se produit avec probabilit\'e
$1-q$. Si par exemple l'exp\'erience consiste \`a jeter un d\'e \'equilibr\'e,
et que l'on consid\`ere comme succ\`es uniquement l'obtention de $6$ points, on
aura $q=\frac16$. 

Soit $X$ la variable al\'eatoire donnant le nombre de succ\`es de
l'exp\'erience de longueur $n$. Elle pourra prendre les valeurs $0, 1, 2, \dots,
n$, la valeur $k$ \'etant obtenue avec probabilit\'e 
\begin{equation}
 \label{disc1}
\prob{X=k} = \binom nk q^k(1-q)^{n-k} \bydef b_{n,q}(k)\;,
\end{equation} 
o\`u nous avons not\'e les coefficients binomiaux 
\begin{equation}
 \label{disc2}
\binom nk \equiv C^k_n = \frac{n!}{k!(n-k)!}\;.
\end{equation} 
En effet, il y a $\binom nk$ mani\`eres d'arranger les $k$ succ\`es et $n-k$
\'echecs parmi les $n$ exp\'eriences, et chacun de ces arrangements se produit
avec probabilit\'e $q^k(1-q)^{n-k}$. 

\begin{definition}[Loi binomiale]
\label{def_binom}
Soit $X$ une variable al\'eatoire prenant ses valeurs dans
$\set{0,1,\dots,n}$ satisfaisant~\eqref{disc1}. Nous dirons que $X$ suit une
\defwd{loi binomiale de param\`etres $(n,q)$} et noterons 
$X \sim b_{n,q}$. 
\end{definition}

On peut repr\'esenter $X$ comme somme de $n$ variables al\'eatoires $Y_i$ 
ind\'ependantes et identiquement distribu\'ees (i.i.d.), de loi de Bernoulli de
param\`etre $q$, c'est-\`a-dire telles que $\prob{Y_i=1}=q=1-\prob{Y_i=0}$.
Alors comme l'esp\'erance 
de chaque $Y_i$ vaut $\expec{Y_i}=0\cdot\prob{Y_i=0}+1\cdot\prob{Y_i=1}=q$, on
obtient pour l'esp\'erance de $X$
\begin{equation}
 \label{disc4}
\expec{X} \defby \sum_{k=0}^n k \, \prob{X=k} 
= \sum_{i=1}^n \expec{Y_i}
= nq\;.
\end{equation} 
De plus, comme la variance de chaque $Y_i$ vaut $\variance(Y_i) \defby
\expec{Y_i^2} - \expec{Y_i}^2 = q(1-q)$, on voit que la variance de $X$ est
donn\'ee par  
\begin{equation}
 \label{disc5}
\variance(X) \defby \expec{X^2}-\expec{X}^2 = nq(1-q)\;. 
\end{equation}

\goodbreak
Une deuxi\`eme loi importante dans ce cours est la loi de Poisson. 

\begin{definition}[Loi de Poisson]
\label{def_poisson}
Nous dirons que la variable al\'eatoire $X$ suit une \defwd{loi de Poisson de
param\`etre $\lambda>0$}, et noterons $X\sim \pi_\lambda$, si elle prend des
valeurs enti\`eres non-n\'egatives, avec probabilit\'e 
\begin{equation}
 \label{disc6}
\prob{X=k} = \e^{-\lambda} \frac{\lambda^k}{k!} \bydef \pi_\lambda(k)\;. 
\end{equation} 
\end{definition}

Avant de discuter sa signification, mentionnons quelques propri\'et\'es de base
de cette loi.

\begin{prop} \hfill
\label{prop_disc1}
\begin{enum}
\item	Si $X$ suit une loi de Poisson de param\`etre $\lambda$, alors 
\begin{equation}
 \label{disc7}
\expec{X} = \variance(X) = \lambda\;. 
\end{equation} 
\item	Si $X$ et $Y$ sont ind\'ependantes, et suivent des lois de Poisson de
param\`etres $\lambda$ et $\mu$ respectivement, alors $X+Y$ suit une loi de
Poisson de param\`etre $\lambda+\mu$. 
\end{enum}
\end{prop}
\begin{proof}
Voir exercices~\ref{exo_generatrice2} et~\ref{exo_generatrice3}.
\end{proof}

L'importance de la loi de Poisson $\pi_\lambda$ vient du fait qu'elle donne une
bonne approximation de la loi binomiale $b_{n,q}$ lorsque la longueur $n$ de
l'exp\'erience est grande et que la probabilit\'e de succ\`es $q$ est faible,
avec $nq=\lambda$. En effet nous avons le r\'esultat de convergence suivant~: 

\begin{prop}
\label{prop_disc2}
Soit $\set{q_n}_{n\geqs0}$ une suite telle que
$\lim_{n\to\infty}nq_n=\lambda>0$. Alors, pour tout $k\in\N$, 
\begin{equation}
\label{disc8}
\lim_{n\to\infty} b_{n,q_n}(k) = \pi_\lambda(k)\;.
\end{equation}
\end{prop}
\begin{proof}
Soit $\lambda_n=nq_n$. Alors 
\begin{align}
\nonumber
b_{n,q_n}(k) 
&= b_{n,\lambda_n/n}(k) \\
\nonumber
&= \frac{n(n-1)\dots(n-k+1)}{k!} \frac{\lambda_n^k}{n^k}
\biggpar{1-\frac{\lambda_n}n}^{n-k} \\
&= \frac{\lambda_n^k}{k!} \frac1{(1-\lambda_n/n)^k}
\biggpar{1-\frac1n}\dots\biggpar{1-\frac{k-1}n}\biggpar{1-\frac{\lambda_n}n}
^n\;.
\label{disc8:1}
\end{align}
Lorsque $n\to\infty$, on a $\lambda_n\to\lambda$, $\lambda_n/n\to0$ et
$j/n\to0$ pour $j=0,\dots,k-1$, donc  
\begin{equation}
\label{disc8:2}
\lim_{n\to\infty} b(k; n, q_n) = \frac{\lambda^k}{k!} 
\lim_{n\to\infty} \biggpar{1-\frac{\lambda_n}n}^n 
= \frac{\lambda^k}{k!}\e^{-\lambda}\;,
\end{equation}
la derni\`ere \'egalit\'e pouvant se montrer par un d\'eveloppement limit\'e de
$\log\brak{(1-\lambda_n/n)^n}$. 
\end{proof}

\begin{example}
\label{ex_lpn}
La probabilit\'e de gagner au tierc\'e est de $1/1000$. Quelle est la
probabilit\'e de gagner $k$ fois en jouant $2000$ fois? 

Nous mod\'elisons la situation par une exp\'erience de Bernoulli de longueur
$n=2000$ et param\`etre $q=0.001$. La probabilit\'e de gagner $k$ fois sera
donn\'ee par 
\begin{equation}
\label{disc8:3}
b_{2000,0.001}(k) \simeq \pi_2(k) = \e^{-2}\frac{2^k}{k!}\;.
\end{equation}
Le tableau suivant compare quelques valeurs des deux lois. 

\begin{center}
\begin{tabular}{|l||l|l|l|l|l|l|}
\hline 
\myvrule{12pt}{6pt}{0pt}
$k$ & $0$ & $1$ & $2$ & $3$ & $4$ & $5$ \\
\hline 
\myvrule{12pt}{7pt}{0pt}
$b_{2000,0.001}(k)$ 
& $0.13520$ & $0.27067$ & $0.27081$ & $0.18053$ & $0.09022$ & $0.03605$
\\
\myvrule{7pt}{7pt}{0pt} 
$\pi_2(k)$ 
& $0.13534$ & $0.27067$ & $0.27067$ & $0.18045$ & $0.09022$ & $0.03609$
\\
\hline
\end{tabular}
\end{center}

L'avantage de la loi de Poisson est qu'elle est bien plus rapide \`a
calculer, puisqu'elle ne fait pas intervenir de coefficients binomiaux. 
\end{example}

Nous donnons maintenant un r\'esultat bien plus fort, \`a savoir la convergence
dans $L^1$ de la loi de Bernoulli vers la loi de Poisson. Ce r\'esultat est
entre autres int\'eressant pour sa preuve, qui peut \^etre faite exclusivement
avec des arguments de probabilit\'es, et un minimum d'analyse. 

\begin{theorem}
\label{thm_lpn}
On a 
\begin{equation}
\label{disc9}
\sum_{k=0}^\infty \bigabs{b_{n,q}(k) - \pi_{nq}(k)} \leqs 2nq^2\;.
\end{equation}
\end{theorem}
\begin{proof}
Nous commen\c cons par introduire des espaces probabilis\'es
$(\Omega_i,p_i)$, pour $i=1,\dots,n$, donn\'es par 
$\Omega_i = \set{-1,0,1,2,\dots}$ et 
\begin{equation}
\label{disc9:1}
p_i(k) = 
\begin{cases}
\myvrule{10pt}{12pt}{0pt}
\e^{-q} - (1-q) 
& \text{si $k=-1$\;,} \\
\myvrule{14pt}{10pt}{0pt} 
1-q 
& \text{si $k=0$\;,} \\
\myvrule{10pt}{10pt}{0pt} 
\e^{-q} \dfrac{q^k}{k!}
& \text{si $k\geqs 1$\;.}
\end{cases}
\end{equation}
On v\'erifiera que les $p_i$ d\'efinissent bien une distribution de
probabilit\'e. Sur chaque $\Omega_i$, nous introduisons les deux variables
al\'eatoires
\begin{equation}
\label{disc9:2}
X_i(\omega_i) = 
\begin{cases}
0 & \text{si $\omega_i=0$\;,} \\
1 & \text{sinon\;,}
\end{cases}
\qquad\qquad
Y_i(\omega_i) = 
\begin{cases}
\omega_i & \text{si $\omega_i\geqs 1$\;,} \\
0 & \text{sinon\;.}
\end{cases}
\end{equation}
De cette mani\`ere, on a $\prob{X_i=0}=1-q$, $\prob{X_i=1}=q$, 
et $\prob{Y_i=k}= \pi_q(k)$ pour tout $k\geqs0$. De plus, 
\begin{align}
\nonumber
\prob{X_i=Y_i}
&= \prob{X_i=0, Y_i=0} + \prob{X_i=1, Y_i=1} \\
&= p_i(0) + p_i(1) = 1 - q + q \e^{-q}\;,
\label{disc9:3}
\end{align}
donc 
\begin{equation}
\label{disc9:4}
\prob{X_i \neq Y_i} = q (1-\e^{-q}) \leqs q^2\;. 
\end{equation}
Soit $(\Omega,p)$ l'espace produit des $(\Omega_i,p_i)$. Alors 
\begin{itemiz}
\item	$X=X_1+\dots+X_n$ suit la loi binomiale $\prob{X=k}=b_{n,q}(k)$; 
\item	$Y=Y_1+\dots+Y_n$ suit la loi de Poisson $\prob{Y=k}=\pi_{nq}(k)$,
en vertu de la Proposition~\ref{prop_disc1}. 
\end{itemiz}
Comme $X\neq Y$ implique que $X_i\neq Y_i$ pour un $i$ au moins, il suit
de~\eqref{disc9:4} que 
\begin{equation}
\label{disc9:5}
\prob{X\neq Y} \leqs \sum_{i=1}^n \prob{X_i\neq Y_i} \leqs nq^2\;.
\end{equation} 
Il reste donc \`a montrer que le membre de gauche de~\eqref{disc9} est
major\'e par $2\prob{X\neq Y}$. Un tel proc\'ed\'e s'appelle un argument de
couplage. Nous posons, pour abr\'eger l'\'ecriture, $f(k)=\prob{X=k}$,
$g(k)=\prob{Y=k}$ et  $A = \setsuch{k}{f(k)>g(k)}$. Alors 
\begin{align}
\nonumber
\sum_{k=0}^\infty \bigabs{b_{n,q}(k) - \pi_{nq}(k)} 
&= \sum_{k=0}^\infty \bigabs{f(k)-g(k)} \\
\nonumber
&= \sum_{k\in A} \bigpar{f(k)-g(k)} - \sum_{k\notin A} (f(k)-g(k)) \\
\label{disc9:6}
&= 2 \sum_{k\in A} \bigpar{f(k)-g(k)} - 
\underbrace{\sum_{k\in\N} \bigpar{f(k)-g(k)}}_{=1-1=0}\;.
\end{align}
Or nous pouvons \'ecrire 
\begin{align}
\nonumber
\sum_{k\in A} \bigpar{f(k)-g(k)}
&= \prob{X\in A} - \prob{Y\in A} \\
\nonumber
&\leqs \prob{X\in A, Y\in A} + \prob{X\in A, Y\neq X} - \prob{Y\in A} \\
\nonumber
&\leqs \prob{X\in A, Y\in A} + \prob{X\neq Y} - \prob{Y\in A} \\
&\leqs \prob{X \neq Y}\;,
\label{disc9:7}
\end{align}
ce qui conclut la d\'emonstration.
\end{proof}

Si nous prenons par exemple $q=\lambda/n$, la borne~\eqref{disc9} nous fournit 
\begin{equation}
\label{disc10}
\sum_{k=0}^\infty \bigabs{b_{n,\lambda/n}(k) - \pi_{\lambda}(k)} \leqs
2\frac{\lambda^2}n\;.
\end{equation}
Dans l'exemple~\ref{ex_lpn}, $\lambda^2/n$ vaut $4/2000=0.002$. Ainsi la somme
de toutes les valeurs absolues de diff\'erences $b_{2000,0.001}(k) -
\pi_{2}(k)$ est born\'ee par $0.004$, et comme tous ces termes sont positifs, la
plupart seront bien plus petits encore. 

%%%%%%%%%%%%%%%%%%%%%%%%%%%%%%%%%%%%%%%%%%%%%%%%%%%%%%%%%%%%%%%%%%%%%%%%%%%

\section{Loi normale et loi exponentielle}
\label{sec_cont}

Nous aurons \'egalement affaire \`a des variables al\'eatoires r\'eelles
continues. Pour les d\'efinir, le plus simple est de passer par la notion de
\defwd{fonction de r\'epartition}\/.

\begin{definition}[Fonction de r\'epartition]
\label{def_vard2}
Une fonction $F: \R\to[0,1]$ est une\/ \defwd{fonction de r\'epartition} si 
\begin{itemiz}
\item	$F$ est croissante: $x\leqs y \Rightarrow F(x)\leqs F(y)$. 
\item	$F$ est continue \`a droite: $\lim_{y\to x+} F(y)=F(x)$ $\forall x$. 
\item	$\lim_{x\to-\infty}F(x)=0$ et $\lim_{x\to+\infty}F(x)=1$. 
\end{itemiz}
Une fonction de r\'epartition 
$F$ est dite\/ \defwd{absolument continue de densit\'e $f$}\/ si 
\begin{equation}
\label{vard9}
F(x) = \int_{-\infty}^x f(y)\,\6y\;.
\end{equation}
\end{definition}

Le lien entre la notion de fonction de r\'epartition et les variables
al\'eatoires vient du fait que pour toute variable al\'eatoire r\'eelle,
$\prob{X\leqs t}$ est une fonction de r\'epartition. 

\goodbreak
En effet, 
\begin{itemiz}
\item	si $s\leqs t$, alors $\set{X\leqs s} \subset \set{X\leqs t}$, et
donc $\prob{X\leqs s} \leqs \prob{X\leqs t}$;
\item	$\lim_{s\to t+}\prob{X\leqs s} - \prob{X\leqs t}
=\lim_{s\to t+}\prob{t<X\leqs s}=0$;
\item	$\lim_{t\to-\infty}\prob{X\leqs t}=0$ et
$\lim_{t\to+\infty}\prob{X\leqs t}=1$. 
\end{itemiz}
Ceci motive la d\'efinition suivante. 

\begin{definition}[Variable al\'eatoire \`a densit\'e]
\label{def_vard3}
Si $X$ est une variable al\'eatoire, 
\begin{equation}
\label{vard10}
F_X(t) = \prob{X\leqs t}
\end{equation}
est appel\'ee\/ \defwd{fonction de r\'epartition de $X$}\/. Si $F_X$ est
absolument continue de densit\'e $f$, on dit que\/ \defwd{$X$ admet la
densit\'e $f$} et on a les relations
\begin{align}
\label{vard11}
\prob{X\leqs t} &= \int_{-\infty}^t f(s)\,\6s\;, \\
\prob{a < X \leqs b} &= \prob{X\leqs b} - \prob{X\leqs a} 
= \int_a^b f(s)\,\6s\;.
\label{vard12}
\end{align}
Dans ce cas, on peut remplacer $<$ par $\leqs$ et inversement. 
\end{definition}

L'esp\'erance d'une variable al\'eatoire r\'eelle $X$ de densit\'e $f$ est
d\'efinie par l'int\'egrale 
\begin{equation}
 \label{vard14}
\expec{X} = \int_{-\infty}^\infty x f(x) \,\6x\;,
\end{equation}
(pourvu que la fonction $x f(x)$ soit absolument int\'egrable, sinon on dit
que $X$ n'admet pas d'esp\'erance).

Toutes les variables al\'eatoires r\'eelles ne sont par des variables \`a
densit\'e : Par exemple, les fonctions de r\'epartition de variables \`a
valeurs discr\`etes, comme celles vues dans la section pr\'ec\'edente, sont
constantes par morceaux et admettent des discontinuit\'es, et ne peuvent donc
par \^etre \'ecrites comme l'int\'egrale d'une fonction continue.

Un premier exemple important de variable al\'eatoire r\'eelle \`a densit\'e est
celui des variables gaussiennes ou normales. 

\begin{definition}[Loi normale]
\label{def_normal} 
On dit que la variable al\'eatoire $X$ suit une \defwd{loi normale de moyenne
$\mu$ et \'ecart-type $\sigma$}\/, et on notera $X\sim\cN(\mu,\sigma^2)$ si elle
admet la densit\'e 
\begin{equation}
 \label{vard21}
f(x) = \frac{1}{\sqrt{2\pi\sigma^2}} \e^{-(x-\mu)^2/2\sigma^2}\;. 
\end{equation}  
Si $X\sim\cN(0,1)$, on dit qu'elle suit une loi normale centr\'ee r\'eduite, ou
standard.
\end{definition}

On v\'erifie que si $X\sim\cN(\mu,\sigma^2)$, alors son esp\'erance vaut
$\expec{X}=\mu$ et sa variance vaut $\variance(X)=\sigma^2$. 
L'importance de la loi normale vient avant tout du th\'eor\`eme de la limite
centrale :

\begin{theorem}[Th\'eor\`eme central limite]
\label{thm_tlc}
Soit $X_1, X_2, \dots$ une suite de variables al\'ea\-toires ind\'ependantes,
identiquement distribu\'ees (abr\'eg\'e\/ \emph{i.i.d.}\/), d'esp\'erance finie
$\mu$ et de variance finie $\sigma^2$. Alors la variable al\'eatoire $S_n =
\sum_{i=1}^n X_i$ satisfait
\begin{equation}
\label{tlc11}
\lim_{n\to\infty} \biggprob{a\leqs \frac{S_n-n\mu}{\sqrt{n\sigma^2}}\leqs b} 
= \int_a^b \frac{\e^{-x^2/2}}{\sqrt{2\pi}}\,\6x\;,
\end{equation} 
c'est-\`a-dire que $(S_n-n\mu)/\sqrt{n\sigma^2}$ converge en loi vers une
variable normale standard.
\end{theorem}

Un second exemple de loi \`a densit\'e, particuli\`erement important dans ce
cours, est celui des variables exponentielles. 

\begin{definition}[Variable exponentielle]
\label{def_exp}
On dit que la variable al\'eatoire $X$ suit une \defwd{loi exponentielle de
para\-m\`etre $\lambda>0$}, et on note $X\sim\cE\!xp(\lambda)$, si elle
satisfait 
\begin{equation}
 \label{vard22}
\prob{X>t} = \e^{-\lambda t} 
\end{equation}  
pour tout $t\geqs0$. Sa fonction de r\'epartition est donc
$F_X(t)=1-\e^{-\lambda t}$ pour $t>0$, et sa densit\'e est $\lambda\e^{-\lambda
t}$, toujours pour $t>0$. 
\end{definition}

On v\'erifie qu'une variable de loi exponentielle a esp\'erance $1/\lambda$ et
variance $1/\lambda^2$. Une propri\'et\'e remarquable de la loi exponentielle
est la \defwd{propri\'et\'e de Markov}\/ : Pour $t>s\geqs0$, 
\begin{equation}
 \label{vard23}
\pcond{X>t}{X>s} = \e^{-\lambda(t-s)} = \prob{X>t-s}\;. 
\end{equation} 

Nous aurons parfois affaire \`a des couples, ou des $n$-uplets de variables
al\'eatoires \`a densit\'e, aussi appel\'es \defwd{vecteurs al\'eatoires}\/.
Leur \defwd{densit\'e conjointe}\/ est d\'efinie comme la fonction de $n$
variables $f$ telle que 
\begin{equation}
 \label{vard24}
\prob{X_1\leqs t_1, X_2\leqs t_2, \dots, X_n\leqs t_n} 
= \int_{-\infty}^{t_1}\int_{-\infty}^{t_2}\dots\int_{-\infty}^{t_n}
f(x_1,x_2,\dots,x_n) \,\6x_n\dots\6x_2\6x_1 
\end{equation}  
pour tout choix de $(t_1,t_2,\dots,t_n)$. Autrement dit, on a 
\begin{equation}
 \label{vard25}
f(t_1,t_2,\dots,t_n) = 
\dpar{^n}{t_1\partial t_2\dots\partial t_n} 
\prob{X_1\leqs t_1, X_2\leqs t_2, \dots, X_n\leqs t_n}\;. 
\end{equation} 
Les variables al\'eatoires $X_1, X_2, \dots, X_n$ sont dites
\defwd{ind\'ependantes} si on a 
\begin{equation}
 \label{vard26} 
\prob{X_1\leqs t_1, X_2\leqs t_2, \dots, X_n\leqs t_n} 
= \prob{X_1\leqs t_1} \prob{X_2\leqs t_2} \dots \prob{X_n\leqs t_n}
\end{equation} 
pour tout choix de $t_1, t_2, \dots, t_n$. On montre que c'est \'equivalent \`a
ce que la densit\'e conjointe s'\'ecrive sous la forme 
\begin{equation}
 \label{vard27}
f(x_1,x_2,\dots,x_n) = f_1(x_1) f_2(x_2) \dots f_n(x_n) 
\end{equation} 
pour des densit\'es $f_1, f_2, \dots, f_n$ (appel\'ees \defwd{densit\'es
marginales}). 

Un autre r\'esultat important est le suivant~:

\begin{prop}[Convolution]
\label{dor_verd1}
Si $X_1$ et $X_2$ sont des variables al\'eatoires ind\'epen\-dan\-tes, de
densit\'es respectives $f_1$ et $f_2$, alors $X_1+X_2$ admet une densit\'e 
donn\'ee par la convolution 
\begin{equation}
\label{verd16}
(f_1*f_2)(z) = \int_{-\infty}^\infty f_1(z-x_2)f_2(x_2)\,\6x_2\;.
\end{equation}
\end{prop}

\begin{example}[Loi Gamma]
Soient $X_1, \dots, X_n$ des variables i.i.d.\ de loi $\cE\!xp(\lambda)$. 
En calculant la convolution de leurs densit\'es, on montre par r\'ecurrence
sur $n$ que la somme $S_n=X_1+\dots+X_n$ admet la densit\'e 
\begin{equation}
 \label{gamma}
\gamma_{\lambda,n}(x) = 
\frac{\lambda^n}{(n-1)!} x^{n-1}\e^{-\lambda x}
\end{equation} 
pour $x\geqs0$. On dit que $S_n$ suit une \defwd{loi Gamma}\/ de param\`etres
$(\lambda,n)$. 
\end{example}

%%%%%%%%%%%%%%%%%%%%%%%%%%%%%%%%%%%%%%%%%%%%%%%%%%%%%%%%%%%%%%%%%%%%%%%%%%%

\section{Exercices}
\label{sec_exo_rappel}

Soit $X$ une variable al\'eatoire \`a valeurs dans $\N$.
On appelle \defwd{fonction g\'en\'eratrice}\/ de $X$ la fonction $G_X:\C\to\C$
d\'efinie par
\[
G_X(z) = \bigexpec{z^X} = \sum_{k=0}^\infty z^k \, \prob{X=k}\;. 
\]

\begin{exercice}
\label{exo_generatrice1} 
Calculer les fonctions g\'en\'eratrices des distributions suivantes:
\begin{enum}
\item	Loi de Bernoulli: $\prob{X=0}=1-q$, $\prob{X=1}=q$, o\`u $q\in[0,1]$.
\item	Loi binomiale: $\prob{X=k} = b_{n,q}(k) = \binom nk q^k(1-q)^{n-k}$,
pour $k=0,1,\dots,n$.
\item	Loi de Poisson: $\prob{X=k} = \pi_\lambda(k)
= \e^{-\lambda}\lambda^k/k!$, o\`u $\lambda>0$ et $k\in\N$. 
\item	Loi g\'eom\'etrique: $\prob{X=k} = q(1-q)^{k-1}$, o\`u $q\in[0,1]$ et
$k\in\N^*$.
\end{enum}
\end{exercice}

\begin{exercice}
\label{exo_generatrice2} 
On suppose que $G_X$ existe pour tous les $z$ dans un disque de rayon
strictement sup\'erieur \`a $1$. Montrer que 
\[
G_X(1)=1\;,
\qquad
G_X'(1) = \expec{X}\;,
\qquad
G_X''(1) = \expec{X^2} - \expec{X}\;,
\]
et en d\'eduire une expression de la variance de $X$ en termes de sa fonction
g\'en\'eratrice. En d\'eduire les esp\'erances et variances de variables
al\'eatoires de Bernoulli, binomiale, de Poisson et g\'eom\'etrique.
\end{exercice}

\begin{exercice}
\label{exo_generatrice3} 
Soient $X$ et $Y$ deux variables al\'eatoires ind\'ependantes \`a valeurs dans
$\N$, et de fonctions g\'en\'eratrices $G_X$ et $G_Y$ respectivement. Montrer
que $G_{X+Y}=G_X G_Y$. 

\medskip
\noindent
Application: V\'erifier les assertions suivantes.
\begin{enum}
\item	La somme de $n$ variables al\'eatoires de Bernoulli
ind\'ependantes suit une loi binomiale;
\item	La somme de deux variables al\'eatoires binomiales 
ind\'ependantes suit une loi binomiale;
\item	La somme de deux variables al\'eatoires ind\'ependantes de loi de
Poisson suit une loi de Poisson.
\end{enum}
\end{exercice}

\begin{exercice}
\label{exo_generatrice4} 

Soit $N$ une variable al\'eatoire \`a valeurs dans $\N$, et soit
$G_N(z)=\expec{z^N}$ sa fonction g\'en\'eratrice. 
Soient $X_1, X_2, \dots$ des variables al\'eatoires \`a valeurs dans $\N$,
ind\'epen\-dantes et identiquement distribu\'ees, et ind\'ependantes de $N$.
Soit $G_X(z)=\expec{z^X}$ leur fonction g\'en\'eratrice. 

\begin{enum}
\item	Soit $n\in\N^*$ et 
\[
S_n = X_1 + \dots + X_n\;.
\]
Ecrire la fonction g\'en\'eratrice $\expec{z^{S_n}}$ de $S_n$ en fonction de
$G_X(z)$. 

\item	Soit 
\[
 S_N = X_1 + \dots + X_N\;.
\]
Montrer que sa fonction g\'en\'eratrice $ G_S(z)=\expec{z^{S_N}}$ est donn\'ee
par 
\[
 G_S(z) = G_N(G_X(z))\;.
\]
{\it Indication:} Ecrire $\prob{S_N=k}$ en fonction des $\pcond{S_N=k}{N=n}$,
$n\in\N^*$.

\item	On suppose que $N$ suit une loi de Poisson de param\`etre $\lambda>0$,
et que les $X_i$ suivent des lois de Bernoulli de param\`etre $q\in[0,1]$.
D\'eterminer la loi de $S_N$. 
\end{enum}
\end{exercice}

\begin{exercice}
\label{exo_rappel05} 

Soit $U$ une variable al\'eatoire uniforme sur $[0,1]$,
c'est-\`a-dire de densit\'e donn\'ee par la fonction
indicatrice $\indicator{[0,1]}$.
\begin{enum}
\item	Donner la fonction de r\'epartition de $U$.
\item	Soit $\varphi:[0,1]\to\R$ une fonction strictement croissante
et admettant une r\'eciproque
$\varphi^{-1}$. D\'eterminer la fonction de r\'epartition de la
variable al\'eatoire $Y=\varphi(U)$. 
\item	D\'eterminer $\varphi$ de telle mani\`ere que $Y$ soit exponentielle de
param\`etre $\lambda$.
\end{enum}
\end{exercice}

%%%%%%%%%%%%%%%%%%%%%%%%%%%%%%%%%%%%%%%%%%%%%%%%%%%%%%%%%%%%%%%%%%%%%%%%%%%

\chapter{Le processus ponctuel de Poisson}
\label{chap_ppp}

Le processus ponctuel de Poisson est un processus stochastique qui associe une
distribution de probabilit\'e aux configurations de points sur $\R_+$. Ces
points mod\'elisent, par exemple, les temps de passage d'un bus \`a un arr\^et,
les instants d'arriv\'ee d'appels t\'el\'ephoniques dans une centrale, et ainsi
de suite. 

Dans certains cas, par exemple lorsque les bus suivent un horaire r\'egulier et
qu'il n'y a pas de perturbation du trafic, les instants d'arriv\'ee sont assez
r\'eguli\`erement espac\'es. Dans d'autres situations, par exemple lorsque des
travaux perturbent le trafic, les instants d'arriv\'ee deviennent beaucoup plus
irr\'eguliers, et on observe \`a la fois des longs temps d'attente et des bus
qui se suivent presque imm\'ediatement. Le processus ponctuel de Poisson
mod\'elise la situation la plus al\'eatoire possible, dans un sens qui reste \`a
d\'efinir. 

\begin{figure}[h]
%  \vspace{10mm}
%  \centerline{
%  \includegraphics*[clip=true,width=120mm]{figs/Poisson} 
%  }
%  \figtext{
%  	\writefig	1.6	0.9	{$X_0=0$}
%  	\writefig	3.1	0.9	{$X_1(\omega)$}
%  	\writefig	5.4	0.9	{$X_2(\omega)$}
%  	\writefig	6.4	0.9	{$X_3(\omega)$}
%  	\writefig	8.5	0.9	{$X_4(\omega)$}
%  	\writefig	10.1	0.9	{$X_5(\omega)$}
%  }
\begin{center}
\begin{tikzpicture}[-,scale=0.9,
auto,node distance=1.0cm, thick,main node/.style={draw,circle,fill=white,minimum
size=5pt,inner sep=0pt}]

  \path[->,>=stealth'] 
     (-1,0) edge (14,0)
  ;

%   \node at (12.0,0.5) {$n$};
%   \node at (-1.0,2.5) {$X_n$};

  \draw (0,0) node[main node] {} node[above] {$X_0=0$} 
  -- (2,0) node[main node] {} node[above] {$X_1(\omega)$} 
  -- (5,0) node[main node] {} node[above] {$X_2(\omega)$} 
  -- (6.3,0) node[main node] {} node[above] {$X_3(\omega)$} 
  -- (10,0) node[main node] {} node[above] {$X_4(\omega)$} 
  -- (12.5,0) node[main node] {} node[above] {$X_5(\omega)$} 
   ;
\end{tikzpicture}
\end{center}
\vspace{-2mm}
 \caption[]{Une r\'ealisation d'un processus de Poisson.}
 \label{fig_Poisson}
\end{figure}
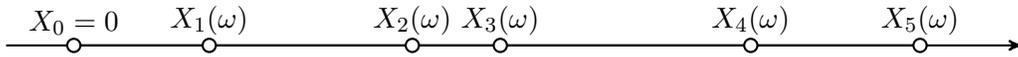

Le processus peut \^etre caract\'eris\'e de plusieurs mani\`eres
diff\'erentes. Une r\'ealisation peut \^etre sp\'ecifi\'ee par une suite
croissante de nombres r\'eels positifs
\begin{equation}
\label{ppp01}
X_0 = 0 < X_1(\omega) < X_2(\omega) < X_3(\omega) < \dots \;, 
\end{equation} 
d\'esignant les positions des points dans $\R_+$. Alternativement, on peut
d\'ecrire une r\'ealisa\-tion en donnant le nombre de points $N_I(\omega)$
contenus dans chaque intervalle $I$ de la forme $I=]t,t+s]$. Si nous
abr\'egeons $N_{]0,t]}$ par $N_t$ (commun\'ement appel\'e \defwd{fonction de
comptage}\/), nous aurons $N_{]t,t+s]}=N_{t+s}-N_t$,
et les $N_t$ sont donn\'es en fonction des $X_n$ par
\begin{equation}
\label{ppp02}
N_t(\omega) = \sup\setsuch{n\geqs0}{X_n(\omega)\leqs t}\;.
\end{equation}
Inversement, les $X_n$ se d\'eduisent des $N_t$ par la relation 
\begin{equation}
\label{ppp3}
X_n(\omega) = \inf\setsuch{t\geqs0}{N_t(\omega)\geqs n}\;.
\end{equation}
Nous allons voir deux constructions \'equivalentes du processus de Poisson.
La premi\`ere construction part de la distribution des $N_t$.

\begin{figure}
% \vspace{10mm}
%  \centerline{
%  \includegraphics*[clip=true,width=120mm]{figs/poisson_Nt} 
%  }
%  \figtext{
%  	\writefig	4.55	0.75	{$X_1$}
%  	\writefig	5.55	0.75	{$X_2$}
%  	\writefig	7.9	0.75	{$X_3$}
%  	\writefig	10.2	0.75	{$X_4$}
%  	\writefig	10.85	0.75	{$X_5$}
%  	\writefig	12.5	0.75	{$t$}
%  	\writefig	2.3	2.05	{$1$}
%  	\writefig	2.3	3.05	{$2$}
%  	\writefig	2.3	4.05	{$3$}
%  	\writefig	2.3	5.05	{$4$}
%  	\writefig	2.3	6.05	{$5$}
%  	\writefig	2.1	7.0	{$N_t$}
%  }
 \begin{center}
\begin{tikzpicture}[-,auto,node distance=1.0cm, thick,
main node/.style={draw,circle,fill=white,minimum size=5pt,inner sep=0pt},
full node/.style={draw,red,circle,fill=red,minimum size=5pt,inner sep=0pt}]

  \path[->,>=stealth'] 
     (-1,0) edge (10,0)
     (0,-1) edge (0,6)
  ;

  \node at (9.5,-0.3) {$t$};
  \node at (-0.4,5.7) {$N_t$};

  \path[red,very thick,-(] 
    (2,1) edge (3,1)
    (3,2) edge (5,2)
    (5,3) edge (7,3)
    (7,4) edge (7.8,4)
  ;
  
  \path[red,very thick,-]
    (-1,0) edge (2,0) 
    (7.8,5) edge (9.5,5)
  ;
  
  \draw (2,0) node[main node] {} node[below=0.1cm] {$X_1$} 
  -- (3,0) node[main node] {} node[below=0.1cm] {$X_2$} 
  -- (5,0) node[main node] {} node[below=0.1cm] {$X_3$} 
  -- (7,0) node[main node] {} node[below=0.1cm] {$X_4$} 
  -- (7.8,0) node[main node] {} node[below=0.1cm] {$X_5$} 
   ;
   
  \draw (0,0) node[main node] {} 
  -- (0,1) node[main node] {} node[left=0.1cm] {$1$} 
  -- (0,2) node[main node] {} node[left=0.1cm] {$2$} 
  -- (0,3) node[main node] {} node[left=0.1cm] {$3$} 
  -- (0,4) node[main node] {} node[left=0.1cm] {$4$} 
  -- (0,5) node[main node] {} node[left=0.1cm] {$5$} 
   ;
   
  \node[full node] at (2,1) {};
  \node[full node] at (3,2) {};
  \node[full node] at (5,3) {};
  \node[full node] at (7,4) {};
  \node[full node] at (7.8,5) {};

\end{tikzpicture}
\end{center}
\vspace{-4mm}
 \caption[]{Fonction de comptage d'une r\'ealisation d'un processus de Poisson.}
 \label{fig_Poisson_Nt}
\end{figure}
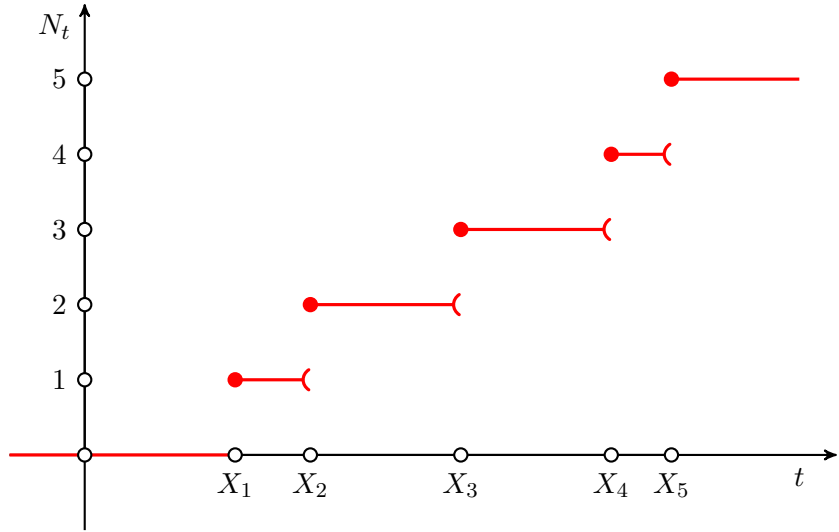

%%%%%%%%%%%%%%%%%%%%%%%%%%%%%%%%%%%%%%%%%%%%%%%%%%%%%%%%%%%%%%%%%%%%%%%%%%%

\section{Construction par la fonction de comptage}
\label{sec_ppp}

%%%%%%%%%%%%%%%%%%%%%%%%%%%%%%%%%%%%%%%%%%%%%%%%%%%%%%%%%%%%%%%%%%%%%%%%%%%

\begin{definition}[Processus de Poisson]
\label{def_ppp1}
Le processus ponctuel de Poisson satisfait les conditions suivantes:
\begin{enum}
\item	$N_I$ ne d\'epend que de la longueur de $I$, i.e. $N_{]t,t+s]}$ a la
m\^eme loi que $N_s$.
\item	Si $I_1,\dots,I_k$ sont deux \`a deux disjoints,
$N_{I_1},\dots,N_{I_k}$ sont ind\'ependants. 
\item	$\expec{N_I}$ existe pour tout $I$ (de longueur finie).
\item	Il existe un intervalle $I$ tel que $\prob{N_I>0}>0$.
\item	Absence de points doubles:
$\lim_{\eps\to0}\frac1\eps\prob{N_\eps\geqs2}=0$.  
\end{enum}
\end{definition}  

Supposant qu'un tel processus existe bel et bien, nous pouvons d\'eriver
quelques-unes de ses propri\'et\'es.

\begin{prop}\hfill
\label{prop_ppp1}
\begin{enum}
\item	Soit $\alpha(t)=\expec{N_t}$. Alors il existe $\lambda>0$ tel que 
$\alpha(t)=\lambda t$.
\item	Pour tout intervalle born\'e $I\subset\R_+$, on a $\prob{N_I\geqs
1}\leqs \expec{N_I}$. 
\end{enum}
\end{prop}
\begin{proof}\hfill
\begin{enum}
\item	Comme $N_0=0$, on a $\alpha(0)=0$. De plus, comme 
$N_{t+s}=N_t + N_{]t,t+s]}$ on a 
\begin{equation}
\label{ppp1:1}
\alpha(t+s)=\alpha(t)+\expec{N_{]t,t+s]}}=\alpha(t)+\alpha(s)\;,
\end{equation}
en vertu de la condition 1. Par un r\'esultat d'analyse, ceci implique
n\'ecessairement que $\alpha(t)=\lambda t$ pour un $\lambda\geqs 0$, et la
propri\'et\'e 4. implique $\lambda>0$. 
\item	Comme $\expec{N_I}$ est suppos\'e fini, on a 
\begin{equation}
 \label{ppp1:2}
\expec{N_I} = \sum_{k=0}^\infty k \prob{N_I=k} 
\geqs \sum_{k=1}^\infty \prob{N_I=k} = \prob{N_I\geqs 1}\;,
\end{equation}
ce qui montre \`a la fois que $\prob{N_I\geqs 1}$ est finie et
qu'elle est born\'e par
$\expec{N_I}$.
\qed
\end{enum}
\renewcommand{\qed}{}
\end{proof}

La propri\'et\'e remarquable du processus de Poisson est alors que les
variables al\'eatoires $N_{]t,t+s]}$ suivent n\'ecessairement une loi de
Poisson.
% de param\`etre $\lambda s$. 

\begin{theorem}
\label{thm_ppp1}
Si le processus satisfait les 5 conditions de la
D\'efinition~\ref{def_ppp1}, alors les variables al\'eatoires $N_{]t,t+s]}$
suivent des lois de Poisson de param\`etre $\lambda s$: 
\begin{equation}
\label{ppp2}
\prob{N_{]t,t+s]}=k} = \pi_{\lambda s}(k)
= \e^{-\lambda s}\frac{(\lambda s)^k}{k!}\;.
\end{equation}
\end{theorem}
\begin{proof}
Par la propri\'et\'e 1., il suffit de montrer le r\'esultat pour $t=0$,
c'est-\`a-dire pour $N_s$. Partageons $]0,s]$ en $k$ intervalles de longueur
\'egale, de la forme 
\begin{equation}
\label{ppp2:1}
]s_{j-1},s_j] 
\qquad
\text{o\`u $s_j = \dfrac{js}k$ pour $0\leqs j\leqs k$\;.}
\end{equation}
L'id\'ee de la preuve est que pour $k$ suffisamment grand, il est peu
probable d'avoir plus d'un point par intervalle, donc la loi de  $Y^{(k)}_j
= N_{]s_{j-1},s_j]}$ est \`a peu pr\`es une loi de Bernoulli. La loi de
$N_s$ est donc proche d'une loi binomiale, que l'on peut approximer par la
loi de Poisson pour $k$ grand. 

Il suit des conditions 1. et 2. que les $Y^{(k)}_j$ sont i.i.d., 
de m\^eme loi que $N_{s_1}=N_{s/k}$, et on a
\begin{equation}
\label{ppp2:2}
N_s = \sum_{j=1}^k Y^{(k)}_j\;.
\end{equation}
Introduisons alors des variables al\'eatoires 
\begin{equation}
\label{ppp2:3}
\overline Y^{(k)}_j = 
\begin{cases}
0 & \text{si $Y^{(k)}_j=0$\;,}\\
1 & \text{si $Y^{(k)}_j\geqs1$\;.}
\end{cases}
\end{equation}  
Les $\overline Y^{(k)}_j$ sont \'egalement i.i.d., et suivent une loi de
Bernoulli. La variable al\'eatoire 
\begin{equation}
\label{ppp2:4}
\overbar N^{(k)}_s = \sum_{j=1}^k \overline Y^{(k)}_j\;, 
\end{equation} 
satisfaisant $\overbar N^{(k)}_s \leqs N_s$ pour tout $k$, on a  
\begin{equation}
\label{ppp2:5}
\prob{\overbar N^{(k)}_s \geqs m} \leqs \prob{N_s \geqs m}
\end{equation}
pour tout $k$ et tout $m$. De plus, $\overbar N^{(k)}_s$ suit une loi
binomiale de param\`etre 
\begin{equation}
\label{ppp2:6}
p_k = \prob{\overline Y^{(k)}_j=1} = \prob{Y^{(k)}_j\geqs1} 
= \prob{N_{s/k}\geqs1}\;.
\end{equation}
Estimons maintenant la diff\'erence entre les lois de $\overbar N^{(k)}_s$
et $N_s$. Nous avons 
\begin{align}
\nonumber
\prob{\overbar N^{(k)}_s\neq N_s}
&= \prob{\exists j\in\set{1,\dots,k}:\,Y^{(k)}_j\geqs2} \\
\nonumber
&\leqs \sum_{j=1}^k \prob{Y^{(k)}_j\geqs2}  \\
%\nonumber
&= k \prob{Y^{(k)}_1\geqs2} 
= k \prob{N_{s/k}\geqs2}\;.
\label{ppp2:7}
\end{align}
La condition 5. avec $\eps=s/k$ implique alors 
\begin{equation}
\label{ppp2:8}
\lim_{k\to\infty} \prob{\overbar N^{(k)}_s\neq N_s} = 0\;.
\end{equation}
Comme on a d'une part la minoration 
\begin{equation}
\label{ppp2:9}
\prob{N_s = m} \geqs \prob{N_s = \overbar N^{(k)}_s = m} 
\geqs \prob{\overbar N^{(k)}_s = m} - \prob{\overbar N^{(k)}_s \neq N_s}\;,
\end{equation}
et d'autre part la majoration 
\begin{align}
\nonumber
\prob{N_s = m} 
&= \prob{\overbar N^{(k)}_s = N_s = m} + 
\prob{\overbar N^{(k)}_s \neq N_s = m} \\
&\leqs \prob{\overbar N^{(k)}_s = m} + \prob{N_s \neq \overbar N^{(k)}_s}\;,
\label{ppp2:10}
\end{align}
il suit que
\begin{equation}
\label{ppp2:11}
\lim_{k\to\infty} \prob{\overbar N^{(k)}_s = m} = \prob{N_s = m}\;.
\end{equation}
Il reste \`a montrer que $kp_k$ tend vers $\lambda s$ pour $k\to\infty$. Si
c'est le cas, alors la Proposition~\ref{prop_disc2} montre que $N_s$ suit une
loi de Poisson de param\`etre $\lambda s$. 
Or nous avons 
\begin{align}
\nonumber
kp_k 
&= \expec{\overbar N^{(k)}_s}
= \sum_{j=1}^\infty j\prob{\overbar N^{(k)}_s=j} 
= \sum_{j=1}^\infty \sum_{\ell=1}^j\prob{\overbar N^{(k)}_s=j} \\
&= \sum_{\ell=1}^\infty \sum_{j=\ell}^\infty\prob{\overbar N^{(k)}_s=j}
= \sum_{\ell=1}^\infty \prob{\overbar N^{(k)}_s\geqs\ell}\;. 
\label{ppp2:12}
\end{align}
Un calcul analogue montre que  
\begin{equation}
\label{ppp2:13}
\lambda s = \expec{N_s} = \sum_{\ell=1}^\infty \prob{N_s\geqs\ell}\;.
\end{equation}
Il suit de~\eqref{ppp2:5} que $kp_k \leqs \lambda s$ pour tout $k$. 
Par un th\'eor\`eme d'analyse, on peut alors intervertir limite et somme, et
\'ecrire 
\begin{equation}
\label{ppp2:14}
\lim_{k\to\infty} \sum_{\ell=1}^\infty \prob{\overbar N^{(k)}_s\geqs\ell} 
= \sum_{\ell=1}^\infty \lim_{k\to\infty} \prob{\overbar N^{(k)}_s\geqs\ell}
= \sum_{\ell=1}^\infty \prob{N_s\geqs\ell} = \lambda s\;,  
\end{equation}
en vertu de~\eqref{ppp2:11}. Ceci montre que $kp_k$ converge bien vers
$\lambda s$, et par cons\'equent que la loi de $N_s$, \'etant la limite
d'une loi binomiale de param\`etre $kp_k$, est une loi de Poisson de
param\`etre $\lambda s$. 
\end{proof}

%%%%%%%%%%%%%%%%%%%%%%%%%%%%%%%%%%%%%%%%%%%%%%%%%%%%%%%%%%%%%%%%%%%%%%%%%%%

\section{Construction par les temps d'attente}
\label{sec_ppt}

%%%%%%%%%%%%%%%%%%%%%%%%%%%%%%%%%%%%%%%%%%%%%%%%%%%%%%%%%%%%%%%%%%%%%%%%%%%

La seconde construction du processus ponctuel de Poisson se base sur la
distribution des  diff\'erences de position $Z_n=X_n-X_{n-1}$. Celles-ci
caract\'erisent \'egalement de mani\`ere univoque le processus, via la
relation
\begin{equation}
\label{ppp4}
X_n(\omega) = \sum_{j=1}^n Z_j(\omega)\;.
\end{equation}
Le r\'esultat remarquable est alors que les $Z_j$ sont i.i.d.\ et suivent une
loi bien particuli\`ere, \`a savoir une loi exponentielle de param\`etre
$\lambda$. 

%\goodbreak
\begin{theorem}
\label{thm_ppp2}
Pour tout $n$, les variables al\'eatoires $Z_1,\dots,Z_n$ sont
ind\'ependantes, et suivent la m\^eme loi exponentielle $\cE\!xp(\lambda)$. 
\end{theorem}
\begin{proof}
Fixons des instants 
\begin{equation}
\label{ppp5:1}
t_0=0 < s_1 < t_1 < s_2 < t_2 < \dots < s_n < t_n\;.
\end{equation}
Nous pouvons alors calculer 
\begin{align}
\nonumber
&\bigprob{X_1\in]s_1,t_1], X_2\in]s_2,t_2], \dots, X_n\in]s_n,t_n]} \\
\nonumber
&{\qquad}= \bigprob{N_{]0,s_1]}=0, N_{]s_1,t_1]}=1, N_{]t_1,s_2]}=0, \dots, 
N_{]t_{n-1},s_n]}=0, N_{]s_n,t_n]}\geqs1} \\
\nonumber
&{\qquad}= \prod_{k=1}^n \bigprob{N_{]t_{k-1},s_k]}=0}
\prod_{k=1}^{n-1} \bigprob{N_{]s_k,t_k]}=1} 
\,\bigprob{N_{]s_n,t_n]}\geqs 1} \\
\nonumber
&{\qquad}= \prod_{k=1}^n \e^{-\lambda(s_k-t_{k-1})}
\prod_{k=1}^{n-1} \lambda(t_k-s_k)\e^{-\lambda(t_k-s_k)} 
\,\bigbrak{1 - \e^{-\lambda(t_n-s_n)}} \\
\nonumber
&{\qquad}= \lambda^{n-1} \prod_{k=1}^{n-1} (t_k-s_k)
\bigbrak{\e^{-\lambda s_n} - \e^{-\lambda t_n}} \\ 
&{\qquad}= \int_{s_1}^{t_1} \int_{s_2}^{t_2} \dots \int_{s_n}^{t_n} 
\lambda^n \e^{-\lambda x_n} \,\6x_n \6x_{n-1} \dots \6x_2 \6x_1\;. 
\label{ppp5:2}
\end{align}
La loi conjointe de $(X_1,\dots,X_n)$ admet donc la densit\'e 
\begin{equation}
\label{ppp5:3}
f(x_1,\dots,x_n) =
\begin{cases}
\lambda^n \e^{-\lambda x_n} & \text{si $0<x_1<\dots<x_n$\;,}\\
0 & \text{sinon\;.}
\end{cases}
\end{equation}
Nous pouvons alors calculer la fonction de r\'epartition des $Z_k$: 
\begin{align}
\nonumber
\bigprob{Z_1\leqs z_1,\dots,Z_n\leqs z_n}
&= \bigprob{X_1\leqs z_1,X_2-X_1\leqs z_2,\dots,X_n-X_{n-1}\leqs z_n} \\
\nonumber
&= \bigprob{X_1\leqs z_1,X_2\leqs z_2+X_1,\dots,X_n\leqs z_n+X_{n-1}} \\
&= \int_0^{z_1} \int_{x_1}^{z_2+x_1} \dots \int_{x_{n-1}}^{z_n+x_{n-1}}
f(x_1,\dots,x_n) \6x_n \dots \6x_1 \;.
\label{ppp5:4}
\end{align}
La densit\'e conjointe de $(Z_1,\dots,Z_n)$ s'obtient alors en calculant la
d\'eriv\'ee 
\begin{align}
\nonumber
\dpar{^n}{z_1\dots\partial z_n}\bigprob{Z_1\leqs z_1,\dots,Z_n\leqs z_n}
&= f(z_1,z_1+z_2,\dots,z_1+\dots+z_n) \\
&= \lambda^n \e^{-\lambda (z_1+\dots+z_n)}\;,
\label{ppp5:5}
\end{align}
Or cette densit\'e est bien la densit\'e conjointe de $n$ variables
exponentielles ind\'ependantes de param\`etre $\lambda$. 
\end{proof}

Ce r\'esultat fournit une m\'ethode permettant de construire le processus de
Poisson: chaque $X_n$ est obtenu \`a partir de $X_{n-1}$ en lui ajoutant une
variable al\'eatoire de loi exponentielle, ind\'ependante des variables
pr\'ec\'edentes. On notera que $X_n$ suit une loi Gamma de param\`etres
$(\lambda, n)$. 

% Ceci pr\'esuppose l'existence d'un espace probabilis\'e
% pouvant contenir une infinit\'e de variables al\'eatoires exponentielles
% ind\'ependantes: C'est l'espace produit que nous avons \'evoqu\'e (sans
% prouver son existence) dans l'Exemple~\ref{ex_mpep1}. 

Notons que la propri\'et\'e de Markov~\eqref{vard23} peut s'\'ecrire
\begin{equation}
 \label{ppp6}
\bigpcond{Z_n>t+\eps}{Z_n>t} = \e^{-\lambda\eps}
\qquad\forall\eps>0\;,
\end{equation} 
et donc 
\begin{equation}
 \label{ppp7}
\lim_{\eps\to0}\frac{\bigpcond{Z_n\leqs t+\eps}{Z_n>t}}{\eps} 
= \lim_{\eps\to0}\frac{1-\e^{-\lambda\eps}}{\eps} 
= \lambda\;. 
\end{equation} 
Le param\`etre $\lambda$ repr\'esente donc le taux d'apparition d'un nouveau
point, qui est ind\'epen\-dant du pass\'e. Ce r\'esultat peut bien s\^ur
\'egalement \^etre obtenu directement \`a l'aide de la fonction de comptage. 

%%%%%%%%%%%%%%%%%%%%%%%%%%%%%%%%%%%%%%%%%%%%%%%%%%%%%%%%%%%%%%%%%%%%%%%%%%%

\goodbreak
\section{G\'en\'eralisations}
\label{sec_ppgen}

%%%%%%%%%%%%%%%%%%%%%%%%%%%%%%%%%%%%%%%%%%%%%%%%%%%%%%%%%%%%%%%%%%%%%%%%%%%

Il existe plusieurs g\'en\'eralisations du processus ponctuel de Poisson
discut\'e ici:
\begin{enum}

\item	Le \defwd{processus de Poisson inhomog\`ene}~: Dans ce cas le
nombre de points $N_{]t,t+s]}$ suit une loi de Poisson de param\`etre 
\begin{equation}
 \label{ppp8}
\int_t^{t+s} \lambda(u)\,\6u\;, 
\end{equation}  
o\`u $\lambda(u)$ est une fonction positive, donnant le taux au temps $u$. Ce
processus permet de d\'ecrire des situations o\`u les points apparaissent avec
une intensit\'e variable, par exemple si l'on veut tenir compte des variations
journali\`eres du trafic influen\c cant les horaires de passage de bus. On
retrouve le processus de Poisson homog\`ene si $\lambda(u)$ est constant.  

\item	Le \defwd{processus de Poisson de dimension $n\geqs2$}~: Ce processus
peut \^etre d\'efini par sa fonction de comptage, en rempla\c cant les
intervalles $I$ par des sous-ensembles (mesurables) de $\R^n$. Les nombres de
points dans deux ensembles disjoints sont \`a nouveau ind\'epen\-dants, et le
nombre de points dans un ensemble est proportionnel \`a son volume. Ce processus
peut par exemple mod\'eliser la distribution des \'etoiles dans une r\'egion de
l'espace ou du ciel.

\item	Le \defwd{processus de naissance et de mort}~: Le processus ponctuel de
Poisson peut \^etre consid\'er\'e comme un processus de naissance pur~: Si
$N_t$ est interpr\'et\'e comme le nombre d'individus dans une population au
temps $t$, ce nombre augmente avec un taux constant $\lambda$. Plus
g\'en\'eralement, dans un processus de naissance et de mort, de nouveaux
individus naissent avec un taux $\lambda$ et meurent avec un taux $\mu$;
\'eventuellement, $\lambda=\lambda(N_t)$ et $\mu=\mu(N_t)$ peuvent d\'ependre de
la taille actuelle de la population.

\item	Le \defwd{processus de Poisson compos\'e}~: Soit $N_t$ la fonction de
comptage d'un processus de Poisson simple, et soient $Y_1, Y_2, \dots$, des
variables al\'eatoires ind\'ependantes et identiquement distribu\'ees (i.i.d.)
et ind\'ependantes de $N_t$. Alors le processus 
\begin{equation}
 \label{ppp9}
S_t = \sum_{i=1}^{N_t} Y_i 
\end{equation}  
est appel\'e un processus de Poisson compos\'e. A chaque instant $X_n$ o\`u le
processus de Poisson sous-jacent augmente d'une unit\'e, le processus compos\'e
$S_t$ augmente d'une quantit\'e al\'eatoire $Y_n$. Si $N_t$ d\'ecrit les
instants d'arriv\'ee de clients dans une station service, et $Y_n$ d\'ecrit la
quantit\'e d'essence achet\'ee par le client $n$, alors $S_t$ est la quantit\'e
d'essence totale vendue au temps $t$. Si tous les $Y_i$ valent $1$ presque
s\^urement, on retrouve le processus de Poisson de d\'epart. 

\item	Le \defwd{processus de renouvellement}~: Un tel processus est d\'efini
comme le processus de Poisson \`a partir de ses temps d'attente $Z_n$, sauf que
les $Z_n$ ne suivent pas forc\'ement une loi exponentielle. Il mod\'elise par
exemple les instants de remplacement de machines de dur\'ee de vie al\'eatoire
(d'o\`u le terme de \lq\lq renouvellement\rq\rq~: chaque nouvelle machine
\'etant suppos\'ee ind\'ependante des pr\'ec\'edentes, le processus oublie son
pass\'e \`a chaque instant de renouvellement). Ce processus n'a en
g\'en\'eral plus les propri\'et\'es d'ind\'ependance de la fonction de
comptage, mais sous certaines hypoth\`eses sur la loi des $Z_n$
(existence des deux premiers moments), il existe des versions asymptotiques de
ces propri\'et\'es. 
\end{enum}

%%%%%%%%%%%%%%%%%%%%%%%%%%%%%%%%%%%%%%%%%%%%%%%%%%%%%%%%%%%%%%%%%%%%%%%%%%%

\goodbreak
\section{Exercices}
\label{sec_exo_poisson}

\begin{exercice}
\label{exo_poisson01} 
Des clients arrivent dans une banque suivant un processus de Poisson
d'intensit\'e $\lambda$. Sachant que deux clients sont arriv\'es dans la
premi\`ere heure, quelle est la probabilit\'e que 
\begin{enum}
\item 	les deux soient arriv\'es dans les 20 premi\`eres minutes?
\item	l'un au moins soit arriv\'e dans les 20 premi\`eres minutes?
\end{enum}
\end{exercice}

\goodbreak

\begin{exercice}
\label{exo_poisson02} 
Dans un serveur informatique, les requ\^etes arrivent selon un processus
ponctuel de Poisson, avec un taux de $60$ requ\^etes par heure. D\'eterminer
les probabilit\'es suivantes~:
\begin{enum}
\item	L'intervalle entre les deux premi\`eres requ\^etes est compris entre 2
et 4 minutes.
\item	Aucune requ\^ete n'arrive entre 14h et 14h05.
\item	Sachant que deux requ\^etes sont arriv\'ees entre 14h et 14h10, les
deux sont arriv\'ees dans les 5 premi\`eres minutes.
\item	Sachant que deux requ\^etes sont arriv\'ees entre 14h et 14h10, au
moins une est arriv\'ee dans les 5 premi\`eres minutes.
\end{enum}
\end{exercice}

\goodbreak

\begin{exercice}
\label{exo_poisson03} 
Un serveur informatique envoie des messages selon un processus ponctuel de
Poisson. En moyenne, il envoie un message toutes les 30 secondes. 

\begin{enum}
\item	Quelle est la probabilit\'e que le serveur n'envoie aucun message au
cours des 2 premi\`eres minutes de sa mise en service?
\item	A quel moment esp\'erez-vous le second message (quel est le temps moyen
de l'envoi du second message)?
\item	Quelle est la probabilit\'e que le serveur n'ait pas envoy\'e de
message durant la premi\`ere minute, sachant qu'il a envoy\'e  3
messages au cours des 3 premi\`eres minutes?
\item	Quelle est la probabilit\'e qu'il y ait moins de 3
messages au cours des 2 premi\`eres minutes, sachant qu'il y en a eu au moins 1
au cours de la premi\`ere minute?
\end{enum}
\end{exercice}

\goodbreak

\begin{exercice}
\label{exo_poisson031} 

Dans un centre d'appel t\'el\'ephonique, les appels arrivent selon un processus
de Poisson de taux 10 appels par heure. 

\begin{enum}
\item	Si une t\'el\'ephoniste fait une pause de 10 heures \`a 10h30,
combien d'appels va-t-elle rater en moyenne durant sa pause? 
\item 	Quelle est la probabilit\'e qu'elle a rat\'e au plus 2 appels? 
\item 	Sachant que 4 appels arrivent entre 10 heures et 11 heures, quelle est
la probabilit\'e qu'elle n'a rat\'e aucun appel durant sa pause? Qu'elle n'a
rat\'e qu'un appel? 
\item 	Sachant qu'il y aura 2 appels entre 10h30 et 11 heures, quelle est la
probabilit\'e qu'ils arrivent tous entre 10h30 et 10h45?
\end{enum}
\end{exercice}

\goodbreak

\begin{exercice}
\label{exo_poisson04} 
L'\'ecoulement des voitures le long d'une route est mod\'elis\'e par un
processus de Poisson d'intensit\'e $\lambda=2$ voitures par minute. A cause
d'un chantier, le trafic est arr\^et\'e alternativement dans chaque direction. 
On admet qu'\`a l'arr\^et, chaque v\'ehicule occupe une longueur de $8$
m\`etres en moyenne. 
\begin{enum}
\item	Quelle est la loi du temps d'arriv\'ee $X_n$ de la $n$i\`eme voiture? 
\item	A l'aide du th\'eor\`eme central limite, donner une
approximation gaussienne de la loi de $X_n$.
\item	Pendant combien de temps peut-on arr\^eter le trafic si l'on d\'esire
que la queue ainsi form\'ee ne d\'epasse la longueur de $250$m qu'avec une
probabilit\'e de $0.2$? (La valeur de $x$ pour laquelle $\prob{\cN(0,1)<x}=0.2$
est $x\simeq-0.85$).
\end{enum}
\end{exercice}

\goodbreak

\begin{exercice}[Le paradoxe de l'autobus]
\label{exo_poisson05} 
Les temps d'arriv\'ee d'autobus \`a un arr\^et sont d\'ecrits par un
processus de Poisson $(X_n)_n$ d'intensit\'e $\lambda$. Un client arrive \`a
l'instant $t$ apr\`es le d\'ebut du service.
\begin{enum}
\item	Calculer la probabilit\'e que le client rate le $n$i\`eme bus
mais attrape le $(n+1)$i\`eme bus.
\item	Calculer la probabilit\'e qu'il rate le $n$i\`eme bus et doive
attendre le $(n+1)$i\`eme bus pendant un temps $s$ au moins.
\item	Calculer la probabilit\'e qu'il doive attendre le bus suivant pendant un
temps $s$ au moins.
\item	En d\'eduire le temps d'attente moyen, et comparer ce temps avec le
temps de passage moyen entre bus. Qu'en pensez-vous? 
\end{enum}
\end{exercice}

\goodbreak

\begin{exercice}
\label{exo_poisson06} 
Soient $(X_n)_n$ et $(Y_n)_n$ deux processus de Poisson ind\'ependants,
d'intensit\'es respectives $\lambda$ et $\mu$. Soit $(Z_n)_n$ le processus
obtenu
en superposant $(X_n)_n$ et $(Y_n)_n$. Montrer qu'il s'agit encore d'un
processus de Poisson et donner son intensit\'e. 
\end{exercice}

\goodbreak

\begin{exercice}
\label{exo_poisson07} 
Soit $X_n$ un processus ponctuel de Poisson d'intensit\'e $\lambda$. 

\noindent
Soit $Y_n$ le processus obtenu en effa\c cant chaque $X_n$ ($n\geqs1$) avec
probabilit\'e $1/2$, ind\'epen\-dam\-ment de tous les autres, puis en
renum\'erotant les points restants par les entiers positifs. On note $N_t$,
respectivement $M_t$, le nombre de $X_n$, respectivement $Y_n$, dans
l'intervalle $]0,t]$. 

\begin{enum}
\item	Donner la loi de $N_t$.
\item	Montrer que pour tout $k$, $l\mapsto \pcond{M_t=l}{N_t=k}$ suit une loi
binomiale, et d\'eterminer ses param\`etres.
\item	En d\'eduire la loi de $M_t$.
\item	Montrer que $Y_n$ est un processus ponctuel de Poisson, et d\'eterminer
son intensit\'e. 
\item	Que se passe-t-il si on efface chaque $X_n$ avec probabilit\'e $1-q$,
$0<q<1$?
\end{enum}
\end{exercice}

%%%%%%%%%%%%%%%%%%%%%%%%%%%%%%%%%%%%%%%%%%%%%%%%%%%%%%%%%%%%%%%%%%%%%%%%%%%

\chapter{Processus markoviens de sauts}
\label{chap_psm}

Les processus markoviens de sauts sont un analogue en temps continu des
cha\^ines de Markov. Rappelons que la propri\'et\'e essentielle permettant de
d\'efinir ces cha\^ines est la propri\'et\'e de Markov~:
\begin{align}
\nonumber
\bigpcond{X_n=j}{X_{n-1}=i,X_{n-2}=i_{n-2},\dots,X_0=i_0} 
&= \bigpcond{X_n=j}{X_{n-1}=i} \\
&= p_{ij}
\label{psm1}
\end{align}
pour tout temps $n$. 
Un processus markovien de sauts est d\'efini par la condition 
\begin{align}
\nonumber
\bigpcond{X_{s+t}=j}{X_s=i,X_{s_{n-2}}=i_{n-2},\dots,X_{s_0}=i_0} 
&= \bigpcond{X_{s+t}=j}{X_s=i} \\
&= P_t(i,j)
\label{psm2}
\end{align}
pour tout choix de temps $0\leqs s_0 < s_1 < \dots < s_{n-2} < s < s+t \in
\R_+$. Pour chaque $t\geqs0$, l'ensemble des $P_t(i,j)$ doit former une matrice
stochastique, appel\'ee \defwd{noyau de transition}. Ses \'el\'ements sont
appel\'es \defwd{probabilit\'es de transition}.

\begin{example}[Cha\^ine de Markov en temps continu]
\label{ex_cmtc} 
Soit $Y_n$ une cha\^ine de Markov de matrice de transition $P$, et soit $N_t$
la fonction de comptage d'un processus ponctuel de Poisson de taux $\lambda$.
Alors 
\begin{equation}
 \label{psm3}
X_t = Y_{N_t}
\end{equation}
est un processus markovien de sauts, qu'on appelle une \defwd{cha\^ine de
Markov en temps continu}. Ce processus effectue des transitions \`a des temps
al\'eatoires, donn\'es par le processus de Poisson, au lieu de temps
r\'eguli\`erement espac\'es comme dans le cas de la cha\^ine de Markov de
d\'epart. La propri\'et\'e~\eqref{psm2} est une cons\'equence de deux 
propri\'et\'es de Markov~: celle de la cha\^ine $Y_n$, et celle
de la loi exponentielle, c.f.~\eqref{vard23}. 

Comme le nombre de sauts $N_t$ jusqu'au temps $t$ suit une loi de 
Poisson, les probabilit\'es de transition sont donn\'ees par 
\begin{equation}
 \label{psm4}
P_t(i,j)
= \sum_{n=0}^\infty \bigprob{N_t=n} \bigprobin{i}{X_n=j}
= \sum_{n=0}^\infty \e^{-\lambda t} \frac{(\lambda t)^n}{n!}
p^{(n)}_{ij}\;, 
\end{equation} 
o\`u les $p^{(n)}_{ij}$ sont les \'el\'ements de la matrice $P^n$, d\'ecrivant
les probabilit\'es de transition en $n$ pas de la cha\^ine de Markov $Y_n$.  
\end{example}

%%%%%%%%%%%%%%%%%%%%%%%%%%%%%%%%%%%%%%%%%%%%%%%%%%%%%%%%%%%%%%%%%%%%%%%%%%%

\section{Taux de transition}
\label{sec_tt}

Les noyaux de transition $P_t$ pour diff\'erents $t$ ne sont pas arbitraires,
mais ob\'eissent \`a une \'equation dite de Chapman--Kolmogorov (qui est
l'analogue de la relation~\eqref{fdef11} pour les cha\^ines de Markov). 

\begin{prop}[Equation de Chapman--Kolmogorov]
 \label{prop_Chapman_Komogorov}
La famille $\set{P_t}_{t\geqs0}$ des noyaux de transition d'un processus
markovien de sauts sur un ensemble d\'enombrable $\cX$ satisfait 
\begin{equation}
 \label{tt01}
P_{s+t}(i,j) = \sum_{k\in\cX} P_s(i,k)P_t(k,j) 
\end{equation} 
pour tout choix de $s,t\geqs0$ et de $i,j\in\cX$. 
\end{prop}
\begin{proof}
On a 
\begin{align}
\nonumber
P_{s+t}(i,j)
&= \bigpcond{X_{s+t}=j}{X_0=i} \\
\nonumber
&= \sum_{k\in\cX} \frac{\prob{X_{s+t}=j,X_s=k,X_0=i}}{\prob{X_0=i}} \\
\nonumber
&= \sum_{k\in\cX} \frac{\prob{X_{s+t}=j,X_s=k,X_0=i}}{\prob{X_s=k,X_0=i}} 
\frac{\prob{X_s=k,X_0=i}}{\prob{X_0=i}}\\
&= \sum_{k\in\cX} \underbrace{\bigpcond{X_{s+t}=j}{X_s=k,X_0=i}}_{=P_t(k,j)}
\underbrace{\bigpcond{X_s=k}{X_0=i}}_{=P_s(i,k)}\;,
\label{tt02}
\end{align}
en vertu de la propri\'et\'e de Markov~\eqref{psm2}.
\end{proof}

L'\'equation de Chapman--Kolmogorov montre que si l'on conna\^it les noyaux
$P_t$ pour tous les temps $t$ dans un intervalle $]0,t_0]$, aussi petit soit-il,
alors on conna\^it les noyaux pour tous les temps $t\in\R_+$. Ceci sugg\`ere de
consid\'erer la limite 
\begin{equation}
 \label{tt03}
q(i,j) = \lim_{h\to0} \frac{P_h(i,j)}{h} \;,
\qquad
i\neq j\in \cX\;.
\end{equation} 
Si la limite existe, alors c'est par d\'efinition la d\'eriv\'ee \`a droite de
$P_t(i,j)$ en $t=0$. Dans ce cas, $q(i,j)$ est appel\'e le \defwd{taux de
transition de l'\'etat $i$ vers l'\'etat $j$}.

\begin{example}[Cha\^ine de Markov en temps continu]
 \label{ex_cmct2} 
En effectuant un d\'eveloppement limit\'e de l'expression~\eqref{psm4} en
$h=0$, on obtient (comme $\smash{p^{(0)}_{ij}}=0$ pour $i\neq j$)
\begin{equation}
 \label{tt04}
P_h(i,j) = (1-\lambda h + \dots)(\lambda h p_{ij} + \dots)
= \lambda h p_{ij} + \order{h}\;, 
\end{equation} 
ce qui donne le taux de transition
\begin{equation}
 \label{tt05}
q(i,j) = \lambda p_{ij}\;. 
\end{equation} 
\end{example}

\begin{example}[Processus de Poisson]
\label{ex_poisson_markov} 
Soit $N_t$ la fonction de comptage d'un processus ponctuel de Poisson de
param\`etre $\lambda$. On peut la consid\'erer comme un processus markovien de
sauts, dans lequel seules les transitions entre $n$ et $n+1$ sont permises. Les
probabilit\'es de transition sont alors donn\'ees par 
\begin{equation}
 \label{tt05B}
P_s(n,n+1) = \bigpcond{N_{t+s}=n+1}{N_t=n}
= \bigprob{N_{]t,t+s]}=1} = \lambda s \e^{-\lambda s}\;. 
\end{equation} 
Les taux de transition correspondants sont donc 

\begin{equation}
 \label{tt05C}
q(n,n+1) = \lim_{h\to0} \frac{\lambda h\e^{-\lambda h}}{h} = \lambda\;. 
\end{equation} 
La cha\^ine peut \^etre repr\'esent\'ee graphiquement comme dans la
\figref{fig_markov_Poisson}.
\end{example}

\begin{figure}[ht]
%  \centerline{
%  \includegraphics*[clip=true,height=10mm]{figs/markov_poisson}
%  }
%  \figtext{
%  \writefig       2.0     0.9     $0$
%  \writefig       4.9     0.9     $1$
%  \writefig       7.8     0.9     $2$
%  \writefig       10.7    0.9     $3$
%  \writefig       3.3     1.1     $\lambda$
%  \writefig       6.17    1.1     $\lambda$
%  \writefig       9.07    1.1     $\lambda$
%  \writefig       11.9    1.1     $\lambda$
%  }
\begin{center}
\begin{tikzpicture}[->,>=stealth',shorten >=2pt,shorten <=2pt,auto,node
distance=3.0cm, thick,main node/.style={circle,scale=0.7,minimum size=1.1cm,
fill=violet!20,draw,font=\sffamily\Large}]

  \node[main node] (0) {$0$};
  \node[main node] (1) [right of=0] {$1$};
  \node[main node] (2) [right of=1] {$2$};
  \node[main node] (3) [right of=2] {$3$};
  \node[node distance=2cm] (4) [right of=3] {$\dots$};

  \path[every node/.style={font=\sffamily\small}]
    (0) edge [above] node {$\lambda$} (1)
    (1) edge [above] node {$\lambda$} (2)
    (2) edge [above] node {$\lambda$} (3)
    (3) edge [above] node {$\lambda$} (4)
    ;
\end{tikzpicture}
\end{center}
\vspace{-3mm}
 \caption[]{Graphe repr\'esentant le processus de Poisson, ou processus de
naissance pur.}
 \label{fig_markov_Poisson}
\end{figure}
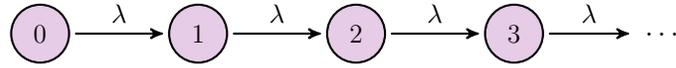

En r\`egle g\'en\'erale, on construit le processus de sauts directement \`a
partir des taux de transition $q(i,j)$, au lieu de partir des noyaux de
transition $P_t(i,j)$. Pour cela, on commence par d\'efinir 
\begin{equation}
 \label{tt06}
\lambda(i) = \sum_{j\neq i} q(i,j)\;, 
\end{equation}  
qui donne le taux avec lequel le processus quitte l'\'etat $i$. Supposons que
$0<\lambda(i)<\infty$. Dans ce cas 
\begin{equation}
 \label{tt07}
r(i,j) = \frac{q(i,j)}{\lambda(i)} 
\end{equation}  
donne la probabilit\'e que le processus choisisse l'\'etat $j$, sachant qu'il
quitte l'\'etat $i$. 

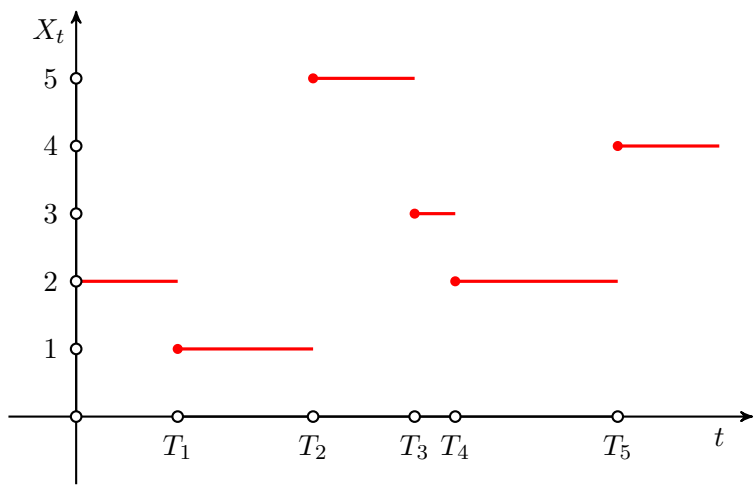
\begin{figure}[b]
\begin{center}
\begin{tikzpicture}[-,auto,scale=0.9,node distance=1.0cm, thick,
main node/.style={draw,circle,fill=white,minimum size=4pt,inner sep=0pt},
full node/.style={draw,red,circle,fill=red,minimum size=3pt,inner sep=0pt}]

  \path[->,>=stealth'] 
     (-1,0) edge (10,0)
     (0,-1) edge (0,6)
  ;

  \node at (9.5,-0.3) {$t$};
  \node at (-0.4,5.7) {$X_t$};

  \path[red,very thick,-] 
    (1.5,1) edge (3.5,1)
    (3.5,5) edge (5,5)
    (5,3) edge (5.6,3)
    (5.6,2) edge (8,2)
  ;
  
  \path[red,very thick,-]
    (0,2) edge (1.5,2) 
    (8,4) edge (9.5,4)
  ;
  
  \draw (1.5,0) node[main node] {} node[below=0.1cm] {$T_1$} 
  -- (3.5,0) node[main node] {} node[below=0.1cm] {$T_2$} 
  -- (5,0) node[main node] {} node[below=0.1cm] {$T_3$} 
  -- (5.6,0) node[main node] {} node[below=0.1cm] {$T_4$} 
  -- (8,0) node[main node] {} node[below=0.1cm] {$T_5$} 
   ;
   
  \draw (0,0) node[main node] {} 
  -- (0,1) node[main node] {} node[left=0.1cm] {$1$} 
  -- (0,2) node[main node] {} node[left=0.1cm] {$2$} 
  -- (0,3) node[main node] {} node[left=0.1cm] {$3$} 
  -- (0,4) node[main node] {} node[left=0.1cm] {$4$} 
  -- (0,5) node[main node] {} node[left=0.1cm] {$5$} 
   ;
   
  \node[full node] at (1.5,1) {};
  \node[full node] at (3.5,5) {};
  \node[full node] at (5,3) {};
  \node[full node] at (5.6,2) {};
  \node[full node] at (8,4) {};

\end{tikzpicture}
\end{center}
\vspace{-4mm}
 \caption[]{Une r\'ealisation d'un processus de sauts markovien, de cha\^ine de
Markov sous-jacente $(Y_n)_n = (2,1,5,3,2,4,\dots)$.}
 \label{fig_processus_sauts}
\end{figure}

Soit $\set{Y_n}_{n\geqs0}$ une cha\^ine de Markov de matrice de transition
$R=\set{r(i,j)}$. A partir de cette cha\^ine, on peut construire (et simuler) le
processus de saut de taux $\set{q(i,j)}$ de la mani\`ere suivante~: 
\begin{itemiz}
\item 	On se donne des variables i.i.d.\ $\tau_0, \tau_1, \dots$ de loi
exponentielle $\cE\!xp(1)$, ind\'ependantes des probabilit\'es r\'egissant la
cha\^ine $Y_n$. 
\item	Si le
processus part de l'\'etat $X_0=Y_0$, il doit quitter cet \'etat avec un taux
$\lambda(Y_0)$, c'est-\`a-dire apr\`es un temps al\'eatoire $t_1$ de loi
$\cE\!xp(\lambda(Y_0))$. Il suffit pour cela de choisir
$t_1=\tau_0/\lambda(Y_0)$. 
\item	Au temps $t_1$, le processus saute dans l'\'etat
$Y_1$, dans lequel il doit rester pendant un temps al\'eatoire $t_2$ de loi
$\cE\!xp(\lambda(Y_1))$. 
\end{itemiz}

En continuant de cette mani\`ere, on est donc amen\'es
\`a d\'efinir des temps de s\'ejour 
\begin{equation}
 \label{tt08}
t_n = \frac{\tau_{n-1}}{\lambda(Y_{n-1})}\;, 
\end{equation} 
et de poser 
\begin{equation}
 \label{tt09}
X_t = Y_n 
\qquad
\text{pour $T_n \leqs t < T_{n+1}$}\;, 
\end{equation} 
o\`u les instants des sauts sont donn\'es par 
\begin{equation}
 \label{tt10}
T_n = \sum_{i=1}^n t_i\;. 
\end{equation} 
La seule diff\'erence entre le cas g\'en\'eral et le cas particulier de la
cha\^ine de Markov en temps continu (Exemple~\ref{ex_cmtc}) est que les
diff\'erents temps de s\'ejour $t_n$ n'ont pas n\'ecessairement le m\^eme
param\`etre. 

%%%%%%%%%%%%%%%%%%%%%%%%%%%%%%%%%%%%%%%%%%%%%%%%%%%%%%%%%%%%%%%%%%%%%%%%%%%

\section{G\'en\'erateur et \'equations de Kolmogorov}
\label{sec_ek}

Nous montrons maintenant comment d\'eterminer le noyau de transition $P_t$ \`a
partir des taux de transition $q(i,j)$. Un r\^ole important est jou\'e par le
g\'en\'erateur du processus.

\begin{definition}[G\'en\'erateur]
\label{def_generateur}
Soit $X_t$ un processus markovien de sauts de taux de transition
$\set{q(i,j)}$. On appelle \defwd{g\'en\'erateur (infinit\'esimal)} du
processus la matrice $L$ d'\'el\'ements 
\begin{equation}
 \label{ek01}
L(i,j) = 
\begin{cases}
q(i,j) & \text{si $i\neq j$\;,} \\
-\lambda(i) & \text{si $i=j$\;.}
\end{cases} 
\end{equation}  
\end{definition}

On remarquera que la d\'efinition~\eqref{tt06} des $\lambda(i)$ implique 
\begin{equation}
 \label{ik02}
\sum_{j\in\cX} L(i,j) = 0 
\qquad \forall i\in\cX\;. 
\end{equation} 

\begin{theorem}[Equations de Kolmogorov]
 \label{thm_Kolmogorov}
Le noyau de transition satisfait l'\/\defwd{\'equation de Kolmogorov
r\'etrograde}
\begin{equation}
 \label{ik03}
\dtot{}{t} P_t = L P_t\;, 
\end{equation}  
ainsi que l'\/\defwd{\'equation de Kolmogorov progressive}
\begin{equation}
 \label{ik04}
\dtot{}{t} P_t = P_t L\;. 
\end{equation} 
\end{theorem}
\begin{proof}
Pour l'\'equation r\'etrograde, commen\c cons par \'ecrire, en utilisant
l'\'equa\-tion de Chapman--Kolmogorov, 
\begin{align}
\nonumber
P_{t+h}(i,j) - P_t(i,j)
&= \sum_{k\in\cX} P_h(i,k)P_t(k,j) - P_t(i,j) \\
&= \sum_{k\neq i} P_h(i,k)P_t(k,j) 
+ \bigbrak{P_h(i,i) - 1} P_t(i,j)\;.
\label{ik05:01} 
\end{align}
La d\'eriv\'ee s'obtient en divisant par $h$ des deux c\^ot\'es et en prenant la
limite $h\to0$. Pour le premier terme de droite, on a 
\begin{equation}
 \label{ik05:02}
\lim_{h\to0} \frac{1}{h}  \sum_{k\neq i} P_h(i,k)P_t(k,j)
= \sum_{k\neq i} q(i,k)P_t(k,j)\;.
\end{equation} 
Pour traiter le second terme, on observe que la condition 
$\sum_k P_h(i,k)=1$ implique 
\begin{equation}
 \label{ik05:03}
\lim_{h\to0} \frac{1}{h} \Bigpar{P_h(i,i) - 1} 
= \lim_{h\to0} \frac{1}{h} 
\biggpar{-\sum_{k\neq i}P_h(i,k)}
= -\sum_{k\neq i} q(i,k)
= -\lambda(i)\;.
\end{equation} 
En ins\'erant~\eqref{ik05:02} et~\eqref{ik05:03} dans \eqref{ik05:01}, on
obtient 
\begin{equation}
 \label{ik05:04}
\dtot{}{t} P_t(i,j) 
%= \lim_{h\to0} \frac1h  \bigbrak{P_{t+h}(i,j) - P_t(i,j)}
= \sum_{k\neq i} q(i,k) P_t(k,j) - \lambda(i) P_t(i,j)
= \sum_{k\in\cX} L(i,k) P_t(k,j)\;,
\end{equation} 
ce qui n'est autre que l'\'equation~\eqref{ik03} \'ecrite en composantes. 
La preuve de l'\'equation progressive est analogue, en utilisant la
d\'ecomposition 
\begin{align}
\nonumber
P_{t+h}(i,j) - P_t(i,j)
&= \sum_{k\in\cX} P_t(i,k)P_h(k,j) - P_t(i,j) \\
&= \sum_{k\neq j} P_t(i,k)P_h(k,j) 
+ \bigbrak{P_h(j,j) - 1} P_t(i,j)
\label{ik05:05} 
\end{align}
au lieu de~\eqref{ik05:01}.
\end{proof}

Les \'equations de Kolmogorov sont des syst\`emes d'\'equations lin\'eaires
coupl\'ees pour les \'el\'ements de matrice $P_t(i,j)$. Comme $P_0 = I$ est la
matrice identit\'e, la solution peut s'\'ecrire sous la forme de la matrice
exponentielle 
\begin{equation}
 \label{ik06}
P_t = \e^{L t} = \sum_{n=0}^\infty \frac{(Lt)^n}{n!}\;. 
\end{equation} 
Pour le v\'erifier, il suffit de d\'eriver la s\'erie terme \`a terme (ce qui
est justifi\'e si la s\'erie converge absolument). Il est difficile de
trouver une forme plus explicite de $\e^{Lt}$ en toute g\'en\'eralit\'e, mais le
calcul peut \^etre fait dans certains cas particuliers.

\begin{example}[Cha\^ine de Markov en temps continu]
En appliquant la d\'efinition du g\'en\'erateur \`a~\eqref{tt05}, et en
utilisant le fait que $\sum_{j\neq i} q(i,j)=\lambda(1-p_{ii})$ puisque $P$ est
une matrice stochastique, on obtient 
\begin{equation}
 \label{ik10A}
L = \lambda(P-I)\;. 
\end{equation} 
En g\'en\'eral,
$\e^{AB}$ n'est pas \'egal \`a $\e^A\e^B$. Mais si $AB=BA$, on a effectivement 
$\e^{AB}=\e^A\e^B$. Comme $PI=IP$ et $\e^{-\lambda I t}=\e^{-\lambda t}I$, on
obtient 
\begin{equation}
 \label{ik10B}
P_t = \e^{-\lambda t}\e^{\lambda t P}\;, 
\end{equation} 
ce qui est \'equivalent \`a~\eqref{psm4}. 
\end{example}

\begin{example}[Processus de Poisson]
 \label{ex_poisson_markov2}
Pour le processus de Poisson, le g\'en\'erateur est donn\'e par 
\begin{equation}
 \label{ik07}
L(i,j) = 
\begin{cases}
-\lambda & \text{si $j=i$\;,} \\
\lambda & \text{si $j=i+1$\;,} \\
0 &\text{sinon\;.}
\end{cases} 
\end{equation}  
L'\'equation de Kolmogorov r\'etrograde s'\'ecrit donc 
\begin{equation}
 \label{ik08}
\dtot{}{t} P_t(i,j) = -\lambda P_t(i,j) + \lambda P_t(i+1,j)\;. 
\end{equation} 
Ecrivons le g\'en\'erateur sous la forme $L=\lambda(R-I)$ o\`u $I$ est la
matrice identit\'e, et $(R)_{ij}=1$ si $j=i+1$, $0$ sinon. Comme $IR=RI$, on
a
\begin{equation}
 \label{ik09}
\e^{Lt} = \e^{-\lambda t I}\e^{\lambda t R} 
= \e^{-\lambda t} \sum_{n\geqs0} \frac{(\lambda t)^n}{n!} R^n\;,
\end{equation}  
puisque $\e^{-\lambda t I}=\e^{-\lambda t}I$. Les \'el\'ements de la matrice
$R^n$ valent $1$ si $j-i=n$ et $0$ sinon. Par cons\'equent, en prenant
l'\'el\'ement de matrice $(i,j)$ de~\eqref{ik09}, on trouve 
\begin{equation}
 \label{ik10}
P_t(i,j) = 
\begin{cases}
\e^{-\lambda t} \dfrac{(\lambda t)^{j-i}}{(j-i)!}
& \text{si $j\geqs i$\;,} \\
0 & \text{sinon\;.}
\end{cases} 
\end{equation} 
On v\'erifie que c'est bien une solution du syst\`eme~\eqref{ik08}. On aurait
\'egalement pu obtenir ce r\'esultat directement, en observant que $P_t(i,j)$
est la probabilit\'e que le processus de Poisson admette $j-i$ points dans
l'intervalle $[0,t[$, qui est donn\'ee par la loi de Poisson.  
\end{example}

\begin{figure}
% \vspace{2mm}
%  \centerline{
%  \includegraphics*[clip=true,height=10mm]{figs/markov_twostate}
%  }
%  \figtext{
%  \writefig       6.0     0.9     $1$
%  \writefig       8.6     0.9     $2$
%  \writefig       7.3     1.6     $\lambda$
%  \writefig       7.3     0.2     $\mu$
%  }
\begin{center}
\begin{tikzpicture}[->,>=stealth',shorten >=2pt,shorten <=2pt,auto,node
distance=3.0cm, thick,main node/.style={circle,scale=0.7,minimum size=1.1cm,
fill=violet!20,draw,font=\sffamily\Large}]

  \node[main node] (1) {$1$};
  \node[main node] (2) [right of=0] {$2$};

  \path[every node/.style={font=\sffamily\small}]
    (1) edge [bend left,above] node {$\lambda$} (2)
    (2) edge [bend left,below] node {$\mu$} (1)
    ;
\end{tikzpicture}
\end{center}
\vspace{-5mm}
 \caption[]{Graphe du processus de sauts \`a deux \'etats.}
 \label{fig_markov_twostate}
\end{figure}
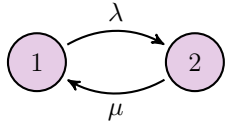

\begin{example}[Processus de sauts \`a deux \'etats]
Soit $\cX=\set{1,2}$, avec des taux de transition $q(1,2)=\lambda$ et
$q(2,1)=\mu$ (\figref{fig_markov_twostate}). Le g\'en\'erateur a alors la forme 
\begin{equation}
 \label{ik11}
L = 
\begin{pmatrix}
-\lambda & \lambda \\ \mu & -\mu
\end{pmatrix} \;.
\end{equation} 
Un calcul montre que $L^2=-(\lambda+\mu)L$, donc 
$L^n=(-(\lambda+\mu))^{n-1}L$ pour tout $n\geqs1$. Il suit que 
\begin{align}
\nonumber
P_t = 
\e^{Lt} &= I + \sum_{n\geqs1}\frac{t^n}{n!}(-(\lambda+\mu))^{n-1}L  \\
\nonumber
&= I - \frac{1}{\lambda+\mu} \bigpar{\e^{-(\lambda+\mu)t}-1}L
\\
&= \frac{1}{\lambda+\mu}
\begin{pmatrix}
\mu + \lambda\e^{-(\lambda+\mu)t} & \lambda - \lambda\e^{-(\lambda+\mu)t} \\
\mu - \mu\e^{-(\lambda+\mu)t} & \lambda + \mu\e^{-(\lambda+\mu)t}
\end{pmatrix}\;.
 \label{ik12}
\end{align} 
On remarquera en particulier que si $\lambda+\mu>0$, alors 
\begin{equation}
 \label{ik13}
\lim_{t\to\infty}  P_t =
\begin{pmatrix}
\myvrule{10pt}{15pt}{0pt}
\dfrac{\mu}{\lambda+\mu} & \dfrac{\lambda}{\lambda+\mu} \\
\dfrac{\mu}{\lambda+\mu} & \dfrac{\lambda}{\lambda+\mu}
\end{pmatrix}\;.
\end{equation} 
Cela signifie qu'asymptotiquement, le syst\`eme sera dans l'\'etat $1$ avec
probabilit\'e $\mu/(\lambda+\mu)$, et dans l'\'etat $2$ avec
probabilit\'e $\lambda/(\lambda+\mu)$, quel que soit l'\'etat initial.
\end{example}

%%%%%%%%%%%%%%%%%%%%%%%%%%%%%%%%%%%%%%%%%%%%%%%%%%%%%%%%%%%%%%%%%%%%%%%%%%%

\section{Distributions stationnaires}
\label{sec_mi}

\begin{definition}[Distribution stationnaire]
\label{def_mi01}
Une distribution de probabilit\'e $\pi$ sur $\cX$ est dite \defwd{stationnaire}
pour le processus de noyau de transition $P_t$ si 
\begin{equation}
 \label{mi01}
\pi P_t = \pi 
\qquad \forall t>0\;, 
\end{equation} 
c'est-\`a-dire, en composantes, 
\begin{equation}
 \label{mi02}
\sum_{i\in\cX} \pi(i) P_t(i,j) = \pi(j)
\qquad \forall j\in\cX,\; \forall t>0\;. 
\end{equation} 
\end{definition}

Comme en g\'en\'eral, on n'a pas acc\`es aux noyaux de transition, il est plus
utile de disposer d'un crit\`ere faisant intervenir le g\'en\'erateur. 

\begin{theorem}
\label{thm_mi1}
$\pi$ est une distribution stationnaire si et seulement si $\pi L=0$,
c'est-\`a-dire
 \begin{equation}
 \label{mi03}
\sum_{i\in\cX} \pi(i) L(i,j) = 0
\qquad \forall j\in\cX\;. 
\end{equation}
\end{theorem}
\begin{proof} \hfill
\begin{itemiz}
\item[$\Rightarrow:$]
Soit $\pi$ une distribution stationnaire. 
L'\'equation de Kolmogorov progressive s'\'ecrit en composantes 
\begin{equation}
 \label{mi04:1}
\dtot{}{t} P_t(i,j) = \sum_{k\in\cX} P_t(i,k)L(k,j)\;. 
\end{equation} 
En multipliant par $\pi(i)$ et en sommant sur $i$, on obtient 
\begin{equation}
 \label{mi04:2}
\sum_{i\in\cX} \pi(i) \dtot{}{t} P_t(i,j) 
= \sum_{i,k\in\cX} \pi(i)P_t(i,k)L(k,j)\;. 
\end{equation} 
Le membre de gauche peut s'\'ecrire 
\begin{equation}
 \label{mi04:3}
\dtot{}{t}  \sum_{i\in\cX} \pi(i) P_t(i,j)
= \dtot{}{t} \pi(j) = 0\;.
\end{equation} 
En effectuant la somme sur $i$ dans le membre de droite, on obtient 
\begin{equation}
 \label{mi04:4}
 \sum_{i,k\in\cX} \pi(i)P_t(i,k)L(k,j)
= \sum_{k\in\cX} \pi(k) L(k,j)\;.
\end{equation} 
L'\'egalit\'e de~\eqref{mi04:3} et~\eqref{mi04:4} est bien \'equivalente
\`a~\eqref{mi02}. 

\item[$\Leftarrow:$]
Supposons que $\pi L=0$. Alors on a par l'\'equation de Kolmogorov r\'etrograde 
\begin{align}
\nonumber
\dtot{}{t} \sum_{i\in\cX} \pi(i)P_t(i,j) 
&= \sum_{i\in\cX} \pi(i) \sum_{k\in\cX} L(i,k)P_t(k,j) \\
&= \sum_{k\in\cX} 
\underbrace{\biggpar{\sum_{i\in\cX}\pi(i)L(i,k)}}_{=0} P_t(k,j) = 0\;.
 \label{mi04:5}
\end{align}
Par cons\'equent, $\sum_i \pi(i)P_t(i,j)$ est constante, et donc \'egale \`a sa
valeur en $t=0$. Or cette valeur est $\pi(j)$ puisque $P_0(i,j)=\delta_{ij}$. 
Cela prouve que $\pi$ est stationnaire. 
\qed
\end{itemiz}
\renewcommand{\qed}{}
\end{proof}

\begin{example}[Cha\^ine de Markov en temps continu]
\label{ex_cmct4} 
Comme $L=\lambda(P-I)$, l'\'equation $\pi L=0$ est \'equivalente \`a $\pi
P=\pi$. La distribution $\pi$ est donc stationnaire pour la cha\^ine en temps
continu $X_t=Y_{N_t}$ si et seulement si elle est stationnaire pour la cha\^ine
en temps discret $Y_n$.
\end{example}

\begin{example}[Processus de Poisson]
\label{ex_markov_poisson3}
Le processus de Poisson n'admet pas de distribution stationnaire. En effet,
nous avons vu que le g\'en\'erateur \'etait donn\'e par $L=\lambda(R-I)$, o\`u
$(R)_{ij}=1$ si $j=i+1$ et $0$ sinon. La condition $\pi L=0$ est \'equivalente
\`a $R \pi=\pi$, donc \`a $\pi_{i+1}=\pi_i$ pour tout $i$. Or il n'existe pas
de distribution de probabilit\'e sur $\N$ dont tous les \'el\'ements seraient
\'egaux. L'expression~\eqref{ik08} des probabilit\'es de transition $P_t(i,j)$
montre d'ailleurs que celles-ci tendent vers $0$ lorsque $t\to\infty$. Cela
est d\^u au fait que la distribution \lq\lq s'en va vers l'infini\rq\rq.
\end{example}

\begin{example}[Processus \`a deux \'etats]
 \label{ex_twostate2} 
Dans le cas du processus \`a deux \'etats, l'\'equation $\pi L=0$ s'\'ecrit 
\begin{equation}
 \label{mi05}
\pi(1) (-\lambda) + \pi(2) \mu = 0\;.
\end{equation} 
Comme $\pi(1)+\pi(2)=1$, on a n\'ecessairement 
\begin{equation}
 \label{mi06}
\pi = \biggpar{\frac{\mu}{\lambda+\mu}, \frac{\lambda}{\lambda+\mu}}\;. 
\end{equation} 
Cela correspond bien \`a la valeur limite de $P_t$ trouv\'ee dans~\eqref{ik13}.
\end{example}

Dans le cas des cha\^ines de Markov en temps discret, nous avons vu que toute
distribution initiale converge vers une unique distribution stationnaire, \`a
condition que la cha\^ine soit irr\'eductible, ap\'eriodique et r\'ecurrente
positive. Dans le cas en temps continu, la condition d'ap\'eriodicit\'e n'est
plus n\'ecessaire. Cela est d\^u au caract\`ere al\'eatoire des intervalles de
temps entre transitions. 

On dira qu'un processus markovien de sauts est \defwd{irr\'eductible}\/ s'il est
possible de trouver, pour toute paire d'\'etats $(i,j)$, un chemin
$(i_0=i,i_1,\dots,i_n=j)$ tel que $q(i_k,i_{k+1})>0$ pour tout $k=0,\dots,n-1$. 

\begin{theorem}
 \label{thm_saut_stat}
Si le processus de sauts est irr\'eductible et admet une distribution
stationnaire $\pi$, alors 
\begin{equation}
 \label{mi07}
\lim_{t\to\infty} P_t(i,j) = \pi(j)
\qquad
\forall i,j\in\cX\;. 
\end{equation} 
De plus, si $r~:\cX\to\R$ est une fonction telle que
$\sum_i\pi(i)\abs{r(i)}<\infty$, alors 
\begin{equation}
 \label{mi07B}
\lim_{t\to\infty} \frac{1}{t} \int_0^t r(X_s)\6s
= \sum_{i} \pi(i)r(i) =: \E_\pi(r) \;. 
\end{equation} 
\end{theorem}

Nous admettrons ce r\'esultat. La relation~\eqref{mi07B} est un
analogue du Th\'eor\`eme \ref{thm_firred3}. En particulier, si
$r(i)=\delta_{ij}$, elle montre que le temps moyen pass\'e dans l'\'etat $j$
est \'egal \`a~$\pi(j)$. 

Nous avons vu que la distribution stationnaire d'une cha\^ine de Markov
r\'eversible \'etait plus facile \`a d\'eterminer, car il suffit de
r\'esoudre les conditions d'\'equilibre d\'etaill\'e. De mani\`ere tout \`a
fait analogue, on dira que le processus de sauts de taux de transition
$\set{q(i,j)}$ est \defwd{r\'eversible}\/ s'il existe une distribution $\pi$
telle que 
\begin{equation}
 \label{mi08}
\pi(i)q(i,j) = \pi(j)q(j,i)
\qquad
\forall i,j\in\cX\;. 
\end{equation} 

\begin{theorem}
 \label{thm_saut_reversible}
Si la condition  d'\'equilibre d\'etaill\'e~\eqref{mi08} est satisfaite, alors
$\pi$ est une distribution stationnaire. 
\end{theorem}
\begin{proof}
Sommant la condition sur tous les $i\neq j$, on obtient 
\begin{equation}
 \label{mi09}
\sum_{i\neq j}  \pi(i)q(i,j) = \pi(j) \sum_{i\neq j} q(j,i) =
\pi(j)\lambda(j) = -\pi(j)L(j,j)\;.
\end{equation} 
Mais ceci est \'equivalent \`a $\sum_i \pi(i)L(i,j)=0$. 
\end{proof}

\begin{figure}
%  \vspace{2mm}
%  \centerline{
%  \includegraphics*[clip=true,height=10mm]{figs/birth_death}
%  }
%  \figtext{
%  \writefig       2.4     0.9     $0$
%  \writefig       5.05    0.9     $1$
%  \writefig       7.65    0.9     $2$
%  \writefig       10.3    0.9     $3$
%  \writefig       3.7     1.6     $\lambda_0$
%  \writefig       6.3     1.6     $\lambda_1$
%  \writefig       8.9     1.6     $\lambda_2$
%  \writefig       11.5    1.6     $\lambda_3$
%  \writefig       3.7     0.2     $\mu_1$
%  \writefig       6.3     0.2     $\mu_2$
%  \writefig       8.9     0.2     $\mu_3$
%  \writefig       11.5    0.2     $\mu_4$
%  }
\begin{center}
\begin{tikzpicture}[->,>=stealth',shorten >=2pt,shorten <=2pt,auto,node
distance=3.0cm, thick,main node/.style={circle,scale=0.7,minimum size=1.1cm,
fill=violet!20,draw,font=\sffamily\Large}]

  \node[main node] (0) {$0$};
  \node[main node] (1) [right of=0] {$1$};
  \node[main node] (2) [right of=1] {$2$};
  \node[main node] (3) [right of=2] {$3$};
  \node[node distance=2cm] (4) [right of=3] {$\dots$};

  \path[every node/.style={font=\sffamily\small}]
    (0) edge [bend left, above] node {$\lambda_0$} (1)
    (1) edge [bend left, above] node {$\lambda_1$} (2)
    (2) edge [bend left, above] node {$\lambda_2$} (3)
    (3) edge [bend left, above] node {$\lambda_3$} (4)
    (4) edge [bend left, below] node {$\mu_4$} (3)
    (3) edge [bend left, below] node {$\mu_3$} (2)
    (2) edge [bend left, below] node {$\mu_2$} (1)
    (1) edge [bend left, below] node {$\mu_1$} (0)
   ;
\end{tikzpicture}
\end{center}
 \vspace{-5mm}
 \caption[]{Graphe repr\'esentant le processus de naissance et de mort.}
 \label{fig_birthdeath}
\end{figure}
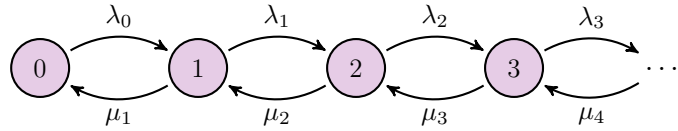

\begin{example}[Processus de naissance et de mort]
Soit $\cX=\set{0,\dots,N}$, et supposons que les seuls taux de transition non
nuls sont 
\begin{align}
\nonumber
q(n,n+1) &= \lambda_n 
&&
\text{pour $0\leqs n<N$\;,} \\
q(n,n-1) &= \mu_n 
&&
\text{pour $0< n\leqs N$\;.} 
\label{mi10} 
\end{align}
Ici $X_t$ repr\'esente par exemple le nombre d'individus d'une population, qui
naissent avec taux $\lambda_n$ et meurent avec taux $\mu_n$, pouvant d\'ependre
du nombre $n$ d'individus. Le graphe associ\'e est repr\'esent\'e dans
la~\figref{fig_birthdeath}. 

Supposons tous les taux strictement positifs. Dans ce cas, on peut satisfaire
la condition d'\'equilibre d\'etaill\'e en posant 
\begin{equation}
 \label{mi11}
\pi(n) = \frac{\lambda_{n-1}}{\mu_n} \pi(n-1)\;. 
\end{equation} 
Par r\'ecurrence, on obtient 
\begin{equation}
 \label{mi12}
\pi(n) = \frac{\lambda_{n-1}\lambda_{n-2}\dots\lambda_0}
{\mu_n\mu_{n-1}\dots\mu_1} \pi(0)\;,
\end{equation} 
et $\pi(0)$ est d\'etermin\'e par la condition de normalisation
$\sum_i\pi(i)=1$.
\end{example}

%%%%%%%%%%%%%%%%%%%%%%%%%%%%%%%%%%%%%%%%%%%%%%%%%%%%%%%%%%%%%%%%%%%%%%%%%%%

\section{Exercices}
\label{sec_exo_saut}

\begin{exercice}
\label{exo_saut01} 
On consid\`ere un processus de sauts markovien $X_t$ sur $\cX=\set{1,2,3,4}$,
de g\'en\'erateur infinit\'esimal 
\[
L = 
\begin{pmatrix}
-2 & 1 & 1 & 0 \\
2 & -5 & 1 & 2 \\
2 & 0 & -3 & 1 \\
0 & 0 & 1 & -1
\end{pmatrix}
\]

\begin{enum}
\item	Repr\'esenter le processus de sauts sous forme de graphe. 
\item	D\'eterminer la distribution stationnaire du processus.
\item	Le processus $X_t$ est-il irr\'eductible? 
\item	Le processus $X_t$ est-il r\'eversible? 
\end{enum}
\end{exercice}

\goodbreak

\begin{exercice}
\label{exo_saut02} 
Un homme d'affaires voyage entre Paris, Bordeaux et Marseille. Il passe dans
chaque ville un temps de loi exponentielle, de moyenne $1/4$ de mois pour Paris
et Bordeaux, et de $1/5$ de mois pour Marseille. S'il est \`a Paris, il va \`a
Bordeaux ou Marseille avec probabilit\'e $1/2$. S'il est \`a Bordeaux, il va
\`a Paris avec probabilit\'e $3/4$ et \`a Marseille avec probabilit\'e $1/4$.
Apr\`es avoir visit\'e Marseille, il retourne toujours \`a Paris.
\begin{enum}
\item	Donner le g\'en\'erateur du processus markovien de sauts d\'ecrivant
l'itin\'eraire de l'homme d'affaires.
\item	D\'eterminer la fraction de temps qu'il passe dans chaque ville.
\item	Combien de voyages fait-il en moyenne de Paris \`a Bordeaux par ann\'ee?
\end{enum}
\end{exercice}

\goodbreak

\begin{exercice}
\label{exo_saut03} 
Un petit magasin d'informatique peut avoir au plus trois ordinateurs en stock.
Des clients arrivent avec un taux de $2$ clients par semaine. Si au moins un
ordinateur est en stock, le client l'ach\`ete. S'il reste au plus un
ordinateur, le tenancier du magasin commande deux nouveaux ordinateurs, qui
sont livr\'es apr\`es un temps de loi exponentielle de moyenne $1$ semaine. 
\begin{enum}
\item	Donner le g\'en\'erateur du processus d\'ecrivant le nombre
d'ordinateurs en stock.
\item	D\'eterminer la distribution stationnaire.
\item	Quel est le taux de vente d'ordinateurs? 
\end{enum}
\end{exercice}

\goodbreak

\begin{exercice}
\label{exo_saut04} 
Un cycliste peut \^etre victime de deux types de pannes, qui se produisent
selon un processus de Poisson~:
\begin{itemiz}
\item	Il d\'eraille en moyenne une fois toutes les $10$ heures. 
Dans ce cas, il lui faut un temps exponentiel de moyenne $0.1$ heure pour
remettre la cha\^ine en place. 
\item	Il cr\`eve une chambre \`a air en moyenne une fois toutes les $25$
heures. Il lui faut alors un temps exponentiel de moyenne $1$ heure pour
r\'eparer la panne. 
\end{itemiz}

\begin{enum}
\item	D\'eterminer la distribution invariante. 
\item	Pendant quelle proportion du temps le cycliste peut-il rouler?
\item	G\'en\'eraliser la solution au cas de $N$ diff\'erentes pannes, chaque
panne se produisant avec un taux $\mu_i$ et n\'ecessitant un temps de
r\'eparation exponentiel de moyenne $1/\lambda_i$.
\end{enum}
\end{exercice}

\goodbreak

\begin{exercice}
\label{exo_saut05} 
Une mol\'ecule d'h\'emoglobine peut fixer soit une mol\'ecule d'oxyg\`ene,
soit une mol\'ecule de monoxyde de carbone. On suppose que ces mol\'ecules
arrivent selon des processus de Poisson de taux $\lambda_{\text{O}_2}$ et
$\lambda_{\text{CO}}$ et restent attach\'ees pendant des temps exponentiels de
taux $\mu_{\text{O}_2}$ et $\mu_{\text{CO}}$ respectivement. D\'eterminer la
fraction de temps pass\'ee dans chacun des trois \'etats~: h\'emoglobine seule,
h\'emoglobine et oxyg\`ene, h\'emoglobine et monoxyde de carbone.
\end{exercice}

\goodbreak

\begin{exercice}
\label{exo_saut06} 
Des voitures arrivent dans une station service avec un taux de $20$ voitures
par heure. La station comporte une seule pompe. Si un conducteur trouve la
pompe libre, il s'arr\^ete pour faire le plein. L'op\'eration lui prend un
temps exponentiel de moyenne $6$ minutes. Si la pompe est occup\'ee mais
qu'aucune voiture n'attend, le conducteur attend que la pompe se lib\`ere. Si
par contre il y a d\'ej\`a deux voitures sur place, l'une dont on fait le plein
et l'autre qui attend, la voiture qui arrive repart aussit\^ot. 
\begin{enum}
\item	Formuler le probl\`eme comme un processus markovien de sauts et trouver
sa distribution stationnaire. 
\item	D\'eterminer le nombre moyen de clients servis par heure. 
\end{enum}
\end{exercice}

\goodbreak

\begin{exercice}
\label{exo_saut08}
On mod\'elise la d\'esint\'egration radioactive de $N$ atomes par un processus
de sauts markovien $\set{X_t}_{t\geqs0}$ sur $\cX=\set{0,1,\dots,N}$ de taux
$q(n,n-1)=\mu$, $X_t$ d\'esignant le nombre d'atomes non d\'esint\'egr\'es au
temps $t$. 

\begin{enum}
\item	On consid\`ere le cas $N=1$~:

%\bigskip

% \centerline{
%  \includegraphics*[clip=true,height=8mm]{figs/exam_IIa}
%  \figtext{
%  \writefig       -3.05     0.3     $0$
%  \writefig       -0.75     0.3     $1$
%  \writefig       -1.85     0.55    $\mu$
%  }
%  }
 
\begin{center}
\begin{tikzpicture}[->,>=stealth',shorten >=2pt,shorten <=2pt,auto,node
distance=3.0cm, thick,main node/.style={circle,scale=0.7,minimum size=1.1cm,
fill=violet!20,draw,font=\sffamily\Large}]

  \node[main node] (0) {$0$};
  \node[main node] (1) [right of=0] {$1$};

  \path[every node/.style={font=\sffamily\small}]
    (1) edge [left, above] node {$\mu$} (0)
   ;
\end{tikzpicture}
\end{center}

D\'eterminer le g\'en\'erateur $L$. Calculer $L^2$, puis $L^n$ pour tout $n$.
En d\'eduire le noyau de transition $P_t$. 

\item	On consid\`ere maintenant le cas $N=2$~:

% \bigskip
% 
% \centerline{
%  \includegraphics*[clip=true,height=8mm]{figs/exam_IIb}
%  \figtext{
%  \writefig       -5.35     0.3     $0$
%  \writefig       -3.05     0.3     $1$
%  \writefig       -0.75     0.3     $2$
%  \writefig       -1.85     0.55    $\mu$
%  \writefig       -4.15     0.55    $\mu$
%  }
%  }
 
\begin{center}
\begin{tikzpicture}[->,>=stealth',shorten >=2pt,shorten <=2pt,auto,node
distance=3.0cm, thick,main node/.style={circle,scale=0.7,minimum size=1.1cm,
fill=violet!20,draw,font=\sffamily\Large}]

  \node[main node] (0) {$0$};
  \node[main node] (1) [right of=0] {$1$};
  \node[main node] (2) [right of=1] {$2$};

  \path[every node/.style={font=\sffamily\small}]
    (2) edge [left, above] node {$\mu$} (1)
    (1) edge [left, above] node {$\mu$} (0)
   ;
\end{tikzpicture}
\end{center}

D\'eterminer le g\'en\'erateur $L$ et \'ecrire les \'equations de Kolmogorov
progressives. R\'esou\-dre ces \'equations pour la condition initiale $X_0=2$,
c'est-\`a-dire trouver $P_t(2,j)$ pour $j=2, 1, 0$. 

{\it Indication~:} La solution de l'\'equation diff\'erentielle 
$\dtot xt = -\mu x + f(t)$ s'\'ecrit 
\[
x(t) = x(0)\e^{-\mu t} + \int_0^t \e^{-\mu(t-s)} f(s) \6s \;.
\]

\item	Par le m\^eme proc\'ed\'e, calculer $P_t(N,j)$ pour $j=N,N-1,\dots,0$
pour $N$ quelconque. 

\item	Calculer 
\[
\lim_{N\to\infty} \expec{Y_t} 
\]
o\`u $Y_t=N-X_t$ est le nombre d'atomes d\'esint\'egr\'es au temps $t$, s'il
y a $N$ atomes au temps~$0$. 

\end{enum}
\end{exercice}

\goodbreak

\begin{exercice}
\label{exo_saut07} 
On consid\`ere une cha\^ine de mort pure, de taux de mort $q(n,n-1)=\mu$ pour
tout $n\geqs1$. D\'eterminer les noyaux de transition $P_t(i,j)$. 
\end{exercice}

\goodbreak

\begin{exercice}
\label{exo_saut09}
On consid\`ere le processus de sauts markovien sur $\cX=\set{1,2,3}$ dont le
graphe est le suivant:

% \bigskip
% 
% \centerline{
%  \includegraphics*[clip=true,height=40mm]{figs/jump_dec12a}
%  }

\begin{center}
\begin{tikzpicture}[->,>=stealth',shorten >=2pt,shorten <=2pt,auto,node
distance=3.0cm, thick,main node/.style={circle,scale=0.7,minimum size=1.1cm,
fill=violet!20,draw,font=\sffamily\Large}]

  \node[main node] (1) at (0,0) {$1$};
  \node[main node] (2) at (1.5,-2.25) {$2$};
  \node[main node] (3) at (-1.5,-2.25){$3$};

  \path[every node/.style={font=\sffamily\small}]
    (1) edge [above right] node {$\lambda$} (2)
    (2) edge [below] node {$\lambda$} (3)
    (3) edge [above left] node {$\lambda$} (1)
    ;
\end{tikzpicture}
\end{center}

%\medskip

\begin{enum}
\item 	Donner le g\'en\'erateur $L$ de ce processus.
\item 	D\'eterminer la distribution stationnaire $\pi$ du processus.
\item 	Soit $R$ la matrice telle que $L = -\lambda I + \lambda R$. 
Calculer $R^2$, $R^3$, puis $R^n$ pour tout $n\in\N$. 
\item 	En d\'eduire $\e^{\lambda t R}$, puis le noyau de transition $P_t$ pour
tout $t$. 
\goodbreak
\item 	On consid\`ere maintenant le processus de sauts sur
$\cX=\set{1,\dots,N}$ dont le graphe est le suivant:
% \bigskip
% 
% \centerline{
%  \includegraphics*[clip=true,height=50mm]{figs/jump_dec12b}
%  }
 
\begin{center}
\begin{tikzpicture}[->,>=stealth',shorten >=2pt,shorten <=2pt,auto,node
distance=4.0cm, thick,main node/.style={circle,scale=0.7,minimum size=1.2cm,
fill=violet!20,draw,font=\sffamily\Large}]

  \node (0) {};
  \node[main node] (1) [above of=0] {1};
  \node[main node] (2) [above right of=0] {2};
  \node[main node] (3) [right of=0] {3};
  \node[main node] (N) [above left of=0] {$N$};
  \node[main node] (N-1) [left of=0] {{\small $N-1$}};
  \node[node distance=2.6cm] (d1) [below left of=0] {$\dots$};
  \node[node distance=2.6cm] (d2) [below right of=0] {$\dots$};

  \path[every node/.style={font=\sffamily\small}]
    (1) edge [above right] node {$\lambda$} (2)
    (2) edge [right] node {$\lambda$} (3)
    (3) edge [right] node {$\lambda$} (d2)
    (d1) edge [left] node {$\lambda$} (N-1)
    (N-1) edge [left] node {$\lambda$} (N)
    (N) edge [above left] node {$\lambda$} (1)
    ;

\end{tikzpicture}
\end{center}

% \bigskip
Donner le g\'en\'erateur $L$ de ce processus.
\item 	D\'eterminer la distribution stationnaire $\pi$ du processus.
\item 	Indiquer la forme du noyau de transition $P_t$ pour
tout $t$. 
\end{enum}

\noindent
{\it Indication:} On pourra admettre les d\'eveloppements limit\'es suivants.

\noindent
Soit $\omega=\e^{2\icx\pi/3}=-\frac12+\frac{\sqrt3}{2}\icx$ et 
\[
f(x) = \frac13 \Bigpar{\e^x + \e^{\omega x} + \e^{\bar\omega x}}
= \frac13 \Bigpar{\e^x + 2 \e^{-x/2}\cos\bigpar{\sqrt{3}x/2}}\;.
\]
Alors
\begin{align*}
f(x) &= 1 + \frac{1}{3!}x^3 + \frac{1}{6!}x^6 + \dots + \frac{1}{(3n)!}x^{3n} +
\dots \\ 
f'(x) &= \frac{1}{2!}x^2 + \frac{1}{5!}x^5 + \dots + \frac{1}{(3n-1)!}x^{3n-1} +
\dots \\ 
f''(x) &= x + \frac{1}{4!}x^4 + \dots + \frac{1}{(3n-2)!}x^{3n-2}
+ \dots \;. 
\end{align*}
\end{exercice}

\goodbreak

\begin{exercice}
\label{exo_saut10}

Soient $\lambda, \mu>0$. 
On consid\`ere le processus de sauts markovien sur $\cX=\set{0,1,\dots,N}$ dont
le graphe est le suivant:

%\bigskip

\begin{center}
\begin{tikzpicture}[->,>=stealth',shorten >=2pt,shorten <=2pt,auto,node
distance=5.0cm, thick,main node/.style={circle,scale=0.7,minimum size=1.1cm,
fill=violet!20,draw,font=\sffamily\Large}]

  \node[main node] (0) {0};
  \node[main node] (1) [above of=0] {1};
  \node[main node] (2) [above right of=0] {2};
  \node[main node] (3) [right of=0] {3};
  \node[main node] (N) [above left of=0] {$N$};
  \node[node distance=4.0cm] (d1) [left of=0] {$\dots$};
  \node[node distance=4.0cm] (d2) [below right of=0] {$\dots$};

  \draw[every node/.style={font=\sffamily\small}]
    (0) to[out=100,in=-100] node[left] {$\lambda$} (1);
  \draw[every node/.style={font=\sffamily\small}]
    (1) to[out=-80,in=80] node[right] {$\mu$} (0);
    
  \draw[every node/.style={font=\sffamily\small}]
    (0) to[out=55,in=-145] node[above left] {$\lambda$} (2);
  \draw[every node/.style={font=\sffamily\small}]
    (2) to[out=-125,in=35] node[below right] {$\mu$} (0);
    
  \draw[every node/.style={font=\sffamily\small}]
    (0) to[out=10,in=170] node[above] {$\lambda$} (3);
  \draw[every node/.style={font=\sffamily\small}]
    (3) to[out=-170,in=-10] node[below] {$\mu$} (0);
    
  \draw[every node/.style={font=\sffamily\small}]
    (0) to[out=145,in=-55] node[below left] {$\lambda$} (N);
  \draw[every node/.style={font=\sffamily\small}]
    (N) to[out=-35,in=125] node[above right] {$\mu$} (0);
    
\end{tikzpicture}
\end{center}
%\medskip

\begin{enum}
\item 	Donner le g\'en\'erateur $L$ de ce processus.
\item 	D\'eterminer la distribution stationnaire $\pi$ du processus.
\item 	Ce processus est-il r\'eversible?
\item 	On consid\`ere le cas $N=1$. 
Calculer $L^n$ pour tout $n\in\N$, et en d\'eduire le noyau de
transition $P_t$ pour tout $t\geqs0$. 

\item 	On consid\`ere le cas $N>1$, avec $\lambda=\mu$. Pour tout $j\in\cX$  on
pose 

\[
 Q_t(j) = \frac{1}{N} \sum_{k=1}^N P_t(k,j)\;.
\]
A l'aide des \'equations de Kolmogorov, exprimer  
\[
\dtot{}{t} P_t(0,j) 
\qquad\text{et}\qquad
\dtot{}{t} Q_t(j)
\]
en fonction de $P_t(0,j)$ et $Q_t(j)$. A l'aide du r\'esultat du point 4., 
en d\'eduire $P_t(0,j)$ et $Q_t(j)$ pour tout $j\in\cX$.
On distinguera les cas $j=0$ et $j>0$. 

\item 	Soit $i\in\set{1,\dots,N}$. Calculer 
\[
\dtot{}{t} P_t(i,j)
\]
et en d\'eduire $P_t(i,j)$ pour tout $j\in\cX$. 
On distinguera les cas $j=0$ et $j>0$. 

\item 	Que vaut 
\[
 \lim_{t\to\infty} P_t\;?
\]

\end{enum}
\end{exercice}

%%%%%%%%%%%%%%%%%%%%%%%%%%%%%%%%%%%%%%%%%%%%%%%%%%%%%%%%%%%%%%%%%%%%%%%%%%%

\chapter{Files d'attente}
\label{chap_tfa}

%%%%%%%%%%%%%%%%%%%%%%%%%%%%%%%%%%%%%%%%%%%%%%%%%%%%%%%%%%%%%%%%%%%%%%%%%%%

\section{Classification et notation de Kendall}
\label{sec_faK}

La th\'eorie des files d'attente (\lq\lq queueing theory\rq\rq\ en anglais)
permet de mod\'eliser des situations o\`u des clients arrivent \`a des temps
al\'eatoires \`a un serveur. Le serveur peut \^etre le guichet d'une banque,
d'un bureau de poste, la caisse d'un supermarch\'e, une station service par
exemple, mais \'egalement un serveur informatique dont les clients sont des
t\^aches que le serveur doit traiter. Le temps n\'ecessaire \`a servir chaque
client est \'egalement suppos\'e al\'eatoire. 

Les questions auxquelles l'on voudrait r\'epondre sont par exemple~:
\begin{itemiz}
\item	Quelle est la dur\'ee d'attente moyenne d'un client arrivant dans la
file?
\item	Pendant quelle fraction du temps le serveur est-il occup\'e?
\item	Quelle est la distribution de probabilit\'e de la longueur de la file? 
\item	Selon quel processus les clients quittent-ils le syst\`eme, une fois
servis?
\end{itemiz}

Il existe toute une s\'erie de mod\`eles de files d'attente, qui se distinguent
par la loi des temps d'arriv\'ee, la loi des temps de service, le nombre de
serveurs, l'\'eventuelle longueur maximale de la file, l'ordre de passage des
clients. 

La notation de Kendall permet de sp\'ecifier le mod\`ele de mani\`ere
compacte. Le plus souvent, cette notation se compose de trois symboles~:
\begin{equation}
 \label{faK1}
\text{A/B/}s\;, 
\end{equation} 
o\`u A d\'esigne la loi des intervalles de temps entre arriv\'ees des clients, B
d\'esigne la loi des temps de service, et $s$ est le nombre de serveurs. Les
valeurs les plus usuelles de A et B sont les suivantes~:
\begin{itemiz}
\item	{\bf M (Markov)~:} Loi exponentielle. En effet, nous avons vu que cette
loi implique la propri\'et\'e de Markov. Il s'agit du cas le plus simple \`a
analyser.
\item	{\bf D (D\'eterministe)~:} Temps constant, c'est-\`a-dire que les
arriv\'ees des clients sont r\'eguli\`erement espac\'ees, respectivement le
temps de service est le m\^eme pour tous les clients.
\item	{\bf E (Erlang)~:} Loi dite d'Erlang, qui est en fait une loi Gamma
(loi d'une somme de variables exponentielles).
\item	{\bf G (G\'en\'erale)~:} Loi arbitraire. Ce symbole s'utilise quand on
d\'erive des propri\'et\'es ne d\'ependant pas de la loi particuli\`ere
consid\'er\'ee. 
\end{itemiz}

Si la longueur de la file est limit\'ee \`a une valeur finie $N$, on utilise la
notation 
\begin{equation}
 \label{faK02}
\text{A/B/}s\text{/}N\;. 
\end{equation} 
Par d\'efaut, on suppose donc que $N=\infty$. 

Finalement, on peut aussi sp\'ecifier en dernier l'ordre dans lequel les clients
sont servis, la valeur par d\'efaut \'etant FIFO (first in, first out), ce qui
signifie que le premier client arriv\'e est aussi le premier servi. Une
alternative est LIFO (last in, first out). Dans ce cas, les nouveaux arrivants
se placent en t\^ete de la file, et seront donc servis d\`es que le serveur se
lib\`ere, \`a moins que d'autres clients n'arrivent entretemps.

%%%%%%%%%%%%%%%%%%%%%%%%%%%%%%%%%%%%%%%%%%%%%%%%%%%%%%%%%%%%%%%%%%%%%%%%%%%

\section{Cas markoviens~: Files d'attente M/M/$s$}
\label{sec_faMM}

Les files M/M/$s$ sont les plus simples \`a analyser, puisque le caract\`ere
markovien des temps d'arriv\'ee et des temps de service implique que la
longueur de la file est un processus markovien de sauts. On peut donc appliquer
les outils d\'evelopp\'es au chapitre pr\'ec\'edent. 

\begin{example}[File M/M/1]
\label{ex_MM1} 
Le cas le plus simple se pr\'esente pour un serveur unique, dont les clients
arrivent selon un processus de Poisson de param\`etre $\lambda$, et dont le
temps de service suit une loi exponentielle de param\`etre $\mu$. L'hypoth\`ese
de temps d'arriv\'ee poissoniens est relativement r\'ealiste, d\`es lors qu'on
suppose que les clients proviennent d'un grand r\'eservoir d'individus
ind\'ependants. Celle des temps de service exponentiels est beaucoup plus
discutable. On la fait surtout parce qu'elle permet des calculs plus
explicites. 

Soit $N_t$ la longueur de la file au temps $t$. Elle \'evolue selon un
processus de sauts de taux de transition 
\begin{align}
\nonumber
q(n,n+1) &= \lambda
\qquad \text{si $n\geqs 0$\;,}\\
q(n,n-1) &= \mu
\qquad \text{si $n\geqs 1$\;.}
\label{MM01} 
\end{align}
En d'autres termes, $N_t$ suit un processus de naissance et de mort, chaque
arriv\'ee d'un nouveau client \'etant assimil\'e \`a une naissance, et chaque
d\'epart d'un client servi \'etant assimil\'e \`a une mort (\figref{fig_MM1}).

La relation~\eqref{mi12} montre que si le processus admet une distribution
invariante $\pi$, alors celle-ci satisfait 
\begin{equation}
 \label{MM02}
\pi(n) = \biggpar{\frac{\lambda}{\mu}}^n \pi(0)\;. 
\end{equation} 
La constante $\pi(0)$ se d\'etermine en exigeant que la somme des $\pi(n)$ soit
\'egale \`a $1$. Cette somme est une s\'erie g\'eom\'etrique, et on obtient 
\begin{equation}
 \label{MM03}
\pi(n) = \biggpar{1-\frac{\lambda}{\mu}} \biggpar{\frac{\lambda}{\mu}}^n
\qquad
\text{si $\lambda < \mu$\;.} 
\end{equation} 
La distribution invariante suit donc une loi g\'eom\'etrique (d\'ecal\'ee en
$0$). Si $\lambda\geqs\mu$, la s\'erie diverge, et il n'existe pas de
distribution stationnaire. Dans ce cas, le taux d'arriv\'ee de nouveaux clients
d\'epasse la capacit\'e de traitement du serveur, et la longueur de la file
cro\^it ind\'efiniment. 

Supposons $\lambda<\mu$, et que la file d'attente a atteint l'\'equilibre. On
peut alors calculer diff\'erentes quantit\'es d'int\'er\^et.

\begin{itemiz}
\item	La probabilit\'e que le serveur soit occup\'e est donn\'ee par 
\begin{equation}
 \label{MM04}
\bigprobin{\pi}{N_t>0} = 
1 - \pi(0) = \frac{\lambda}{\mu}\;. 
\end{equation} 
C'est \'egalement la fraction du temps pendant laquelle le serveur est
occup\'e. 

\item	La longueur moyenne de la file est donn\'ee par 
\begin{equation}
 \label{MM05}
\expecin{\pi}{N_t} = \sum_{n=0}^\infty n\pi(n) 
= \frac{\lambda/\mu}{1-\lambda/\mu}\;, 
\end{equation} 
o\`u l'on a utilis\'e la valeur de l'esp\'erance d'une loi g\'eom\'etrique. On
notera que ce nombre diverge lorsque $\lambda/\mu$ tend vers $1$. 

\item	Soit $W$ le temps d'attente d'un client avant d'\^etre servi. Sa loi se
calcule en distinguant deux cas. Si la file est vide \`a son arriv\'ee, alors
le temps d'attente est nul, et on a 
\begin{equation}
 \label{MM06}
\bigprobin{\pi}{W=0} = \bigprobin{\pi}{N_t=0} = 1 - \frac{\lambda}{\mu}\;. 
\end{equation} 
En revanche, si la file a une longueur $n>0$ \`a l'arriv\'ee du client, celui-ci
devra attendre que les $n$ clients le pr\'ec\'edant dans la file soient servis.
Le temps d'attente est donc la somme de $n$ variables exponentielles
ind\'ependantes de param\`etre $\mu$, qui suit une loi Gamma de param\`etres
$(n,\mu)$. La densit\'e de $W$ est donn\'ee par 
\begin{align}
\nonumber
f(t) 
&= \sum_{n=1}^\infty \pi(n) \e^{-\mu t} \frac{\mu^n t^{n-1}}{(n-1)!} \\
\nonumber
&= \biggpar{1-\frac{\lambda}{\mu}} \e^{-\mu t}
\lambda \sum_{n=1}^\infty \frac{\lambda^{n-1}t^{n-1}}{(n-1)!} \\
&= \frac{\lambda}{\mu} (\mu-\lambda) \e^{-(\mu-\lambda)t}\;.
 \label{MM07}
\end{align} 
En d'autres termes, conditionnellement \`a $W>0$, $W$ suit une loi
exponentielle de param\`etre $\mu-\lambda$. 

\begin{figure}
%  \vspace{2mm}
%  \centerline{
%  \includegraphics*[clip=true,height=10mm]{figs/birth_death}
%  }
%  \figtext{
%  \writefig       2.4      0.9     $0$
%  \writefig       5.05     0.9     $1$
%  \writefig       7.65     0.9     $2$
%  \writefig       10.3     0.9     $3$
%  \writefig       3.75     1.6     $\lambda$
%  \writefig       6.35     1.6     $\lambda$
%  \writefig       8.95     1.6     $\lambda$
%  \writefig       11.65    1.6     $\lambda$
%  \writefig       3.75     0.2     $\mu$
%  \writefig       6.35     0.2     $\mu$
%  \writefig       8.95     0.2     $\mu$
%  \writefig       11.65    0.2     $\mu$
%  \writefig       13.0     0.9     $\dots$
%  }
\begin{center}
\begin{tikzpicture}[->,>=stealth',shorten >=2pt,shorten <=2pt,auto,node
distance=3.0cm, thick,main node/.style={circle,scale=0.7,minimum size=1.1cm,
fill=violet!20,draw,font=\sffamily\Large}]

  \node[main node] (0) {$0$};
  \node[main node] (1) [right of=0] {$1$};
  \node[main node] (2) [right of=1] {$2$};
  \node[main node] (3) [right of=2] {$3$};
  \node[node distance=2cm] (4) [right of=3] {$\dots$};

  \path[every node/.style={font=\sffamily\small}]
    (0) edge [bend left, above] node {$\lambda$} (1)
    (1) edge [bend left, above] node {$\lambda$} (2)
    (2) edge [bend left, above] node {$\lambda$} (3)
    (3) edge [bend left, above] node {$\lambda$} (4)
    (4) edge [bend left, below] node {$\mu$} (3)
    (3) edge [bend left, below] node {$\mu$} (2)
    (2) edge [bend left, below] node {$\mu$} (1)
    (1) edge [bend left, below] node {$\mu$} (0)
   ;
\end{tikzpicture}
\end{center}
 \vspace{-5mm}
 \caption[]{Graphe associ\'e \`a la file d'attente M/M/1.}
 \label{fig_MM1}
\end{figure}

\item	Le temps d'attente moyen avant d'\^etre servi est donn\'e par 
\begin{equation}
 \label{MM08}
\expecin{\pi}{W} = \int_0^\infty tf(t)\6t = \frac{\lambda}{\mu(\mu-\lambda)}\;. 
\end{equation} 
Le temps d'attente total moyen, en comptant le temps de service, vaut donc 
\begin{equation}
 \label{MM09}
\expecin{\pi}{W} + \frac{1}{\mu} = \frac{1}{\mu-\lambda}\;. 
\end{equation} 
\end{itemiz}
\end{example}

\begin{example}[File M/M/$1$/$N$]
\label{ex_MM1N} 
Consid\'erons maintenant le cas o\`u la longueur de la file est limit\'ee \`a
$N$, c'est-\`a-dire que si un client arrive en trouvant une file de longueur
$N$, alors il repart sans rejoindre la file. Dans ce cas, le syst\`eme est
d\'ecrit par un processus markovien de saut de taux de transition 
\begin{align}
\nonumber
q(n,n+1) &= \lambda
\qquad \text{si $0\leqs n<N$\;,}\\
q(n,n-1) &= \mu
\qquad \text{si $0<n\leqs N$\;.}
\label{MM10} 
\end{align}
C'est encore un processus de naissance et de mort (\figref{fig_MM1N}), et la
relation~\eqref{MM02} reste valable pour $1\leqs n\leqs N$. La seule
diff\'erence est la normalisation, et en effectuant la somme on obtient comme
distribution stationnaire 
\begin{equation}
 \label{MM11}
\pi(n) = 
\begin{cases}
\myvrule{10pt}{15pt}{0pt}
\dfrac{1-\lambda/\mu}{1-(\lambda/\mu)^{N+1}} \biggpar{\dfrac{\lambda}{\mu}}^n
& \text{si $\lambda\neq\mu$\;,} \\
\dfrac{1}{N+1}
&\text{si $\lambda=\mu$\;.}
\end{cases} 
\end{equation} 
Contrairement au cas d'une file de longueur arbitraire, la distribution
stationnaire existe toujours. Si $\lambda<\mu$ et qu'on fait tendre $N$ vers
l'infini, on retrouve la distribution stationnaire de la file M/M/1.
\end{example}

\begin{figure}
%  \vspace{2mm}
%  \centerline{
%  \includegraphics*[clip=true,height=10mm]{figs/MM1N}
%  }
%  \figtext{
%  \writefig       1.95     0.9     $0$
%  \writefig       4.55     0.9     $1$
%  \writefig       7.2      0.9     $2$
%  \writefig       12.55    0.9     $N$
%  \writefig       3.25     1.6     $\lambda$
%  \writefig       5.85     1.6     $\lambda$
%  \writefig       8.45     1.6     $\lambda$
%  \writefig       11.15    1.6     $\lambda$
%  \writefig       3.25     0.2     $\mu$
%  \writefig       5.85     0.2     $\mu$
%  \writefig       8.45     0.2     $\mu$
%  \writefig       11.15    0.2     $\mu$
%  \writefig       9.75     0.9     $\dots$
%  }
\begin{center}
\begin{tikzpicture}[->,>=stealth',shorten >=2pt,shorten <=2pt,auto,node
distance=3.0cm, thick,main node/.style={circle,scale=0.7,minimum size=1.1cm,
fill=violet!20,draw,font=\sffamily\Large}]

  \node[main node] (0) {$0$};
  \node[main node] (1) [right of=0] {$1$};
  \node[main node] (2) [right of=1] {$2$};
  \node[main node] (3) [right of=2] {$3$};
  \node[node distance=2cm] (4) [right of=3] {$\dots$};
  \node[main node] (N) [right of=4] {$N$};

  \path[every node/.style={font=\sffamily\small}]
    (0) edge [bend left, above] node {$\lambda$} (1)
    (1) edge [bend left, above] node {$\lambda$} (2)
    (2) edge [bend left, above] node {$\lambda$} (3)
    (3) edge [bend left, above] node {$\lambda$} (4)
    (4) edge [bend left, above] node {$\lambda$} (N)
    (N) edge [bend left, below] node {$\mu$} (4)
    (4) edge [bend left, below] node {$\mu$} (3)
    (3) edge [bend left, below] node {$\mu$} (2)
    (2) edge [bend left, below] node {$\mu$} (1)
    (1) edge [bend left, below] node {$\mu$} (0)
   ;
\end{tikzpicture}
\end{center}
 \vspace{-5mm}
 \caption[]{Graphe associ\'e \`a la file d'attente M/M/1/$N$.}
 \label{fig_MM1N}
\end{figure}
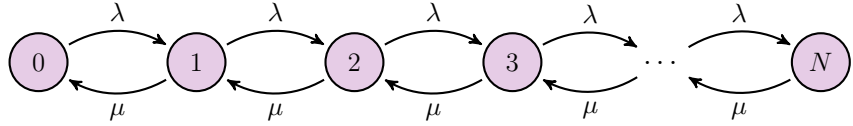

\begin{example}[File M/M/$s$]
\label{ex_MMs} 
Supposons que les clients forment une seule file d'attente, mais qu'il existe
un nombre $s$ de serveurs en parall\`ele. D\`es qu'un serveur se lib\`ere, le
client en t\^ete de la file le rejoint. Dans ce cas, on a toujours
$q(n,n+1)=\lambda$, mais pour les taux $q(n,n-1)$ il faut distinguer deux cas. 
Si la longueur $n$ de la file est inf\'erieure \`a $s$, alors seuls $n$
serveurs sont actifs, et les d\'eparts ont lieu au taux $n\mu$. En revanche, si
$n$ est sup\'erieur ou \'egal \`a $s$, tous les serveurs sont occup\'es, et le
taux des d\'eparts est $s\mu$. On a donc 
\begin{equation}
q(n,n-1) = 
\begin{cases}
n\mu
&\text{si $n< s$\;,}\\
s\mu
&\text{si $n\geqs s$\;.}
\end{cases}
\label{MM12} 
\end{equation}
On a encore affaire \`a un processus de naissance et de mort (\figref{fig_MMs}),
et la distribution stationnaire, si elle existe, satisfait 
\begin{equation}
 \label{MM13}
\pi(n) = 
\begin{cases}
\displaystyle
\prod_{k=1}^n \biggpar{\frac{\lambda}{k\mu}} \pi(0)
= \frac{1}{n!} \biggpar{\frac{\lambda}{\mu}}^n \pi(0) 
& \text{si $n\leqs s\;,$} \\
\displaystyle
\prod_{k=1}^s \biggpar{\frac{\lambda}{k\mu}} 
\prod_{k=s+1}^n \biggpar{\frac{\lambda}{s\mu}}
\pi(0)
= \frac{1}{s!s^{n-s}} \biggpar{\frac{\lambda}{\mu}}^n \pi(0) 
& \text{si $n>s\;.$}
\end{cases} 
\end{equation}
Si $\lambda < \mu s$, on peut trouver $\pi(0)$ tel que $\pi$ soit une
distribution de probabilit\'e, et alors il existe un \'etat stationnaire.
Sinon, la file ne poss\`ede pas de distribution stationnaire.  

On peut calculer les m\^emes quantit\'es que pour la file M/M/1, qui
ont toutefois des expressions plus compliqu\'ees. Une exception est le
nombre moyen de serveurs occup\'es, qui est donn\'e par
\begin{equation}
 \label{MM13A}
\sum_{n=1}^s n\pi(n) + \sum_{n=s+1}^\infty s\pi(n)\;. 
\end{equation}  
En observant sur l'expression~\eqref{MM13} de la distribution stationnaire que 
\begin{equation}
 \label{MM13B}
\frac{\lambda}{\mu} \pi(n-1) = 
\begin{cases}
n\pi(n) & \text{si $n\leqs s$\;,}\\
s\pi(n) & \text{si $n>s$\;,}
\end{cases}
\end{equation} 
on peut r\'ecrire~\eqref{MM13A} sous la forme 
\begin{equation}
 \label{M13C}
\sum_{n=1}^\infty \frac{\lambda}{\mu} \pi(n-1)
= \frac{\lambda}{\mu} \sum_{n=0}^\infty \pi(n) = \frac{\lambda}{\mu}\;. 
\end{equation} 
Puisqu'il n'existe une distribution stationnaire que sous la
condition $\lambda/\mu<s$, on obtient bien un nombre moyen de serveurs
occup\'es inf\'erieur \`a $s$. La proportion de serveurs occup\'es est \'egale
\`a $\lambda/\mu s$.  
\end{example}

\begin{figure}
%  \vspace{2mm}
%  \centerline{
%  \includegraphics*[clip=true,height=10mm]{figs/MMs}
%  }
%  \figtext{
%  \writefig       1.05     0.9     $0$
%  \writefig       3.65     0.9     $1$
%  \writefig       6.3      0.9     $2$
%  \writefig       11.75    0.9     $s$
%  \writefig       2.35     1.6     $\lambda$
%  \writefig       4.95     1.6     $\lambda$
%  \writefig       7.55     1.6     $\lambda$
%  \writefig       10.35    1.6     $\lambda$
%  \writefig       12.95    1.6     $\lambda$
%  \writefig       2.35     0.2     $\mu$
%  \writefig       4.85     0.2     $2\mu$
%  \writefig       7.45     0.2     $3\mu$
%  \writefig       10.25    0.2     $s\mu$
%  \writefig       12.85    0.2     $s\mu$
%  \writefig       8.9      0.9     $\dots$
%  \writefig       14.2     0.9     $\dots$
%  }
\begin{center}
\begin{tikzpicture}[->,>=stealth',shorten >=2pt,shorten <=2pt,auto,node
distance=3.0cm, thick,main node/.style={circle,scale=0.7,minimum size=1.1cm,
fill=violet!20,draw,font=\sffamily\Large}]

  \node[main node] (0) {$0$};
  \node[main node] (1) [right of=0] {$1$};
  \node[main node] (2) [right of=1] {$2$};
  \node[node distance=2cm] (4) [right of=2] {$\dots$};
  \node[main node] (s) [right of=4] {$s$};
  \node[node distance=2cm] (s+1) [right of=s] {$\dots$};

  \path[every node/.style={font=\sffamily\small}]
    (0) edge [bend left, above] node {$\lambda$} (1)
    (1) edge [bend left, above] node {$\lambda$} (2)
    (2) edge [bend left, above] node {$\lambda$} (3)
    (3) edge [bend left, above] node {$\lambda$} (s)
    (s) edge [bend left, above] node {$\lambda$} (s+1)
    (s+1) edge [bend left, below] node {$s\mu$} (s)
    (s) edge [bend left, below] node {$s\mu$} (3)
    (3) edge [bend left, below] node {$3\mu$} (2)
    (2) edge [bend left, below] node {$2\mu$} (1)
    (1) edge [bend left, below] node {$\mu$} (0)
   ;
\end{tikzpicture}
\end{center}
 \vspace{-5mm} 
 \caption[]{Graphe associ\'e \`a la file d'attente M/M/$s$.}
 \label{fig_MMs}
\end{figure}
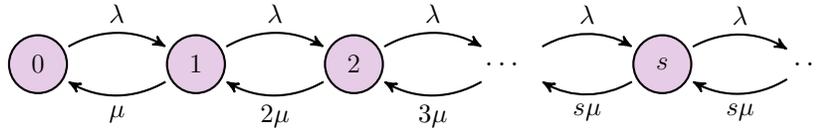

\begin{example}[File M/M/$\infty$]
\label{ex_MM8} 
L'existence d'un nombre infini de serveurs peut sembler fantaisiste, mais en
fait il s'agit d'une bonne approximation de situations dans lesquelles le
nombre de serveurs est grand par rapport au nombre de clients, comme par
exemple pour certaines centrales t\'el\'epho\-niques. 

Dans ce cas, les taux de transition sont donn\'es par 
\begin{align}
\nonumber
q(n,n+1) &= \lambda \phantom{n}
\qquad\text{si $n\geqs 0$\;,}\\
q(n,n-1) &= n\mu 
\qquad\text{si $n\geqs 1$\;.}
\label{MM14} 
\end{align}
La distribution stationnaire peut \^etre calcul\'ee comme dans les cas
pr\'ec\'edents, avec le r\'esultat explicite
\begin{equation}
 \label{MM15}
\pi(n) = \e^{-\lambda/\mu} \frac{(\lambda/\mu)^n}{n!}\;. 
\end{equation} 
La longueur de la file, qui est dans ce cas \'egale au nombre de serveurs
occup\'es, suit donc une loi de Poisson d'esp\'erance
$\lambda/\mu$. 
\end{example}

\begin{figure}[h]
%  \vspace{2mm}
%  \centerline{
%  \includegraphics*[clip=true,height=10mm]{figs/birth_death}
%  }
%  \figtext{
%  \writefig       2.4      0.9     $0$
%  \writefig       5.05     0.9     $1$
%  \writefig       7.65     0.9     $2$
%  \writefig       10.3     0.9     $3$
%  \writefig       3.75     1.6     $\lambda$
%  \writefig       6.35     1.6     $\lambda$
%  \writefig       8.95     1.6     $\lambda$
%  \writefig       11.65    1.6     $\lambda$
%  \writefig       3.75     0.2     $\mu$
%  \writefig       6.25     0.2     $2\mu$
%  \writefig       8.85     0.2     $3\mu$
%  \writefig       11.45    0.2     $4\mu$
%  \writefig       13.0     0.9     $\dots$
%  }
\begin{center}
\begin{tikzpicture}[->,>=stealth',shorten >=2pt,shorten <=2pt,auto,node
distance=3.0cm, thick,main node/.style={circle,scale=0.7,minimum size=1.1cm,
fill=violet!20,draw,font=\sffamily\Large}]

  \node[main node] (0) {$0$};
  \node[main node] (1) [right of=0] {$1$};
  \node[main node] (2) [right of=1] {$2$};
  \node[main node] (3) [right of=2] {$3$};
  \node[node distance=2cm] (4) [right of=3] {$\dots$};

  \path[every node/.style={font=\sffamily\small}]
    (0) edge [bend left, above] node {$\lambda$} (1)
    (1) edge [bend left, above] node {$\lambda$} (2)
    (2) edge [bend left, above] node {$\lambda$} (3)
    (3) edge [bend left, above] node {$\lambda$} (4)
    (4) edge [bend left, below] node {$4\mu$} (3)
    (3) edge [bend left, below] node {$3\mu$} (2)
    (2) edge [bend left, below] node {$2\mu$} (1)
    (1) edge [bend left, below] node {$\mu$} (0)
   ;
\end{tikzpicture}
\end{center}
 \vspace{-5mm} 
 \caption[]{Graphe associ\'e \`a la file d'attente M/M/$\infty$.}
 \label{fig_MM8}
\end{figure}

Nous mentionnons un r\'esultat remarquable s'appliquant \`a toutes les files
d'attente M/M/$s$. 

\begin{theorem}
\label{thm_MMs}
Si $\lambda < s\mu$, alors les clients servis quittent la file M/M/$s$ selon un
processus de Poisson d'intensit\'e $\lambda$.  
\end{theorem}

\begin{proof}
Dans le cas d'un seul serveur, $s=1$, l'assertion peut \^etre v\'erifi\'ee par
un calcul explicite. Il y a deux cas \`a consid\'erer. 
\begin{enum}
\item	Il y a $n\geqs1$ clients dans la file. Dans ce cas, le prochain d\'epart
d'un client servi aura lieu apr\`es un temps exponentiel de taux $\mu$. 
\item	La file est vide. Alors il faut d'abord attendre un temps $T_1$ de loi
$\cE\!xp(\lambda)$ avant l'arriv\'ee d'un client, puis un temps $T_2$
ind\'ependant, de loi $\cE\!xp(\mu)$, jusqu'\`a ce que ce client ait \'et\'e
servi. La densit\'e de $T_1+T_2$ s'obtient par convolution,
\begin{equation}
 \label{MM20}
\int_0^t \lambda\e^{-\lambda s} \mu\e^{-\mu(t-s)} \6s
= \frac{\lambda\mu}{\lambda-\mu} \bigpar{\e^{-\mu t} - \e^{-\lambda t}}\;. 
\end{equation} 
\end{enum}
Le premier cas se produit avec probabilit\'e $\lambda/\mu$, et le second avec
probabilit\'e $1-\lambda/\mu$. 
La densit\'e du temps jusqu'au prochain d\'epart s'obtient en combinant ces
deux cas, et vaut 
\begin{equation}
 \label{MM21}
\frac{\lambda}{\mu} \mu\e^{-\mu t} 
+ \biggpar{1-\frac{\lambda}{\mu}} 
\frac{\lambda\mu}{\lambda-\mu} \bigpar{\e^{-\mu t} - \e^{-\lambda t}}
= \lambda \e^{-\lambda t}\;.
\end{equation} 
C'est bien la densit\'e d'une variable exponentielle de param\`etre $\lambda$.

De mani\`ere g\'en\'erale, le r\'esultat se d\'emontre en utilisant la
r\'eversibilit\'e. De mani\`ere analogue \`a~\eqref{rev2}, en partant avec la
distribution stationnaire, le processus $N_t$ a la m\^eme loi que le processus
renvers\'e dans le temps $N_{T-t}$. Le renversement du temps intervertit les
clients arrivant dans la file et les clients quittant la file. Les deux ont
donc la m\^eme loi. 
\end{proof}

Ce r\'esultat est important dans l'\'etude des r\'eseaux de files d'attente,
car il caract\'erise la distribution des clients rejoignant une seconde file
apr\`es avoir \'et\'e servis dans une premi\`ere file. 

%%%%%%%%%%%%%%%%%%%%%%%%%%%%%%%%%%%%%%%%%%%%%%%%%%%%%%%%%%%%%%%%%%%%%%%%%%%

\section{Cas g\'en\'eral~: Files d'attente G/G/1}
\label{sec_faGG}

Le cas d'une file d'attente g\'en\'erale est sensiblement plus difficile \`a
\'etudier, car si les temps entre arriv\'ees de clients et les temps de service
ne sont pas exponentiels, on perd la propri\'et\'e de Markov. En effet, 
l'\'etat actuel du syst\`eme ne suffit pas \`a d\'eterminer son \'evolution
future, qui d\'epend entre autres du temps que le client actuellement servi a
d\'ej\`a pass\'e au serveur. 

Dans certains cas particuliers, on peut se ramener \`a un syst\`eme markovien
en introduisant des \'etats suppl\'ementaires.

\begin{example}[File d'attente M/E$_r$/1]
\label{ex_ME1} 
Supposons que les clients arrivent selon un processus de Poisson d'intensit\'e
$\lambda$, mais que le temps de service suit une loi Gamma de param\`etres
$(r,\mu)$ avec $r\geqs2$. Une interpr\'etation possible de cette loi est que le
service du client requiert $r$ actions successives, prenant chacune un temps
exponentiel de param\`etre $\mu$. Si $N_t$ d\'esigne la somme du nombre de
clients dans la file, et du nombre d'actions encore n\'ecessaires pour finir de
servir le client actuel, son \'evolution suit un processus markovien de sauts
de taux de transition 
\begin{align}
\nonumber
q(n,n-1) &= \mu 
\qquad\text{si $n\geqs 1$\;,} \\
q(n,n+r) &= \lambda
\qquad\text{si $n\geqs 0$\;.}
\label{GG01} 
\end{align}
En effet, le nombre d'actions \`a accomplir diminue de $1$ avec un taux $\mu$,
et augmente de $r$ \`a l'arriv\'ee de chaque nouveau client.  Si $\lambda<\mu$,
on peut montrer que ce syst\`eme admet une distribution stationnaire, qui est
une combinaison lin\'eaire de lois g\'eom\'etriques. 
\end{example}
 
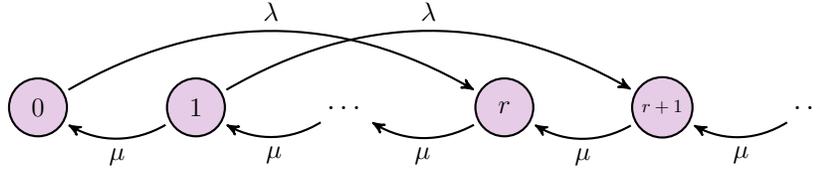
\begin{figure}
%  \vspace{2mm}
%  \centerline{
%  \includegraphics*[clip=true,height=15mm]{figs/MEr1}
%  }
%  \figtext{
%  \writefig       1.5      0.9     $0$
%  \writefig       3.9      0.9     $1$
%  \writefig       6.3      0.9     $\dots$
%  \writefig       8.9      0.9     $r$
%  \writefig       11.15    0.9     {\footnotesize $r+1$}
%  \writefig       5.05     2.1     $\lambda$
%  \writefig       7.45     2.1     $\lambda$
%  \writefig       2.65     0.2     $\mu$
%  \writefig       5.15     0.2     $\mu$
%  \writefig       7.55     0.2     $\mu$
%  \writefig       10.15    0.2     $\mu$
%  \writefig       12.65    0.2     $\mu$
%  \writefig       14.0     0.9     $\dots$
%  }
\begin{center}
\begin{tikzpicture}[->,>=stealth',shorten >=2pt,shorten <=2pt,auto,node
distance=3.0cm, thick,main node/.style={circle,scale=0.7,minimum size=1.1cm,
fill=violet!20,draw,font=\sffamily\Large}]

  \node[main node] (0) {$0$};
  \node[main node] (1) [right of=0] {$1$};
  \node[node distance=2cm] (2) [right of=1] {$\dots$};
  \node[main node] (r) [right of=2] {$r$};
  \node[main node] (r+1) [right of=r] {{\small $r+1$}};
  \node[node distance=2cm] (r+2) [right of=r+1] {$\dots$};

  \path[every node/.style={font=\sffamily\small}]
    (0) edge [bend left, above] node {$\lambda$} (r)
    (1) edge [bend left, above] node {$\lambda$} (r+1)
    (r+2) edge [bend left, below] node {$\mu$} (r+1)
    (r+1) edge [bend left, below] node {$\mu$} (r)
    (r) edge [bend left, below] node {$\mu$} (2)
    (2) edge [bend left, below] node {$\mu$} (1)
    (1) edge [bend left, below] node {$\mu$} (0)
   ;
\end{tikzpicture}
\end{center}
 \vspace{-5mm}  
\caption[]{Graphe associ\'e \`a la repr\'esentation markovienne
de la file d'attente M/E$_r$/1.}
 \label{fig_MEr1}
\end{figure}

En r\`egle g\'en\'erale, toutefois, une telle repr\'esentation markovienne
n'est pas possible, et on fait appel \`a la th\'eorie des \defwd{processus de
renouvellement}. Un tel processus est caract\'eris\'e par une suite de temps
al\'eatoires $0<T_1<T_2<\dots$, appel\'es temps de renouvellement, tels que le
comportement sur chaque intervalle de temps $[T_n,T_{n+1}[$ soit ind\'ependant
et de m\^eme loi que sur les autres intervalles. 

Dans la suite, nous nous servirons du r\'esultat suivant, qui suit de la loi
forte des grands nombres. 

\begin{theorem}[Loi des grands nombres pour processus de renouvellement]
Soit $1/\mu$ l'esp\'erance des intervalles $T_{n+1}-T_n$, et soit
$N_t=\sup\setsuch{n}{T_n\leqs t}$ la fonction de comptage des temps de
renouvellement. Alors 
\begin{equation}
 \label{GG02}
\biggprob{
\lim_{t\to\infty} \frac{1}{t} N_t = \mu } = 1\;.
\end{equation} 
De plus, 
\begin{equation}
 \label{GG03}
 \lim_{t\to\infty} \frac{1}{t} \expec{N_t} = \mu\;.
\end{equation} 
\end{theorem}

La relation~\eqref{GG03} affirme que le nombre moyen de
temps de renouvellement converge vers $\mu$ en moyenne ergodique. $\mu$ peut
donc \^etre consid\'er\'e comme le taux du processus. 

Consid\'erons alors une file d'attente dans laquelle les clients arrivent selon
un processus de renouvellement de taux $\lambda$, et dans laquelle ils sont
servis pendant une dur\'ee al\'eatoire de moyenne $1/\mu$. 

\begin{theorem}
\label{thm_GG1}
Si $\lambda<\mu$, et si la longueur initiale de la file est finie, alors la
longueur de la file atteindra $0$ en un temps fini presque s\^urement. De plus,
le serveur sera occup\'e pendant une fraction de temps $\lambda/\mu$.  
\end{theorem}
\begin{proof}
Le temps $T_n$ de l'arriv\'ee du $n$i\`eme client est la somme de $n$ variables
al\'eatoires ind\'ependantes d'esp\'erance $1/\lambda$. La loi forte des grands
nombres affirme qu'avec probabilit\'e $1$, on a 
\begin{equation}
 \label{GG04:1}
\lim_{n\to\infty} \frac{T_n}{n} = \frac{1}{\lambda}\;. 
\end{equation} 
Soit $Z_0$ le temps n\'ecessaire \`a servir les clients initialement dans la
file, et soit $s_i$ le temps n\'ecessaire \`a servir le $i$\`eme client
arrivant apr\`es le temps $0$. On sait que $\expec{s_i}=1/\mu$. Supposons par
l'absurde que le serveur reste toujours occup\'e. Alors le $n$i\`eme client
partira au temps $Z_0+S_n$, o\`u $S_n=s_1+\dots+s_n$. La loi forte des grands
nombres implique que 
\begin{equation}
 \label{GG04:2}
\lim_{n\to\infty} \frac{Z_0+S_n}{n} = \frac{1}{\mu}\;.  
\end{equation} 
Comme $1/\mu<1/\lambda$, cela implique que pour $n$ assez grand, le $n$i\`eme
client quitte la file avant son arriv\'ee, ce qui est absurde. Le serveur ne
peut donc pas \^etre toujours occup\'e. 

Soit maintenant $A_n$ la dur\'ee pendant laquelle le serveur a \'et\'e occup\'e
au temps $T_n$. Alors on a 
\begin{equation}
 \label{GG04:3}
A_n = S_n - Z_n\;, 
\end{equation} 
o\`u $Z_n$ est le temps n\'ecessaire \`a vider la file des clients pr\'esents
au temps $T_n$. Or on a 
\begin{equation}
 \label{GG04:4}
\lim_{n\to\infty} \frac{S_n}{T_n}
=  \lim_{n\to\infty} \frac{S_n/n}{T_n/n} 
= \frac{\lambda}{\mu}\;.
\end{equation} 
D'autre part, si la file atteint l'\'equilibre, $\expec{Z_n}$ doit rester
born\'e, donc $Z_n/n$ doit tendre vers~$0$. 
\end{proof}

Dans le cas des files M/M/$s$, nous avons montr\'e par un calcul direct que le
serveur est occup\'e pendant une fraction du temps $\lambda/\mu$. Le r\'esultat
ci-dessus montre que c'est vrai pour toutes les files d'attente G/G/$1$. 

Soit $X_t$ la longueur de la file au temps $t$, et soit $W_n$ le temps
d'attente du $n$i\`eme client. Deux quantit\'es importantes
sont la longueur moyenne de la file
\begin{equation}
 \label{GG05}
L = \lim_{t\to\infty} \frac{1}{t} \int_0^t X_s\6s\;, 
\end{equation}  
et la moyenne des temps d'attente 
\begin{equation}
 \label{GG06}
W = \lim_{n\to\infty} \frac{1}{n} \sum_{i=1}^n W_n\;. 
\end{equation} 

\begin{theorem}[Loi de Little]
\label{thm_Little} 
Soit $\lambda^*$ le taux des clients qui arrivent et joignent la file. Alors
\begin{equation}
 \label{GG07}
L = \lambda^* W\;. 
\end{equation} 
\end{theorem}

Dans le cas de la file M/M/1, on a $\lambda^*=\lambda$. Par calcul explicite,
nous avons obtenu une longueur de file moyenne de
$(\lambda/\mu)/(1-(\lambda/\mu))$, et un temps d'attente moyen de
$1/(\mu-\lambda)$. La loi de Little est donc v\'erifi\'ee dans ce cas. On
notera toutefois qu'il s'agit l\`a de valeurs moyennes par rapport \`a la
distribution stationnaire. Mais il se trouve que ces moyennes sont \'egales aux
moyennes ergodiques pour les processus markoviens de sauts. 

Consid\'erons finalement le cas particulier important des files M/G/1,
c'est-\`a-dire que les clients arrivent selon un processus de Poisson
d'intensit\'e $\lambda$. Nous savons que le serveur passe alternativement par
des phases occup\'ees et libres. Soit $O_n$ la longueur de la $n$i\`eme phase
occup\'ee. Les phases libres ont une dur\'ee moyenne de $1/\lambda$. La
proportion du temps pendant laquelle le serveur est libre est donn\'ee par 
\begin{equation}
 \label{GG08}
\frac{1/\lambda}{1/\lambda + \expec{O_n}}\;. 
\end{equation} 
D'autre part, le Th\'eor\`eme~\ref{thm_GG1} montre que cette proportion est
\'egale \`a $1-\lambda/\mu$. On en conclut que 
\begin{equation}
 \label{GG09}
\expec{O_n} = \frac{1/\mu}{1-\lambda/\mu}\;. 
\end{equation} 
Enfin, une autre propri\'et\'e importante des files M/G/1 est la suivante.

\begin{theorem}[Propri\'et\'e PASTA]
\label{thm_PASTA} 
Soit $\pi(n)$ la proportion du temps pendant laquelle la file a une longueur
$n$, et soit $a_n$ la proportion asymptotique des clients trouvant une file de
longueur $n$ \`a leur arriv\'ee. Alors 
\begin{equation}
 \label{GG10}
a_n = \pi(n)\;. 
\end{equation} 
\end{theorem}

L'acronyme PASTA vient de l'anglais \lq\lq Poisson arrivals see time
averages\rq\rq. Cette propri\'et\'e n'est pas n\'ecessairement vraie lorsque les
temps d'arriv\'ee ne suivent pas une loi de Poisson.

%%%%%%%%%%%%%%%%%%%%%%%%%%%%%%%%%%%%%%%%%%%%%%%%%%%%%%%%%%%%%%%%%%%%%%%%%%%

\section{Exercices}
\label{sec_faexo}

\begin{exercice}
\label{exo_fa01} 
Une station service comporte une seule pompe \`a essence. Des voitures arrivent
selon un processus de Poisson de taux 20 voitures par heure. Le temps de
service suit une loi exponentielle d'esp\'erance 2 minutes.
\begin{enum}
\item	Donner la distribution stationnaire du nombre de voitures dans la
station.
\item	D\'eterminer le temps d'attente moyen avant d'\^etre servi, et le temps
de s\'ejour total.
\item	Quelle proportion des voitures doit attendre avant
de pouvoir faire le plein? Quelle proportion doit attendre plus de 2 minutes?
\end{enum}
On suppose maintenant que tout conducteur trouvant 2 voitures dans la station
repart aussit\^ot. 
\begin{enum}
\setcounter{enumi}{3}
\item	Donner la distribution stationnaire du nombre de voitures dans la
station. Quelle est la probabilit\'e qu'une voiture reparte sans faire le plein?
\item	D\'eterminer le temps d'attente et le temps de s\'ejour moyens. 
\end{enum}
\end{exercice}

\goodbreak

\begin{exercice}
\label{exo_fa02} 
Des clients arrivent dans un salon de coiffure selon un processus de Poisson de
taux 5 clients par heure. On suppose qu'il y a un seul coiffeur, qui met un
temps exponentiel de moyenne un quart d'heure pour coiffer un client. La salle
d'attente comporte deux chaises. Si un client arrive et que toutes les chaises
sont occup\'ees, il repart. 
\begin{enum}
\item	Calculer la distribution stationnaire. 
\item	Quelle est la probabilit\'e qu'un client doive attendre avant d'\^etre
servi?
\item	D\'eterminer le temps d'attente moyen.
\item	Quel est le nombre moyen de clients servis par heure?
\item	On suppose maintenant qu'il y a deux coiffeurs. Chacun met un temps
exponentiel de moyenne une demi-heure pour s'occuper d'un client. Calculer le
nombre moyen de clients servis par heure. 
\end{enum}
\end{exercice}

\goodbreak

\begin{exercice}
\label{exo_fa03} 
Le centre d'appel d'un compagnie d'assurance re\c coit en moyenne 40 appels par
heure. Il y a trois op\'erateurs pour r\'epondre aux appels. Le temps des appels
est exponentiel de moyenne 3 minutes. 
\begin{enum}
\item	Quel est le nombre moyen d'op\'erateurs occup\'es?
\item	Quelle est la probabilit\'e qu'un client
doive attendre avant qu'on lui r\'eponde?
\end{enum}
\end{exercice}

\goodbreak

\begin{exercice}
\label{exo_fa04} 

La salle d'attente du Docteur H comprend 2 chaises. Les patients arrivent selon
un processus de Poisson de taux 6 patients par heure. Les patients trouvant les
3 chaises occup\'ees partent chercher un autre m\'edecin. Les consultations
suivent une loi exponentielle de moyenne 15 minutes. 

\begin{enum}
\item	Quelle est la probabilit\'e que la salle d'attente soit pleine?
\item	Calculer le temps d'attente moyen d'un patient avant la consultation.
\item	Combien de patients le Docteur traite-t-il par heure en moyenne?
\end{enum}

\end{exercice}

\goodbreak

\begin{exercice}
\label{exo_fa041}

Thelma et Louise tiennent un salon de coiffure, dont la salle d'attente
comporte deux chaises. Pour coiffer un client, chacune passe un temps de loi
exponentielle, de moyenne 30 minutes. 
Les clients arrivent selon un processus de Poisson avec un taux de 5 par heure. 
Si les deux chaises de la salle d'attente sont occup\'ees lors de l'arriv\'ee
d'un client, celui-ci repart aussit\^ot. 

\begin{enum}
\item	D\'eterminer le distribution stationnaire du processus.
\item	Quelle est la probabilit\'e que la salle d'attente soit pleine? 
\item	Quelle est la probabilit\'e que les deux coiffeuses, l'une des deux, ou
aucune des deux ne soit occup\'ee? 
\item 	Quel est le temps d'attente moyen des clients?
\item 	Pendant quelle fraction de temps Louise est-elle occup\'ee \`a coiffer
un client? Avez-vous fait une hypoth\`ese particuli\`ere pour arriver \`a ce
r\'esultat?
\end{enum}
\end{exercice}

\goodbreak

\begin{exercice}
\label{exo_fa042} 

Madame Jamilah, diseuse de bonne aventure, offre ses services \`a la f\^ete
foraine de Patelin-sur-Loire. On suppose que les clients arrivent selon un
processus ponctuel de Poisson d'intensit\'e 4 clients par heure, et que les
consultations ont une dur\'ee de loi exponentielle de moyenne 10 minutes.

\begin{enum}
\item	On suppose que la longueur de le file d'attente devant la tente de
Madame Jamilah est illimit\'ee. Calculer 
\begin{enum}
\item 	la distribution stationnaire de la longueur de la file;
\item 	le temps d'attente moyen d'un client;
\item 	le nombre moyen de clients par heure.
\end{enum}

\item 	Suite \`a des probl\`emes avec le service d'ordre, les organisateurs de
la f\^ete interdisent toute file d'attente. Madame Jamilah \'etablit alors une
salle d'attente dans sa tente, avec une seule place. Toute personne arrivant
alors que Madame Jamilah et la salle d'attente sont occup\'ees repart
aussit\^ot. D\'eterminer 
\begin{enum}
\item 	la distribution stationnaire du nombre de clients;
\item 	le temps d'attente moyen d'un client;
\item 	le nombre moyen de clients par heure.
\end{enum}
\end{enum}
\end{exercice}

\goodbreak

\begin{exercice}
\label{exo_fa05} 

On consid\`ere une file d'attente M/M/2 traitant les clients au taux $\mu$, et
une file d'attente M/M/1 traitant les clients au taux $2\mu$. Pour laquelle de
ces files le serveur a-t-il la plus grande probabilit\'e d'\^etre occup\'e?  
\end{exercice}

\goodbreak

\begin{exercice}[File d'attente M/M/$s$/0]
\label{exo_fa06} 
Des appels arrivent dans une centrale t\'el\'epho\-nique selon un processus de
Poisson de taux $\lambda$. Il y a $s$ lignes disponibles, et les appels ont une
dur\'ee exponentielle de moyenne $1/\mu$. Un appel arrivant alors que toutes les
lignes sont occup\'ees est refus\'e. 
\begin{enum}
\item	Trouver la distribution stationnaire. 
\item	Calculer la probabilit\'e qu'un appel soit rejet\'e. 
\end{enum}
\end{exercice}

\goodbreak

\begin{exercice}
\label{exo_fa07} 
Le but du probl\`eme est de comparer deux types de files d'attente \`a deux
serveurs.

Dans le premier cas, les clients forment une seule file et choisissent le
premier serveur qui se lib\`ere (file M/M/2). On suppose que les clients
arrivent selon un processus de Poisson de taux $\lambda$, et qu'ils sont servis
pendant un temps exponentiel de param\`etre $\mu=\lambda$.
% \medskip
% 
% \centerline{
%  \includegraphics*[clip=true,height=25mm]{figs/file_ex1}
%  }

\begin{center}
\begin{tikzpicture}[->,>=stealth',shorten >=10pt,shorten <=10pt,auto,node
distance=3.0cm, thick,main node/.style={circle,scale=0.7,minimum size=0.4cm,
fill=green!20,draw,font=\sffamily\Large}]

  \node[main node] (1) at(0,0) {};
  \node[main node] (2) at(0.5,0) {};
  \node[main node] (3) at(1,0) {};
  \node[main node] (4) at(1.5,0) {};
  \node[main node] (s1) at(-1.5,0.75) {};
  \node[main node] (s2) at(-1.5,-0.75) {};

  \node (r) at(3.5,0) {};
  \node (l1) at(-3.8,0.75) {};
  \node (l2) at(-3.8,-0.75) {};
  \node (s-1) at(-1.8,0.75) {};
  \node (s-2) at(-1.8,-0.75) {};

  \path[thick,-,shorten >=0pt, shorten <=0pt]
    (-1.5,1.2) edge (-1.8,1.2)
    (-1.8,1.2) edge (-1.8,0.3)
    (-1.8,0.3) edge (-1.5,0.3)
    (-1.5,-1.2) edge (-1.8,-1.2)
    (-1.8,-1.2) edge (-1.8,-0.3)
    (-1.8,-0.3) edge (-1.5,-0.3)
  ;

  \path[every node/.style={font=\sffamily\footnotesize}]
    (1) edge (s1)
    (1) edge (s2)
    (r) edge [above] node {$\lambda$} (4)
    (s-1) edge [above] node {$\lambda$} (l1)
    (s-2) edge [above] node {$\lambda$} (l2)
    ;
\end{tikzpicture}
\end{center}
 
\begin{enum}
\item	D\'eterminer la distribution stationnaire $\pi$ de la file.
\item	Quelle est la probabilit\'e qu'un client ne doive pas attendre avant
d'\^etre servi?
\item	Quel est le temps d'attente moyen avant d'\^etre servi?
\item	Soit $S$ le nombre de serveurs occup\'es. D\'eterminer
$\expecin{\pi}{S}$. 
\end{enum}

Dans le second cas, il y a une file distincte devant chaque serveur. 
Les clients choisissent une file ou l'autre avec probabilit\'e $1/2$. 

% \bigskip
% 
% \centerline{
%  \includegraphics*[clip=true,height=25mm]{figs/file_ex2}
%  }

\begin{center}
\begin{tikzpicture}[->,>=stealth',shorten >=10pt,shorten <=10pt,auto,node
distance=3.0cm, thick,main node/.style={circle,scale=0.7,minimum size=0.4cm,
fill=green!20,draw,font=\sffamily}]

  \node[main node] (+1) at(0,0.75) {};
  \node[main node] (+2) at(0.5,0.75) {};
  \node[main node] (+3) at(1,0.75) {};
  \node[main node] (-1) at(0,-0.75) {};
  \node[main node] (-2) at(0.5,-0.75) {};
  \node[main node] (s1) at(-1.5,0.75) {};
  \node[main node] (s2) at(-1.5,-0.75) {};

  \node (-3) at(1,-0.75) {};
  \node (r) at(4.0,0) {};
  \node (r-) at(2.0,0) {};
  \node (r+) at(2.5,0) {};
  \node (l1) at(-3.8,0.75) {};
  \node (l2) at(-3.8,-0.75) {};
  \node (s-1) at(-1.8,0.75) {};
  \node (s-2) at(-1.8,-0.75) {};

  \path[thick,-,shorten >=0pt, shorten <=0pt]
    (-1.5,1.2) edge (-1.8,1.2)
    (-1.8,1.2) edge (-1.8,0.3)
    (-1.8,0.3) edge (-1.5,0.3)
    (-1.5,-1.2) edge (-1.8,-1.2)
    (-1.8,-1.2) edge (-1.8,-0.3)
    (-1.8,-0.3) edge (-1.5,-0.3)
  ;

  \path[every node/.style={font=\sffamily\footnotesize}]
    (r) edge [above] node {$\lambda$} (r-)
    (r+) edge [above right] node {$\lambda/2$} (+3)
    (r+) edge [below right] node {$\lambda/2$} (-3)
    (s-1) edge [above] node {$\lambda$} (l1)
    (s-2) edge [above] node {$\lambda$} (l2)
    ;
\end{tikzpicture}
\end{center}

\begin{enum}
\setcounter{enumi}{4}
\item	Expliquer pourquoi du point de vue du client, ce cas est \'equivalent
\`a une file M/M/1 avec taux $\lambda/2$ et $\lambda$.
\item	D\'eterminer la distribution stationnaire $\pi$ du syst\`eme.
\item	Quelle est la probabilit\'e qu'un client ne doive pas attendre avant
d'\^etre servi?
\item	Quel est le temps d'attente moyen avant d'\^etre servi?
\item	Soit $S$ le nombre de serveurs occup\'es. D\'eterminer
$\expecin{\pi}{S}$. 
\item	Comparer les deux syst\`emes. 
\end{enum}
\end{exercice}

\renewcommand{\appendixname}{Appendice}
\appendix

\chapter{Solution de quelques exercices}

Cet appendice contient les solutions de certains exercices. Les r\'eponses
donn\'ees servent seulement \`a v\'erifier vos calculs, mais ne constituent pas
une correction d\'etaill\'ee.

\section{Exercices du Chapitre \ref{chap_fini}}

\subsection*{Exercice~\ref{exo_mf01}}
\begin{enum}
\item 	On num\'erote les \'etats dans l'ordre 
T\^ete Rousse, Aiguille du Go\^uter, Nid d'Aigle, Sommet du Mont Blanc. 
\[
P = 
\begin{pmatrix}
0 & p & q & 0 \\
q & 0 & 0 & p \\
0 & 0 & 1 & 0 \\
0 & 0 & 0 & 1
\end{pmatrix}\;, \qquad
Q = 
\begin{pmatrix}
0 & p \\ q & 0 
\end{pmatrix}\;, \qquad
R = 
\begin{pmatrix}
q & 0 \\ 0 & p 
\end{pmatrix}\;.
\]
Matrice fondamentale~:
\[
F = \frac{1}{1-pq}
\begin{pmatrix}
1 & p \\ q & 1
\end{pmatrix}\;.
\]

\item 	$\probin{1}{X_\tau=4} = p^2/(1-pq)$. 

\item 	$p^\star = (\sqrt{5}-1)/2$. 

\item 	$\expecin{1}{\tau} = (1+p^\star)/[1-p^\star(1-p^\star)]
= (1+\sqrt{5})/[2(3-\sqrt{5})] \cong 2.118$. 

\end{enum}

\subsection*{Exercice~\ref{exo_mf02}}

\begin{enum}
\item 	
\[
F = 
\begin{pmatrix}
3/2 & 1/2 & 1/2 \\
3/4 & 5/4 & 1/4 \\
3/4 & 1/4 & 5/4 
\end{pmatrix}\;.
\]

\item 	$\probin{1}{X_\tau=4} = 1/2$, 
$\probin{2}{X_\tau=4} = 3/4 = \probin{3}{X_\tau=4}$.

\item 	$\expecin{1}{\tau} = 5/2$, 
$\expecin{2}{\tau} = 9/4 = \expecin{3}{\tau}$. 
\end{enum}

\subsection*{Exercice~\ref{exo_mf03}}

\begin{enum}
\item 	Avec l'ordre des \'etats~: Egalit\'e, Avantage A, Avantage B, A gagne, 
B gagne,  
\[
Q = 
\begin{pmatrix}
0 & 3/5 & 2/5 \\
2/5 & 0 & 0 \\
3/5 & 0 & 0 
\end{pmatrix}\;, \qquad
R = 
\begin{pmatrix}
0 & 0 \\ 3/5 & 0 \\ 0 & 2/5 
\end{pmatrix}\;.
\]

\addtocounter{enumi}{1}
\item $\probin{1}{A \text{ gagne}} = 9/13$.

\item $\expecin{1}{\tau} = 50/13$. 
\end{enum}

\subsection*{Exercice~\ref{exo_mf031}}

\begin{enum}
\item 	La matrice de transition sous forme canonique et la matrice
fondamentale sont donn\'ees par 
\[
Q = 
\begin{pmatrix}
0 & 1/3 & 1/3 \\
1/4 & 0 & 1/4 \\
1/3 & 1/3 & 0 
\end{pmatrix}\;, \qquad
R = 
\begin{pmatrix}
1/3 & 0 \\ 1/4 & 1/4 \\ 0 & 1/3 
\end{pmatrix}\;, \qquad
F = \frac{1}{24}
\begin{pmatrix}
33 & 16 & 15 \\ 12 & 32 & 12 \\ 15 & 16 & 33 
\end{pmatrix}\;.
\]

\item 	$\expecin{1}{\tau} = 8/3$. 
\item 	$\probin{1}{X_\tau=5} = 3/8$. 
\end{enum}

\subsection*{Exercice~\ref{exo_mf04}}

\begin{center}
\begin{tikzpicture}[->,>=stealth',shorten >=2pt,shorten <=2pt,auto,node
distance=4.0cm, thick,main node/.style={circle,scale=0.7,minimum size=1.0cm,
fill=blue!20,draw,font=\sffamily\Large}]

  \node[main node] (1) {1};
  \node[main node] (2) [below left of=1] {2};
  \node[main node] (3) [below right of=1] {3};
  \node[main node] (4) [below right of=2] {4};

  \path[every node/.style={font=\sffamily\small}]
    (1) edge [bend right, above left] node {$1/2$} (2)
    (2) edge [below right, out=35, in=-125] node {$1/16$} (1)
    (1) edge [below left, out=-55, in=145] node {$1/2$} (3)
    (3) edge [bend right, above right] node {$1/16$} (1)
    (2) edge [bend right, below left] node {$1/2$} (4)
    (4) edge [above right, out=125, in=-35] node {$1/4$} (2)
    (3) edge [above left, out=-145, in=55] node {$1/2$} (4)
    (4) edge [bend right, below right] node {$1/4$} (3)
    (2) edge [loop right, left,distance=1.5cm,out=150,in=-150] node {$7/16$} (2)
    (3) edge [loop right, right,distance=1.5cm,out=-30,in=30] node {$7/16$} (3)
   ;
\end{tikzpicture}
\end{center}

\begin{enum}
\addtocounter{enumi}{1}
\item 	$P^2$ n'a que des \'el\'ements positifs, donc la \chaine\ est
r\'eguli\`ere. 

\item 	$\pi = \frac{1}{33}(1,8,8,16)$. 
\end{enum}

\subsection*{Exercice~\ref{exo_mf05}}

\begin{enum}
\item 
\[
P = 
\begin{pmatrix}
0 & 1 & 0 & 0 \\
1/4 & 1/2 & 1/4 & 0 \\
0 & 1/2 & 0 & 1/2 \\
0 & 0 & 1 & 0 
\end{pmatrix}\;.
\]

\item 	La \chaine\ est irr\'eductible.

\item 	La \chaine\ est r\'eguli\`ere ($P^4$ n'a que des \'el\'ements
strictement positifs). 

\item 	$\pi = (\frac18, \frac12, \frac14, \frac18)$. 

\item 	La \chaine\ est r\'eversible. 
\end{enum}

\subsection*{Exercice~\ref{exo_mf051}}

\begin{enum}
\item 
\[
P = 
\begin{pmatrix}
0 & 1/2 & 0 & 1/2 \\
1/2 & 0 & 1/2 & 0 \\
1/2 & 0 & 1/2 & 0 \\
0 & 1/2 & 0 & 1/2
\end{pmatrix}\;.
\]
La \chaine\ est irr\'eductible.

\item 	La \chaine\ est r\'eguli\`ere ($P^2$ n'a que des \'el\'ements
strictement positifs). 

\item 	$\pi = (\frac14, \frac14, \frac14, \frac14)$. 

\item 	La \chaine\ n'est pas r\'eversible. 
Par exemple $\pi_1p_{13} = 0 \neq \pi_3 p_{31}$. 

\end{enum}

\subsection*{Exercice~\ref{exo_mf06}}

\begin{center}
\begin{tikzpicture}[->,>=stealth',shorten >=2pt,shorten <=2pt,auto,node
distance=4.0cm, thick,main node/.style={circle,scale=0.7,minimum size=1.0cm,
fill=blue!20,draw,font=\sffamily\Large}]

  \node[main node] (B) {B};
  \node[main node] (P) [right of=B] {P};
  \node[main node] (N) [right of=P] {N};

  \path[every node/.style={font=\sffamily\small}]
    (B) edge [bend left, above] node {$3/4$} (P)
    (P) edge [bend left, above] node {$1/6$} (N)
    (N) edge [bend left, below] node {$1/2$} (P)
    (P) edge [bend left, below] node {$1/2$} (B)
    (P) edge [loop right, above,distance=1.5cm,out=120,in=60] node {$1/3$} (P)
    (B) edge [loop right, left,distance=1.5cm,out=210,in=150] node {$1/4$} (B)
    (N) edge [loop right, right,distance=1.5cm,out=30,in=-30] node {$1/2$} (N)
  ;
\end{tikzpicture}
\end{center}

\begin{enum}
\item  C'est une \chaine\ r\'eguli\`ere de matrice (sur 
$\cX=(\text{Beau},\text{Pluie},\text{Neige})$) 
\[
P = 
\begin{pmatrix}
1/4 & 3/4 & 0 \\
1/2 & 1/3 & 1/6 \\
0 & 1/2 & 1/2 
\end{pmatrix}\;.
\]

\item 	$\pi = (\frac13,\frac12,\frac16)$. 

\item 	$\expecin{\text{Neige}}{\tau_{\text{Neige}}} = 6$. 
\end{enum}

\subsection*{Exercice~\ref{exo_mf07}}

\begin{center}
\begin{tikzpicture}[->,>=stealth',shorten >=2pt,shorten <=2pt,auto,node
distance=4.0cm, thick,main node/.style={circle,scale=0.7,minimum size=1.0cm,
fill=blue!20,draw,font=\sffamily\Large}]

  \node[main node] (0) {0};
  \node[main node] (3) [right of=0] {3};
  \node[main node] (1) [right of=3] {1};
  \node[main node] (2) [right of=1] {2};

  \path[every node/.style={font=\sffamily\small}]
    (0) edge [bend left, above] node {$1$} (3)
    (3) edge [bend left, above] node {$1/3$} (1)
    (1) edge [bend left, above] node {$2/3$} (2)
    (2) edge [bend left, below] node {$2/3$} (1)
    (1) edge [bend left, below] node {$1/3$} (3)
    (3) edge [bend left, below] node {$2/3$} (0)
    (2) edge [loop right, right,distance=1.5cm,out=30,in=-30] node {$1/3$} (2)
  ;
\end{tikzpicture}
\end{center}

\begin{enum}
\item  
\[
P = 
\begin{pmatrix}
0 & 0 & 0 & 1 \\
0 & 0 & 2/3 & 1/3 \\
0 & 2/3 & 1/3 & 0 \\
2/3 & 1/3 & 0 & 0 
\end{pmatrix}\;.
\]

\item	La \chaine\ est r\'eguli\`ere. 

\item 	$\pi=(\frac{2}{11}, \frac{3}{11}, \frac{3}{11}, \frac{3}{11})$, donc 
$\pi(0) = \frac{2}{11}$. 

\item 	$\frac13\cdot\pi(0) = \frac{2}{33}$. 
\end{enum}

\subsection*{Exercice~\ref{exo_mf08}}

\[
p_{ij} = 
\begin{cases}
\frac{i}{N} & \text{si $j=i-1$\;,} \\
1-\frac{i}{N} & \text{si $j=i+1$\;,} \\
0 & \text{sinon\;,} \\
\end{cases} 
\quad \text{donc} \quad
p_{ij} = 
\begin{cases}
\frac{j+1}{N} & \text{si $i=j+1$\;,} \\
\frac{N-j+1}{N} & \text{si $i=j-1$\;,} \\
0 & \text{sinon\;.} \\
\end{cases}
\]
Soit 
\[
\pi_i = \frac{1}{2^N} \frac{N!}{i!(N-i)!}\;.
\]
Alors 
\begin{align*}
\sum_{i=0}^N \pi_i p_{ij} 
&= \frac{N!}{2^N} 
\biggbrak{\frac{1}{(j-1)!(N-j+1)!}\frac{N-j+1}{N} + 
\frac{1}{(j+1)!(N-j-1)!}\frac{j+1}{N}} \\
&= \frac{N!}{N\cdot 2^N} \frac{j+N-j}{j!(N-j)!}
= \pi_j\;.
\end{align*}

\subsection*{Exercice~\ref{exo_mf09}}

\begin{enum}
\item 	La \chaine\ est
\begin{itemiz}
\item 	absorbante si $p=0$ ou $q=0$;
\item 	irr\'eductible non r\'eguli\`ere si $p=q=1$;
\item 	r\'eguli\`ere dans les autres cas.
\end{itemiz}

\item 	Le noyau est engendr\'e par $\transpose{(p,-q)}$; l'image par
$\transpose{(1,1)}$.

\item 	On obtient par calcul explicite 
\[
Q = \frac{1-(p+q)}{p+q}
\begin{pmatrix}
p & -p \\ -q & q 
\end{pmatrix}
\]
puis $Q^2=[1-(p+q)]Q$ et donc $Q^n=[1-(p+q)]^{n-1}Q$ pour tout $n\geqs2$.

\item 	$P^2=(Q+\Pi)^2=Q^2+\Pi Q+Q\Pi+\Pi^2 = Q^2 + \Pi$, et donc 
\[
P^n = Q^n + \Pi 
= \frac{[1-(p+q)]^n}{p+q}
\begin{pmatrix}
p & -p \\ -q & q 
\end{pmatrix}
+ \frac{1}{p+q}
\begin{pmatrix}
q & p \\ q & p 
\end{pmatrix}\;.
\]
Par cons\'equent, si $0<p+q<2$ on a $\lim_{n\to\infty}P^n=\Pi$. 

Si $p=q=1$, alors $P^n=I$ si $n$ est pair 
et $P^n=P=\bigpar{\begin{smallmatrix}0&1\\1&0\end{smallmatrix}}$
si $n$ est impair. 

Si $p=q=0$, la matrice $Q$ n'est pas d\'efinie, mais $P=I$ donc $P^n=I$
pour tout $n$.
\end{enum}

\subsection*{Exercice~\ref{exo_mf10}}

Le temps de r\'ecurrence moyen vaut $60/2=30$. 

\subsection*{Exercice~\ref{exo_mf11}}

\begin{enum}
\item 	$420/3=140$.
\item 	$208/3$. 
\item 	$280/7=40$. (Attention, le fou ne se d\'eplace que sur les cases d'une
couleur!)
\end{enum}

\subsection*{Exercice~\ref{exo_mf12}}

\begin{enum}
\item 	La distribution stationnaire est uniforme~: 
$\pi_i=1/N \:\forall i$. 

\item 	La \chaine\ est r\'eversible si et seulement si $p=1/2$.  
\end{enum}

\subsection*{Exercice~\ref{exo_mf13}}

\begin{enum}
\addtocounter{enumi}{1}
\item 	$\pi_i = N_i/\sum_{j\in V}N_j$ o\`u $N_i$ est le nombre de voisins du
sommet $i$. 
\end{enum}

\section{Exercices du Chapitre \ref{chap_den}}

%\subsection*{Exercice~\ref{exo_md01}}

\subsection*{Exercice~\ref{exo_md02}}

\begin{enum}
\item 	La \chaine\ est irr\'eductible pour $0<p<1$.
\item 	$\probin{0}{\tau_0=n}=p^{n-1}(1-p)$ (c'est une loi g\'eom\'etrique). 
\item 	En sommant la s\'erie g\'eom\'etrique, on v\'erifie que 
$\sum_{n\geqs0}\probin{0}{\tau_0=n}=1$. 
\item 	On a $\expecin{0}{\tau_0}=\sum_{n\geqs0}np^{n-1}(1-p) = 1/(1-p) <
\infty$. 
\item 	Comme $\probin{0}{X_1=0}>0$ et $\probin{0}{X_2=0}>0$, l'\'etat $0$ est
ap\'eriodique. 
\item 	$\pi_0 = 1/\expecin{0}{\tau_0} = 1-p$. 
\item 	$\pi_i = (1-p)p^i$ pour tout $i\geqs0$ (c'est encore une loi
g\'eom\'etrique). 
\end{enum}

\subsection*{Exercice~\ref{exo_md03}}

\begin{enum}
\item 	La \chaine\ est irr\'eductible si 
pour tout $n\geqs0$, il existe $m\geqs n$ tel que $p_m>0$.
\item 	$\probin{0}{\tau_0=n}=p_{n-1}(1-p)$. 
\item 	$\probin{0}{\tau_0<\infty} = \sum_{n\geqs2}p_{n-1}=1$. 
\item 	On a $\expecin{0}{\tau_0}=\sum_{n\geqs2}np_{n-1} 
=\sum_{n\geqs1}(n+1)p_n =\expec{L}+1<\infty$. 
\item 	Il suffit que $\pgcd\setsuch{k}{p_k>0}=1$. 
\item 	$\pi_0 = 1/(\expec{L}+1)$. 
\item 	$\pi_1=\pi_0$ et 
$\pi_i = (1-p_1-p_2-\dots-p_{i-1})\pi_0$ pour tout $i\geqs2$.
\end{enum}

\subsection*{Exercice~\ref{exo_md04}}

\begin{enum}
\item 	La \chaine\ est irr\'eductible.
\item 	$\probin{1}{\tau_1=2n}=2^{-n}$ donc $1$ est r\'ecurrent. 
\item 	$\expecin{1}{\tau_1} = 4$ donc $1$ est r\'ecurrent positif. 
\item 	L'\'etat $1$ n'est pas ap\'eriodique, il est de p\'eriode $2$. 
\item 	$\pi_1 = 1/4$. 
\item 	$\pi_i = 2^{-(\abs{i}+1)}$ pour tout $i\neq0$.
\end{enum}

\subsection*{Exercice~\ref{exo_md05}}

\begin{enum}
\item 	$p(0)=0$ et $p(N)=1$. 
\addtocounter{enumi}{5}
\item 	$p(i) = i/N$ pour $i=0,1,\dots N$. 
\end{enum}

\subsection*{Exercice~\ref{exo_md06}}

\begin{enum}
\addtocounter{enumi}{1}
\item 	$f(0)=f(N)=0$. 
\addtocounter{enumi}{4}
\item 	$f(i) = i(N-i)$ pour $i=0,1,\dots N$. 
\end{enum}

\subsection*{Exercice~\ref{exo_md07}}

\begin{enum}
\item 	$h(0)=1$ puisque $\tau_0=1$. 
\item 	$h(i)=1$ pour tout $i$.
\addtocounter{enumi}{3}
\item 	$p\alpha^2-\alpha+1-p=0$ donc $\alpha\in\set{1,(1-p)/p}$.
\item 	Lorsque $p<1/2$, $\alpha=1$ est la seule solution admissible, 
elle donne $h(i)=1$ $\forall i$.

Lorsque $p>1/2$, les deux solutions $\alpha\in\set{1,(1-p)/p}$ fournissent une
probabilit\'e. En fait on sait montrer que $\alpha=(1-p)/p$ est la solution
correcte, donc $h(i)=\brak{(1-p)/p}^i$. 
\end{enum}

\subsection*{Exercice~\ref{exo_md08}}

\begin{enum}
\item 	La \chaine\ est irr\'eductible pour $0<p<1$.
\item 	La \chaine\ est ap\'eriodique.
\item 	$\alpha_n = (p/(1-p))^n \alpha_0$. 
\item 	Il faut que $\sum_{n\geqs0}\alpha_n < \infty$, donc $p<1/2$. 
Dans ce cas on obtient  
\[
\pi_n = \frac{\alpha_n}{\sum_{m\geqs0}\alpha_m}
= \frac{1-2p}{1-p} \biggpar{\frac{p}{1-p}}^n\;.
\]
\item 	La \chaine\ est r\'ecurrente positive pour $p<1/2$.
\item 	$\expecin{0}{\tau_0} = 1/\pi_0 = (1-p)/(1-2p)$. 
\item 	La position moyenne est  
\[
\expecin{\pi}{X} = \sum_{n\geqs0} n \pi_n = \frac{p}{1-2p}\;.
\]
\end{enum}

\subsection*{Exercice~\ref{exo_md09}}

\begin{enum}
\item 	$f(0,\lambda)=0$ et $f(N,\lambda)=1$. 
\addtocounter{enumi}{2}
\item 	$\cosh(c)=\e^\lambda$. 
\item 	$a=-b=2/\sinh(cN)$, et donc 
\[
f(i,\lambda) = \frac{\sinh(ci)}{\sinh(cN)}\;.
\]
\item 	On a 
\[
f(i,\lambda) = \frac{i}{N} - \frac{i(N^2-i^2)}{3N} \lambda 
+ \Order{\lambda^2}\;,
\]
donc $\probin{i}{X_\tau=N} = i/N$ et 
$\expecin{i}{\tau\indexfct{X_\tau=N}} = i(N^2-i^2)/(3N)$. 

\end{enum}

\subsection*{Exercice~\ref{exo_md10}}

\begin{enum}
\item 	La \chaine\ est irr\'eductible pour $0<p<1$.
\item 	La \chaine\ n'est pas ap\'eriodique (sa p\'eriode est $2$). 
\item 	
\[
\bigprobin{0}{X_{2n}=0} = p^n (1-p)^n \binom{2n}{n}
\]
\item 	
\[
\bigprobin{0}{X_{2n}=0} \simeq \frac{\bigpar{4p (1-p)}^n}{\sqrt{\pi n}}\;.
\]
La cha\^ine est transiente pour $p\neq1/2$, r\'ecurrente pour $p=1/2$. 
\addtocounter{enumi}{1}
\item 	$Y_n$ est transiente si $p>1/2$. 
\item 	$Y_n$ admet une distribution stationnaire et est donc r\'ecurrente
positive si et seulement si $p<1/2$. 
\item 	$Y_n$ est r\'ecurrente nulle si $p=1/2$. 
\end{enum}

\subsection*{Exercice~\ref{exo_md11}}

\begin{enum}
\item 	$f(0)=0$ et $f(N)=1$.
\addtocounter{enumi}{1}
\item 	$z_+=1$ et $z_-=\rho$.
\item 	$a=1/(1-\rho^N)=-b$, donc 
\[
\probin{i}{X_\tau=N} = \frac{1-\rho^i}{1-\rho^N}\;.
\]
\addtocounter{enumi}{1}
\item 	$\gamma=1/(1-2p)$.
\item 	$a=-N/((1-2p)(1-\rho^N))=-b$, donc
\[
\expecin{i}{\tau} = \frac{N}{1-2p}
\biggbrak{\frac iN - \frac{1-\rho^i}{1-\rho^N}}\;.
\]

\end{enum}

\section{Exercices du Chapitre \ref{chap_rappel}}

\subsection*{Exercice~\ref{exo_generatrice1}}

\begin{enum}
\item 	Bernoulli: $G_X(z)=1-q+qz$.
\item 	Binomiale: $G_X(z)=(1-q+qz)^n$.
\item 	Poisson: $G_X(z)=\e^{\lambda(z-1)}$.
\item 	G\'eom\'etrique: $G_X(z)=qz/[1-(1-q)z]$.
\end{enum}

\subsection*{Exercice~\ref{exo_generatrice2}}

\begin{enum}
\item 	Bernoulli: $\expec{X}=q$, $\variance(X)=q(1-q)$.
\item 	Binomiale: $\expec{X}=nq$, $\variance(X)=nq(1-q)$.
\item 	Poisson: $\expec{X}=\lambda$, $\variance(X)=\lambda$.
\item 	G\'eom\'etrique: $\expec{X}=1/q$, $\variance(X)=(1-q)/q^2$.
\end{enum}

%\subsection*{Exercice~\ref{exo_generatrice3}}

\subsection*{Exercice~\ref{exo_generatrice4}}

\begin{enum}
\item 	$\expec{z^{S_n}} = G_X(z)^n$.
\addtocounter{enumi}{1}
\item 	$G_S(z) = \e^{\lambda q(z-1)}$ donc $S$ suit une loi de Poisson 
de param\`etre $q\lambda$. 
\end{enum}

\subsection*{Exercice~\ref{exo_rappel05}}

\begin{enum}
\addtocounter{enumi}{1}
\item 	$\prob{Y\leqs y} = \prob{U\leqs \ph^{-1}(y)} = \ph^{-1}(y)$. 
\item 	$\ph(u) = -\log(1-u)/\lambda$. 
\end{enum}

\section{Exercices du Chapitre \ref{chap_ppp}}

\subsection*{Exercice~\ref{exo_poisson01}}

\begin{enum}
\item 	$1/9$. 
\item 	$5/9$. 
\end{enum}

\subsection*{Exercice~\ref{exo_poisson02}}

\begin{enum}
\item 	$\e^{-2} - \e^{-4}$. 
\item 	$\e^{-5}$.
\item 	$1/4$.
\item 	$3/4$.
\end{enum}

\subsection*{Exercice~\ref{exo_poisson03}}

\begin{enum}
\item 	$\e^{-4}$. 
\item 	Apr\`es $1$ minute. 
\item 	$(2/3)^3=8/27$. 
\item 	$8\e^{-4}/(1-\e^{-2})$.
\end{enum}

\subsection*{Exercice~\ref{exo_poisson031}}

\begin{enum}
\item 	$5$. 
\item 	$(37/2)\e^{-5}$. 
\item 	$1/16$ et $1/4$. 
\item 	$1/4$.
\end{enum}

\subsection*{Exercice~\ref{exo_poisson04}}

\begin{enum}
\item 	Une loi Gamma de param\`etres $(n,2)$. 
\item 	La loi de $(X_n - n/2)/\sqrt{n/4}$ est approximativement 
normale centr\'ee r\'eduite.
\item 	$T\cong 13.59$ minutes.
\end{enum}

\subsection*{Exercice~\ref{exo_poisson05}}

\begin{enum}
\item 	$\e^{-\lambda t}(\lambda t)^n/n!$. 
\item 	$\e^{-\lambda (t+s)}(\lambda t)^n/n!$. 
\item 	$\e^{-\lambda s}$. 
\item 	$1/\lambda$, donc \'egal au temps moyen entre passages. 
\end{enum}

\subsection*{Exercice~\ref{exo_poisson06}}

C'est un processus de Poisson d'intensit\'e $\lambda+\mu$. 

\subsection*{Exercice~\ref{exo_poisson07}}

\begin{enum}
\item 	$N_t$ suit une loi de Poisson de param\`etre $\lambda t$.
\item 	On a 
\[
\pcond{M_t=l}{N_t=k} = \frac{1}{2^k} \frac{k!}{l!(k-l)!}\;. 
\]
\item 	$M_t$ suit une loi de Poisson de param\`etre $\lambda t/2$.
\item 	L'intensit\'e de $Y_n$ est $\lambda/2$. 
\item 	$Y_n$ est un processus de Poisson d'intensit\'e $q\lambda$. 
\end{enum}

\section{Exercices du Chapitre \ref{chap_psm}}

\subsection*{Exercice~\ref{exo_saut01}}

\begin{enum}
\item 	\hfill

\begin{center}
\begin{tikzpicture}[->,>=stealth',shorten >=2pt,shorten <=2pt,auto,node
distance=4.0cm, thick,main node/.style={circle,scale=0.7,minimum size=1.0cm,
fill=violet!20,draw,font=\sffamily\Large}]

  \node[main node] (1) {$1$};
  \node[main node] (2) [right of=1] {$2$};
  \node[main node] (3) [below of=1] {$3$};
  \node[main node] (4) [below of=2] {$4$};

  \path[every node/.style={font=\sffamily\small}]
    (1) edge [bend left, above] node {$1$} (2)
    (2) edge [below, out=-170, in=-10] node {$2$} (1)
    (2) edge [below right] node {$1$} (3)
    (2) edge [right] node {$2$} (4)
    (3) edge [above, out=10, in=170] node {$1$} (4)
    (4) edge [bend left, below] node {$1$} (3)
    (1) edge [right, out=-80, in=80] node {$1$} (3)
    (3) edge [bend left, left] node {$2$} (1)
    ;
\end{tikzpicture}
\end{center}
\item 	$\pi = (\frac{5}{16},\frac{1}{16},\frac{4}{16},\frac{6}{16})$. 
\item 	Le processus est irr\'eductible.
\item 	Le processus n'est pas r\'eversible. 
\end{enum}

\subsection*{Exercice~\ref{exo_saut02}}

\begin{enum}
\item  \hfill
\begin{minipage}{7cm}
\begin{center}
$
L = 
\begin{pmatrix}
-4 & 2 & 2 \\
3 & -4 & 1 \\
5 & 0 & -5
\end{pmatrix}\;.
$
\end{center}
\end{minipage}
\begin{minipage}{7cm}
\begin{center}
\begin{tikzpicture}[->,>=stealth',shorten >=2pt,shorten <=2pt,auto,node
distance=4.0cm, thick,main node/.style={circle,scale=0.7,minimum size=1.0cm,
fill=violet!20,draw,font=\sffamily\Large}]

  \node[main node] (P) at (0,0) {P};
  \node[main node] (M) at (1.5,-2.25) {M};
  \node[main node] (B) at (-1.5,-2.25){B};

  \path[every node/.style={font=\sffamily\small}]
    (P) edge [bend left, above right] node {$2$} (M)
    (P) edge [bend right, above left] node {$2$} (B)
    (B) edge [below] node {$1$} (M)
    (B) edge [below right, out=45, in=-115] node {$3$} (P)
    (M) edge [below left, out=135, in=-65] node {$5$} (P)
    ;
\end{tikzpicture}
\end{center}
\end{minipage}

\item 	$\pi = (\frac12,\frac14,\frac14)$. 

\item 	$12$ voyages par an. 
\end{enum}

\subsection*{Exercice~\ref{exo_saut03}}

\begin{enum}
\item  \hfill
\begin{center}
\begin{tikzpicture}[->,>=stealth',shorten >=2pt,shorten <=2pt,auto,node
distance=3.0cm, thick,main node/.style={circle,scale=0.7,minimum size=1.0cm,
fill=violet!20,draw,font=\sffamily\Large}]

  \node[main node] (0) {$0$};
  \node[main node] (1) [right of=0] {$1$};
  \node[main node] (2) [right of=1] {$2$};
  \node[main node] (3) [right of=2] {$3$};

  \path[every node/.style={font=\sffamily\small}]
    (0) edge [bend left, above] node {$1$} (2)
    (1) edge [bend left, above] node {$1$} (3)
    (3) edge [bend left, below] node {$2$} (2)
    (2) edge [bend left, below] node {$2$} (1)
    (1) edge [bend left, below] node {$2$} (0)
   ;
\end{tikzpicture}
\end{center}

\[
L = 
\begin{pmatrix}
-1 & 0 & 1 & 0 \\
2 & -3 & 0 & 1 \\
0 & 2 & -2 & 0 \\
0 & 0 & 2 & -2 
\end{pmatrix}\;.
\]

\item 	$\pi = (0.4,0.2,0.3,0.1)$. 

\item 	$1.2$ ordinateurs par mois. 
\end{enum}

\subsection*{Exercice~\ref{exo_saut04}}

\begin{enum}
\item 	$\pi = \frac{1}{105}(100,1,4)$. 

\item 	Pendant une proportion $\frac{100}{105}=\frac{20}{21}$ du temps. 

\item 	En g\'en\'eral, 
\[
\pi(0) = \frac{1}{1+\sum_j \mu_j/\lambda_j}
\qquad\text{et}\qquad
\pi(i) = \frac{\mu_i/\lambda_i}{1+\sum_j \mu_j/\lambda_j}
\quad\text{pour $i=1,\dots,N$\;.}
\]

\end{enum}

\subsection*{Exercice~\ref{exo_saut05}}

Les fractions de temps sont \'egales aux valeurs de $\pi$, qui s'expriment
comme dans l'exercice pr\'ec\'edent, cas g\'en\'eral. 

\subsection*{Exercice~\ref{exo_saut06}}

\begin{enum}
\item 	$\pi = (\frac{1}{7},\frac{2}{7},\frac{4}{7})$. 
\item 	$50/7$ clients par heure. 
\end{enum}

\subsection*{Exercice~\ref{exo_saut08}}

\begin{enum}
\item 	Pour $N=1$, 
\[
L = 
\begin{pmatrix}
0 & 0 \\ -\mu & \mu 
\end{pmatrix}\;, 
\qquad
L^n = (-\mu)^{n-1}L\quad\forall n>1\;,
\qquad
P_t = 
\begin{pmatrix}
1 & 0 \\ 1-\e^{-\mu t} & \e^{-\mu t}
\end{pmatrix}\;.
\]

\item 	Pour $N=2$, 
\begin{align*}
P_t(2,2) &= \e^{-\mu t}\;, \\
P_t(2,1) &= \mu t \e^{-\mu t}\;, \\
P_t(2,0) &= 1 - (1+\mu t)\e^{-\mu t}\;.
\end{align*}

\item 	Pour $N$ quelconque, 
\begin{align*}
P_t(N,j) &= \frac{(\mu t)^{N-j}}{(N-j)!} \e^{-\mu t}\;, 
&& j=1,\dots, N\;, \\
P_t(N,0) &= 1 - 
\biggpar{1+\mu t + \dots + \frac{(\mu t)^{N-1}}{(N-1)!}}\e^{-\mu t}\;.
\end{align*}

\item 	On a 
\[
\expec{Y_t} = \sum_{j=0}^N (N-j) P_t(N,j) = 
\sum_{k=1}^{N-1} \frac{(\mu t)^{k}}{(k-1)!}\e^{-\mu t}
+ N P_t(N,0)\;,
\]
ce qui implique 
\[
\lim_{N\to\infty} \expec{Y_t} = \mu t\;.
\]

\end{enum}

\subsection*{Exercice~\ref{exo_saut07}}

\[
P_t = 
\begin{pmatrix}
1 & 0 & 0 & \dots & \dots \\
1-\e^{-\mu t} & \e^{-\mu t} & 0 & \dots  & \dots\\
1-(1+\mu t)\e^{-\mu t} & \mu t \e^{-\mu t} & \e^{-\mu t} & 0 & \dots \\
\vdots &   & \ddots  & \ddots  & \ddots
\end{pmatrix}\;.
\]

\subsection*{Exercice~\ref{exo_saut09}}

\begin{enum}
\item 	
\[
L = 
\begin{pmatrix}
-\lambda & \lambda & 0 \\
0 & -\lambda & \lambda \\
\lambda & 0 & -\lambda
\end{pmatrix}\;.
\]
\item 	$\pi = (\frac13,\frac13,\frac13)$. 
\item 
\[
R = 
\begin{pmatrix}
0 & 1 & 0 \\ 0 & 0 & 1 \\ 1 & 0 & 0
\end{pmatrix}\;, 
\qquad
R^2 = 
\begin{pmatrix}
0 & 0 & 1 \\ 1 & 0 & 0 \\ 0 & 1 & 0
\end{pmatrix}\;, 
\qquad
R^3 = I\;,
\]
puis $R^n = R^{n\pmod 3}$. 

\item 
\[
\e^{\lambda t R} = 
\begin{pmatrix}
f(\lambda t) & f''(\lambda t) & f'(\lambda t) \\
f'(\lambda t) & f(\lambda t) & f''(\lambda t) \\
f''(\lambda t) & f'(\lambda t) & f(\lambda t) 
\end{pmatrix}
\]
et $P_t = \e^{-\lambda t}\e^{\lambda t R}$. 

\end{enum}

\subsection*{Exercice~\ref{exo_saut10}}

\begin{enum}
\item 	
\[
L = 
\begin{pmatrix}
-N\lambda & \lambda & \lambda & \dots & \lambda \\
\mu       & -\mu    & 0       & \dots & 0 \\
\mu       & 0       & -\mu    & \dots & 0 \\
\vdots    & \vdots  &         & \ddots &  \\
\mu       & 0       & 0       & \dots & -\mu 
\end{pmatrix}\;.
\]

\item 	
\[
 \pi = \frac{1}{\mu+N\lambda} (\mu,\lambda,\dots,\lambda)\;.
\]

\item 	Le syst\`eme est r\'eversible.

\item 	
\[
P_t = \frac{1}{\lambda+\mu}
\begin{pmatrix}
\mu + \lambda\e^{-(\lambda+\mu)t} & \lambda - \lambda\e^{-(\lambda+\mu)t} \\
\mu - \mu\e^{-(\lambda+\mu)t} & \lambda + \mu\e^{-(\lambda+\mu)t} 
\end{pmatrix}\;.
\]

\item 	Pour $j=0$,
\begin{align*}
P_t(0,0) &= \frac{1+N\e^{-(N+1)\mu t}}{N+1}\;, \\
Q_t(0) &= \frac{1-\e^{-(N+1)\mu t}}{N+1}\;,
\end{align*}
et pour $j\neq0$
\begin{align*}
P_t(0,j) &= \frac{1-\e^{-(N+1)\mu t}}{N+1}\;, \\
Q_t(j) &= \frac{1+\frac1N\e^{-(N+1)\mu t}}{N+1}\;.
\end{align*}

\item 	Pour tout $i\neq 0$, $P_t(i,0)=Q_t(0)$ et 
\[
P_t(i,j) = \Bigpar{\delta_{ij}-\frac1N} \e^{-\mu t} + Q_t(j)
\]
pour tout $j\neq0$. 

\item 	$\displaystyle\lim_{t\to\infty} P_t(i,j) = \frac{1}{N+1}$ pour tout
$i,j$.
\end{enum}

\section{Exercices du Chapitre \ref{chap_tfa}}

\subsection*{Exercice~\ref{exo_fa01}}

\begin{enum}
\item 	$\pi(n) = (1/3)(2/3)^n$.
\item 	$4$ minutes et $6$ minutes.
\item 	$2/3$ et $(2/3)\e^{-1/3}\cong 0.48$. 
\item 	$\pi=\frac{1}{19}(9,6,4)$. La probabilit\'e de repartir est $4/19$. 
\item 	$28/19\cong1.47$ minutes et $66/19\cong3.47$ minutes. 
\end{enum}

\subsection*{Exercice~\ref{exo_fa02}}

\begin{enum}
\item 	$\pi(0) = 64/369$ et $\pi(n)=(5/4)^n\pi(0)$.
\item 	$305/369 \cong 0.827$. 
\item 	$\pi(0) \cdot 655/256 \cong 0.444$ heures $=26.6$ minutes. 
\item 	$1220/369 \cong 3.31$ clients par heure.
\item 	$4\cdot1685/1973 \cong 3.41$ clients par heure.
\end{enum}

\subsection*{Exercice~\ref{exo_fa03}}

\begin{enum}
\item 	$47/24 \cong 1.958$ op\'erateurs occup\'es.
\item 	$11/24 \cong 0.458$.
\end{enum}

\subsection*{Exercice~\ref{exo_fa04}}

\begin{enum}
\item 	$\pi(3) = 27/65 \cong 0.415$.
\item 	$387/13 \cong 29.77$ minutes. 
\item 	$228/65 \cong 3.5$ clients par heure.
\end{enum}

\subsection*{Exercice~\ref{exo_fa041}}

\begin{enum}
\item 	$\pi(0) = 625/1364\,(1,2/5,8/25,32/125,128/625)$.
\item 	$128/1363 \cong 9.39\%$. 
\item 	$35.8\%$, $18.34\%$ et $45.85\%$. 
\item 	$(904/1363)\cdot 15$ minutes $\cong 9.95$ minutes.
\item 	$613/1363 \cong 44.97\%$. 
\end{enum}

\subsection*{Exercice~\ref{exo_fa042}}

\begin{enum}
\item
\begin{enum}
\item 	$\pi_n = \frac13\bigpar{\frac23}^n$. 
\item 	$\expecin{\pi}{W} = \frac13$ heure $=20$ minutes.
\item 	$4$ clients par heure.
\end{enum}

\item
\begin{enum}
\item 	$\pi = \frac{1}{19}(9,6,4)$.
\item 	$\expecin{\pi}{W} = \frac{140}{19}\cong 7.37$ minutes.
\item 	$\frac{60}{19} \cong 3.16$ clients par heure.
\end{enum}
\end{enum}

\subsection*{Exercice~\ref{exo_fa05}}

La probabilit\'e est deux fois plus grande pour la \chaine\ M/M/1.

\subsection*{Exercice~\ref{exo_fa06}}

\begin{enum}
\item 
\[
\pi(n) = \frac{\dfrac{1}{n!} \biggpar{\dfrac{\lambda}{\mu}}^n}
{\displaystyle\sum_{i=1}^s \dfrac{1}{i!} \biggpar{\dfrac{\lambda}{\mu}}^i}
\;, \qquad
n=0,\dots, s\;.
\]
\item 	Un appel est rejet\'e avec probabilit\'e $\pi(s)$. 
\end{enum}

\subsection*{Exercice~\ref{exo_fa07}}

\begin{enum}
\item 	$\pi= \bigpar{\frac13,\frac13,\frac16,\frac1{12},\frac{1}{24}\dots}$.
\item 	$\probin{\pi}{W=0} = 2/3$.
\item 	$\expecin{\pi}{W} = 1/(3\lambda)$.
\item 	$\expecin{\pi}{S} = 1$. 
\addtocounter{enumi}{1}
\item 	$\pi(n,m) = 2^{-(n+m+2)}$, $n$ et $m$ \'etant le nombre de clients 
dans chaque file.
\item 	$\probin{\pi}{W=0} = 1/2$.
\item 	$\expecin{\pi}{W} = 2/\lambda$.
\item 	$\expecin{\pi}{S} = 5/4$. 
\item 	Du point de vue du client, le premier syst\`eme est bien plus
avantageux: le temps d'attente moyen est plus court, et la probabilit\'e
d'\^etre servi tout de suite est plus grande. Le second syst\`eme donne un taux
d'occupation l\'eg\`erement plus grand. 
\end{enum}

\newpage
\chapter*{Bibliographie et sources d'inspiration}

\begin{itemize}
\item	I.\ Adan, J.\ Resing, {\it Queueing Theory}, notes de cours, 
Eindhoven University of Technology (2002)
\item 	E.\ Bolthausen, {\it Einf\"uhrung in die Stochastik}, notes de
cours, Universit\'e de Zurich (2007)
\item	P.\ Bougerol, {\it Processus de Sauts et Files d'Attente}, notes de
cours, Universit\'e Pierre et Marie Curie (2002)
\item	R.\ Durrett, {\it Essentials of Stochastic Processes}, Springer, 1999.
\item	C.\,M.\ Grinstead, J.\,L.\ Snell, {\it Introduction to
Probability}, web book, 

{\tt http://math.dartmouth.edu/$\sim$doyle/docs/prob/prob.pdf}
\item	J.\ Lacroix, {\it \Chaine s de Markov et Processus de Poisson},
notes de cours, Universit\'e Pierre et Marie Curie (2002)
\end{itemize}

\end{document}